\documentclass[11pt]{amsart}
\usepackage{graphicx} 
\usepackage{subcaption}
\usepackage[utf8]{inputenc}
\usepackage[margin=1in]{geometry}
\usepackage{amsfonts,amssymb,amsmath,amsthm,tikz,comment,mathtools,setspace,float,stmaryrd,datetime,bbm,adjustbox}
\usepackage[toc]{appendix}
\usepackage{enumitem}
\usepackage{xcolor}

\usepackage[hidelinks]{hyperref}
\usepackage{scalerel} 
\numberwithin{equation}{section}
\usepackage{chngcntr}
\counterwithin{figure}{section}

\allowdisplaybreaks

\newcommand{\N}{\mathbb{N}}
\newcommand{\R}{\mathbb{R}}
\newcommand{\Q}{\mathbb{Q}}
\newcommand{\Z}{\mathbb{Z}}
\newcommand{\F}{\mathcal{F}}

\newcommand{\I}{\mathcal{I}}
\newcommand{\ve}{\varepsilon}
\newcommand{\Ll}{\mathcal{L}}
\newcommand{\Pp}{\mathbb P}
\newcommand{\f}{\frac}

\newcommand{\mbf}{\mathbf}

\newcommand{\wt}{\widetilde}

\newcommand{\Rup}{\R_{{\scaleobj{0.7}{\uparrow}}}^4}

\newcommand{\W}{W}
\newcommand{\NU}{\operatorname{NU}}
\newcommand{\DLBusedc}{\Xi}
\newcommand{\dir}{\xi}
\newcommand{\Split}{\mathfrak S}
\newcommand{\Branch}{\mathfrak B}
\newcommand{\h}{\mathfrak h}

\newcommand{\graph}{\mathcal G}

\newcommand{\UC}{\operatorname{UC}}
\newcommand{\fxi}{\mathcal{F}^\xi}

\newcommand{\sig}{{\scaleobj{0.8}{\boxempty}}} 
\newcommand{\sigg}{{\scaleobj{0.9}{\boxempty}}} 

\newcommand{\kpzs}{\mathfrak{h}} 

\newcommand{\intc}{\tau}

\newcommand{\ltriv}{\nu_{-\infty}}
\newcommand{\rtriv}{\nu_{+\infty}}

\newcommand{\be}{\begin{equation}}
\newcommand{\ee}{\end{equation}}

\DeclareMathOperator*{\argmax}{arg\,max}   

\def\Intc{\mathcal{T}}
\def\Intr{\mathfrak{I}}
\def\a{\mathfrak{a}}
\def\b{\mathfrak{b}}
\newtheorem{theorem}{Theorem}[section]
\newtheorem{proposition}[theorem]{Proposition}
\newtheorem{corollary}[theorem]{Corollary}
\newtheorem{lemma}[theorem]{Lemma}

\theoremstyle{definition}
\newtheorem{definition}[theorem]{Definition}

\theoremstyle{remark}
\newtheorem{remark}[theorem]{Remark}

\title{Exceptional force, uncountably many solutions in the KPZ fixed point }

\author{Sudeshna Bhattacharjee, Ofer Busani, and Evan Sorensen}
\address{Sudeshna Bhattacharjee\\Department of Mathematics\\Indian Institute of Science\\CV Raman Road\\Bengaluru, 560012\\Karnataka, India}
\email{sudeshnab@iisc.ac.in}
\urladdr{https://sites.google.com/view/sudeshnabhattacharjee/home}
\address{Ofer Busani \\ University of Edinburgh\\ School of Mathematics \\ 5321, James Clerk Maxwell Building\\
Peter Guthrie Tait Road\\
Edinburgh, EH9 3FD}
\email{obusani@ed.ac.uk}
\urladdr{https://oferbusani.github.io/}
\address{Evan Sorensen \\ Columbia University \\ Department of Mathematics \\ Room 624, MC 4432, 2990 Broadway, New York, NY 10027, USA}
\email{evan.sorensen@columbia.edu}
\urladdr{https://sites.google.com/view/evan-sorensen}
\begin{document}

\begin{abstract} We give a complete characterization of all eternal solutions $b(x,t)$ of the KPZ fixed point satisfying the asymptotic slope condition $\lim_{|x| \to \infty} \frac{b(x,0)}{x} = 2\xi$. For fixed $\xi$, there is exactly one eternal solution with probability one. However, in the second and third authors' work with Sepp\"al\"ainen, it was shown that there exists a random, countably infinite set of slopes, for which there exist at least two eternal solutions. These correspond to two non-coalescing families of infinite geodesics in the same direction for the directed landscape. We denote the two eternal solutions as $b^{\xi-}$ and $b^{\xi +}$. In the present paper, we show that, for the exceptional slopes, there are in fact uncountably many eternal solutions. To give the characterization, we show that these eternal solutions are in bijection with a certain set of bi-infinite competition interfaces. Each bi-infinite interface separates the plane into two connected components--a left component and a right component. A general eternal solution with slope $\xi$ is equal to $b^{\xi-}$ on the  left component and equal to $b^{\xi +}$ on the right component. For these bi-infinite interfaces in the exceptional directions, we uncover new geometric phenomena that is not present for directed landscape geodesics. Additionally, we show that this set of eternal solutions appears as the Busemann limits $\Ll(\mbf v_n;\mbf p) - \Ll(\mbf v_n;\mbf q)$ for sequences $\mbf v_n$ going to $-\infty$ in direction $\dir$. 
\end{abstract}

\maketitle

\tableofcontents

\section{Introduction} 
\subsection{Invariant measures and eternal solutions of Hamilton-Jacobi equations} One of the most fundamental and natural questions for a given Markov process is the identification of its invariant measures. We obtain a very interesting class of examples when we consider Hamilton-Jacobi equations with random forcing. For the $1 + 1$-dimensional Burger's equation with random periodic forcing, in both the viscous and inviscid cases, the existence of such measures was obtained in \cite{Sinai1991,Sinai1996,Iturriaga-Khanin-2003} and \cite{EKMS-1997,EKMS-2000,Iturriaga-Khanin-2003,GIKP-2005,TR22}, respectively. In non-compact settings, this has been achieved in \cite{Bakhtin-Cator-Konstantin-2014,bakhtin2019thermodynamic,dunlap2021stationary,Janj-Rass-Sepp-22}. A fundamental idea common to each of these works is known as the `one force--one solution (1F1S) principle: if we start with an initial condition with a fixed conserved property at a time in the distant past, the system will converge almost surely to the unique eternal space-time solution associated with the same preserved quantity.

\subsection{1F1S and the KPZ fixed point} 
In the present paper, we study a Markov process known as the \textbf{KPZ fixed point}. This process does not arise directly from a Hamilton-Jacobi equation, but can be solved via a Lax-Oleinik-type formula. That is, if we start from an initial condition $f$ at time $s$, the solution at time $t > s$ and location $x$ is given as 
\be \label{eqref:kpzs_def}
\kpzs_t(x;f,s)=\sup_{z\in\R}\{f(z)+\Ll(z, s; x,t)\}.
\ee
The object $\Ll$ is a random function known as the \textbf{directed landscape (DL)}, and is discussed further in Section \ref{sec:DL}. In this setting, a eternal solution to the KPZ fixed point is a function $b:\R^2 \to \R$ so that for any $s < t$ and $x \in \R$,
\be \label{b_global}
b(x,t) = \sup_{z\in\R} \{b(z,s) + \Ll(z,s;x,t)\}.
\ee
In words, $b$ is an eternal solution because we can start from any time in the past and evolve the equation to any time in the future. The 1F1S principle for the KPZ fixed point is as follows: Given an upper-semi-continuous initial condition $f:\R\to \R$ satisfying the asymptotic slope condition 
\be \label{eq:slope_cond}
\lim_{|x| \to \infty} \f{f(x)}{x} = 2\dir,
\ee
for some $\dir \in \R$, the recentered solution
\be \label{hxminh0}
(x,t) \mapsto \kpzs_t(x;f,s) - \kpzs_t(0;f,s)
\ee
converges, as $s \to -\infty$, to an eternal solution $b^\dir$ satisfying
\be \label{eq:global_soln_slope}
\lim_{|x| \to \infty} \f{b^\dir(x,t)}{x} = 2\dir,\quad\text{for all}\quad t \in \R.
\ee
This convergence was proved in \cite[Theorem 5.1(viii)]{Busa-Sepp-Sore-22a}. This slope condition is a manifestation of the fact that asymptotic slope is a conserved quantity for the KPZ fixed point \cite[Lemma B.9]{Busa-Sepp-Sore-22a}. In showing convergence to the eternal solution, the slope assumption \eqref{eq:slope_cond} can be relaxed slightly. See Equation (2.4) in \cite{Busa-Sepp-Sore-22a} for the more general condition. When an eternal solution satisfies the slope condition \eqref{eq:global_soln_slope} for $t = 0$, we call it a \textbf{$\dir$-eternal solution.} 

We recenter the function $\kpzs$ in \eqref{hxminh0} because the KPZ fixed point is a growth model; without recentering, the solution will explode as $t \to \infty$. The spatial derivative of $\kpzs_t$ can be thought of as the solution to an analogue of the stochastic Burgers equation. More precisely, the KPZ fixed point has been shown to be the scaling limit of the KPZ equation \cite{KPZ_equation_convergence,heat_and_landscape,Wu-23}, and the spatial derivative of the KPZ equation formally solves the viscous stochastic Burgers' equation. More generally, the KPZ fixed point is conjectured to arise as the universal scaling limit of random growth models in the KPZ universality class. 

The eternal solution $b^\dir$ we obtain here is a random object, and for all $t \in \R$, the function $x \mapsto b^\dir(x,t) - b^\dir(0,t)$ is a two-sided Brownian motion with drift $2\dir$. Since the function $b^\dir$ evolves in time as the KPZ fixed point, we see that Brownian motion with drift is invariant (up to height shifts) for the KPZ fixed point. This invariance was previously known from the first construction of the KPZ fixed point \cite{KPZfixed}. One of the contributions of \cite{Busa-Sepp-Sore-22a} is that the 1F1S principle gives uniqueness of the invariant measure under the slope condition.

In \cite{Busa-Sepp-Sore-22a}, the second and third author, along with Sepp\"al\"ainen, constructed a family of eternal solutions $\{b^{\dir \sig}: \dir \in \R, \sigg \in \{-,+\}\}$. The process $\dir \mapsto b^{\dir -}$ is left-continuous and the process $\dir \mapsto b^{\dir +}$ is right-continuous (with respect to the topology of uniform convergence on compact sets). We let $\DLBusedc$ be the set of $\dir \in \R$ so that $b^{\dir -} \neq b^{\dir +}$. With probability one, $\DLBusedc$ is countably infinite and dense in $\R$, and for any fixed $\dir \in \R$, $\Pp(\dir \in \DLBusedc) = 0$ \cite[Theorem 5.5]{Busa-Sepp-Sore-22a}. Thus, for $\dir \in \DLBusedc$, there are at least two eternal solutions with the same asymptotic slope.  We note here, however, that $b^{\dir -}$ and $b^{\dir +}$ do \textit{not} correspond to distinct invariant measures for the KPZ fixed point, as $b^{\dir-} = b^{\dir +}$ with probability one. Specifically, the problem of existence of these two eternal solutions is of a different nature than the problem of finding invariant measures. In particular, on an event of full probability, given a realization of the directed landscape, we seek to understand the set of continuous functions $b:\R^2 \to \R$ satisfying the conditions \eqref{b_global} and \eqref{eq:global_soln_slope}. This is a related, but notably distinct problem than that of classifying the invariant measures. For work on classifying invariant measures in this context, see \cite{Dunlap-Sorensen-2024}, which characterizes the invariant measures for the KPZ equation by resolving a case left open in the previous work \cite{Janj-Rass-Sepp-22}.

In the KPZ equation, the works \cite{Janj-Rass-Sepp-22,GRASS-23} proved an analogous statement; that is, there exists a countably infinite dense set of directions for which there are at least two eternal solutions with the same asymptotic slope. Analogous  results have also been shown for exponential last-passage percolation \cite{Janjigian-Rassoul-Seppalainen-19}, Brownian last-passage percolation \cite{Seppalainen-Sorensen-21b}, and the inverse-gamma polymer \cite{Bates-Fan-Seppalainen-2025}. In the context of the KPZ equation, it was presented as an open problem in \cite{Janj-Rass-Sepp-22} whether, for $\dir \in \DLBusedc$, there exist more than two eternal solutions satisfying the condition \eqref{eq:global_soln_slope}. This problem remains open for all of these models, except the KPZ fixed point, which is handled in the present paper. We discuss our main result in the following section.

\subsection{Summary of the main result} Our first main result obtains, on an event of probability one, a complete characterization of $\dir$-eternal solutions, simultaneously for all $\dir \in \R$. In particular, our result states that for $\dir \in \DLBusedc$, there are in fact, uncountably many eternal solutions to the KPZ fixed point satisfying the condition \eqref{eq:global_soln_slope}. We emphasize again that this result is not a result about the invariant measures of the KPZ fixed point. Instead, it characterizes, on an event of probability one, the continuous functions $b:\R^2 \to \R$ satisfying the eternal solution evolution \eqref{b_global} and the slope condition \eqref{eq:global_soln_slope}.

Loosely speaking, the characterization states that each eternal solution of the KPZ fixed point is a patching together of the two eternal solutions $b^{\dir -}$ and $b^{\dir +}$. In particular, we consider continuous functions $\intc:\R\rightarrow\R$ whose graph $(\intc(t),t)$ forms a bi-infinite path in the plane and therefore separates $\R^2$ into two infinite connected sets. Let us denote by $H_-$ and $H_+$, respectively, the sets that are to the left and right of the bi-infinite path. We now define the function
\begin{equation}\label{eq62}
b=b^{\dir-}|_{H_-}+b^{\dir+}|_{H_+},
\end{equation}
that is, our new solution is the sum of the  restriction of $b^{\dir-}$ on $H_-$ and the restriction of $b^{\dir+}$ on $H_+$. 

Such a function is not a $\dir$-eternal solution for any bi-infinite path, but will be shown to be such when we let $\intc$ be a bi-infinite interface for a mixed Busemann initial condition. These interfaces are discussed in more detail in Section \ref{sec:eternal}. Key to the proof is showing that such bi-infinite interfaces do, in fact, exist and that there are uncountably many of them. We work with a particular type of interfaces, which we term mixed $\dir$-Busemann bi-infinite interfaces, defined precisely in Section \ref{sec:eternal} below.

Loosely stated, our main result, stated precisely in Theorem \ref{thm:mr}, is the following. 
\begin{align*}
    &\text{For all $\dir \in \DLBusedc$,  there exists a one-to-one correspondence between the set of $\dir$-eternal}\\
    &\text{   solutions of the KPZ fixed-point and the set  of mixed $\dir$-Busemann bi-infinite interfaces. } 
\end{align*}
The existence of these bi-infinite interfaces is a new and quite surprising phenomenon, given its connection to the question of bi-infinite geodesics. It is widely conjectured that, with probability one, bi-infinite geodesics do not exist in models of the KPZ universality class \cite{Newman-97,Damron_Hanson2016,wehr97}. This conjecture is motivated by the physical explanation that the nonexistence of bi-infinite geodesics for first-passage percolation (FPP) is equivalent to the nonexistence of more than two ground states in disordered ferromagnetic models.
The conjecture has been resolved for some exactly solvable last-passage percolation (LPP) and polymer models \cite{SlyNonexistenceOB,bala-busa-sepp-20,Groathouse-Janjigian-Rassoul-21,Busani-Seppalainen-2020,Sweeney-Rassoul-2024} and for the directed landscape in the work \cite{Bha24}. In FPP, it is known that there are no bi-infinite geodesics in directions of differentiability of the shape function \cite{licea1996,Damron_Hanson2016,Ahlberg_Hoffman}. The recent work \cite{Bhatia-2025} studied a dynamical version of exponential LPP and obtained a $O(\f{1}{\log n})$ lower bound for the probability that there exists an exceptional time where there exists a geodesic of length $n$ passing through the origin at its midpoint. If one could improve this to an $O(1)$ bound, it would prove the existence of exceptional times with the existence of a bi-infinite geodesic.  

The connection to competition interfaces comes initially from the works \cite{Ferrari-Martin-Pimentel-2009,pimentel2016} in exponential LPP, where it is shown that the tree of semi-infinite geodesics in a given direction is equal in distribution to a dual tree consisting of competition interfaces. Before the works \cite{SlyNonexistenceOB,bala-busa-sepp-20,Sweeney-Rassoul-2024} proved the nonexistence of any bi-infinite geodesics in exponential and Brownian last-passage percolation, the works \cite{Timo_Coalescence,Seppalainen-Sorensen-21} used this duality to rule out the existence of bi-infinite geodesics in fixed directions. The work of \cite{Bha24} constructed a limiting duality between competition interfaces and infinite geodesics in the directed landscape. 

In light of these results, we point out here that the only way that the existence of these bi-infinite competition interfaces can be possible is that they occur in exceptional directions, and therefore do not exhibit a probabilistic duality with a tree of geodesics in the directed landscape. In these exceptional directions, we see additional unique behavior of these interfaces that is not exhibited by geodesics. In particular, Theorem  \ref{thm:int1}\ref{int1:it4} below states that there are points for which these interfaces form interior bubbles. That is, they may travel together for some time, split apart, then come back together and stick together for a nonzero amount of time (potentially to split again). This is in contrast to the results for geodesics, where, with probability one, there are no such geodesic bubbles in the directed landscape \cite{Bha24,Dauvergne-23}. Although the collection of interfaces in this case is not equal in law to a tree of geodesics, there is a geometric sense in which duality holds. For example, the points along these interfaces correspond to splitting points of infinite geodesics (Theorem \ref{thm:BiInt} and Theorem \ref{thm:int1}\ref{int1:it2}). In particular, Theorem \ref{thm:int1}\ref{int1:it2} generalizes a result from \cite{Bhatia-2025} to all directions simultaneously (see Remark \ref{rmk:branchNU}). Proposition \ref{p: interfaces do not meet} (See also Figure \ref{fig:interfaces can split}) describes a duality between splitting of interfaces and splitting of geodesics. 

The connection between points along interfaces and splitting points of semi-infinite geodesics bears some resemblance to the third author's work with Sepp\"al\"ainen \cite{Seppalainen-Sorensen-21b} for Brownian LPP. In particular, Theorem 2.10 in \cite{Seppalainen-Sorensen-21b} states that the set of splitting points of semi-infinite geodesics is the same as the set of points for which there is a nontrivial competition interface. The structure of competition interfaces in Brownian LPP and its relation to points of instability was studied further in \cite{Sweeney-Rassoul-2024}. However, we remark that the competition interfaces considered in \cite{Seppalainen-Sorensen-21b,Sweeney-Rassoul-2024} are different from those studied in the present paper; in \cite{Seppalainen-Sorensen-21b,Sweeney-Rassoul-2024}, interfaces from a Dirac delta initial condition were considered, while we consider interfaces from a continuous initial condition.   

\subsection{Eternal solutions and interfaces} \label{sec:eternal}
The eternal solutions $b^{\dir \sig}$ constructed in \cite{Busa-Sepp-Sore-22a} are known as Busemann functions. These are defined as follows. For a fixed direction $\dir \in \R$, the following limit was shown to exist almost surely in \cite[Proposition 3.21]{Rahman-Virag-21}:
  \begin{equation} \label{eq:Buselim1}
    W^{\dir}(z,s;x,s)=\lim_{t\rightarrow -\infty}\Ll(-t\dir,t;x,s)-\Ll(-t\dir,t;z,s).
\end{equation}
We then set
 \begin{equation}
     b^\dir(x,t) :=W^{\dir}(0,0;x,t).
 \end{equation}
 In \cite{Busa-Sepp-Sore-22a}, to construct the Busemann function in all directions, the procedure was to construct the Busemann function $W^\dir$ for all $\dir$ in a countable dense set, then define $W^{\dir \sig}$ for all $\dir \in \R$ and $\sigg \in \{-,+\}$ by taking left an right limits in the direction parameter $\dir$.  The fact that these functions are indeed eternal solutions with the specified slopes was shown first for fixed directions $\dir$ in \cite{Rahman-Virag-21}, then simultaneously for all directions in \cite{Busa-Sepp-Sore-22a}. For $\DLBusedc$, the limit in \eqref{eq:Buselim1} need not exist, but the sequence is bounded, and we show in Section \ref{sec:limits} that the limit points are exactly the differences of $\dir$-eternal solutions.

The paths $\intc$ discussed in the previous section are competition interfaces for a specialized initial condition. Competition interfaces in the directed landscape were first studied in \cite{Rahman-Virag-21}, which we review now. An \textbf{initial condition} $f: \R \rightarrow \R \cup \{ -\infty\}$  is an upper semi continuous function, such that $f(x)$ that is finite at some point and satisfies
\begin{equation}\label{eq:IC}
    f(x) \leq c(1+|x|)
\end{equation}
for all $x$ and some constant $c$.  Consistent with the notations in \cite{Rahman-Virag-21}, for $t>s$ and a continuous function $f$, we define
  \begin{equation}\label{eq:d}
     d_{(x_0,s)}(f;x,t)=\sup_{z \geq x_0} \left\{f(z)+\Ll(z,s; x,t) \right\}-\sup_{z \leq x_0} \left\{f(z)+\Ll(z,s; x,t)\right\}. 
  \end{equation}
It follows immediately from the definition that $x \mapsto d_{(x_0,s)}(f;x,t)$ is nondecreasing.  For an initial condition $f$ and $t > s$, let us define 
 \be \label{eq:taupm_def}
\begin{aligned}
& \intc^-_{f;(x_0,s)}(t)=\inf \{ x \in \R : d_{(x_0,s)}(f;x,t) \geq 0\},\\
& \intc^+_{f;(x_0,s)}(t)=\sup \{x \in \R: d_{(x_0,s)}(f;x,t) \leq 0 \},
\end{aligned}
\ee
while for $t = s$, we define
\be \label{intcat0}
 \intc^-_{f;(x_0,s)}(s) = \intc^+_{f;(x_0,s)}(s)=x_0
\ee
so that the interface starts at the point $(x_0,s)$. We  call $\intc^-_{f;(x_0,s)}$ (resp.\ $\intc^+_{f;(x_0,s)}$) the \textbf{\textit{leftmost} (resp.\ \textit{rightmost}) competition interface} from initial condition $f$. The monotonicity of $d_{(x_0,s)}(f;x,t)$ implies that  $\intc^-_{f;(x_0,s)}(t) \le \intc^+_{f;(x_0,s)}(t)$ for all $t \ge s$.

In the present paper, we work with what we call mixed $\dir$-Busemann interfaces. To define these, we choose a direction $\dir \in \R$ and an initial point $(x_0,s)$, and we take our initial condition to be
\begin{equation}
\label{eq: initial condition}
    f^{\dir}_{(x_0,s)}(z):=
    \begin{cases}
         W^{\dir+}(x_0,s;z,s) & z\geq x_0,\\
         W^{\dir-}(x_0,s;z,s) & z \le x_0.
     \end{cases},\qquad \text{ for } (x_0,s)\in\R^2, \dir\in\R.
\end{equation}
When $\dir \notin \DLBusedc$, we have $\W^{\dir -} = W^{\dir +}$, so $f_{(x,s)}^\dir(z) = W^\dir(x_0,s;z,s)$ in this case. 

We now define the \textbf{leftmost and rightmost mixed $\dir$-Busemann interfaces} as 
\be \label{eq:Buse_interface}
\tau_{(x_0,s)}^{\dir,\sig} := \intc^\sig_{f^{\dir}_{(x_0,s)};(x_0,s)},\quad\text{for}\quad \sigg \in \{-,+\}.
\ee
In words, started at the space-time point $(x_0,s)\in\R^2$, the interface `sees' the initial condition $W^{\dir-}$ to its left and $W^{\dir+}$ to its right. We also make the abbreviations
\be \label{kpzs_abbr}
\kpzs^\dir_{(x_0,s),t}(x) := \kpzs_{t}(x;f^\dir_{(x_0,s)},s) ,\quad\text{and}\quad d^\dir_{(x_0,s)}(x,t):= d_{(x_0,s)}(f^\dir_{(x_0,s)};x,t),
\ee
where $\kpzs_t$ is defined in \eqref{eqref:kpzs_def}, and $d_{(x_0,s)}$ is defined in \eqref{eq:d}. 
In the case $\dir \notin \DLBusedc$, then  $f_{(x,s)}^\dir(z) = W^\dir(x_0,s;z,s)$, and we call $\tau_{(x_0,s)}^{\dir,-}$ and $\tau_{(x_0,s)}^{\dir,+}$ the leftmost and rightmost $\dir$-Busemann interfaces. These were studied extensively in \cite{Bha24}.

A main difference between the interfaces studied in this work compared to those studied in \cite{Rahman-Virag-21,Bha24} comes from  our motivation behind studying such interfaces in the first place: their connection to eternal solutions of the KPZ fixed point. As opposed to  studying the interaction of interfaces started from different initial points  and a fixed  initial condition,  we are interested in the dynamics between interfaces when the initial condition varies as the starting  point of the interface varies.

We show in Proposition \ref{p:rest} that the mixed $\dir$-Busemann interfaces satisfy the following consistency property: for $s \le r \le t$, we have
\[
\intc_{(x_0,s)}^{\dir,-}(t) = \intc_{(w,r)}^{\dir,-}(t), \quad\text{where}\quad w = \intc_{(x_0,s)}^{\dir,-}(r).
\]
and the same holds for $-$ replaced by $+$. This naturally brings up the question of the existence of bi-infinite interfaces, which we now define.

 \begin{definition}\label{def:biInf}
A \textbf{leftmost mixed $\dir$-Busemann bi-infinite interface} is a continuous function $\intc:\R \to \R$ such that, for every $s \in \R$, 
\[
\intc|_{[s,\infty)} = \intc_{(\intc(s),s)}^{\dir,-}.
\]
A \textbf{rightmost mixed $\dir$-Busemann bi-infinite interface}  is defined similarly, except that we have $\intc|_{[s,\infty)} = \intc_{(\intc(s),s)}^{\dir,+}$. We let $\Intc^{\dir,-}$ denote the set of leftmost mixed $\dir$-Busemann bi-infinite interfaces, and let $\Intc^{\dir,+}$ denote the set of  rightmost mixed $\dir$-Busemann bi-infinite interfaces.

 Furthermore, we define the sets
\be \label{eq:T-+triv}
\overline \Intc^{\dir,-} := \Intc^{\dir,-} \cup\{\ltriv,\rtriv\},\quad\text{and}\quad \overline \Intc^{\dir,+} := \Intc^{\dir,+} \cup\{\ltriv,\rtriv\},
\ee
where $\ltriv:\R \to \{-\infty\}$ is the trivial path $\ltriv(t) \equiv -\infty$, and $\rtriv:\R \to \{+\infty\}$ is the trivial path $\rtriv(t) \equiv +\infty$.  
\end{definition}
It is shown in Proposition \ref{p: existence of bi-infinite competetion interfaces} that the sets $\Intc^{\dir,-}$ and $\Intc^{\dir,+}$ are nonempty, and, more specifically uncountable. In particular, for each $t \in \R$, these sets are in bijection with particular subsets of $\R$, named $\Split_{t,\dir}^L$ and $\Split_{t,\dir}^R$. See Theorem \ref{thm:BiInt} below. The set $\Split_{t,\dir} = \Split_{t,\dir}^L \cup \Split_{t,\dir}^R$ has Hausdorff dimension $\f{1}{2}$, as proved in \cite{Busa-Sepp-Sore-22a} and recorded here as Proposition \ref{prop:Haus12}.

 It is immediate from the definition that, if $b$ is an eternal solution, then for any constant $C \in \R$, $(x,t) \mapsto b(x,t) + C$ is also an eternal solution. We define an equivalence class on the set of eternal solutions, where $b_1 \sim b_2$ if the difference $(x,t) \mapsto b_1(x,t) - b_2(x,t)$ is a constant. 
Our first main result is the following. 
\begin{theorem}\label{thm:mr}
    On a single event of probability one, for every $\dir\in\DLBusedc$, there exists a bijection between the set of equivalence classes of  $\dir$-eternal solutions of the KPZ fixed-point and the set $\overline \Intc^{\dir,-}$. Similarly, there is a bijection between this set of equivalence classes of $\dir$-eternal solutions and $\overline \Intc^{\dir,+}$. These bijections satisfy the following:
    \begin{enumerate}[label=(\roman*), font=\normalfont]
    \item(Trivial interfaces are mapped to the $\dir \pm$ Busemann functions) \label{it:mainthmit1} Under both bijections, $\ltriv$ is mapped to  the equivalence class of $(x,t) \mapsto W^{\dir +}(0,0;x,t)$, and $\rtriv$ is mapped to the equivalence class of $(x,t) \mapsto W^{\dir-}(0,0;x,t)$.
    \item(General $\dir$-eternal solutions are a stitching together of the two Busemann functions) \label{it:mainthmit2} If $\intc \in \Intc^{\dir,-}$ (resp. $\intc \in \Intc^{\dir +}$), then the image under each of these bijections is the equivalence class of the function
        \be \label{b_def_intro}
        b(x,t)=\begin{cases}
            W^{\dir-}(\intc_0,0;x,t) \text{ if } x \leq \intc_t\\
             W^{\dir+}(\intc_0,0;x,t) \text{ if } x \geq \intc_t,
        \end{cases}
        \ee
        where $\intc_s = \intc(s)$ for $s \in \R$.
    \end{enumerate}
\end{theorem}
\begin{remark}
It is not immediately obvious that the function $b$ in \eqref{b_def_intro} is well-defined and continuous. This is due to a particular fact about mixed Busemann interfaces proved as Lemma \ref{l:MV}\ref{it:Bu3}, which gives us that $W^{\dir -}(\intc_0,0;\intc_t,t) = W^{\dir+}(\intc_0,0;\intc_t,t)$.
\end{remark}
Theorem \ref{thm:mr} is restated and proved as the first two items of Proposition \ref{p:correspondence}. The following is a corollary of the theorem, and is proved in Section \ref{sec:most_main_proofs}. The fact that there are uncountably many equivalence classes of $\dir$-eternal solutions comes by proving another bijection between the set of mixed $\dir$-Busemann bi-infinite interfaces and a certain uncountable subset of $\R$, recorded below as Theorem \ref{thm:BiInt}.
\begin{corollary}\label{cor:mr}
    With probability one, the following hold for all $\dir \in \R$.
    \begin{enumerate} [label=(\roman*), font=\normalfont]
        \item \label{it:1F1S} When $\dir \notin \DLBusedc$, there is exactly one equivalence class of $\dir$-eternal solutions, namely that of $(x,t) \mapsto W^\dir(0,0;x,t)$.
        \item \label{it:1FinfS} When $\dir\in\DLBusedc$, the set of equivalence classes of $\dir$-eternal solutions to the KPZ fixed point is uncountably infinite.
        \item \label{it:not_constant} If $b_1$ and $b_2$ are two non-equivalent $\dir$-eternal solutions, then for every $t\in\R$ the function $x\mapsto b_1(x,t) - b_2(x,t)$ is not constant.
    \end{enumerate} 
\end{corollary}

We now briefly explain in broad strokes where the connection between interfaces and eternal solutions comes from. We show in Lemma \ref{lem:geodesics_from_b} that every $\dir$-eternal solution $b$ gives rise to a forest of semi-infinite geodesics, which we call $b$-geodesics. A semi-infinite geodesic rooted at a point $(x,t) \in \R^2$ is a backwards infinite planar path encoded by a continuous function $g:(-\infty,t] \to \R$ such that the portion of the path between any two of its points is a point-to-point geodesic for $\Ll$. See Section \ref{sec:SIG} for details. For $s < t$, the spatial location, $g(s)$, of a $b$-geodesic rooted at $(x,t)$ is a maximizer of the function
\be \label{bmax1}
z \mapsto b(z,s) + \Ll(z,s;x,t).
\ee
When there is more than one maximizer, some care is needed to ensure that the choices of maximizers at each $s$ indeed gives a continuous function $g$. See Lemma \ref{lem:geodesics_from_b} and Remark \ref{rmk:b_global_solutions} for details. It is shown in  Lemma \ref{lem:dir_from_global} that if $b$ satisfies the slope condition \eqref{eq:slope_cond}, then the associated semi-infinite geodesics all have direction $\dir$; i.e.,
\[
\lim_{s \to -\infty} \f{g(s)}{|s|} = \dir.
\]

When $b \sim W^{\dir -}(0,0;\cdot,\cdot)$ or $b \sim W^{\dir +}(0,0;\cdot,\cdot)$, we call the associated geodesics $\dir-$ geodesics and $\dir +$ geodesics, respectively. It was shown in \cite{Busa-Sepp-Sore-22a} that, when $\dir \in \DLBusedc$, the $\dir-$ and $\dir +$ geodesics form two coalescing families of semi-infinite geodesics. That is, every $\dir -$ geodesic eventually separates from every $\dir +$ geodesic, and the $\dir -$ geodesics all coalesce, while the $\dir +$ geodesics all coalesce  (recorded here as Proposition \ref{prop:DL_all_coal}). In \cite{Busani_N3G}, it was shown that \textit{every} $\dir$-directed geodesic is either a $\dir-$ geodesic or a $\dir +$ geodesic.

These $\dir \pm$ geodesics obey an ordering property that implies the following: if $b$ is an eternal solution, then traveling from left to right on each time level $\{(x,t):x\in\R\}$ there must be a point  $\intc_t\in \R\cup\{\pm\infty\}$ where $b$-geodesics rooted at $(x,t)$ switch from $\dir-$ geodesics to $\dir+$ geodesics. This point is the location of a bi-infinite interface at time $t$.  This is made precise in the following theorem. 

\begin{figure}
    \centering
    \includegraphics[width=0.3\linewidth]{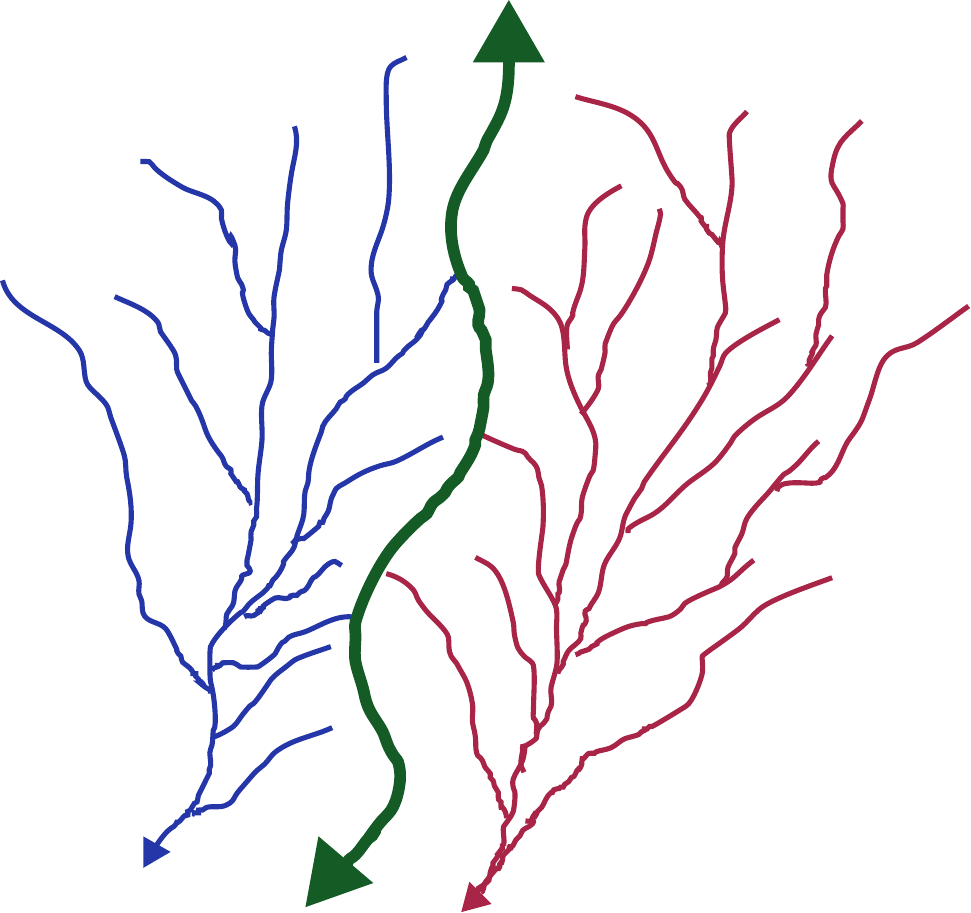}
    \caption{\small A bi-infinite interface separating two coalescing families of backwards semi-infinite geodesics}
    \label{fig:two_coal_fam}
\end{figure}
\begin{theorem} \label{thm:intGeo}
   On an event of probability one, the following holds for every $\dir\in\DLBusedc$. Let $b$ be a $\dir$-eternal solution, and let $\intc^-$ and $\intc^+$ be the associated elements of $\overline \Intc^-$ and $\overline \Intc^+$ from the bijection in Theorem \ref{thm:mr}. 
    \begin{enumerate} [label=(\roman*), font=\normalfont]
        \item If $p$ lies strictly to the left of $\intc^-$, then the set of $b$-geodesics equals the set of all $\dir-$ geodesics starting from $p$, and all $b$-geodesics from $p$ lie strictly to the left of $\intc^-$.
        \item If $p$ lies strictly to the right of $\intc^+$, then the set of $b$-geodesics started from $p$ equals the set of all $\dir+$ geodesics starting from $p$, and all $b$-geodesics from $p$ lie strictly to the right of $\intc^+$.
        \item If $p$ lies  between $\intc^-$ and $\intc^+$ (possibly on one of the interfaces), then the set of $b$-geodesics  started from $p$ equals the union of the set of all $\dir-$ geodesics and the set of all $\dir+$  geodesics starting from $p$.
    \end{enumerate}
\end{theorem}
\begin{remark}
    The interface $\intc^-$ separates the rightmost $b$-geodesics, while the interface $\intc^+$ separates the leftmost $b$-geodesics (i.e., leftmost/rightmost maximizers in \eqref{bmax1}). See Proposition \ref{p:correspondence}\ref{it:Requiveq}. See \eqref{tauLR_Def}, Corollary \ref{c:cutoff_geod}, and for precise statements on what happens for points that lie on each of these interfaces. The $\dir-$ geodesics all coalesce, while the $\dir+$ geodesics all coalesce. See Figure \ref{fig:two_coal_fam} for a visual depiction.  
\end{remark}
Theorem \ref{thm:intGeo} is proved in Section \ref{sec:most_main_proofs}. Given the bijection between $\dir$-eternal solutions and bi-infinite interfaces, the key is Lemma \ref{lem:geod} in Section \ref{s:G_and_I}.

\subsection{Geometry of interfaces}
As we have seen, the study of competition interfaces is closely connected with that of infinite geodesics in the directed landscape.  There are four particular $\dir-$ and $\dir +$ geodesics that we make particular use of,  namely
\[
g_{(x,t)}^{\dir-,L}, g_{(x,t)}^{\dir-,R}, g_{(x,t)}^{\dir +,L},g_{(x,t)}^{\dir +,R}.
\]
These are defined and discussed more in Section \ref{sec:SIG}. For $\sigg \in \{-,+\}$, $g_{(x,t)}^{\dir\sig,L}$ is the the leftmost $\dir \sig$ geodesic, and $g_{(x,t)}^{\dir\sig,R}$ is the rightmost $\dir \sig$ geodesic. Furthermore, $g_{(x,t)}^{\dir-,L}$ is the leftmost $\dir$-directed geodesic, and $g_{(x,t)}^{\dir +,R}$ is the rightmost $\dir$-directed geodesic. As we have discussed previously, when $\dir \in \DLBusedc$,  the $\dir-$ and $\dir +$ geodesics each eventually split from each other. We say that two geodesics $g_1,g_2$ rooted at $(x,t)$ are disjoint if they only share the common initial point $(x,t)$. In other words, $g_1(s) \neq g_2(s)$ for all $s < t$. See Figure \ref{fig:split_points}.  We define the following sets of \textbf{splitting points:}
\be
\begin{aligned}
    &\Split_{t,\dir} := \{x \in \R:  \exists \text{
    disjoint}   \text{   }\text{semi-infinite  geodesics from  }(x,t) \text{ in direction }\dir\}, \label{Split_sdir}\\
    &\Split^L_{t,\dir}:= \{ x \in \R: g^{\dir-,L}_{(x,t)}(s) < g^{\dir+,L}_{(x,t)}(s) \quad \forall s < t\}, \\
    &\Split^R_{t,\dir}:= \{ x \in \R: g^{\dir-,R}_{(x,t)}(s) < g^{\dir+,R}_{(x,t)}(s) \quad \forall s < t\},\\
    &\Split^M_{t,\dir}=\Split^L_{t,\dir} \cap \Split^R_{t,\dir},\\
    &\Split^S_{\dir}=\bigcup_{t\in\R}(\Split^S_{t,\dir} \times \{t\}),\quad S \in \{L,R,M\}.
    \end{aligned}
    \ee

\begin{figure*}[t!]
    \centering
    \begin{subfigure}[t]{0.3\textwidth}
        \centering
        \includegraphics[height = 4cm]{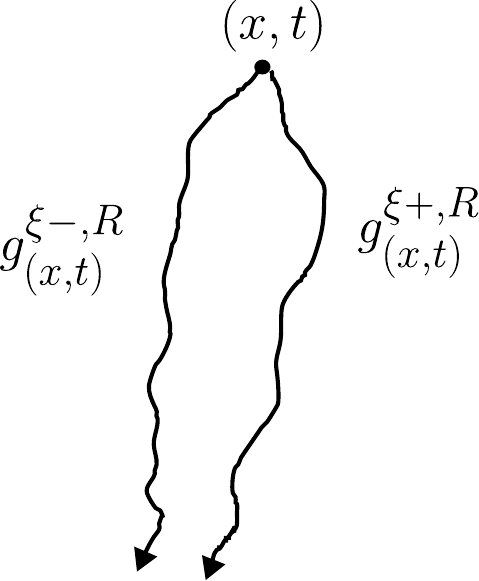}

        \caption{$(x,t) \in \Split_{\dir}^R$}
    \end{subfigure}
    ~ 
    \hspace{2pt}
    \begin{subfigure}[t]{0.3\textwidth}
        \centering
        \includegraphics[height = 4cm]{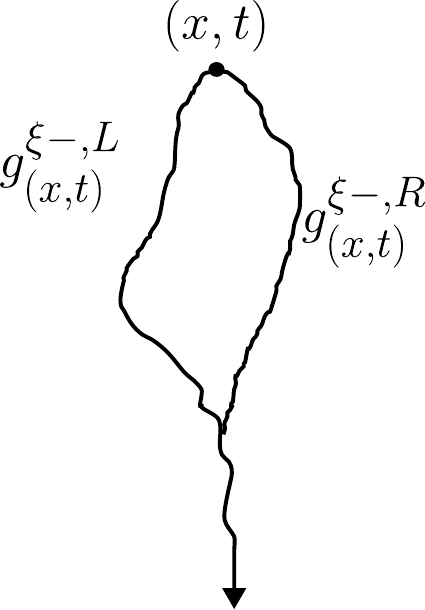}

        \caption{$(x,t) \in \NU_{\dir-}$}
    \end{subfigure}
     ~ 
    \hspace{2pt}
    \begin{subfigure}[t]{0.3\textwidth}
        \centering
        \includegraphics[height = 4cm]{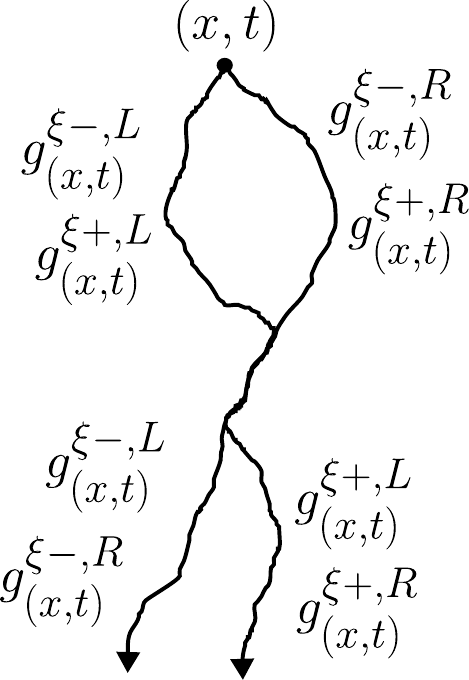}

        \caption{$(x,t) \in \Branch_\dir$}
    \end{subfigure}
    \caption{Examples of points in the sets $\Split_\dir^R,\NU_{\dir-}$, and $\Branch_\dir$}\label{fig:split_points}
\end{figure*}

    By  Lemma \ref{l:equ}, for each $t \in \R$, the set $\Split^L_{t,\dir}$ (resp. $\Split^R_{t,\dir}$) is the set of points of increase of a certain nondecreasing continuous function $D^\dir_t$ defined below in \eqref{Dfunc}. It is shown in  Proposition \ref{theorem: local variation and splitting points} that $\Split_{t,\dir} = \Split_{t,\dir}^L \cup \Split_{t,\dir}^R$.
    Therefore, for each $t \in \R$, we can write 
    \begin{equation}\label{eq77}
        \R=\bigcup_{\alpha\in \I^{t,\dir}} \Intr^{t,\dir}_\alpha,
    \end{equation}
    for some (uncountable) index set $\I^{t,\dir}$, where  $\Intr^{t,\dir}_\alpha=[\a^t_\alpha,\b^t_\alpha]$ are mutually disjoint closed intervals (possibly containing just a single point) such that $\a^t_\alpha\in \Split^L_{t,\dir}$, $\b^t_\alpha\in \Split^R_{t,\dir}$ and $(\a^t_\alpha,\b^t_\alpha)\subseteq \Split^c_{t,\dir}$.  Two points $x,y$ lie in the same interval $\Intr_\alpha^{t,\dir}$ if and only if $D_t^\dir(x) = D_t^\dir(y)$. We denote this family of intervals as $\Intr^{t,\dir}=\{\Intr^{t,\dir}_\alpha\}_{\alpha \in \I^{t,\dir}}$.

    Our next theorem describes the points through which bi-infinite interfaces traverse. Item \ref{itm:Interface_splitting} states that the set $\Split^L_{t,\dir}$ (resp. $\Split_{t,\dir}^R)$ is exactly the set of points at which interfaces in $\Intc^{\dir,-}$ (resp. $\Intc^{\dir,+}$) travel through at time $t$. Item \ref{BiInt:it4} states that for a $\dir$-eternal solution $b$, the time $t$ locations of the associated leftmost and rightmost bi-infinite interfaces are the endpoints of one of the intervals $\Intr_\alpha^{t,\dir}$.
\begin{theorem}\label{thm:BiInt}
    With probability one the following holds for every $\dir\in\DLBusedc$ 
    \begin{enumerate} [label=(\roman*), font=\normalfont]
       \item \label{itm:Interface_splitting} (Bi-infinite interfaces live on the set of splitting points) For each $t \in \R$, the map $\intc \mapsto \intc(t)$ is a bijection between the set $\Intc^{\dir,-}$ (resp. $\Intc^{\dir,+}$) and the set $\Split_{t,\dir}^L$ (resp. $\Split_{t,\dir}^R$). In particular, the associated paths in $\Intc^{\dir,-}$ live in the set $\Split_\dir^L$, and the associated paths in $\Intc^{\dir,+}$ live in the set $\Split_\dir^R$.
        \item (Left and right bi-infinite interfaces corresponding to the same eternal solution are the endpoints of an interval in $\Intr^{t,\dir})$ \label{BiInt:it4} Let $b$ be a $\dir$-eternal solution and let $\intc^- \in \Intc^{\dir,-},\intc^+ \in \Intc^{\dir,+}$ be the associated interfaces from the bijections in Theorem \ref{thm:mr}. Assume that $\intc^- \in \Intc^{\dir,-}$ and $\intc^+ \in \Intc^{\dir,+}$  (i.e., $\intc^-$ and $\intc^+$ are not $\ltriv$ or $\rtriv$). Then, for each $t \in \R$, there exists an index $\alpha \in \mathcal \I^{t,\dir}$ such that 
        \[
        \intc^-(t) = \a^t_\alpha,\quad\text{and}\quad \intc^+(t) = \b^t_\alpha.
        \]
    \end{enumerate}
\end{theorem}
Item \ref{itm:Interface_splitting} of Theorem \ref{thm:BiInt} is restated in Section \ref{sec:along_interface} as Corollary \ref{cor:disInt}. 
 Item \ref{BiInt:it4} is proved in Section \ref{sec:most_main_proofs}. 

Our next theorem describes more general geometric properties of both bi-infinite and semi-infinite interfaces. While the bi-infinite interfaces live on the set of splitting points, the semi-infinite interfaces live on a larger set, known as the set of \textbf{branching points}. For $t \in \R$, we   define this set as
\begin{equation}
\label{eq:bran_def}
\begin{aligned}
    \Branch_{t,\dir}&=\{x \in\R: \exists \delta>0 \text{ so that } g^{\dir-,L}_{(x,t)}(s) < g^{\dir+,R}_{(x,t)}(s) \quad \forall s \in(t-\delta,t)\},\quad\text{and} \\
    \Branch_\dir &= \bigcup_{t \in \R} (\Branch_{t,\dir} \times\{t\}).
    \end{aligned}
\end{equation}
We observe first that $\Split_{t,\dir} \subseteq \Branch_{t,\dir}$. Indeed, if there are disjoint $\dir$-directed geodesics from $(x,t)$, then since $g_{(x,t)}^{\dir-,L}$ and $g_{(x,t)}^{\dir+,R}$ are the leftmost and rightmost such geodesics, we have that $g_{(x,t)}^{\dir-,L}(s) < g_{(x,t)}^{\dir+,R}(s)$ for \textit{all} $s < t$.  We show in Remark \ref{rmk:strictsub}  that $\Split_{t,\dir}$ is in fact a strict subset of $\Branch_{t,\dir}$. To understand the set $\Branch_{t,\dir} \setminus \Split_{t,\dir}$, we define the following sets of points, originally from \cite{Busa-Sepp-Sore-22a}: 
\be \label{eq:NUdef}
   \begin{aligned}
       &\NU_{t,\dir+}:=\{x \in \R: \exists \delta > 0 \text{ so that } g_{(x,t)}^{\dir+,L}(s)<g_{(x,t)}^{\dir+,R}(s) \text{ for all } s \in (t-\delta,t)\}. \\
       & \NU_{t,\dir-}:=\{x \in \R: \exists \delta > 0 \text{ so that } g_{(x,t)}^{\dir-,L}(s)<g_{(x,t)}^{\dir-,R}(s) \text{ for all } s \in (t-\delta,t)\}, \quad\text{and} \\
       &\NU_{\dir \sig} := \bigcup_{t \in \R} \NU_{t,\dir \sig} \times \{t\}\quad \text{for }\sigg \in \{-,+\}.
   \end{aligned}  
   \ee
   Equivalently, as shown in \cite{Bha24}, $\NU_{\dir \pm}$ are precisely the sets of points $(x,t) \in \R^2$ such that $g_{(x,t)}^{\dir \pm,L} \neq g_{(x,t)}^{\dir \pm,R}$. All $\dir -$ geodesics coalesce, and all $\dir +$ geodesics coalesce, so the separation of geodesics from points in $\NU_{\dir \sig}$ is only temporary. See Figure \ref{fig:split_points}. 
   
   If $x \in \Branch_{t,\dir}\setminus \Split_{t,\dir}$, then $g_{(x,t)}^{\dir-,R}$ and $g_{(x,t)}^{\dir+,R}$ are not disjoint, and $g_{(x,t)}^{\dir-,L}$ and $g_{(x,t)}^{\dir+,L}$ are also not disjoint. It turns out that each of these pairs of geodesics cannot split and then come back together, so there is some $\delta > 0$ such that 
   \[
   g_{(x,t)}^{\dir-,L}(s) = g_{(x,t)}^{\dir+,L}(s) < g_{(x,t)}^{\dir-,R}(s) = g_{(x,t)}^{\dir+,R}(s),\quad \forall s \in (t-\delta,t).
   \]
   Thus, $\Branch_{t,\dir} \setminus \Split_{t,\dir} \subseteq \NU_{t,\dir-} \cap \NU_{t,\dir +}$. For each $t \in \R$, this is a countable set (Lemma \ref{lem:NUcount} below).

\begin{theorem}\label{thm:int1}
   On a single event of full probability, the following holds for all $\dir \in \R$:  
    \begin{enumerate} [label=(\roman*), font=\normalfont]
        \item (Interfaces have directions) \label{int1:it1}   For every and $(x_0,s)\in\R^2$ the interfaces $\intc_{(x_0,s)}^{\dir,-}$ and $\intc_{(x_0,s)}^{\dir,+}$ are $-\dir$-directed continuous functions,   i.e.\
        \be \label{eq:forw_dir}
            \lim_{t\rightarrow \infty}\intc_{(x_0,s)}^{\dir,\sig}(t)/t=-\dir, \qquad \sig\in\{-,+\}.
        \ee
        Furthermore, if $\intc \in \Intc^{\dir,-} \cup \Intc^{\dir,+}$, then  we also have
        \be \label{eq:back_dir}
        \lim_{t \to -\infty} \f{\intc(t)}{t} = -\dir.
        \ee
        \item \label{int1:it2}(Interfaces  live on the set of branching points $\Branch_{\dir}$) For all $t \in \R$ and $x \in \R$, there exists $s < t$, $x_0 \in \R$, and $\sigg \in \{-,+\}$ such that  $x = \intc_{(x_0,s)}^{\dir, \sig}(t)$ if and only if $x \in \Branch_{t,\dir}$.
        \item \label{int1:it3} (Interfaces are ordered) For $s\in\R$, if $x_0<y_0$, then
        \[
            \intc^{\dir,-}_{(x_0,s)}(r)\leq \intc^{\dir,-}_{(y_0,s)}(r), \qquad\text{and}\qquad\intc^{\dir,+}_{(x_0,s)}(r)\leq \intc^{\dir,+}_{(y_0,s)}(r), \quad \forall r>s.
        \]
        \item (Interfaces through common intervals $\Intr^{s,\dir}$ coalesce) \label{BiInt:it3} If $\dir \in \DLBusedc$, then for any $s\in \R$ and  $\alpha\in \I^{s,\dir}$, if $x_0,y_0 \in \Intr_\alpha^{s,\dir}$, then there exists $t > s$ such that
        \[
        \intc_{(x_0,s)}^{\dir,-}(r) = \intc_{(y_0,s)}^{\dir,-}(r),\quad\text{and}\quad \intc_{(x_0,s)}^{\dir,+}(r) = \intc_{(y_0,s)}^{\dir,+}(r),\quad\text{for all }r \ge t
        \]
        Additionally, in this case, there exists an infinite sequence $\{t_n\}$ with $t_n \to \infty$ such that 
        \[
        \intc_{(x_0,s)}^{\dir,-}(t_n) = \intc_{(y_0,s)}^{\dir,+}(t_n)\quad \text{for all }n \in \N.
        \]
        \item \label{itm:disj_int_disj}(Interfaces through different intervals $\Intr^{s,\dir}$ are disjoint).
        If $\dir \in \DLBusedc$ and $x_0 < y_0$, where $x_0 \in \Intr_\alpha^{s,\dir}$ and $y_0 \in \Intr_\beta^{s,\dir}$ for $\alpha \neq \beta$, then for all $t > s$,
        \[
        \intc_{(x_0,s)}^{\dir,-}(t) \le \intc_{(x_0,s)}^{\dir,+}(t) < \intc_{(y_0,s)}^{\dir,-}(t) \le \intc_{(y_0,s)}^{\dir,+}(t).
        \]
        In particular, combined with the previous item, the set of all leftmost $\dir$-mixed Busemann interfaces and the set of all rightmost $\dir$-mixed Busemann interfaces are forests,  and for each $(x_0,s) \in \R^2$, there is a unique $\intc^- \in \Intc^{\dir,-}$ such that $\intc_{(x_0,s)}^{\dir,-}$ coalesces with $\intc^-$ and a unique $\intc^+ \in \Intc^{\dir,+}$ such that $\intc_{(x_0,s)}^{\dir,+}$ coalesces with $\intc^+$.  
        \item \label{int1:it4} (Interfaces form interior bubbles)  If $\dir \in \DLBusedc$, there exist infinitely many points $(x_0,s)\in\Split_{\dir}^L$ (and also infinitely many points in $\Split_\dir^R$) such that the following holds:  there exist $t_2 > t_1 > s$ and $\delta > 0$ so that
        \begin{enumerate}  [label=(\alph*), font=\normalfont]
        \item $\intc_{(x_0,s)}^{\dir,-}(r)=\intc_{(x_0,s)}^{\dir,+}(r)$ for $r \in [t_1-\delta,t_1]$,
        \item $\intc_{(x_0,s)}^{\dir,-}(r)=\intc_{(x_0,s)}^{\dir,+}(r)$ for $r \in [t_2,t_2 + \delta]$, and
        \item $\intc_{(x_0,s)}^{\dir,-}(r)<\intc_{(x_0,s)}^{\dir,+}(r)$,  for all $r\in(t_1,t_2)$.
        \end{enumerate} 
        See Figure \ref{fig:bubbles} for an illustration. 
    \end{enumerate}
\end{theorem}
\begin{figure*}[t!]
    \centering
        \includegraphics[page=1, scale=0.5, trim=50 250 100 200, clip]{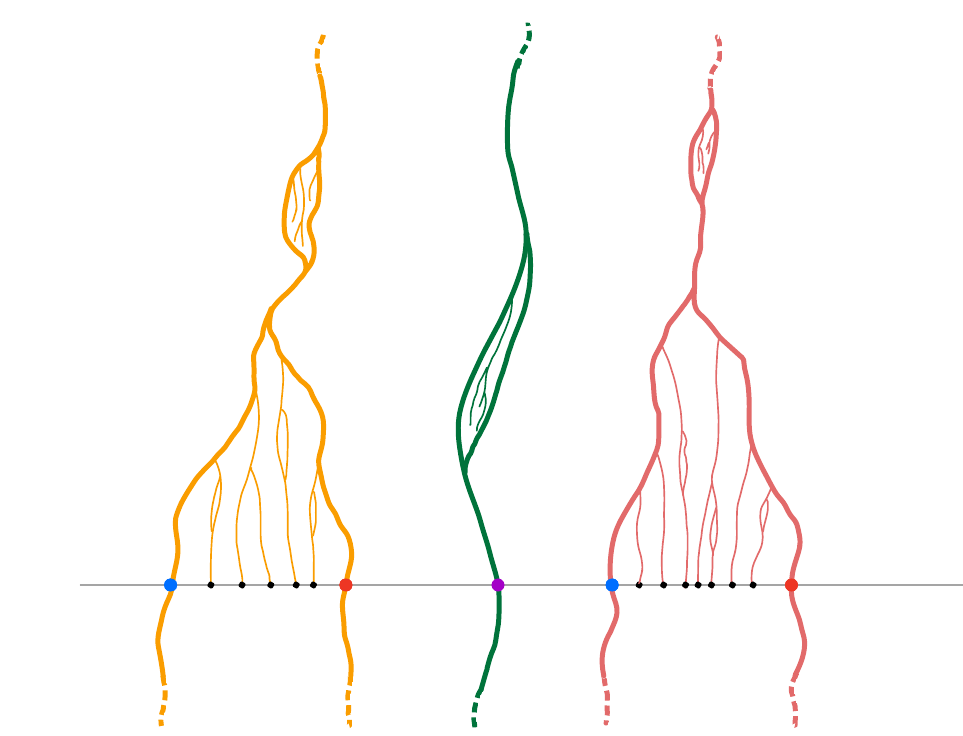}

        \caption{Three networks of competition interfaces in yellow, green and red intersecting a time horizon at time $t$ (black line). Bi-infinite interfaces are in bold and live on points in $\Split_{t,\dir}$ (colored points).  Leftmost and rightmost interfaces cross the time horizon at $\Split^L_{t,\dir}$ and $\Split^R_{t,\dir}$ respectively (points in blue and red respectively). The points where both leftmost and rightmost interfaces cross are $\Split^M_{t,\dir}$ (purple). By Theorems \ref{thm:BiInt}\ref{BiInt:it4}-\ref{itm:disj_int_disj} and \ref{thm:int1}\ref{BiInt:it3}, if $\intc^- \in \Intc^{\dir,-}$ and $\intc^+ \in \Intc^{\dir,+}$ are interfaces corresponding to the same eternal solution $b$, then $\intc^-$ and $\intc^+$ meet infinitely often and are disjoint from all other bi-infinite interfaces. }
    \label{Fig1}
\end{figure*}

Figure \ref{Fig1} gives a visual depiction of the interface network described in Theorem \ref{thm:int1}. The theorem  is proved in Section  \ref{sec:geom_proof}. Equation \eqref{eq:forw_dir} of Item \ref{int1:it1} follows from \cite[Corollary 4.14]{Rahman-Virag-21}, which is actually a more general fact about competition interface directions from sloped initial conditions. Some additional justification is needed to show that the slope condition holds for interfaces from all initial points simultaneously, and we make some brief remarks on this in the proof.

\begin{figure}
    \centering
    \includegraphics[width=5 cm]{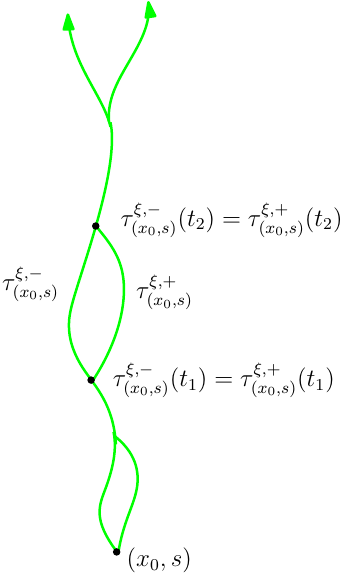}
    \caption{There are infinitely many points from which the leftmost and rightmost interfaces form interior bubbles. In particular, the interfaces stay together for some time, then split, and then come back together and stay together for some time, possibly to separate again. }
    \label{fig:bubbles}
\end{figure}
\begin{remark}
Item \ref{int1:it4} of Theorem \ref{thm:int1} shows a stark difference between the behavior of Busemann interfaces and mixed Busemann interfaces.  Indeed, it was shown in \cite{Bha24} that, for a given $\dir \in \R$, the collection of $\dir$-Busemann interfaces are distributed as infinite Busemann geodesics, and the latter do not form interior bubbles. Since the points $(x_0,s)$ lie on $\Split_\dir^L$ and $\Split_\dir^R$, Theorem \ref{thm:BiInt}\ref{itm:Interface_splitting} implies that these interior bubbles occur on bi-infinite interfaces. See Proposition \ref{p: interfaces do not meet} for a precise description of how we obtain these points where bubbles form.  We currently do not know whether all elements of $\Intc^{\dir,\pm}$ form bubbles, whether these bi-infinite interfaces form more than one bubble, or whether the number of bubbles formed on a given bi-interface is finite. We do know, however from Proposition \ref{p:fin_bubble} that if interfaces split and then meet, they must stay together for some time before splitting again. 
\end{remark}

\begin{remark} \label{rmk:branchNU}
In the case when $\dir \notin \DLBusedc$, Item \ref{int1:it2} is previously known. Indeed, when $\dir \notin \DLBusedc$, $\Branch_{t,\dir}$ is exactly the set $\NU_{t.\dir-} = \NU_{t,\dir+}$. It was shown in \cite[Theorem 5]{Bha24} that these are exactly the points along 
$\dir$-Busemann interfaces for $\dir \notin \DLBusedc$.
\end{remark}

\begin{remark} \label{rmk:strictsub}
Item \ref{int1:it2} of Theorem \ref{thm:int1} along with Theorem \ref{thm:BiInt} imply that $\Split_{t,\dir} \neq \Branch_{t,\dir} $. Indeed, for $s \in \R$, take an index $\alpha \in \I^{s,\dir}$ such that $\a_\alpha^s < \b_\alpha^s$, and choose $x \in (\a_\alpha^s,\b_\alpha^s)$. Let $\intc^-,\intc^+$ be the elements of $\Intc^{\dir,-},\Intc^{\dir,+}$ passing through the endpoints $\a_\alpha^s,\b_\alpha^s$ at time $s$.   Then, by Theorem \ref{thm:BiInt}\ref{BiInt:it4} and continuity, for all $t > s$ with $t$ sufficiently close to $s$, we have \[
\a_\beta^t = \intc^-(t) < \intc_{(x,s)}^{\dir,-}(t) < \intc^{+}(t) = \b_\beta^t.
\]
for some $\beta \in \I^{t,\dir}$. 
Then, we have $\intc_{(x,s)}^{\dir,-}(t) \in \Branch_{t,\dir}$ by Theorem \ref{thm:int1}\ref{int1:it2}, but $\intc_{(x,s)}^{\dir,-}(t)\notin \Split_{t,\dir}$ because there are no points of $\Split_{t,\dir}$ in the open interval $(\a_\beta^t,\b_\beta^t)$.  
\end{remark}

\subsection{Eternal solutions as Busemann limits} \label{sec:limits}

In \cite[Theorem 5.1(vi)]{Busa-Sepp-Sore-22a}, it is shown that, whenever $\dir \notin \DLBusedc$, for any sequence $\mbf v_n = (x_n,t_n)$ with $t_n \to -\infty$ and $\f{x_n}{|t_n|} \to \dir$, we have 
\[
W^{\dir}(\mbf p;\mbf q) = \lim_{n \to \infty} \Ll(\mbf v_n;\mbf q) - \Ll(\mbf v_n;\mbf p),\quad\text{for all }\mbf p,\mbf q \in \R^2.
\]
In this section, we describe the limit points of such sequences when $\dir \in \DLBusedc$. The results of this section are proved in Section \ref{sec:lim_proofs}. In proper metric spaces, the set of these Busemann limits is exactly the max-plus Martin boundary (see \cite[page 4]{Akian-Gaubert-Walsh-2009}).

\begin{theorem} \label{thm:bn_gen_lim}
    The following holds on a single event of full probability. Let $\dir \in \DLBusedc$, and let $\mbf v_n = (x_n,t_n)$ be a sequence with $t_n \to -\infty$ and $\f{x_n}{|t_n|} \to \dir$ as $n \to \infty$. For $n \ge 1$, define the function $b_n:\R^2 \to \R$ by
    \[
    b_n(x,t) := \Ll(\mbf v_n;x,t) - \Ll(\mbf v_n;0,0),
    \]
    for $t \ge t_n + 1$, and define $b_n$ for $t < t_n + 1$ in an arbitrary way so that $b_n$ is continuous. 
    Then, the following hold.
    \begin{enumerate}[label=(\roman*), font=\normalfont]
    \item \label{itm:bn_precompact} The sequence $(b_n)_{n \ge 1}$ is precompact in the uniform-on-compacts topology.
    \item \label{itm:limits_eternal} All subsequential limits of $(b_n)_{n \ge 1}$ are $\dir$-eternal solutions. 
    \end{enumerate}
\end{theorem}
\begin{remark}
Theorem \ref{thm:bn_gen_lim} implies the following weaker statement: For $\mbf p,\mbf q \in \R^2$, the $\R$-valued sequence 
\[
\Ll(\mbf v_n;\mbf q) - \Ll(\mbf v_n;\mbf p)
\]
is bounded, and if $\wt b(\mbf p;\mbf q)$ is a subsequential limit, then there exists a $\dir$-eternal solution $b$ such that $b(\mbf q) - b(\mbf p) = \wt b(\mbf p;\mbf q)$. Indeed, let $\wt b(\mbf p;\mbf q)$ be any subsequential limit. Then, by Theorem \ref{thm:bn_gen_lim}, there exists a further subsequence along which the function $b_n$ converges with respect to the uniform-on-coompact topology to a $\dir$-eternal solution $b$. That function $b$ must then satisfy $b(\mbf q) - b(\mbf p) = \wt b(\mbf p;\mbf q)$. By Theorem \ref{thm:BiInt}, for any such $\mbf p,\mbf q$, there exists $\mbf r \in \Split_\dir \cup\{\mbf p,\mbf q\}$ such that
\[
\text{Either}\quad\wt b(\mbf p;\mbf q) = W^{\dir-}(\mbf p;\mbf r) + W^{\dir +}(\mbf r;\mbf q),\quad\text{ or }\quad \wt b(\mbf p;\mbf q) = W^{\dir+}(\mbf p;\mbf r) + W^{\dir -}(\mbf r;\mbf q).
\]
Note that $\mbf r$ could be $\mbf p$ or $\mbf q$, in which case $\wt b(\mbf p;\mbf q)$ is one of the Busemann functions $W^{\dir-}(\mbf p;\mbf q)$ or $W^{\dir +}(\mbf p;\mbf q)$. \
\end{remark}

The next result gives a partial converse to Theorem \ref{thm:bn_gen_lim}. That is, given any $\dir$-eternal solution $b$, it can be realized as a limit of the sequence $\Ll(\mbf v_n;\mbf p) - \Ll(\mbf v_n;\mbf q)$ for appropriate choices of the sequence $\mbf v_n$ and $\mbf p,\mbf q$.
\begin{theorem} \label{thm:existence_of_v_seq}
The following holds on a single event of full probability. Let $\dir \in \R$, and let $b$ be a $\dir$-eternal solution. Then, for all $t \in \R$ and compact interval $[x,y] \in \R$, there exists a sequence $\mbf v_n = (x_n,t_n)$ with $t_n \to -\infty$ and $\f{x_n}{|t_n|} \to \dir$ and $N \in \N$ so that
\[
b(w,t) - b(x,t) = \Ll(\mbf v_n;w,t) - \Ll(\mbf v_n;x,t),\quad\text{ for all }w\in [x,y] \text{ and } n \ge N. 
\]
\end{theorem}
\begin{remark}
    For each $\dir$-eternal solution $b$, there is an associated bi-infinite interface $\intc^- \in \Intc^-$ by Theorem \ref{thm:mr}. Let $t \in \R$, and $z = \intc^-(t)$. Then, $(z,t)$ is a splitting point of $\dir$-directed infinite geodesics by Theorem \ref{thm:BiInt}. In the proof of Theorem \ref{thm:existence_of_v_seq}, the sequence $\mbf v_n$ is a particularly chosen sequence that lies between these two geodesics (see Figure \ref{fig:converse_subsequence_problem}). If we take our compact interval $[x,y]$ to contain $z$, then 
    \[
    b(y,t) - b(x,t) = b(y,t) - b(z,t) + b(z,t) - b(x,t) = W^{\dir-}(x,t;z,t) + W^{\dir+}(z,t;y,t).
    \]
    By Corollary \ref{cor:mr}\ref{it:not_constant}, for $x$ sufficiently less than $z$ and $y$ sufficiently greater than $z$, 
    \[
    b(y,t) - b(x,t) \notin \{W^{\dir-}(x,t;y,t),W^{\dir+}(x,t;y,t)\}.
    \]
    In particular, the Busemann functions $W^{\dir-}(x,t;y,t)$ and $W^{\dir+}(x,t;y,t)$ are not the only limits of differences
    $\Ll(\mbf v_n;y,t) - \Ll(\mbf v_n;x,t)$,
    where $\mbf v_n$ is a $\dir$-directed sequence. In fact, there are uncountably many possible limit points (across all possible sequences $\mbf v_n$).
\end{remark}

\subsection{Organization of the paper}
The starting point for this paper are the previous results about geodesics in the directed landscape from \cite{Directed_Landscape,Dauvergne-Sarkar-Virag-2022}, as well as detailed information on the Busemann process and semi-infinite geodesics from \cite{Busa-Sepp-Sore-22a}. The needed inputs are described in Section \ref{sec:prelim}, with some more technical results given in the Appendix. Section \ref{sec:interface_start} develops the basic algebraic properties of mixed $\dir$-Busemann interfaces, culminating in the proof of the existence of bi-infinite interfaces in Proposition \ref{p: existence of bi-infinite competetion interfaces}. Section  \ref{s:G_and_I} describes the relationship between infinite geodesics and mixed $\dir$-Busemann interfaces and provides the groundwork for the proof of Theorem \ref{thm:intGeo}. Section \ref{s:int_for} describes the network of mixed $\dir$-Busemann interfaces by describing when interfaces intersect and the points through which bi-infinite interfaces travel. This section lays the groundwork for the proof of Theorem \ref{thm:BiInt} and Items \ref{BiInt:it3}-\ref{itm:disj_int_disj} of Theorem \ref{thm:int1}. Section \ref{sec:gen_soln} builds the theory of general $\dir$-eternal solutions, which leads into the proof of Theorem \ref{thm:mr} in Section \ref{sec:global_proofs}. Section \ref{sec:most_main_proofs} completes the proofs of Corollary \ref{cor:mr}, Theorem \ref{thm:intGeo}, and Theorem \ref{thm:BiInt}. Section \ref{sec:geom_fine_prop} develops the finer geometric properties of interfaces needed to complete the proof of Theorem \ref{thm:int1}.  The results of Section \ref{sec:limits} are proved in Section \ref{sec:lim_proofs}. 

\subsection{Acknowledgments and funding}
We wish to thank Manan Bhatia, Riddhipratim Basu, Alexander Dunlap, Chris Janjigian, and Firas Rassoul-Agha for helpful conversations. We especially thank Chris Janjigian and Ivan Corwin for bringing questions about Busemann limits to our attention that led us to prove the results in Section \ref{sec:limits} and Chris Janjigian for pointing out the reference \cite{Akian-Gaubert-Walsh-2009}. E.S.\ was partially supported by the Fernholz foundation and by Ivan Corwin's Simons Investigator Grant 929852.  S.B. is supported by scholarship from National Board for Higher Mathematics (NBHM) (ref no: 0203/13(32)/2021-R\&D-II/13158). Part of this work was discussed during separate trips that S.B. and E.S. made to the University of Edinburgh, whom we thank for their hospitality.  During his trip to Edinburgh, E.S. was supported by AMS-Simons travel grant AMMS CU23-2401. Additions to the second version of this paper were discussed during a trip E.S. and S.B. made to the workshop on stochastic interacting particle systems and random matrices at the R\'enyi institute in Budapest, Hungary in June 2025. The travel of E.S. during that trip was supported by the same AMS-Simons travel grant. Travels of S.B. were supported by GARP travel funding by Indian Institute of Science.

\section{Preliminaries} \label{sec:prelim}
\subsection{The directed landscape} \label{sec:DL}
Let $\Rup = \{(x,s;y,t) \in \R^4: s < t\}$. The directed landscape is a random continuous function $\Ll:\Rup \to \R$ that was first introduced in \cite{Directed_Landscape}. It satisfies the following \textbf{metric composition property}: For $s < r < t$ and $x,y \in \R$,
\be \label{L_comp}
\Ll(x,s;y,t) = \sup_{z \in \R}\{\Ll(x,s;z,r) + \Ll(z,r;y,t)\},
\ee
which implies the reverse triangle inequality for $\Ll$: for $s < r < t$ and $x,y,z \in \R$,
\be \label{triangle}
\Ll(x,s;y,t) \ge \Ll(x,s;z,r) + \Ll(z,r;y,t).
\ee
While a key feature of the directed landscape, this metric composition property does not uniquely define its law; we refer the reader to \cite{Directed_Landscape} for a precise characterization. We let $(\Omega,\F,\Pp)$ be a probability space on which $\Ll$ is defined.

The directed landscape gives rise to a notion of geodesics.  We consider a continuous function $\gamma:[s,t] \to \R$ as a directed path in the plane. For such a path, define it's $\Ll$-length as
\[
\Ll(\gamma) = \inf_{k \in \Z_{>0}} \; \inf_{s = t_0 < t_1 < \cdots < t_k = t} \sum_{i = 1}^k \Ll\bigl(\gamma(t_{i - 1}),t_{i - 1};\gamma(t_i),t_i\bigr).
\]
Note that the $\Ll$-length of a path $\gamma$ can be $-\infty$. A path $\gamma$ is called a \textbf{point-to-point geodesic} from $(x,s)$ to $(y,t)$ if its $\Ll$ length is maximal among all paths $\gamma:[s,t] \to \R$ with $\gamma(s) = x$ and $\gamma(t) = y$. Equivalently, 
\[
\Ll(\gamma) = \sum_{i = 1}^k \Ll\bigl(\gamma(t_{i - 1}),t_{i - 1};\gamma(t_i),t_i\bigr)
\]
for all $k \in \Z_{>0}$ and partitions $s = t_0 < t_1 < \cdots < t_k = t$. In particular, if $s < r <t$, a point $z \in \R$ satisfies $\gamma(r) = z$ for some geodesic from $(x,s)$ to $(y,t)$ if and only if equality holds in \eqref{triangle}. See Figure \ref{fig:DL_geod}.
\begin{figure}
    \centering
    \includegraphics[width=0.5\linewidth]{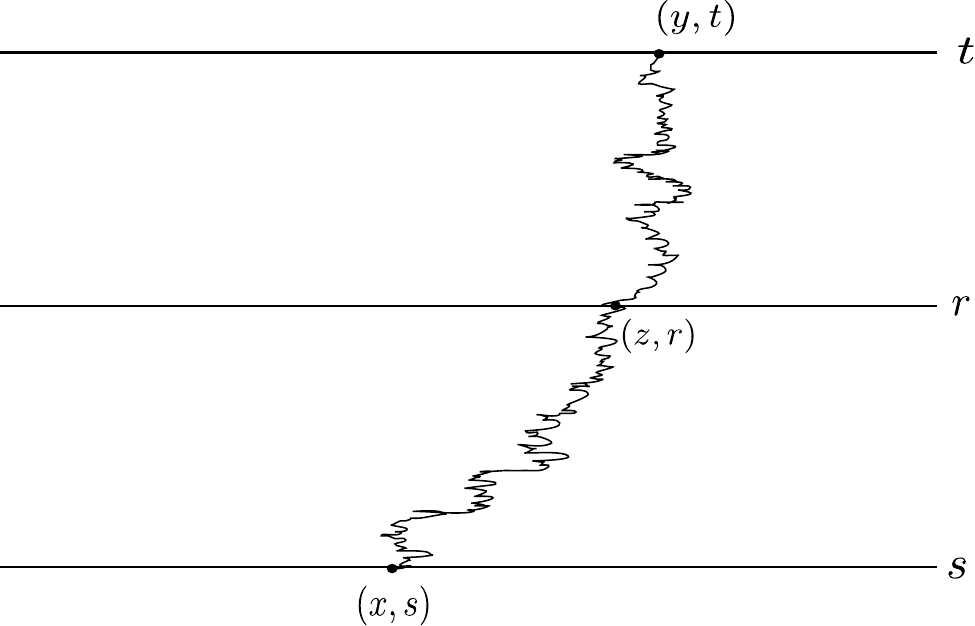}
    \caption{\small A directed landscape geodesic from $(x,s)$ to $(y,t)$ passing through $(z,r)$ at time level (horizontal line) $r$.}
    \label{fig:DL_geod}
\end{figure}

The fact that geodesics exist between any pair of points $(x,s)$ and $(y,t)$ with $s < t$ is proved in \cite[Lemma 13.2]{Directed_Landscape}. More specifically, if we define the path $\gamma^R:[s,t] \to \R$ by setting $\gamma^R(s) = x,\gamma^R(t) = y$, and $\gamma^R(r)$ to be the rightmost maximizer in \eqref{L_comp} for all $r \in (s,t)$, then $\gamma^R$ is continuous and is the rightmost geodesic from $(s,x)$ to $(y,t)$ in the sense that, for any other geodesic $\gamma$, $\gamma(r) \le \gamma^R(r)$ for $r \in [s,t]$. Similarly, one can define the path $\gamma^L$ by choosing leftmost maximizers, and this forms the leftmost geodesic between the two points.

For two fixed points $(x,s)$ and $(y,t)$, there is a unique geodesic between the points, almost surely \cite[Theorem 12.1]{Directed_Landscape}. However, there are exceptional pairs of points for which the geodesic is not unique. In this case, we know that that geodesics do not form interior bubbles. We state this precisely as the following lemma. 
\begin{lemma}\cite[Theorem 1]{Bha24}, \cite[Lemma 3.3]{Dauvergne-23} \label{lem:no_bubbles}
With probability one, there exist no points $(x,s;y,t) \in \Rup$ and distinct geodesics $\gamma_1,\gamma_2$ from $(x,s)$ to $(y,t)$, such that, for some $\delta > 0$, $\gamma_1$ and $\gamma_2$ agree on the set $[s,s+\delta] \cap [t-\delta,t]$.
 \end{lemma}

 For a function $f:\R \to \R$ such that maximizers of 
 \be \label{zf}
 z \mapsto f(z) + \Ll(z,s;x,t) \quad \text{over $z \in \R$}
 \ee
 exist for some $(z,s;x,t) \in \Rup$, we call $\gamma:[s,t] \to \R$ an \textbf{$(f,s)$-to-$(x,t)$ geodesic} or \textbf{point-to-line geodesic} if $\gamma$ is a point-to-point geodesic from $(z^\star,s)$ to $(x,t)$ for some maximizer $z^\star$ of \eqref{zf}.

\subsection{Busemann process}
For a fixed direction $\dir$, we define the Busemann function
\[
W^{\dir}(x,s;y,t)=\lim_{r\rightarrow -\infty}\Ll(-r\dir,r;y,t)-\Ll(-r\dir,r;x,s),
\]
which was first studied in \cite{Rahman-Virag-21}. In \cite{Busa-Sepp-Sore-22a}, this was extended to a process
\[
\Bigl(W^{\dir \sig}(x,s;y,t): \dir \in \R, \sigg \in \{-,+\}, (x,s;y,t) \in \R^4\Bigr).
\]
Here, the sign $\sigg \in \{-,+\}$ specifies the left and right-continuous versions of the process. It was shown in \cite{Busa-Sepp-Sore-22a} that the set of discontinuities in the $\dir$ parameter, named $\Xi$, is a random, countably infinite, dense subset of $\R$, and that any fixed $\dir$ lies in $\Xi$ with probability $0$. We now state the relevant results from \cite{Busa-Sepp-Sore-22a} that are needed for the present paper. We point out, that in \cite{Busa-Sepp-Sore-22a} and in \cite{Rahman-Virag-21}, the Busemann functions are defined differently; there, they are defined as differences of values of $\Ll$ to a terminal point traveling forward in time; i.e.,
\[
W^{\dir}(x,s;y,t)=\lim_{r\rightarrow \infty}\Ll(x,s;r\dir,r)-\Ll(y,t;r\dir,r).
\]
Due to flip symmetry of the directed landscape \cite[Lemma 10.2]{Directed_Landscape}, the results for the version of the Busemann process we use follow immediately.  In the theorem below, we only record the items needed in the present paper. 
\begin{proposition} \cite[Theorem 5.1, Theorem 5.3, Lemma 5.12]{Busa-Sepp-Sore-22a} \label{prop:Buse_basic_properties}
On the probability space of the directed landscape $\Ll$, there exists a process
\[
\Bigl(W^{\dir \sig}(x,s;y,t): \dir \in \R, \sigg \in \{-,+\}, (x,s;y,t) \in \R^4\Bigr)
\]
satisfying the following properties:
\begin{enumerate}[label=(\roman*), font=\normalfont]
\item{\rm(Continuity)} \label{itm:general_cts}  As an $\R^4 \to \R$ function,  $(x,s;y,t) \mapsto \W^{\dir \sig}(x,s;y,t)$ is  continuous. 
 \item {\rm(Additivity)} \label{itm:DL_Buse_add} For all $p,q,r \in \R^2$, 
    $\W^{\dir \sig}(p;q) + \W^{\dir \sig}(q;r) = \W^{\dir \sig}(p;r)$.   In particular, \\ $\W^{\dir \sig}(p;q) = -\W^{\dir \sig}(q;p)$ and $\W^{\dir \sig}(p;p) = 0$.
    \item {\rm(Monotonicity along a horizontal line)}
    \label{itm:DL_Buse_gen_mont} Whenever $\dir_1< \dir_2$, $x < y$, and $t \in \R$,
    \[
    \W^{\dir_1 -}(x,t;y,t) \le \W^{\dir_1 +}(x,t;y,t) \le \W^{\dir_2 -}(x,t;y,t) \le \W^{\dir_2 +}(x,t;y,t).
    \]
    \item {\rm(Evolution as the KPZ fixed point)}\label{itm:Buse_KPZ_description} For 
    all $x,y \in \R$ and $s < t$,
    \be\label{W_var}
    \W^{\dir \sig}(x,s;y,t) = \sup_{z \in \R}\{\W^{\dir \sig}(x,s;z,s) +
 \Ll(z,s;y,t)\}.
    \ee
    \item \label{it:Wslope} {\rm(Law and asymptotic slope)} For fixed $\dir \in \R$, $W^{\dir} := W^{\dir +} = W^{\dir -}$ with probability one, and for each $t \in \R$, the random function $x \mapsto W^{\dir}(0,t;x,t)$ is a two-sided Brownian motion with variance $2$ and drift $2\dir$. Furthermore, on a single event of full probability, for all $t \in \R,\dir \in \R$, and $\sigg \in \{-,+\}$,
    \[
    \lim_{|x| \to \infty} \f{W^{\dir \sig}(0,t;x,t)}{x} = 2\dir.
    \]
    \item {\rm(General Busemann limits)}\label{it:Busliminfsup} for all $\dir \in \R$, $t \in \R$, $x < y$, and sequences $\mbf v_n = (x_n,t_n)$ with $t_n \to -\infty$ and $\f{x_n}{|t_n|}\to \dir$,
\begin{align*}
W^{\dir-}(x,t;y,t) &\le \liminf_{n \to \infty} \Ll(\mbf v_n;y,t) - \Ll(\mbf v_n;x,t) \\&\le \limsup_{n \to \infty} \Ll(\mbf v_n;y,t) - \Ll(\mbf v_n;x,t) \le W^{\dir +}(x,t;y,t).
\end{align*}
\end{enumerate}
\end{proposition}

In the paper, we shall often refer to the first four properties above as \textbf{continuity, additivity, monotonicity, and evolution,} without reference to the theorem.

\begin{definition}
    We define $\DLBusedc$ to be the random set 
    \[
    \{\xi \in \R: W^{\dir -}(p;q) \neq W^{\dir +}(p;q) \quad\text{for some }p,q \in \R^2\}
    \]
    By Theorem 5.5(iii) in \cite{Busa-Sepp-Sore-22a}, with probability one, $\DLBusedc$ is countably infinite and dense in $\R$.  
\end{definition}

\subsection{Semi-infinite geodesics and Busemann geodesics} \label{sec:SIG} 
A \textbf{semi-infinite geodesic}  rooted at the point $(x,t)$ is an infinite continuous path, encoded by the function $g:(-\infty,t] \to \R$ such that for any $s < t$, the restriction $g|_{[s,t]}$ is a geodesic for $\Ll$. We say the infinite geodesic $g$ has \textbf{direction} $\dir$ if 
\[
\lim_{s \to -\infty} \f{g(s)}{|s|} = \dir.
\]
See Figure \ref{fig:SIG_directions}. We will refer to $g$ as a semi-infinite geodesic, sometimes using it to refer to the function $g:(-\infty,t] \to \R$ and at other times identifying it with its graph $\{(g(s),s):s  \in (-\infty,t]\}$

\begin{figure}
    \centering
    \includegraphics[width=0.4\linewidth]{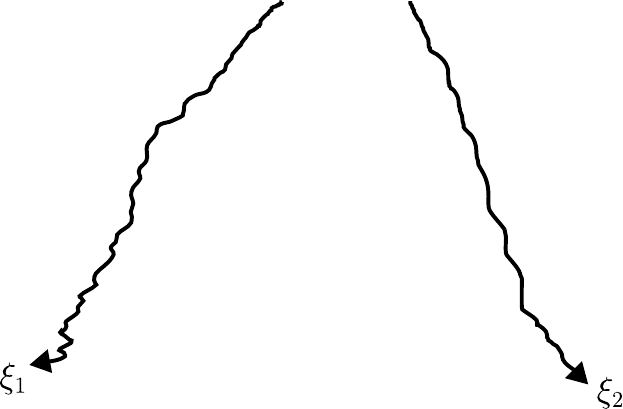}
    \caption{\small Two semi-infinite geodesics in directions $\dir_1 < \dir_2$ (larger directions are farther to the right)}
    \label{fig:SIG_directions}
\end{figure}

The existence of semi-infinite geodesics for the directed landscape was first shown in \cite{Rahman-Virag-21}. It was later shown in \cite{Busa-Sepp-Sore-22a} that there exist semi-infinite geodesics in all directions and from all points in the plane. This was done by constructing the infinite geodesics from the Busemann process. For a given point $(x,t)$, there are two particularly relevant geodesics corresponding to the Busemann function $W^{\dir \sig}$. We denote these as $g_{(x,t)}^{\dir \sig,L}$ and $g_{(x,t)}^{\dir \sig,R}$. They are defined as follows:
\[
\begin{aligned}
&g_{(x,t)}^{\dir \sig,L}(t) = g_{(x,t)}^{\dir \sig,R}(t) = x, \;\; \text{and for }s < t, \\
&g_{(x,t)}^{\dir \sig,L/R}(s) \;\;\text{is the leftmost/rightmost maximizer of} \quad z \mapsto W^{\dir \sig}(0,s;z,s) + \Ll(z,s;x,t) \quad\text{over } z \in \R.\end{aligned}
\]
For fixed $(x,t)$, with probability one, $g_{(x,t)}^{\dir \sig,L}(s) = g_{(x,t)}^{\dir \sig,R}(s)$ for all $s \le t$, $\dir \in \R$, and $\sigg \in \{-,+\}$, but there are exceptional points where this is not the case. This is the set $\NU_{\dir \sig}$ defined in \eqref{eq:NUdef}.

   \begin{lemma}\cite[Theorem 6.1]{Busa-Sepp-Sore-22a} \label{lem:NUcount}
   On a single event of probability one, for every $t \in \R$, the sets $\NU_{t,\dir-}$ and $\NU_{t,\dir +}$ are countably infinite. 
   \end{lemma}
It was later shown in \cite[Theorem 5]{Bha24} that the set of these points in the plane, $\NU_\dir$ has Hausdorff dimension $\f{4}{3}$.

Choosing the rightmost or leftmost maximizer of $z \mapsto W^{\dir \sig}(0,s;z,s) + \Ll(z,s;x,t)$ for all $s < t$ produces a geodesic, but there are potentially other maximizers as well, which also give rise to semi-infinite geodesics. In this case, care has to be made to ensure that the path one builds by choosing maximizers is actually a continuous function. The general form of these geodesics is described in the following proposition.

\begin{proposition}  \cite[Theorem 5.9]{Busa-Sepp-Sore-22a}\label{prop:DL_SIG_cons_intro}
 The following hold on a single event of probability one across all initial points $(x,s) \in \R^2$, times $s < t$, directions $\dir  \in \R$, and signs $\sigg \in \{-,+\}$.
 \begin{enumerate} [label=\rm(\roman{*}), ref=\rm(\roman{*})]  \itemsep=3pt
 \item \label{itm:intro_SIG_bd} 
 All maximizers of $z \mapsto \W^{\dir \sig}(0,s;z,s)+ \Ll(z,s;x,t)$ are finite. Furthermore, as  $x,s,t$ vary over a compact set $K\subseteq \R$ with $s < t$, the set of all maximizers is bounded.
    \item \label{itm:arb_geod_cons}  
    Let $t = s_0 > s_1 > s_2 > \cdots$ be an arbitrary decreasing sequence with $\lim_{n \to \infty} s_n = -\infty$. Set $g(t) = x$, and for each $i \ge 1$, let $g(s_i)$ be \textit{any} maximizer of $z \mapsto W^{\dir \sig}(0,s_{i};z,s_{i}) + \Ll(z,s_i; g(s_{i - 1}),s_{i - 1})$ over $z \in \R$. Then, pick \textit{any} point-to-point geodesic of $\Ll$ from $(g(s_{i}),s_{i})$ to $(g(s_{i-1}),s_{i-1})$, and for $s_{i} < s < s_{i-1}$, let $g(s)$ be the location of this geodesic at time $s$. Then, regardless of the choices made at each step, the following hold.
    \begin{enumerate} [label=\rm(\alph{*}), ref=\rm(\alph{*})]
        \item \label{itm:g_is_geod} The path $g:(-\infty,t]\to \R$ is a semi-infinite geodesic.
        \item \label{itm:weight_of_geod} For all  $ s < u \le t$,
    \be \label{eqn:SIG_weight}
    \Ll(g(s),s;g(u),u) = \W^{\dir \sig}(g(s),s;g(u),u).
    \ee
    \item \label{itm:maxes} For all  $ s < u$ in $(-\infty,s]$, $g(u)$ maximizes $z \mapsto \W^{\dir \sig}(0,s;z,s) + \Ll(z,s;g(u),u) $ over $z \in \R$. 
    \item \label{itm:geo_dir} The geodesic $g$ has direction $\dir$, i.e., $g(s)/|s| \to \dir$ as $s \to -\infty$. 
    \end{enumerate}
    \item \label{itm:DL_all_SIG} For 
    $S \in \{L,R\}$, $g_{(x,t)}^{\dir \sig,S}:(-\infty,t] \to \R$ is a semi-infinite geodesic rooted at $(x,t)$ in direction $\dir$. Moreover, for any $s < u \le t$, we have that 
    \be \label{geod_LR_eq_L}
    \Ll\bigl(g_{(x,t)}^{\dir \sig,S}(s),s;g_{(x,t)}^{\dir \sig,S}(u),u\bigr) = \W^{\dir \sig}\bigl(g_{(x,t)}^{\dir \sig,S}(s),s;g_{(x,t)}^{\dir \sig,S}(u),u\bigr),
    \ee
    and $g_{(x,t)}^{\dir \sig,S}(s)$ is the leftmost/rightmost {\rm(}depending on $S${\rm)} maximizer of \\$ z \mapsto \W^{\dir \sig}(0,s;z,s) + \Ll(z,s;g_{(x,t)}^{\dir \sig,S}(u),u)$ over $z \in \R$. In particular, if $w := g_{(x,t)}^{\dir \sig,S}(u)$, then $g_{(x,t)}^{\dir \sig,S}(s) = g_{(w,u)}^{\dir \sig,S}(s)$.
    \item \label{itm:DL_LRmost_geod} 
    The path $g_{(x,t)}^{\dir \sig,L}$ is the leftmost geodesic between any two of its points, and $g_{(x,t)}^{\dir \sig,R}$ is the rightmost geodesic between any two of its points.
    \end{enumerate}
    \end{proposition}
\begin{definition}
We refer to the  geodesics constructed  in Proposition \ref{prop:DL_SIG_cons_intro}\ref{itm:arb_geod_cons} as $\dir \sig$ \textit{Busemann geodesics}, or simply {\it $\dir \sig$ geodesics}. 
\end{definition} 

An analogue of Proposition \ref{prop:DL_SIG_cons_intro} holds for geodesics constructed from a general eternal solution $b$. This is proved in Lemma \ref{lem:dir_from_global} of Appendix \ref{appx:technical}.

Equation \eqref{eqn:SIG_weight} tells us that whenever $(z,s)$ lies along a $\dir \sig$ Busemann geodesic rooted at $(x,t)$, we have 
\[
\Ll(z,s;x,t) = W^{\dir \sig}(z,s;x,t).
\]
The following states that the converse also holds.
\begin{lemma} \cite[Lemma 5.14]{Busa-Sepp-Sore-22a} \label{lem:Buse_eq}
Let $x \in \R, s < t$, $\dir \in \R$, and $\sigg \in \{-,+\}$. Then,
$
\Ll(z,s;x,t) = W^{\dir \sig}(z,s;x,t)$
if and only if $g(s) = z$ for some $\dir \sigg$ Busemann geodesic rooted at $(x,t)$.
\end{lemma}

The main result of the first author's recent work \cite{Busani_N3G} is the following. 
\begin{proposition} \cite[Theorem 1.7]{Busani_N3G} \label{prop:Busani_N3G}
On a single event of probability one, for all $\dir \in \R$, every $\dir$-directed semi-infinite geodesic is either a $\dir-$ or $\dir +$ Busemann geodesic. 
\end{proposition}

In the following, we also make use of the following results from \cite{Busa-Sepp-Sore-22a}. The first two results are known as \textbf{ordering of geodesics.} 
\begin{proposition} \cite[Theorem 6.3, 7.3]{Busa-Sepp-Sore-22a} \label{prop:g_basic_prop}
The following hold on a single event of full probability. 
\begin{enumerate} [label=\rm(\roman{*}), ref=\rm(\roman{*})]  \itemsep=3pt
    \item \label{itm:DL_mont_dir} For $s < t$, $x \in \R$, $\dir_1 < \dir_2$, and $S \in \{L,R\}$,
    \[
    g_{(x,t)}^{\dir_1 -,S}(s) \le g_{(x,t)}^{\dir_1 +,S}(s) \le g_{(x,t)}^{\dir_2 -,S}(s) \le g_{(x,t)}^{\dir_2 +,S}(s).  
    \]
    \item \label{itm:DL_SIG_mont_x} For all $\dir \in \R$, $\sigg \in \{-,+\}$, $s < t$ and $x < y$, 
    \[
    g_{(x,t)}^{\dir \sig,R}(s) \le g_{(y,t)}^{\dir \sig,L}(s).
    \]
    \item \label{itm:DL_SIG_conv_x} For all $\dir \in \R$, $\sigg \in \{-,+\}$, $S \in \{L,R\}$, $x \in \R$, and $s < t$,
    \be \label{371}
    \lim_{w \nearrow x} g_{(w,t)}^{\dir \sig,S}(s) = g_{(x,t)}^{\dir \sig,L}(s),\qquad\text{and}\qquad \lim_{y \searrow x} g_{(y,t)}^{\dir \sig,S}(s) = g_{(x,t)}^{\dir \sig,R}(s),
    \ee
    \item \label{itm:eventually_less} If $\dir \in \DLBusedc$, then for any $\dir -$ geodesic $g_1$ and any $\dir +$ geodesic $g_2$ (rooted any any two points), there exists $S \in \R$ so that for all $s \le S$,
    \[
    g_1(s) < g_2(s).
    \]
\end{enumerate}
\end{proposition}

We say that two semi-infinite geodesics $g_1$ and $g_2$ (rooted at points $(x,t)$ and $(y,u)$; potentially with $u \neq t$) \textbf{coalesce} if there exists $S \in \R$ so that $g_1(s) = g_2(s)$ for all $s \le S$. If $S$ is the maximal such time, then we call $S$ the \textbf{coalescence time}, and $(S,g_1(S))$ the \textbf{coalescence point}. 

The next result states that geodesics with the same direction $\dir$ and sign $\sigg \in \{-,+\}$ coalesce. 

\begin{proposition} \cite[Theorem 7.1]{Busa-Sepp-Sore-22a}\label{prop:DL_all_coal}
On a single event of full probability, the following hold across all directions $\dir \in \R$ and signs $\sigg \in \{-,+\}$. Let $K$ be a compact set. Then, there exists $S < \inf\{t:(x,t) \in K\}$ such that, for any $\dir \sig$ geodesics $g_1$ and $g_2$ rooted at points in $K$, we have $g_1(s) = g_2(s)$ for all $s \le S$. 
\end{proposition}

The next result, due to Bhatia, states there are no bi-infinite geodesics in the directed landscape. 
\begin{proposition}\cite[Proposition 34]{Bha24} \label{prop:bigeod_non}
With probability one, there are no bi-infinite geodesics in the directed landscape. That is, there are no continuous functions $g:\R\to\R$ such that $g|_{(-\infty,t]}$ is a semi-infinite geodesic for all $t \in \R$. 
\end{proposition}

\subsection{Busemann difference profile}
We often make use of what we call the \textbf{Busemann difference profile.} For $\dir\in\DLBusedc$ and $t\in\R$ we define the nondecreasing function $x\mapsto  D_t^\dir(x)$ by
   \begin{equation}\label{Dfunc}
       D_t^\dir(x) = W^{\dir+}(0,t;x,t)-W^{\dir-}(0,t;x,t).
   \end{equation}
   The fact that this function is nondecreasing is a corollary of the additivity and monotonicity of Proposition \ref{prop:Buse_basic_properties}. Indeed, for $y > x$,
\begin{align*}
D_s^\dir(y) - D_s^\dir(x) = W^{\dir +}(x,s;y,s) - W^{\dir -}(x,s;y,s) \ge 0.
\end{align*}
The function $D_t^\dir$ is uniformly $0$ whenever $\dir \notin \DLBusedc$. When $\dir \in \DLBusedc$, we have the following:
\begin{lemma}\cite[Theorem 5.5(ii)]{Busa-Sepp-Sore-22a} \label{lem:dif_to_infinity}
On a single event of probability one, $\dir \in \DLBusedc$, if and only if $\lim_{x \to \pm \infty} D_t^\dir(x)= \pm \infty$.
\end{lemma}

The next result characterizes the sets $\Split^L_{t,\dir}, \Split^M_{t,\dir}$ and $\Split^R_{t,\dir}$ from \eqref{Split_sdir} via the nondecreasing function $D_t^\dir$. Theorem 8.2 in \cite{Busa-Sepp-Sore-22a}, states that 
\[
\Split_{t,\dir}^L \cup \Split_{t,\dir}^R = \{x \in \R: D_t^\dir(x - \ve) < D_t^\dir(x + \ve) \text{ for all }\ve > 0\},
\]
but the proof of this theorem actually shows the following stronger result:
\begin{lemma} \cite[Theorem 8.2]{Busa-Sepp-Sore-22a} \label{l:equ}
    On a single event of full probability, for every $\dir\in\DLBusedc$ and $t\in\R$,
    \begin{equation}
        \begin{aligned}
            \Split^L_{t,\dir}&=\{x\in\R:D^\dir_t(x-\varepsilon) < D^\dir_t(x), \,\, \forall \varepsilon>0\}\\
            \Split^R_{t,\dir}&=\{x\in\R:D^\dir_t(x) < D^\dir_t(x+\varepsilon), \,\, \forall \varepsilon>0\}.
        \end{aligned}
    \end{equation}
    In particular,
    \[
    \Split^M_{t,\dir} = \{x\in\R:D^\dir_t(x-\varepsilon) < D^\dir_t(x) < D^\dir_t(x + \ve), \,\, \forall \varepsilon>0\}.
    \]
\end{lemma}
In Proposition \ref{theorem: local variation and splitting points}, we show that $\Split_{s,\dir} = \Split_{t,\dir}^L \cup \Split_{t,\dir}^R$. Therefore, $\Split_{t,\dir}$ is the set of points of increase of $D_t^\dir$; namely
\[
\Split_{t,\dir} = \{x\in\R:D^\dir_t(x-\varepsilon) < D^\dir_t(x + \ve), \,\, \forall \varepsilon>0\}.
\]
Along a fixed horizontal line, we know the Hausdorff dimension of the set $\Split_{t,\dir}$:
\begin{proposition} \cite[Theorem 2.10]{Busa-Sepp-Sore-22a} \label{prop:Haus12}
    Fix $t \in \R$. Then, for all $\dir \in \DLBusedc$, on a (potentially $t$-dependent, but not $\dir$-dependent) probability one event, the Hausdorff dimension of $\Split_{t,\dir}$ is $\f{1}{2}$. In particular, $\Split_{t,\dir}$ is uncountable.
\end{proposition}
\begin{remark} \label{rmk:countable}
It is  fairly straightforward to see that the set $\Split_{t,\dir}^L \triangle \Split_{t,\dir}^R$ is countably infinite. Indeed, by Lemma \ref{l:equ}, this set consists of the endpoints of the intervals $\Intr_{\alpha}^{t,\dir} = [\a^t_\alpha,\b^t_\alpha]$ such that $\a^t_\alpha < \b^t_\alpha$, and there can be only countably many such intervals in $\R$ with mutually disjoint interiors. Since $\Split_{t,\dir}= \Split_{t,\dir}^M \cup \bigl(\Split_{t,\dir}^L \triangle \Split_{t,\dir}^R\bigr)$, the set $\Split_{t,\dir}^M$ also has Hausdorff dimension $\f{1}{2}$.
\end{remark}

\subsection{Event of probability one}
For the remainder of this paper, we work on a single probability one event, which we call $\Omega_1$. Given the necessary probabilistic inputs, the results of this paper are deterministic statements. We define $\Omega_1$ as the event on which the metric composition property for the directed landscape holds, geodesics exist, intersected with the full probability events from all results in Section \ref{sec:prelim} (except Proposition \ref{prop:Haus12}, where we just take the full probability event for $t = 0$), and the full probability events from Lemmas \ref{lem:Landscape_global_bound}, \ref{lem:precompact}, \ref{lem:overlap}, \ref{lem:SIGprecompact}, and \ref{lem:bounded_maxes}.

\section{Mixed \texorpdfstring{$\dir$-}{} Busemann interfaces and the existence of bi-infinite interfaces  } \label{sec:interface_start}
In this section, we develop the fundamental algebraic properties of the mixed $\dir$-Busemann interfaces. The lemmas below are the basis for Proposition \ref{p:rest}, which give a key consistency property for these interfaces. This consistency property comes as a corollary of a semigroup property for more general interfaces, given in Lemma \ref{lem:MP}. The consistency property is then used in Lemma \ref{l:LRSP} and Proposition \ref{p: existence of bi-infinite competetion interfaces} to show the existence of bi-infinite interfaces. 

 Recall that in \eqref{eq: initial condition} for $(x_0,s) \in \R^2$ and $\xi\in \R$, we defined the mixed Busemann initial condition $f^\dir_{(x_0,s)}$. Further, recall that a function $f: \R \rightarrow \R \cup \{ -\infty \}$ is said to be an initial condition if it is upper semi-continuous and satisfies \eqref{eq:IC}.
\begin{lemma}
\label{lemma: valid initial condition}
On the event $\Omega_1$, for all $\dir\in \R$ and $(x_0,s) \in \R^2$, $f^{\dir}_{(x_0,s)}$ is an initial condition.
\end{lemma}
\begin{proof}
     Recall that $z \mapsto W^{\dir \sig}(x_0,s;z,s)$ is continuous. Moreover, \[
     W^{\dir+}(x_0,s;x_0,s)=W^{\dir-}(x_0,s;x_0,s)=0,
     \]
     which implies that $f^{\dir}_{(x_0,s)}$ is continuous. As Busemann functions are finite, $f^{\dir}_{(x_0,s)}$ is finite everywhere. Finally, by  Proposition \ref{prop:Buse_basic_properties}\ref{it:Wslope}, for all $(x_0,s) \in \R^2, \dir \in \R$ and $\sig \in \{+,-\}$
    \begin{align*}
    \lim_{|z|\rightarrow \infty} \frac{W^{\dir\sig}(x_0,s;z,s)}{z}=2\dir.
    \end{align*} 
    Hence, $f^\dir_{(x_0,s)}$  satisfies \eqref{eq:IC}. This completes the proof.
\end{proof}
Recall the definitions of the interfaces $\intc^{\dir-}_{(x_0,s)}, \intc^{\dir,+}_{(x_0,s)}$ from \eqref{eq:Buse_interface} and $d^\xi_{(x_0,s)}(x,t)$ and $\kpzs_{(x_0,s),t}^\dir$ in \eqref{kpzs_abbr}. We record the following two lemmas and cite results in \cite{Rahman-Virag-21}. We make the note here that the results cited in \cite{Rahman-Virag-21} give a probability one event for a fixed initial condition, while we are working with a collection of initial conditions. It suffices to use the modulus of continuity bounds for the directed landscape from Lemma \ref{lem:Landscape_global_bound} and a linear bound for the initial condition. By the asymptotic slopes and monotonicity in Proposition \ref{prop:Buse_basic_properties}\ref{itm:DL_Buse_gen_mont}-\ref{it:Wslope}, on the event $\Omega_1$, each initial condition $f_{(x_0,s)}^{\dir}$ satisfies a bound
\[
|f_{(x_0,s)}^{\dir}(x)| \le A + B|x|
\]
for $A,B > 0$ (depending on $\dir$ and $(x_0,s)$). This suffices for the following two results.   
\begin{lemma}\cite[Lemma 3.3 and Proposition 4.1]{Rahman-Virag-21}
\label{l:continuity of d}
On the event $\Omega_1$, for all $\dir\in \R, (x_0,s) \in \R^2$, and $t>s$, the function $x \mapsto d^{\dir}_{(x_0,s)}(x,t)$ is non-decreasing and continuous. 
\end{lemma}

\begin{lemma}[Interfaces are continuous]\cite[Propositions 4.4-4.6]{Rahman-Virag-21} \label{lem:interface_continuous}
On the event $\Omega_1$, for all $\dir\in \R, (x_0,s) \in \R^2$, and $t > s$, $\intc_{(x_0,s)}^{\dir,-}(t)$ and $\intc_{(x_0,s)}^{\dir,+}(t)$ are finite, and the functions
$
\intc^{\dir,\pm}_{f;(x_0,s)}: [s,\infty) \to \R
$
are continuous. 
\end{lemma}

\begin{lemma}
\label{l:MV}
    On the event $\Omega_1$, for all $\dir \in \DLBusedc, (x_0,s) \in \R^2$, and $t>s$, if $f=f^\dir_{(x_0,s)}, \kpzs_{t} = \kpzs_{(x_0,s),t}^\dir$, then
    \begin{enumerate}[label=\textup{(\roman*)}]
    \item \label{it:Bu1}For all $x \geq \intc_{(x_0,s)}^{\dir,-}(t)$, 
    \[
     \kpzs_t(x)=\sup_{z \ge x_0}\{f(z)+\Ll(z, s;x,t)\}=\sup_{z\in\R}\{W^{\dir +}(x_0,s; z, s)+\Ll(z, s;x,t)\}=W^{\dir +}(x_0,s; x,t),
    \]
    \item\label{it:Bu2} For all $x \leq \intc_{(x_0,s)}^{\dir,+}(t)$,
    \[
     \kpzs_t(x)=\sup_{z \le x_0}\{f(z)+\Ll(z, s;x,t)\}=\sup_{z\in\R}\{W^{\dir -}(x_0,s; z, s)+\Ll(z, s;x,t)\}=W^{\dir -}(x_0,s; x,t).
    \]
    \item \label{it:Bu3}
In particular, for $\intc_{(x_0,s)}^{\dir,-}(t)\leq x\leq \intc_{(x_0,s)}^{\dir,+}(t)$,
\begin{equation}
    \kpzs_t(x)=W^{\dir +}(x_0,s; x,t)=W^{\dir -}(x_0,s; x,t).
\end{equation}
    \end{enumerate}
\end{lemma}
\begin{proof}
   We prove Item \ref{it:Bu1}, and Item \ref{it:Bu2} is symmetric, while Item \ref{it:Bu3} follows immediately from the other two items. Assume that $x\geq \intc_{(x_0,s)}^{\dir,-}(t)$. Then, from the definition of the interface \eqref{eq:taupm_def}, we have 
    \begin{equation}
    \begin{aligned}
    \label{eq:RmaxR}
        \kpzs_t(x)&=\sup_{z\geq x_0}\{f(z)+\Ll(z, s;x,t)\}\geq \sup_{z \le x_0}\{f(z)+\Ll(z, s;x,t)\}\\
        &\qquad = \sup_{z\le x_0}\{W^{\dir -}(x_0, s;z,s)+\Ll(z, s;x,t)\}\geq \sup_{z\le x_0}\{W^{\dir +}(x_0,s; z,s)+\Ll(z, s;x,t)\},
        \end{aligned}
    \end{equation}
    where the second equality follows from the definition of $f = f_{(x_0,s)}^{\dir}$ \eqref{eq: initial condition}, and the inequality from the monotonicity of the Busemann functions, using 
    \[
    W^{\dir \sig}(x_0,s;z,s) = -W^{\dir \sig}(z,s;x_0,s).
    \]
    
    Moreover, since $f(z) = W^{\dir +}(x_0,s;z,s)$ for $z \ge x_0$, the first equality and last inequality in \eqref{eq:RmaxR} imply that
    \[
    \begin{aligned}
        \kpzs_t(x)&= \sup_{z\ge x_0}\{W^{\dir +}(x_0,s; z, s)+\Ll(z, s;x,t)\} \\
        &=\sup_{z\in\R}\{W^{\dir +}(x_0,s; z, s)+\Ll(z, s;x,t)\}=W^{\dir +}(x_0,s; x,t).
        \end{aligned}\qedhere
    \]
\end{proof}

Our next result shows that up to an additive constant, the function $x \mapsto \kpzs^\dir_t(x)$ can be determined from $\intc_{(x_0,s)}^{\dir, \sig}(t)$ and $f_{(\cdot,\cdot)}^\dir$.
\begin{lemma}\label{lem:4}
    On the event $\Omega_1$, for all $\dir\in\R, (x_0,s)\in\R^2$, and $\sigg \in \{-,+\}$, if $\intc_t^\sig = \intc_{(x_0,s)}^{\dir,\sig}(t)$, then
    \begin{equation}
        f_{(\intc_t^\sig,t)}^{\dir}(x) = \kpzs^\dir_{(x_0,s),t}(x)-\kpzs_{(x_0,s),t}^\dir(\intc^{\sig}_t), \qquad \forall \,\, x \in \R,\,\, t>s.
    \end{equation}
\end{lemma}
\begin{proof}
    We show the claim for $\sigg = -$, the proof for $\sigg = +$ is symmetric. We first assume that  $x\geq \intc^-_t$. By definition of $f_{(\intc_t^-,t)}^{\dir}$ \eqref{eq: initial condition} and additivity of the Busemann functions, we have 
    \begin{equation}
        f_{(\intc_t^-,t)}^{\dir}(x) =W^{\xi+}(\intc^-_t, t;x,t) = W^{\dir +}(x_0,s;x,t) - W^{\dir-}(x_0,s;\intc_t^-,t).
    \end{equation}
    Then, by Lemma \ref{l:MV}\ref{it:Bu1}, this is 
    \[
    \kpzs^\dir_{(x_0,s),t}(x)-\kpzs_{(x_0,s),t}^\dir(\intc^{\sig}_t),
    \]
    as desired. 
\end{proof}
 Our next result states that left and right interfaces satisfy a semigroup property.

\begin{lemma}\label{lem:MP}
            Let $(x_0,s)\in\R^2$, and let $f$ be a continuous initial condition satisfying $f(x) \le a + b|x|$ for some constants $a,b >0$. Assume that $\Ll$ satisfies the metric composition property of the directed landscape, as well as the modulus of continuity bounds in Lemma \ref{lem:Landscape_global_bound}. Let $\kpzs_t(x):=\kpzs_t(x,f,s)$ be the KPZ fixed point solution at time $t$ with initial data $f$, and for $\sigg \in \{-,+\}$, let $\intc^{\sig}_t := \intc_{f;(x_0,s)}^\sig(t)$ be the interface defined as in \eqref{eq:taupm_def}. Then, for any $s<r<t$,
            \begin{equation}
                \intc_t^\sig =\intc^{\sig}_{\kpzs_r;(\intc^{\sig}_{r},r)}(t).
            \end{equation}
        \end{lemma}
        \begin{proof}
        We prove the result for $\sigg = -$, and the proof for $\sigg = +$ is symmetric. The assumptions on $f$ and $\Ll$ ensure that the following are well-defined. Fix $t>r>s$. It suffices to show that
        \begin{equation}\label{eq35}
            d^\dir_{(\intc^-_{r},r)}(\kpzs_r;x,t)< 0 \iff d^\dir_{(x_0,s)}(f;x,t)< 0, \qquad \forall x\in\R.
        \end{equation}
       First, make the following definitions:
\begin{align*}
A := \sup_{w\geq x_0}\Bigg[f(w)+\sup_{z\geq \intc^-_{r}}[\Ll(w,s;z,r)+\Ll(z,r;x,t)]\Bigg] \\
B := \sup_{w \le x_0}\Bigg[f(w)+\sup_{z \le \intc_r^-}[\Ll(w,s;z,r)+\Ll(z,r;x,t)]\Bigg]  \\
C := \sup_{w \ge x_0} \Biggl[f(w) + \sup_{z \le \intc_r^-}[\Ll(w,s;z,r)+\Ll(z,r;x,t)] \Biggr]\\
D := \sup_{w \le x_0} \Bigg[f(w)+\sup_{z\geq \intc^-_{r}}[\Ll(w,s;z,r)+\Ll(z,r;x,t)]\Bigg]. 
\end{align*}
Since the function $z \mapsto d_{(x_0,s)}(f;z,r)$ is nondecreasing,  the definition of the interfaces \eqref{eq:taupm_def} implies that $d_{(x_0,s)}(f;z,r) \le 0$ for $z \le \intc_r^-$ and $d_{(x_0,s)}(f;z,r) \ge 0$ for $z \ge \intc_r^-$. This further implies that 
\begin{equation} \label{ABCD_ineq}
 A \ge D,\qquad \text{and}\qquad  B \ge C,
\end{equation}
and in particular, 
\begin{equation} \label{ABeq}
\begin{aligned}
A &= \sup_{w \in \R}\Bigg[f(w)+\sup_{z\geq \intc^-_{r}}[\Ll(w,s;z,r)+\Ll(z,r;x,t)]\Bigg],\qquad\text{and} \\
B &=\sup_{w \in \R}\Bigg[f(w)+\sup_{z < \intc^-_{r}}[\Ll(w,s;z,r)+\Ll(z,r;x,t)]\Bigg]
\end{aligned}
\end{equation}
By definition,
        \begin{equation}
            d^\dir_{(\intc^-_{r},r)}(\kpzs_r;x,t)< 0\iff \sup_{z\geq \intc^-_{r}}[\kpzs_r(z)+\Ll(z,r;x,t)]< \sup_{z \le  \intc^-_{r}}[\kpzs_r(z)+\Ll(z,r;x,t)],
        \end{equation}
        and the right-hand side is equivalent to 
        \[
             \sup_{z\geq \intc^-_{r}}\Biggl[\sup_{w\in \R} \bigl[f(w)+\Ll(w,s;z,r)\bigr]+\Ll(z,r;x,t)\Biggr] < \sup_{z \le \intc^-_{r}}\Biggl[\sup_{w\in\R} \bigl[f(w)+\Ll(w,s;z,r)\bigr]+\Ll(z,r;x,t)\Biggr],
        \]
        which, by \eqref{ABeq}, is equivalent to 
        \begin{equation} \label{eq54}
        A < B.
        \end{equation}  
        
It suffices to show that \eqref{eq54} is equivalent to 
\begin{equation}\label{eq55}
            \sup_{w\geq x_0}\Bigg[f(w)+\sup_{z\in\R}[\Ll(w,s;z,r)+\Ll(z,r;x,t)]\Bigg]<\sup_{w\leq x_0}\Bigg[f(w)+\sup_{z\in\R}[\Ll(w,s;z,r)+\Ll(z,r;x,t)]\Bigg].
        \end{equation}
        Indeed, it follows immediately from the metric composition property of $\Ll$, that \eqref{eq55} is equivalent to 
        \begin{equation}
            \sup_{w\geq x_0}\big[f(w)+ \Ll(w,s;x,t)\big]< \sup_{w\leq x_0}\big[f(w)+ \Ll(w,s;x,t)\big],
        \end{equation}
        which in turn is equivalent to $d^\dir_{(x_0,s)}(f;x,t)< 0$. 

        We first observe from the definitions of $A,B$, and $C$, that
       \begin{equation} \label{AveeB}
        \begin{aligned}
            A \le \sup_{w\geq x_0}\Bigg[f(w)+\sup_{z\in \R}[\Ll(w,s;z,r)+\Ll(z,r;x,t)]\Bigg]
            =A\vee C \leq A\vee B, 
        \end{aligned}
        \end{equation}
        where the second inequality holds by \eqref{ABCD_ineq}.  Similarly,
        \begin{equation}\label{eq57}
        \begin{aligned}
            B\le \sup_{w\leq x_0}\Bigg[f(w)+\sup_{z\in \R}[\Ll(w,s;z,r)+\Ll(z,r;x,t)]\Bigg]
            = B \vee D \leq B\vee A.
        \end{aligned}
        \end{equation}

\medskip \noindent \textbf{Proof of }
\eqref{eq55} $\implies$ \eqref{eq54}: 

        We prove the contrapositive. Assume that $A \ge B$. Then, $A \vee B = A$, so the inequalities in \eqref{AveeB} are equalities. In particular,
        \be \label{Asup}
        A = \sup_{w\geq x_0}\Bigg[f(w)+\sup_{z\in \R}[\Ll(w,s;z,r)+\Ll(z,r;x,t)]\Bigg].
        \ee
        On the other hand, \eqref{eq57} implies that 
        \[
             A\geq \sup_{w\leq x_0}\Bigg[f(w)+\sup_{z\in \R}[\Ll(w,s;z,r)+\Ll(z,r;x,t)]\Bigg],
        \]
        which, along with \eqref{Asup}, implies that \eqref{eq55} fails.

        \medskip \noindent \textbf{Proof of }
         \eqref{eq54} $\implies$ \eqref{eq55}: Assume $B > A$. Then, $B \vee A = B$, and \eqref{eq57} implies that
        \begin{equation}\label{beq}
             B= \sup_{w\leq x_0}\Bigg[f(w)+\sup_{z\in \R}[\Ll(w,s;z,r)+\Ll(z,r;x,t)]\Bigg].
        \end{equation}
        It suffices to show that $B > C$. Indeed, if this is the case, then $B > A \vee C$, and  the second inequality in  \eqref{AveeB} is strict, which gives 
        \begin{equation}
           B> \sup_{w\geq x_0}\Bigg[f(w)+\sup_{z\in \R}[\Ll(w,s;z,r)+\Ll(z,r;x,t)]\Bigg].
        \end{equation}
        The last display along with \eqref{beq} implies \eqref{eq55}. To complete the proof, we show that $B > A$ implies $B > C$. Assume, by way of contradiction, that this is not the case. Then, by \eqref{ABCD_ineq}, we have $B = C > A$, so by definition of $C$ and $A$, 
        \be \label{eq:C>Asup}
        \begin{aligned}
        &\sup_{z \le \intc_r^-} \Bigl[\sup_{w \ge x_0} [f(w) + \Ll(w,s;z,r)]+\Ll(z,r;x,t) \Bigr] = C  \\
        &\qquad\qquad\qquad\qquad> A \ge \sup_{w \ge x_0} [f(w) + \Ll(w,s;\intc_r^-,r)] + \Ll(\intc_r^-,r;x,t).
        \end{aligned}
        \ee
        We note that the function $z \mapsto \sup_{w \ge x_0}[f(w) + \Ll(w,s;z,r)] + \Ll(z,r;x,t)$ is continuous. This follows by the continuity and growth assumptions on $f$ and $\Ll$, which imply that the supremum can be taken over a common compact set as $z$ varies over a compact set. Then, by \eqref{eq:C>Asup} and the continuity in Lemma \ref{lem:sup_continuous}, there exists $\delta > 0$ so that 
        \[
        C = \sup_{z \le \intc_r^- - \delta} \Bigl[\sup_{w \ge x_0} [f(w) + \Ll(w,s;z,r)]+\Ll(z,r;x,t) \Bigr].
        \]
        Then, by definition of $\intc_r^-$, $d_{(x_0,s)}(f;z,r) < 0$ for all $z \le \intc_r^- - \delta$, so 
        \begin{align*}
        C &= \sup_{z \le \intc_r^- - \delta} \Bigl[\sup_{w \ge x_0} [f(w) + \Ll(w,s;z,r)]+\Ll(z,r;x,t)\Bigr] \\ 
        &< \sup_{z \le \intc_r^- - \delta} \Bigl[\sup_{w \le x_0} [f(w) + \Ll(w,s;z,r)]+\Ll(z,r;x,t)\Bigr] \le B,
        \end{align*}
        where the last inequality follows because of the growth conditions on $f$ and $\Ll$. which allow the supremum over $z$ to be taken over a compact set. This gives a contradiction to the assumption $B = C$. 
        \end{proof}
        Our next result shows that the restriction of a mixed $\dir$-Busemann interface is also a mixed $\dir$-Busemann interface.
\begin{proposition}
\label{p:rest}
On the event $\Omega_1$, for all $\dir \in \R$, $x_0 \in \R$ and $s \le r \le t$, 
\[
\intc_{(x_0,s)}^{\dir,-}(t) = \intc_{(\intc^-_r,r)}^{\dir,-}(t),\quad\text{where}\quad \intc_r^- := \intc_{(x_0,s)}^{\dir,-}(r).
\]
The statement also holds when $-$ is replaced with $+$.
    \end{proposition}
    \begin{proof}
         The cases $r = s$ and $r = t$ are immediate from the definition \eqref{intcat0}, so we assume $s < r < t$. Let $f = f^\dir_{(x_0,s)}$ be the initial condition from \eqref{eq: initial condition}  associated with the point $(x_0,s)$, and let  $\kpzs_t = \kpzs^\dir_{(x_0,s),t}$ be its KPZ fixed point solution at time $t$ as in \eqref{kpzs_abbr}. By Lemma \ref{lem:MP}, for $s < r < t$,
         \be \label{eq:taueq1}
         \intc^{\dir,-}_{(x_0,s)}(t)=\intc^{-}_{\kpzs_r;(\intc^-_{r},r)}(t).
         \ee
        Next, by Lemma \ref{lem:4},  we have that, for all $x \in \R$,
        \[
        \kpzs_r(x) =f_{(\tau_r^-,r)}^\dir(x) + \kpzs_{r}(\tau_r^-)
        \]
        Hence, the functions $x \mapsto \kpzs_r(x)$ and $x \mapsto f_{(\tau_r^-,r)}(x)$ differ by the addition of a constant. We see immediately from the definition \eqref{eq:taupm_def} that interfaces are invariant with respect to additions of constants of the initial functions, so from \eqref{eq:taueq1}, we have 
        \[
        \intc^{\dir,-}_{(x_0,s)}(t)=\intc^-_{\kpzs_r;(\intc^-_{r},r)}(t) = \intc^-_{f^\dir_{(\tau_r^-r)};(\tau_r^-,r)}(t) = \intc_{(\intc^-(r),r)}^{\dir,-}(t),
        \]
        as desired. The proof for $-$ replaced with $+$ is symmetric. 
    \end{proof}
The following two results shows that for any point in $x \in \Split^L_{t,\dir}$ (resp. $\Split^R_{t,\dir}$), there exists an element $\intc \in \Intc^{\dir,-}$ (resp. $\Intc^{\dir,+}$) with $\intc(t) = y$.   
 \begin{lemma}
     \label{l:LRSP}
    On the event $\Omega_1$, the following holds for all $\dir \in \DLBusedc$.
    \begin{enumerate}[label=(\roman*), font=\normalfont]
        \item \label{it5}  Let $t\in\R$ and  $x \in \Split^L_{t,\dir}$.   Then for any $s<t$, there exists  $x_0 \in \Split^L_{s,\dir}$ such that
        \[
        \intc^{\dir,-}_{(x_0,s)}(t)=x.
        \] 
        \item \label{it6} Let $t\in\R$ and  $x \in \Split_{s,\dir}^R$.  Then for any $s<t$, there exists $x_0 \in \Split^R_{t,\dir}$ such that
        \[
       \intc^{\dir,+}_{(x_0,s)}(t)=x.
         \]
        \end{enumerate}
 \end{lemma}

\begin{proof}
     We prove Item \ref{it5}, and the proof of Item \ref{it6} is similar.
   Fix  $x \in \Split^L_{t,\dir}$ and $s<t$. By Lemma \ref{l:equ}, for all $\varepsilon >0, D_{t}^\dir(x-\varepsilon)<D_{t}^\dir(x)$.  By definition of $ \Split_{t,\dir}^L$ \eqref{Split_sdir}, we have 
    \begin{equation}
 b_2^-:=g^{\dir-,L}_{(x,t)}(s)<g^{\dir+,L}_{(x,t)}(s)=:b_2^+.    
 \end{equation}
We now claim that 
\be \label{eq:-<+}
\Ll(b_2^-,s;x,t)=W^{\dir-}(b_2^-,s;x,t)<W^{\dir+}(b_2^-,s;x,t).
\ee
Indeed,  the equality holds by \eqref{geod_LR_eq_L}. To see that the inequality holds,  note that, otherwise, Lemma \ref{lem:Buse_eq} would imply that $b_2^-$ is a maximizer of $z \rightarrow W^{\dir+}(0,s;z,s)+\Ll(z,s;x,t)$. However, this cannot be the case since $b_2^-<b_2^+$, and $b_2^+$ is the leftmost maximizer   of $z \rightarrow W^{\dir+}(0,s;z,t)+\Ll(z,s;x,t)$ by definition.
Similarly,
\be \label{eq:+<-}
\Ll(b_2^+,s;x,t)=W^{\dir+}(b_2^+,s;x,t) \le W^{\dir-}(b_2^+,s;x,t),
\ee
where here, we only have a weak inequality because $b_2^+$ is to the right of the leftmost maximizer of $z \mapsto W^{\dir -}(0,s;z,t) + \Ll(z,t;x,t)$, so $b_2^+$ could be a maximizer. 
Consider the function 
\[
y \rightarrow \widetilde{F}(y): =W^{\dir+}(y,s;x,t)-W^{\dir-}(y,s;x,t).
\]
Note that $\widetilde{F}$ is continuous by continuity of the Busemann functions, and by \eqref{eq:-<+} and \eqref{eq:+<-}, $\widetilde{F}(b_2^+)\le 0$ and $\widetilde{F}(b_2^-)>0$. Hence,  there exists $b_2 \in (b_2^-,b_2^+]$ such that 
$
W^{\dir+}(b_2,s;x,t)=W^{\dir-}(b_2,s;x,t)
$.
Then, set
\be \label{x0}
x_0=\inf \{b \in (b_2^-,b_2^+]: \widetilde{F}(b) = 0\}.
\ee
Since $\widetilde{F}(b_2^-) > 0$, continuity implies that $x_0 \in (b_2^-,b_2^+]$ and $\widetilde{F}(x_0) = 0$.
For our particular choice of $\dir$, and for $s \in \R$, recall the definition
\begin{equation}
x \rightarrow D_{s}^\dir(x)=W^{\dir+}(0,s;x,s)-W^{\dir-}(0,s;x,s).  \end{equation}
\begin{figure}[t!]
    \includegraphics[width=7 cm]{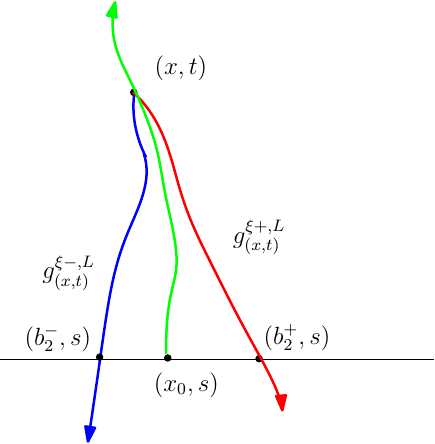}
    \caption{To prove Lemma \ref{l:LRSP}, we fix $(x,t) \in \Split^L_{t,\dir}$ and $s<t$. We argue that we can find $b \in (b_2^-,b_2^+]$ such that $\widetilde{F}(b)=0$. $x_0$ is minimum of such $b$. By the construction the leftmost competition interface (in green) from $(x_0,s)$ passes through $(x,t).$ }
    \label{fig:existence of bi-infinite competetion interface}
\end{figure}
Observe that, using additivity of the Busemann functions, for $b \in \R$, 
\begin{equation}\label{eq37}
 \begin{aligned}
 \widetilde{F}(b) = 0 &\iff W^{\dir+}(b,s;x,t)=W^{\dir-}(b,s;x,t)\\
 &\iff W^{\dir+}(b,s;0,s)+W^{\dir+}(0,s;x,t)=W^{\dir-}(b,s;0,s)+W^{\dir-}(0,s;x,t)\\
 &\iff -W^{\dir+}(0,s;b,s)+W^{\dir+}(0,s;x,t)=-W^{\dir-}(0,s;b,s)+W^{\dir-}(0,s;x,t).\\
 &\iff D_{s}^\dir(b)=W^{\dir+}(0,s;x,t)-W^{\dir-}(0,s;x,t).
\end{aligned}   
\end{equation}

We now claim that  $D^\dir_{s}(x_0-\varepsilon)<D^\dir_{s}(x_0)$, for all $\varepsilon >0$. Indeed, if this were not the case, then by monotonicity of $D_s^\dir$, there exists $y<x_0$  such that $D_s^\dir(b) = D_s^\dir(x_0)$ for $b \in [y,x_0]$, and in particular, $D_s^\dir(b) = D_s^\dir(x_0)$ for some $b \in (b_2^-,x_0)$. Then by \eqref{eq37}, since $\widetilde{F}(x_0) = 0$, we have 
\begin{align*}
    D_{s}^\dir(b)=W^{\dir+}(0,s;x,t)-W^{\dir-}(0,s;x,t) \implies W^{\dir+}(b,s;x,t)=W^{\dir-}(b,s;x,t) \implies \widetilde{F}(b) = 0.
\end{align*}
But as $b\in (b_2^-,x_0)$, this contradicts the definition of $x_0$ \eqref{x0}.  Hence, we indeed have that $D^\dir_{s}(x_0-\varepsilon)<D_{s}^\dir(x_0)$ for all $\varepsilon >0$.
By Lemma \ref{l:equ}, $x_0 \in \Split^L_{s, \dir}.$ It remains to show that 
 \begin{equation}\label{eq44}
     \intc^{\dir,-}_{(x_0,s)}(t)=x.
 \end{equation}
     Let $y: =\intc^{\dir,-}_{(x_0,s)}(t)$, and suppose, by way of contradiction, that $y \neq x$. We consider two cases.

    \medskip \noindent \textbf{Case 1:} $y < x$.
        By the definition of $x_0$ and continuity of Busemann functions, we know that 
        \[
        W^{\dir+}(x_0,s;x,t)=W^{\dir-}(x_0,s;x,t).
        \]
        On the other hand, by Lemma \ref{l:MV}\ref{it:Bu3},
        \[
        W^{\dir+}(x_0,s,y,t)=W^{\dir-}(x_0,s,y,t).
        \]
       By subtracting these two equations and rearranging using additivity, we get
        $
        D_{t}^\dir(x)=D_{t}^\dir(y).
        $
        However, the assumption $x \in  \Split^L_{t, \dir}$ implies $D_t^\dir(y) < D_t^\dir(x)$ for all $y < x$ by Lemma \ref{l:equ}. 

        \medskip \noindent \textbf{Case 2:}  $y > x$: Then,
        \begin{align*}
        &\sup_{z \in \R}\{f_{(x_0,s)}^\dir(z)+\Ll(z,s;x,t)\}\\
        &=W^{\dir-}(x_0,s;x,t) \qquad \text{ by Lemma \ref{l:MV}\ref{it:Bu2}}\\
        &=W^{\dir+} (x_0,s;x,t) \qquad \text{ by definition of } x_0\\
        &=W^{\dir+}(x_0,s;b_2^+,s)+W^{\dir+}(b_2^+,s;x,t) \qquad \text{ by additivity of Busemann functions}\\&=W^{\dir+}(x_0,s;b_2^+,s)+\Ll(b_2^+,s;x,t) \qquad \text{since $b_2^+ = g^{\dir+,L}_{(x,t)}(s)$ and using \eqref{geod_LR_eq_L}} \\&=f^\dir_{(x_0,s)}(b_2^+)+\Ll(b_2^+,s;x,t) \qquad \text{ by definition of $f_{(x_0,s)}^{\dir}$ \eqref{eq: initial condition} because } b_{2}^+\ge x_0.
        \end{align*} 
        This implies that $g^{\dir+,L}_{(x,t)}|_{[s,t]}$ is an $(f^\dir_{(x_0,s)},s)$-to-$(x,t)$ geodesic, with $g^{\dir+,L}_{(x,t)}(s) = b_2^+ \ge x_0$. Hence, $d_{(x_0,s)}^{\dir}(x,t) \geq 0$. However, by the assumption $x < y$ and by definition of $\intc^-$ \eqref{eq:taupm_def} we have  
        \[
        x <  \tau_{(x_0,s)}^{\dir,-}(t) = \inf\{y \in \R: d_{(x_0,s)}^{\dir}(x,t) \geq 0 \},
        \]
        giving a contradiction.
\end{proof}

For the following, we recall the definition of leftmost/rightmost mixed $\dir$-Busemann bi-infinite interfaces from Definition \ref{def:biInf}.
\begin{proposition}[Mixed $\dir$-Busemann bi-infinite interfaces exist]
    \label{p: existence of bi-infinite competetion interfaces}
    On the event $\Omega_1$, for any $\dir \in \DLBusedc$ and $s \in \R$, if $x_0 \in\Split^L_{s,\dir}$,  there exists  $\intc^{\infty} \in \Intc^{\dir,-}$ such that $\intc^{\infty}(s)=x_0$. Similarly, if $x_0 \in\Split^R_{s,\dir}$, then there exists $\intc^\infty \in \Intc^{\dir,+}$ such that $\intc^\infty(s) = x_0$.
\end{proposition}
\begin{proof}
Without loss of generality we assume $s=0$, thus proving the statement for a point $x_0 \in \Split^L_{0,\dir}$. The proof for arbitrary $s$ and $x_0 \in \Split^R_{s,\dir}$ is similar.  The proof consists of constructing a sequence of points $\{(x_n,t_n)\}_{n \geq 1}$ with $0=t_0> t_1>t_2> \dots$, $t_n\rightarrow -\infty$, and such that the interface $\intc^{\dir,-}_{(x_n,t_n)}$ agrees with $\intc^{\dir,-}_{(x_{n+1},t_{n+1})}$ on $[t_{n},\infty)$.  Let $t_i=-i$. 
By Lemma \ref{l:LRSP}, there exists $x_1\in\Split^L_{t_1,\dir}$ such that 
 \begin{equation}
\intc^{\dir,-}_{(x_1,t_1)}(t_0)=x_0.    
 \end{equation}
  Using again Lemma \ref{l:LRSP}, we see that there exists $x_2\in\Split^L_{t_2,\dir}$ such that 
 \begin{equation}
\intc^{\dir,-}_{(x_2,t_2)}(t_1)=x_1.    
 \end{equation}
  More generally, the point $x_i\in\Split^L_{t_i,\dir}$ is chosen such that 
  \begin{equation}\label{eq60}
\intc^{\dir,-}_{(x_i,t_i)}(t_{i-1})=x_{i-1}.    
 \end{equation}
  We now define the leftmost bi-infinite interface $\intc^\infty$ via 
  \begin{equation}
      \intc^\infty(r)=\intc^{\dir,-}_{(x_{i_r},t_{i_r})}(r) \qquad \forall r\in\R.
  \end{equation}
  where $
      i_r=\sup\{i:t_i\leq r\}$. 
  Then, from the definition of $\Intc^{\dir,-}$ (Definition \ref{def:biInf}) and using Proposition \ref{p:rest} and \eqref{eq60}, we have $\intc^\infty \in \Intc^{\dir,-}$.
  \end{proof}

\section{The interaction of geodesics with interfaces}
\label{s:G_and_I}
In this section, we study the interaction between infinite geodesics and mixed $\dir$-Busemann interfaces. Our first  result in this direction, Lemma \ref{lem:geod} below, shows that $f^\dir$-to-point geodesics are the restriction of $\dir\pm$ geodesics. This is the foundation for the proof of Theorem \ref{thm:intGeo}. 

For an initial condition $f$, define
\[
\begin{aligned}
    \chi^L(f,s;x,t)=\min\argmax_z \{f(z)+\Ll(z, s; x,t)\}\\
    \chi^R(f,s;x,t)=\max \argmax_z \{f(z)+\Ll(z, s; x,t)\}
\end{aligned}
\]
be respectively, the leftmost and rightmost exit point of the geodesics associated with the initial conditions $f$. Note that the maximum (and minimum) exists by \cite[Lemma 3.2]{Rahman-Virag-21} (see also \cite[Lemma 5.13]{Busa-Sepp-Sore-22a}) which essentially says that for any initial condition $f$, the leftmost and rightmost geodesics from $f$ to any point exists. The following result connects the exit points of the line-to-point  geodesics and the interfaces. 
\begin{lemma}
    \label{lem:geod}The following holds on the event $\Omega_1$: For all $\dir \in \R,(x_0,s) \in \R^2$ and $t>s$, if $f = f_{(x_0,s)}^\dir$, then any $(f,s)$-to-$(x,t)$ geodesic $g$ satisfying $g(s) \le x_0$ is a $\xi-$ geodesic rooted at $(x_0,s)$, restricted to $[s,t]$, and any $(f,s)$-to-$(x,t)$ geodesic $g$ satisfying $g(s) \ge x_0$ is a $\xi+$ geodesic rooted at $(x_0,s)$, restricted to $[s,t]$. In addition, if we set
    \be \label{tautdef}
    \intc_t^\sig = \intc_{(x_0,s)}^{\dir,\sig}(t),
    \ee
     where the right-hand side is as defined in \eqref{eq:Buse_interface}, then 
    the following holds: 
    \begin{enumerate}[label=(\roman*), font=\normalfont]
        \item \label{it1} If $x<\intc_t^-$ then $\pi$ is an $(f,s)$-to-$(x,t)$ geodesic if and only if it is a restriction to $[s,t]$, of a $\dir-$ geodesic rooted at $(x,t)$. Moreover, in this case,
        \begin{equation}\label{eq74}
            \chi^{R}(f,s;x,t)<x_0
        \end{equation}
        \item \label{it2} If $x > \intc_t^+$ then any $\pi$ is an   $(f,s)$-to-$(x,t)$ geodesic if and only if it is a restriction to $[s,t]$, of a $\dir+$ geodesic rooted at $(x,t)$. Moreover, in this case,
        \begin{equation}
            \chi^{L}(f,s;x,t)>x_0
        \end{equation}
        \item \label{it4} 
        If $\intc_t^-\leq x \leq \intc_t^+$, then the restriction to $[s,t]$ of any $\dir -$ or $\dir +$ geodesic started from $(x,t)$ is an $(f,s)$-to-$(x,t)$ geodesic. Moreover, in this case,
        \begin{equation}
        \label{eq74'}
            \chi^{L}(f,s;x,t)\leq x_0,\,\, \text{ and }\,\, \chi^{R}(f,s;x,t)\geq x_0.
        \end{equation}
    \end{enumerate}
    
\end{lemma}
\begin{proof}
Let $g$ be an $(f,s)$-to-$(x,t)$ geodesic, and assume $z_0=g(s) \geq x_0$. By Lemma \ref{l:MV}\ref{it:Bu1},
 \begin{align*}
     &\qquad \, f(z_0)+\Ll(z_0,s;x,t)=W^{\dir+}(x_0,s;x,t)\\
     &\implies W^{\dir+}(x_0,s;z_0,s)+\Ll(z_0,s;x,t)=W^{\dir+}(x_0,s;x,t)\\
     & \implies \Ll(z_0,s;x,t)=W^{\dir+}(z_0,s;x,t),
 \end{align*}
 and the last line implies that $z_0$ is a maximizer of  $z \mapsto W^{\dir+}(0,s;z,s)+\Ll(z,s;x,t)$ (using Lemma \ref{lem:Buse_eq}). Thus, $g$ is the restriction of a $\dir+$ geodesic by Proposition \ref{prop:DL_SIG_cons_intro}\ref{itm:arb_geod_cons}.  The proof for the statement when $g(s) \le x_0$ is symmetric. We now prove the enumerated items.

\textbf{Item \ref{it1}:} If $x < \intc_t^-$, then by definition of $\intc_t^-$, \[
\sup_{z \le x_0}\{f(z) + \Ll(z,s;x,t)\} < \sup_{z \ge x_0}\{f(z) + \Ll(z,s;x,t)\},
\]
so all maximizers of $z \mapsto f(z) +\Ll(z,s;x,t)$ must lie in the set $(-\infty,x_0)$, which implies \eqref{eq74}. Recall that the locations of $\xi-$ geodesics are exactly the maximizers of $z \mapsto W^{\dir -}(x_0,s;z,s) + \Ll(z,s;x,t)$ over $z \in \R$. By definition, $f(z) = W^{\dir-}(x_0,s;z,s)$ for $z < x_0$, so to show that the $(f,s)$-to-$(x,t)$ geodesics are exactly the restrictions of $\xi-$ geodesics, we must show that all maximizers $z \mapsto W^{\dir -}(x_0,s;z,s) + \Ll(z,s;x,t)$ are also in the set $(-\infty,x_0)$, which also proves $\chi^{R}(f,s;x,t) < x_0$. Recall by monotonicity that $W^{\dir-}(x_0,s;z,s) \le W^{\dir+}(x_0,s;z,s)$ for $z \ge x_0$. Suppose, by way of contradiction, that there is a maximzer of $z \mapsto W^{\dir -}(x_0,s;z,s) + \Ll(z,s;x,t)$ in $[x_0,\infty)$, which we shall call $z^\star$. Then, 
\begin{align*}
&\quad \,f(z^\star) + \Ll(z^\star,s;x,t)  \\
&=W^{\dir +}(x_0,s;z^\star,s) + \Ll(z^\star,s;x,t) \\
&\ge W^{\dir -}(x_0,s;z^\star,s) + \Ll(z^\star,s;x,t) \\
&= \sup_{z \in \R}\{W^{\dir -}(x_0,s;z,s) + \Ll(z,s;x,t)\} \\
&\ge \sup_{z \le x_0}\{W^{\dir -}(x_0,s;z,s) + \Ll(z,s;x,t)\} \\
&= \sup_{z \in \R}\{f(z) + \Ll(z,s;x,t)\}.
\end{align*}
Thus, $z^\star$ is a maximizer of $z \mapsto f(z) + \Ll(z,s;x,t)$ over $z \in \R$, a contradiction since we have seen that all maximizers are in $(-\infty,x_0)$.

\medskip \noindent \textbf{Item \ref{it2}:} This has a symmetric proof to Item \ref{it1}.

\medskip \noindent \textbf{Item \ref{it4}:} 
If $\intc_t^-\leq x \leq \intc_t^+$, then by definition of $\intc_t^-$ and $\intc_t^+$, \[\sup_{z \in \R}\{f(z) + \Ll(z,s;x,t)\}=
\sup_{z \le x_0}\{f(z) + \Ll(z,s;x,t)\} = \sup_{z \ge x_0}\{f(z) + \Ll(z,s;x,t)\},
\]
This shows \eqref{eq74'}.

 Next, we show that when $\intc_t^-\leq x \leq \intc_t^+$, the restriction of any $\dir+$ geodesic starting from $(x,t)$ is an $(f,s)$-to-$(x,t)$ geodesic. The proof that $\dir-$ geodesics are $(f,s)$-to-$(x,t)$ geodesics is symmetric. Let $g$ be an arbitrary $\dir +$ geodesic rooted at $(x,t)$.
 \paragraph{\textbf{Case 1:}} $z_0 := g(s) \geq x_0$. Therefore, we have
 \begin{align*}
     &\sup_{z \in \R}\left(f(z)+\Ll(z,s;x,t)\right)=W^{\dir+}(x_0,s;x,t)\qquad \text{ by Lemma } \ref{l:MV}\ref{it:Bu3}\\&=W^{\dir+}(x_0,s;z_0,s)+W^{\dir+}(z_0,s;x,t)\\&=W^{\dir+}(x_0,s;z_0,s)+\Ll(z_0,s;x,t)\qquad \qquad \text{by \eqref{eqn:SIG_weight} } \\&=f(z_0)+\Ll(z_0,s;x,t).
 \end{align*}
By Lemma \ref{lem:Buse_eq}, $g|_{[s,t]}$ is indeed an $(f,s)$-to-$(x,t)$ geodesic.
 \paragraph{\textbf{Case 2:}} $z_0<x_0$. In this case, 
 \begin{align*}
 &\quad \,f(z_0)+\Ll(z_0,s;x,t) \\
 & =W^{\dir-}(x_0,s;z_0,s)+W^{\dir+}(z_0,s;x,t) \qquad\qquad\qquad \qquad\qquad\;\; \text{by \eqref{eqn:SIG_weight} }\\&=W^{\dir-}  (x_0,s;z_0,s)+W^{\dir+}(z_0,s;x_0,s)+W^{\dir+}(x_0,s;x,t) \qquad\text{ by additivity}\\&=W^{\dir-}(x_0,s;z_0,s)+W^{\dir+}(z_0,s;x_0,s)+W^{\dir-}(x_0,s;x,t) \qquad \text{ by Lemma } \ref{l:MV}\ref{it:Bu3}\\& \geq W^{\dir-}(x_0,s;z_0,s)+W^{\dir-}(z_0,s;x_0,s)+W^{\dir-}(x_0,s;x,t) \qquad \text{ by monotonicity }
 \\&=W^{\dir-}(x_0,s;x,t) \qquad\qquad\qquad\qquad\qquad\qquad\qquad\qquad\qquad\;\;\text{ by additivity}\\
 &=\sup_{z \in \R}\{f(z)+\Ll(z,s;x,t)\} \qquad \qquad\qquad\qquad\qquad\qquad\qquad\quad\;\text{by Lemma \ref{l:MV}}.
 \end{align*}
 Hence, $z_0$ is a maximizer of $z \mapsto f(z) + \Ll(z,s;x,t)$, and $g|_{[s,t]}$ is an $(f,s)$-to-$(x,t)$ geodesic by Lemma \ref{lem:Buse_eq}. 
\end{proof}

\begin{corollary} \label{cor:LRmost_pi} 
    In the setting of Lemma \ref{lem:geod}, the following hold.
    \begin{enumerate} [label=(\roman*), font=\normalfont]
    \item \label{it:LR1} If $x \le \tau_t^+$ and $\pi_{(x,t)}^L$ is the leftmost $(f,s)$-to-$(x,t)$ geodesic, then 
    \[
    \pi_{(x,t)}^L = g_{(x,t)}^{\dir-,L}|_{[s,t]}, \quad\text{and}\quad g_{(x,t)}^{\dir-,L}(s) \le x_0. 
    \]
    \item \label{it:LR2} If $x \ge \tau_t^-$ and $\pi_{(x,t)}^R$ is the rightmost $(f,s)$-to-$(x,t)$ geodesic, then 
    \[
    \pi_{(x,t)}^R = g_{(x,t)}^{\dir+,R}|_{[s,t]},\quad\text{and}\quad g_{(x,t)}^{\dir+,R}(s) \ge x_0.
    \]
    \end{enumerate}
\end{corollary}
\begin{proof}
We prove the first statement, with the second following a symmetric proof. By Items \ref{it1} and \ref{it4} of Lemma \ref{lem:geod}, if $x \le \tau_t^+$, then $\chi^{L}(f;s,x,t) \le x_0$, so the main statement of Lemma \ref{lem:geod} implies that the leftmost $(f,s)$-to-$(x,t)$ geodesic,  $\pi_{(x,t)}^L$, is the restriction of a $\dir -$ geodesic to $[s,t]$. Furthermore, by Items \ref{it1} and \ref{it4} of Lemma \ref{lem:geod}, all $\dir -$ geodesics from $(x,t)$ are $(f,s)$-to-$(x,t)$ geodesics. Hence, $\pi_{(x,t)}$ is the restriction of the leftmost $\dir -$ geodesic from $(x,t)$, which gives the desired statement.  
\end{proof}
    Recall the set $\Branch_{t,\dir}$ defined in \eqref{eq:bran_def} which is the collection of all branching points. Our next result gives two basic properties of interfaces. The first property is that points on the interfaces live on points in $\Branch_{t,\dir}$. Note that this does not necessarily imply that $\intc_t^-\in \Split_{t,\dir}$ (resp.\ $\intc_t^+$) as it is possible for the leftmost $\dir-$ and rightmost $\dir+$ geodesics starting from a point to meet even after they separate. The second property is that some geodesics associated with the competition interfaces do not intersect the interface.

\begin{lemma} 
    \label{lem:Splt}
    The following holds on the event $\Omega_1$. Let $(x_0,s)\in\R^2$  and $\dir\in \R$. For $t \ge s$ and $\sigg \in \{-,+\}$, let $\intc_t^\sig = \intc_{(x_0,s)}^{\dir,\sig}(t)$. Then the following holds:
    \begin{enumerate}[label=(\roman*), font=\normalfont]
        \item \label{itm:branching}(interfaces transverse through branching points) For any $t>s$,
        \begin{equation}
            \intc_t^-, \intc_t^+\in \Branch_{t,\dir}.
        \end{equation}
    \item \label{itm:G1} (geodesics do not intersect interfaces) For all $t>r\geq s$,
     \be \label{eq:no_int}
            g^{\dir-,L}_{(\intc_t^-,t)}(r) < \intc_r^-, \qquad \text{and}\qquad  g^{\dir+,R}_{(\intc_t^+,t)}(r) > \intc_r^+.
            \ee
            In particular, for all $x<\intc_t^-$, $y>\intc_t^+$, and $t>r\geq s$,
            \be \label{eq:othpts}
            g^{\dir-,R}_{(x,t)}(r) < \intc_r^-, \qquad \text{and}\qquad \intc_r^+ < g^{\dir+,L}_{(y,t)}(r).
            \ee   
    \end{enumerate}
\end{lemma}

        \begin{proof} \textbf{Item \ref{itm:branching}:}
            We shall prove the statement for $\intc^-_t$. The case for $\intc^+_t$ will follow similarly. Let $f = f_{(x_0,s)}^\dir$. By Lemma \ref{lem:geod}\ref{it4},  for all $s<t$, $\sigg \in \{-,+\}$, and $S \in \{L,R\}$, we have that
$g_{(\intc^-_t,t)}^{\dir\sig,S}|_{[s,t]}$ is an $f$ to $(\intc_t^-,t)$ geodesic. 

Suppose, in contradiction to the statement of the lemma, that there exists $\delta>0$ so that
            \begin{equation}
                g_{(\intc^-_t,t)}^{\dir -, L}=g_{(\intc^-_t,t)}^{\dir +, R} \text{ on $[t-\delta,t]$}.
            \end{equation}
Proposition \ref{prop:DL_SIG_cons_intro}\ref{itm:DL_LRmost_geod} gives that $g_{(\intc^-_t,t)}^{\dir -, L}$ (resp.\ $g_{(\intc^-_t,t)}^{\dir +, R}$) are leftmost (resp.\ rightmost) geodesics between any of their endpoints. Let $x_1=g_{(\intc^-_t,t)}^{\dir -, L}(s)$ and $x_2=g_{(\intc^-_t,t)}^{\dir +, R}(s).$ By ordering of geodesics (Proposition \ref{prop:g_basic_prop}, \ref{itm:DL_mont_dir}), we have $x_1 \leq x_2$. Now, for $n\in\N$, define the points (see Figure \ref{fig:Splitting points})
            \begin{equation} \label{pn_def}
            \begin{aligned}
                p_n&=(\intc^-_t-n^{-1},t)\\
            \end{aligned}
            \end{equation}
            \begin{figure}[t!]
                \includegraphics[width=7 cm]{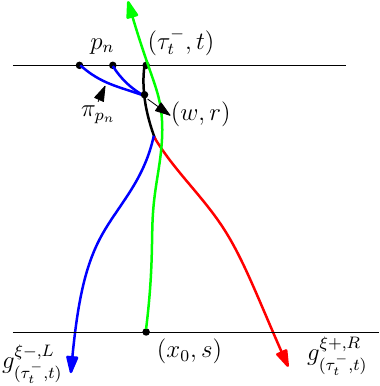}
                \caption{To prove that $(\intc^-_{t},t)$ is a branching point, we use contradiction. We consider the sequence $p_n$ and if $g_{(\intc^-_t,t)}^{\dir -, L}$ and $g_{(\intc^-_t,t)}^{\dir +, R}$ coincide for some time, then there will be $w$ on $g_{(\intc^-_t,t)}^{\dir +, R}$ as in the figure such that the following happens: the rightmost $(f,s)$-to-$p_{n}$ will all intersect $(w,r)$ for large enough $n$. But the $(f,s)$-to-$(w,r)$ geodesic $g_{(\intc^-_t,t)}^{\dir +, R}|_{[s,t-\delta]}$ starts on the right of $x_0.$ Hence, $\pi_{p_{n}}(s) \geq x_0.$ This is a contradiction to the definition of $\intc_t^-.$}
                \label{fig:Splitting points}
            \end{figure}
            For $n \ge 1$, let $\pi_n$ be the rightmost $(f,s)$-to-$p_n$ geodesic. As $p_n \rightarrow (\intc^-_t,t)$, by Lemma \cite[Lemma 3.4]{Rahman-Virag-21}, every subsequence $n_k$ contains a further subsequence $n_{k_\ell}$ such that $\pi_{p_{n_{k_\ell}}}$ converges to an $(f,s)$-to-$(\intc_t^-,t)$ geodesic, where the convergence holds in the Hausdorff topology on the set of paths. Now, by Corollary \ref{cor:LRmost_pi},\, for all $n$,
            \[
            \pi_{p_n}=g^{\dir-,R}_{p_n}|_{[s,t]}.
            \]
            By the ordering of geodesics in Proposition \ref{prop:g_basic_prop}\ref{itm:DL_SIG_mont_x}, for all $r \in [s,t]$,
            \[
            \pi_{p_n}(r) \leq g^{\dir-,L}_{(\intc_t^-,t)}(r),
            \]
            so any subsequential limit of $\pi_{p_n}$ lies to the left of $g_{(\intc_t^-,t)}^{\dir-,L}|_{[s,t]}$. But we have also noted that all subsequential limits (in the Hausdorff topology) of $\pi_{p_n}$ are $(f,s)$-to-$(\intc_t^-,t)$ geodesics, and by Corollary \ref{cor:LRmost_pi}, $g_{(\intc_t^-,t)}^{\dir-,L}|_{[s,t]}$ is the leftmost $(f,s)$-to-$(\intc_t^-,t)$ geodesic.  
             Thus, $\{\pi_{p_{n}}\}_{n \geq 1}$ converges to $g_{(\intc^-_t,t)}^{\dir-,L}|_{[s,t]}$ in the Hausdorff topology.  Next, define $w_0 :=g_{(\intc^-_t,t)}^{\dir -, L}(t-\delta)=g_{(\intc^-_t,t)}^{\dir +, R}(t-\delta)$. By the assumption that $g_{(\intc^-_t,t)}^{\dir -, L}=g_{(\intc^-_t,t)}^{\dir +, R}$ on $[t-\delta,t]$, and the fact that $g_{(\intc_t^-,t)}^{\dir -,L}$ and $g_{(\intc_t^+,t)}^{\dir +,R}$ are the leftmost and rightmost geodesics, respectively, between any two of their points, we see that the restriction of $g_{(\intc^-_t,t)}^{\dir -, L}$ (resp.\ $g_{(\intc^-_t,t)}^{\dir +, R}$) on $[t-\delta, t]$ is the unique geodesic between $(w_0, t-\delta)$ and $(\intc^-_t, t)$. Denote this geodesic by $\pi_{(w_0, t-\delta),(\intc^-_t,t)}$. Since the convergence of $\{\pi_{p_{n}}\}_{n \geq 1}$ holds in the Hausdorff topology, $\pi_{p_n}(t-\delta)$ converges to $w_0$ as $n \to \infty$. Since also $p_n \to (\intc_t^-,t)$, Lemma \ref{lem:overlap} implies that the overlap of $\pi_{p_n}$ and $g_{(\tau_t^-,t)}^{\dir-,L}|_{[s,t]}$ is an interval whose endpoints converge to $s$ and $t$. 
             In particular, if we choose $r \in 
        (t-\delta, t)$, then, for all large enough $n$, 
            \begin{equation} 
            \label{eq: geodesics meet}
            w := \pi_{p_{n}}(r)=\pi_{(w_0, t-\delta),(\intc^-_t,t)}(r) = g_{(\intc_t^-,t)}^{\dir+,R}(r) .
            \end{equation}
        Since $g_{(\intc_t^-,t)}^{\dir+,R}|_{[s,t]}$ is an $(f,s)$-to-$(\intc_t^-,t)$ geodesic, this implies that $g_{(\intc_t^-,t)}^{\dir+,R}|_{[s,r]}$ is an $(f,s)$-to-$(w,r)$ geodesic. Hence, by \cite[Lemma 3.1]{Rahman-Virag-21},
        if we consider $\pi_{p_{n}}$ and replace $\pi_{p_{n}}|_{[s,r]}$ by $g_{(\intc_t^-,t)}^{\dir+,R}|_{[s,r]}$ then the resulting path is an $(f,s)$-to-$p_{n}$ geodesic. But note that by Corollary \ref{cor:LRmost_pi}, $g_{(\intc_t^-,t)}^{\dir+,R}|_{[s,r]}(s)\geq x_0$. Hence, for sufficiently large $n$, $ \chi^{R}(f,s;p_{n}) \geq x_0$. This contradicts Equation \eqref{eq74} of Lemma \ref{lem:geod}\ref{it1}.

            \medskip \noindent 
        \textbf{Item \ref{itm:G1}:} Note that \eqref{eq:othpts} follows immediately from \eqref{eq:no_int} by ordering of geodesics (Proposition \ref{prop:g_basic_prop}, \ref{itm:DL_SIG_mont_x}).
        As we did in \ref{itm:branching}, we will prove the statement for $\intc_t^-$. The other case will follow similarly. We start with $r=s$, i.e., we show that 
        \[
        g^{\dir-,L}_{(\intc_t^-,t)}(s)<x_0.\]
        Assume, by contradiction, that this fails. Then, we can consider the same sequence $p_{n}$ defined in \eqref{pn_def}. As we have shown that for large enough $n$, $\pi_{p_{n}}$ intersects $g^{\dir-,L}_{(\intc_t^-,t)}$, we can apply the same swapping of geodesics as before to conclude that for large enough $n$,
        \[
        \pi_{p_{n}}(s)\geq x_0.
        \]
        This is a contradiction to Equation \eqref{eq74} of Lemma \ref{lem:geod}\ref{it1}.
        
        Now we consider the case when $t>r>s$. By way of contradiction assume that there exists $t>r_0>s$ such that
        \[
        g^{\dir-,L}_{(\intc_t^-,t)}(r_0) \geq \intc^-_{r_0}.
        \]
         Choose $r_1 \in (r_0,t)$, and consider the sequence $p_{n}$ defined in \eqref{pn_def}. Then since geodesics do not form interior bubbles (Lemma \ref{lem:no_bubbles}),  $g^{\dir-,L}_{(\intc_t^-,t)}|_{[r_0,r_1]}$ is the unique geodesic between $(g^{\dir-,L}_{(\intc_t^-,t)}(r_0),r_0)$ and $(g^{\dir-,L}_{(\intc_t^-,t)}(r_1),r_1)$. Now, as $\pi_{p_{n}}$ converges to $g^{\dir-,L}_{(\intc_t^-,t)}|_{[s,t]}$ in the Hausdorff topology, the restrictions $\pi_{p_{n}}|_{[r_0,r_1]}$ also converge in the same topology to $g^{\dir-,L}_{(\intc_t^-,t)}|_{[r_0,r_1]}$, which is the unique geodesic between its two endpoints. By Lemma \ref{lem:overlap}, this implies that the overlap of $\pi_{p_{n}}|_{[r_0,r_1]}$ with $g^{\dir-,L}_{(\intc_t^-,t)}|_{[r_0,r_1]}$ is an interval whose endpoints are converging to $[r_0,r_1]$. So, for sufficiently large $n$, $\pi_{p_{n}}|_{[r_0,r_1]}$ intersects $g^{\dir-,L}_{(\intc_t^-,t)}|_{[r_0,r_1]}$ (see Figure \ref{fig:geodesicsdonotintersectinterface}). By Lemma \ref{lem:geod}\ref{it1}, we know that $\pi_{p_{n}}$ are $\dir-$ geodesics.
        \begin{figure}[t!]
            \includegraphics[width=7cm]{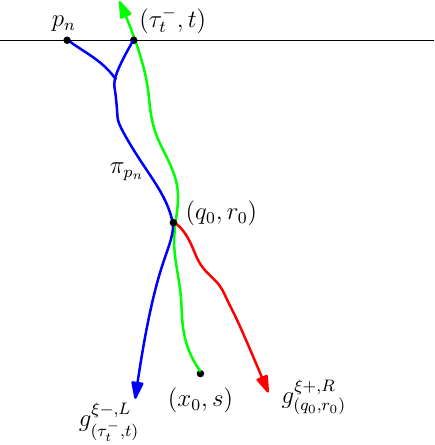}
            \caption{To prove Lemma \ref{lem:Splt}\ref{itm:G1}, we consider the sequence $p_{n}$ as in the first part. As the rightmost $\dir-$ geodesics  $\pi_{p_{n}}$ to $p_{n}$ converge to $g^{\dir-,L}_{(\intc_t^-,t)}|_{[s,t]}$, for large enough $n$ these geodesics have to meet $g^{\dir-,L}_{(\intc_t^-,t)}|_{[s,t]}$ before time level $r_0$ and can not separate again. Now, if $g^{\dir-,L}_{(\intc_t^-,t)}|_{[s,t]}$ intersects $\intc^-|_{[s,t]}$ then we can swap $\pi_{p_{n}}$ on $[s,r_0]$ by $g^{\dir+,R}_{(q_0,r_0)}$ to get an $(f,s)$-to-$p_{n}$ geodesic whose value at time $s$ is at least $x_0$. This is a contradiction as $p_{n}<\intc_t^-.$}
            \label{fig:geodesicsdonotintersectinterface}
        \end{figure}
        By Proposition \ref{prop:DL_all_coal}, all  $\dir-$ geodesics coalesce, so for all large enough $n$, there exists $r_n' \in [r_0,r_1]$ such that
        \[
        \pi_{p_{n}}(r)=g^{\dir-,L}_{(\intc_t^-,t)}(r) \qquad \text{ for all } r \leq r_n'.
        \]
        In particular, for all $n$ large enough (see Figure \ref{fig:geodesicsdonotintersectinterface}), 
        \[
        \pi_{p_{n}}(r_0)=g^{\dir-,L}_{(\intc_t^-,t)}(r_0)\geq \intc^-_{r_0}.
        \]
        Now, let $q_0=g^{\dir-,L}_{(\intc_t^-,t)}(r_0).$ As $q_0 \geq \intc^-_{r_0}$ we know by Lemma \ref{lem:geod}\ref{it2} that $g^{\dir+,R}_{(q_0,r_0)}|_{[s,r_0]}$ is an $f$ to $(q_0,r_0)$ geodesic with 
        \[
        g^{\dir+,R}_{(q_0,r_0)}(s) \geq x_0.
        \]
        As $(g^{\dir+,R}_{(q_0,r_0)}(r_0),r_0)$ lies on all $\pi_{p_{n}}$ for all large enough $n$, by \cite[Lemma 3.1]{Rahman-Virag-21} we can swap $\pi_{p_{n}}$ on $[s,r_0]$ by $g^{\dir+,R}_{(q_0,r_0)}|_{[s,r_0]}$ to obtain an $f$ to $p_{n}$ geodesic $\pi'_{p_{n}}$. Then, $\pi'_{p_{n}}(s) \geq x_0$. But by Equation \eqref{eq74} of Lemma \ref{lem:geod}\ref{it1}, for all $n$,
        \[
        \chi^{R}(f,s;p_{n})<x_0,
        \]
        giving a contradiction.
        \end{proof}
        
    Propositions \ref{p:IntfDist} and \ref{p:>interface} below  consider the case where the leftmost and rightmost interfaces split at a certain point and do not meet again for a certain amount of time. This creates a 'funnel'. Together the proposition show that from any point within the funnel, the $\dir-$ and $\dir+$ geodesics agree within the funnel, trapped in it, and exit through the splitting point (see Figure \ref{fig:funnel}).

    \begin{proposition}
        \label{p:IntfDist}
        The following holds on the event $\Omega_1$: Let $\dir \in \R, (x_0,s) \in \R^2, t>s$, and assume that
        \[
        \intc_{(x_0,s)}^{\dir,-}(r) < \intc_{(x_0,s)}^{\dir,+}(r)\qquad \text{for all }r \in (s,t].
        \]
       Then, letting $\tau_t^{\sig} = \tau_{(x_0,s)}^{\dir,\sig}(t)$ for $\sigg \in \{-,+\}$, we have
        \begin{align}
        \label{eq:5.1eq1}g_{(\intc_t^-,t)}^{\dir+,R}|_{[s,t]}&=g_{(\intc_t^-,t)}^{\dir-,R}|_{[s,t]},\\
        \label{eq:5.1eq2}g_{(\intc_t^+,t)}^{\dir+,L}|_{[s,t]}&=g_{(\intc_t^+,t)}^{\dir-,L}|_{[s,t]}, \qquad\text{and} \\
        \label{eq:5.1eq3} g^{\dir+,R}_{(\intc_t^-,t)}(s)&=g^{\dir-,L}_{(\intc_t^+,t)}(s) = x_0.
        \end{align}
    \end{proposition}
    \begin{proof}
    We will show \eqref{eq:5.1eq1} and $g^{\dir+,R}_{(\intc_t^-,t)}(s)=x_0$ together. Then, the proof of \eqref{eq:5.1eq2} and $g^{\dir-,L}_{(\intc_t^+,t)}(s) = x_0$ follows a symmetric argument. Let $f = f_{(x_0,s)}^\dir$.
By Corollary \ref{cor:LRmost_pi}, $g^{\dir+,R}_{(\intc_t^-,t)}|_{[s,t]}$ is the rightmost $f$ to $(\intc_t^-,t)$ geodesic, and 
    \be \label{eq:g>x0}
    g^{\dir+,R}_{(\intc_t^-,t)}(s) \geq x_0.
    \ee 
    To prove the proposition, we first show that, if $g^{\dir+,R}_{(\intc_t^-,t)}(s) = x_0$, then \eqref{eq:5.1eq1} holds. Then, we prove by way of contradiction that we cannot have $g^{\dir+,R}_{(\intc_t^-,t)}(s) > x_0$.
    \paragraph{\textbf{Step 1: Proving \eqref{eq:5.1eq1} under the assumption $g^{\dir+,R}_{(\intc_t^-,t)}(s) = x_0$. }}
    Under this assumption, we have
    \[
    \Ll(x_0,s;\intc_t^-,t)=W^{\dir+}(x_0,s;\intc_t^-,t)=W^{\dir-}(x_0,s;\intc_t^-,t),
    \]
    where the first equality holds because  $g^{\dir+,R}_{(\intc_t^-,t)}(s) = x_0$ (using Lemma \ref{lem:Buse_eq}) and the second equality holds by Lemma \ref{l:MV}\ref{it:Bu3}. Thus, by Lemma \ref{lem:Buse_eq}, $g_{(\intc_t^-,t)}^{\dir+,R}|_{[s,t]}$ is also a restriction of $\dir-$ geodesic. By ordering of geodesics, $g_{(\intc_t^-,t)}^{\dir+,R}$ lies weakly to the right of every $\dir-$ geodesic, so it is in fact the rightmost $\dir-$ geodesic on $[s,t]$, and 
    \[
g_{(\intc_t^-,t)}^{\dir+,R}|_{[s,t]}=g_{(\intc_t^-,t)}^{\dir-,R}|_{[s,t]}.
    \]
    \paragraph{\textbf{Step 2: Proving $g^{\dir+,R}_{(\intc_t^-,t)}(s)= x_0$}} Assume, by way of contradiction, that this equality does not hold. We will demonstrate in this case that an interior geodesic bubble forms, which we know cannot occur. By \eqref{eq:g>x0}, we must have
    \[
    g^{\dir+,R}_{(\intc_t^-,t)}(s) >  x_0.
    \]
     Let us consider the restriction of the rightmost interface $\tau^+$, denoted by $\intc^+|_{[s,t]}$. Because $\intc_t^-<\intc_t^+$ and $g^{\dir+,R}_{(\intc_t^-,t)}(s)>x_0$, continuity of geodesics and continuity of the interface $\intc^-$ (Lemma \ref{lem:interface_continuous}) ensure that $g^{\dir+,R}_{(\intc_t^-,t)}|_{[s,t]}$ intersects $\intc^+|_{[s,t]}$ (see Figure \ref{fig:unique starting point}). Let
    \begin{figure}[t!]
        \includegraphics[width=10cm]{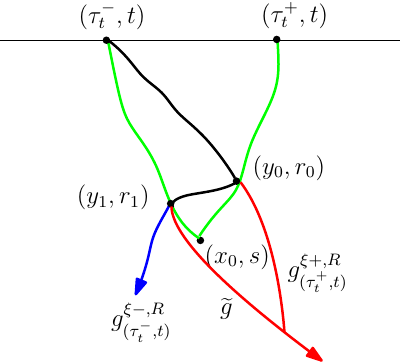}
        \caption{Proof of \eqref{eq:5.1eq3}: If $g^{\dir+,R}_{(\intc_t^-,t)}(s)>x_0$, then under the hypothesis of Proposition \ref{lem:Splt} $g^{\dir+,R}_{(\intc_t^-,t)}$ must intersect $\intc^+|_{[s,t]}$ at some $(y_0,r_0)$. By Lemma \ref{l:MV} $g^{\dir+,R}_{(\intc_t^-,t)}|_{[r_0,t]}$ and $g^{\dir-,R}_{(\intc_t^-,t)}|_{[r_0,t]}$ coincide (denoted by black in the figure). As $g^{\dir-,R}_{(\intc_t^-,t)}(s) \leq x_0$ it must be the case that $g^{\dir-,R}_{(\intc_t^-,t)}$ intersects $\intc^-$ at some point $(y_1,r_1)$. But by Lemma 3.3 again the restriction of $g^{\dir-,R}_{(\intc_t^-,t)}$ between $(y_1,r_1)$ $(y_0,r_0)$ is also restriction of a $\dir+$ geodesic (denoted by black in the figure). But this will lead to a situation where two $\dir+$ geodesics stay together in the beginning and then separate (the geodesic $\widetilde{g}$ in the figure). This will form a geodesic bubble which is not possible.}
        \label{fig:unique starting point}
    \end{figure}
    \[
    r_0=\sup\{r \in [s,t]: g^{\dir+,R}_{(\intc_t^-,t)}(r) = \tau_r^+ \},
    \]
    and set $y_0:=\intc^+_{r_0}$. Since $g^{\dir+,R}_{(\intc_t^-,t)}(s)>x_0 = \tau_{s}^+$, 
    continuity ensures that $r_0 > s$.
    Geometrically, $(y_0,r_0)$ is the first point where $g^{\dir+,R}_{(\intc_t^-,t)}|_{[s,t]}$ and  $\intc^+|_{[s,t]}$  intersect when traveling downward from time $t$. 
    
With this point $(y_0,r_0)$ defined, make the observation that 
    \be \label{eq:LW+-y0r0}
    \Ll(y_0,r_0;\intc_t^-,t)=W^{\dir+}(y_0,r_0;\intc_t^-,t)=W^{\dir-}(y_0,r_0;\intc_t^-,t),
    \ee
    where, as before,  the first equality holds because $ g^{\dir+,R}_{(\intc_t^-,t)}(r_0) = y_0$, and second equality follows by additivity and Lemma \ref{l:MV}\ref{it:Bu3}, which gives
    \[
    W^{\dir+}(x_0,s;\intc^-_t,t)=W^{\dir-}(x_0,s;\intc^-_t,t) \quad\text{and}\quad  W^{\dir+}(x_0,s;y_0,r_0)=W^{\dir-}(x_0,s;y_0,r_0).
    \]
    Recall that $y_0 = g_{(\tau_t^-,t)}^{\dir+,R}(r_0)$, which by definition, is the rightmost maximizer of 
    \[
    z \mapsto W^{\dir +}(0,r_0;z,r_0) + \Ll(z,r_0;\tau_t^-,t) \quad\text{over } z \in \R.
    \]
    By \eqref{eq:LW+-y0r0} and Lemma \ref{lem:Buse_eq}, $y_0$ is also a maximizer of 
    \[
    z \mapsto W^{\dir -}(0,r_0;z,r_0) + \Ll(z,r_0;\tau_t^-,t)\quad\text{over } z \in \R.
    \]
    But by ordering of geodesics, the rightmost such maximizer, $g_{(\tau_t^-,t)}^{\dir-,R}(r_0)$ is no greater than \\ $g_{(\tau_t^-,t)}^{\dir+,R}(r_0) = y_0$. Hence, 
    \[
    g_{(\tau_t^-,t)}^{\dir-,R}(r_0) = y_0 = g_{(\tau_t^-,t)}^{\dir+,R}(r_0).
    \]
    Since these geodesics are the rightmost geodesics between their points (Proposition \ref{prop:DL_SIG_cons_intro}\ref{itm:DL_LRmost_geod}), we have the following (see Figure \ref{fig:unique starting point})
    \be \label{r0teq}
g^{\dir+,R}_{(\intc_t^-,t)}|_{[r_0,t]}=g^{\dir-,R}_{(\intc_t^-,t)}|_{[r_0,t]}.
    \ee
    Next, by Lemma \ref{lem:Splt}\ref{itm:G1}, we know that 
    \be \label{eq:taurgr}
    \intc_r^+ <  g^{\dir+,R}_{(y_0,r_0)}(r)\quad\text{for all }\quad r \in [s,r_0).
    \ee
    In contrast, we now show that \be \label{eq:1567}
    g^{\dir-,R}_{(\intc_t^-,t)}(s) \leq x_0.
    \ee
    To see this, assume by way of contradiction that $g^{\dir-,R}_{(\intc_t^-,t)}(s)>x_0.$ Then pick $x \in (\intc_t^-,\intc_t^+).$ By ordering of $\dir-$ geodesics, we have
    \begin{equation}
    \label{eq: contradicting equation}
    x_0<g^{\dir-,R}_{(\intc_t^-,t)}(s) \leq g^{\dir-,L}_{(x,t)}(s),
    \end{equation}
    and by Corollary \ref{cor:LRmost_pi}, $g^{\dir-,L}_{(x,t)}|_{[s,t]}$ is the leftmost $(f,s)$-to-$(x,t)$ geodesic. Hence, $\chi^{L}(f,s;x,t) > x_0$. However, since $x \in (\intc_t^-,\intc_t^+)$,  we must have  $\chi^{L}(f,s;x,t) \leq x_0$, giving a contradiction. Thus, \eqref{eq:1567} holds. Now, as $g_{(\intc_t^-,t)}^{\dir-,R}(s) \leq x_0$ and $\intc_t^-<\intc_t^+$ on $[s,r_0]$ it must happen that $g_{(\intc_t^-,t)}^{\dir-,R}|_{[s,r_0]}$ intersects $\intc^-|_{[s,r_0]}$ (see Figure \ref{fig:unique starting point}).  Let
    \be \label{eq:r1def}
    r_1=\sup \{r \in [s,r_0]:  g_{(\intc_t^-,t)}^{\dir-,R}(r) = \intc_r^- \},
    \ee
    and set $y_1=\intc^-_{r_1}$. By continuity and since $y_0 = \tau_{r_0}^+ > \tau_{r_0}^-$ (because we showed $r_0 > s$), we have  $r_1<r_0.$ By a similar argument to the proof of \eqref{eq:LW+-y0r0}, we get that 
    \be \label{eq:LW+-y1r1}
    \Ll(y_1,r_1;y_0,r_0)=W^{\dir-}(y_1,r_1;y_0,r_0)=W^{\dir+}(y_1,r_1;y_0,r_0).
    \ee
    In particular, the first equality holds because $g_{(\intc_t^-,t)}^{\dir-,R}(r_0) = y_0$ and $g_{(\intc_t^-,t)}^{\dir-,R}(r_1) = y_1$, and the second equality holds by additivity of the Busemann functions and by Lemma \ref{l:MV}\ref{it4}, which gives
    \[
    W^{\dir+}(x_0,s;y_1,r_1)=W^{\dir-}(x_0,s;y_1,r_1) \quad \text{and}\quad W^{\dir+}(x_0,s;y_0,r_0)=W^{\dir-}(x_0,s;y_0,r_0).
    \]
    Then, \eqref{eq:LW+-y1r1} implies that $y_1$ maximizes 
    \[
    z \mapsto W^{\dir+}(0,r_1;z,r_1)+\Ll(z,r_1;y_0,r_0) \quad\text{over }z \in \R.
    \]
    Hence, $g^{\dir-,R}_{(\intc_t^-,t)}|_{[r_1,t]}$ is  a restriction of a $\dir+$ geodesic rooted at $(\intc_t^-,t)$. We will let $\wt g$ denote this $\dir +$ geodesic. By coalescence of $\dir +$ geodesics (Proposition \ref{prop:DL_all_coal}), there exists $r_2 < t$ such that $\wt g(r) = g^{\dir+,R}_{(\intc_t^-,t)}(r)$ for $r \le r_2$. Combined with \eqref{r0teq},\eqref{eq:taurgr}, along with the definition that $y_1 = \tau_{r_1}^-$ \eqref{eq:r1def}, we have 
    \[
    \begin{aligned}
    \wt g(r) &= g^{\dir+,R}_{(\intc_t^-,t)}(r)\qquad \text{for } r \in [r_0,t] \\
    \wt g(r_1) &= y_1 = \tau_{r_1}^- \le \tau_{r_1}^+ < g_{(y_0,r_0)}^{\dir +,R}(r_1), \qquad\text{and} \\
    \wt g(r) &= g^{\dir+,R}_{(\intc_t^-,t)}(r) \qquad \text{for }r \le r_2.
    \end{aligned}
    \]
    (see Figure \ref{fig:unique starting point}).

    In particular, since $r_2 < r_1 < r_0 < t$, there are two distinct geodesics from $(\wt g(r_2 - 1),r_2 - 1)$ to $(\intc_t^-,t)$ that agree in neighborhoods of both endpoints. This is a contradiction to Lemma \ref{lem:no_bubbles}.
   \end{proof}
   As a corollary, we get the following. 
   \begin{corollary}
       \label{c:in between points}
       The following holds on the event $\Omega_1$: Let $\dir \in \R, (x_0,s) \in \R^2$ and   $t>s$ be such that $\intc_{(x_0,s)}^{\dir,-}(r) < \intc_{(x_0,s)}^{\dir,+}(r)$ for all $r \in (s,t]$. Let $f = f_{(x_0,s)}^\dir$ and $\intc_t^\sig = \intc_{(x_0,s)}^{\dir,\sig}(t)$. Then, for any $x \in (\intc^-_{t},\intc_t^+)$ and for any $(f,s)$-to-$(x,t)$ geodesic $\pi_{(x,t)}$, 
       \[
       \pi_{(x,t)}(s)=x_0.   \]
       Moreover, if $\pi_{(x,t)}$ is any $(f,s)$-to-$(x,t)$ geodesic, it is the restriction of a $\dir-$ geodesic and a $\dir+$ geodesic to $[s,t]$. 
   \end{corollary}
   \begin{proof}
       By Lemma \ref{lem:geod}, we know that $\pi_{(x,t)}$ is either a restriction of a $\dir+$ geodesic or a $\dir-$ geodesic. 
        By ordering of  geodesics along with \eqref{eq:5.1eq1}--\eqref{eq:5.1eq2} in Proposition \ref{p:IntfDist},  we have, for all $r \in [s,t]$,
       \[
       g^{\dir+,R}_{(\intc_t^-,t)}(r)=g^{\dir-,R}_{(\intc_t^-,t)}(r) \leq \pi_{(y,t)}(r)\leq g^{\dir+,L}_{(\intc_t^+,t)}(r)=g^{\dir-,L}_{(\intc_t^+,t)}(r),
       \]
       By Equation \eqref{eq:5.1eq3} of Proposition \ref{p:IntfDist}, 
       \[
       g^{\dir+,R}_{(\intc_t^-,t)}(s)=g^{\dir-,L}_{(\intc_t^+,t)}(s)=x_0.
       \]
       Therefore,$
       \pi_{(x,t)}(s)=x_0$, as desired. 
     
      To prove the second statement, assume that $\pi_{(x,t)}$ is an $(f,s)$-to-$(x,t)$ geodesic. By Lemma \ref{lem:geod}, $\pi_{(y,t)}$ is either a restriction of $\dir+$ geodesic or $\dir-$ geodesic. Now, if it is a restriction of $\dir-$ geodesic then 
      \[
      \Ll(x_0,s;y,t)=W^{\dir-}(x_0,s;y,t)=W^{\dir+}(x_0,s;y,t),
      \]
      where the first equality holds by our assumption that $\pi_{(y,t)}$ is a $\dir +$ geodesic (using \eqref{eqn:SIG_weight}) and the second equality holds by Lemma \ref{l:MV}\ref{it:Bu2}. By Lemma \ref{lem:Buse_eq}, $\pi_{(x,t)}$ is also a restriction of $\dir+$ geodesic. If $\pi_{(y,t)}$ is a restriction of $\dir+$ geodesic, then a symmetric argument follows. This concludes the proof. 
   \end{proof}

  \begin{proposition}
       \label{p:>interface}
       The following holds on the event $\Omega_1$: for all $\dir \in \R, (x_0,s) \in \R^2$ the following is true: assume that there is a $t>s$ such that $\intc_{(x_0,s)}^{\dir,-}(r) < \intc_{(x_0,s)}^{\dir,+}(r)$ for all $r \in (s,t]$. Then, letting $\intc^\sig_r = \tau_{(x_0,s)}^{\dir,-}(r)$, we have
       \begin{align}
       \label{eq:5.3eq1}
           &g^{\dir-,R}_{(\intc_t^-,t)}(r) = g^{\dir+,R}_{(\intc_t^-,t)}(r)>\intc_r^- \qquad \text{ for all } r \in (s,t), \quad\text{and}\\
           \label{eq:5.3eq2}
           &g^{\dir+,L}_{(\intc_t^+,t)}(r) = g^{\dir-,L}_{(\intc_t^+,t)}(r) <\intc_r^+ \qquad \text{ for all } r \in (s,t).
       \end{align}
   \end{proposition}
   \begin{figure}[t!]
    \includegraphics[width=7cm]{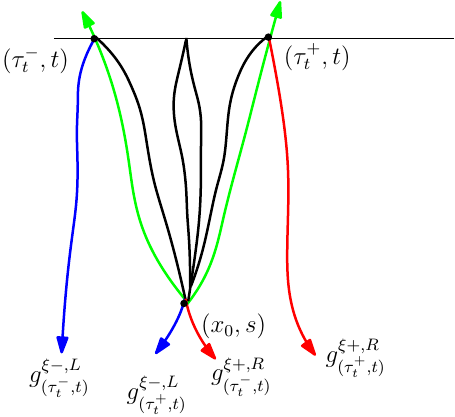}
       \caption{Proposition \ref{p:IntfDist}, Corollary \ref{c:in between points}, Proposition \ref{p:>interface}: When the leftmost interfaces and rightmost interfaces after starting from a point do not meet again for some time then a funnel like the above figure should occur. $g^{\dir-,L}_{(\intc_t^-,t)}$ and $g^{\dir+,R}_{(\intc_t^+,t)}$ will never touch the interfaces (this is proved in Lemma \ref{lem:Splt}, \ref{itm:G1}), $g^{\dir+,R}_{(\intc_t^-,t)}$ and $g^{\dir-,L}_{(\intc_t^+,t)}$ both have to pass through $(x_0,s)$ (this is proved in Proposition \ref{p:IntfDist}). $g^{\dir+,R}_{(\intc_t^-,t)}$ and $g^{\dir-,L}_{(\intc_t^+,t)}$ will not touch the interface except at $(x_0,s)$ (this is proved in Proposition \ref{p:>interface}).  Geodesics to all the in between points in the funnel will be restrictions of both $\dir-$ and $\dir+$ geodesics and they all pass through $(x_0,s)$. These geodesics are coloured in black. This is proved in Corollary \ref{c:in between points}}
       \label{fig:funnel}
   \end{figure}
   \begin{remark}
       Before proving Proposition \ref{p:>interface}, we comment that this result is of a similar flavor, but distinct, from Lemma \ref{lem:Splt}\ref{itm:G1}. Note that Lemma \ref{lem:Splt}\ref{itm:G1} holds without the assumption that $\intc_t^- < \intc_t^+$. Furthermore, Lemma \ref{lem:Splt}\ref{itm:G1} proves that the $\dir -,L$ geodesic from $(\intc^-_t,t)$ lies to the left of $\intc_r^-$ for $r \in [s,t)$, and the $\dir +,R$ geodesic from $(\intc_t^+,t)$ lies to the right of $\intc_r^+$ for $r \in [s,t)$. In Proposition \ref{p:>interface}, we evaluate the $\dir +,R$ and $\dir -,R$  geodesics from $(\intc_t^-,t)$ and the $\dir -,L$ and $\dir +,L$ geodesics from $(\intc_t^+,t)$. Furthermore, unlike Lemma \ref{lem:Splt}\ref{itm:G1}, Equations \eqref{eq:5.3eq1}-\eqref{eq:5.3eq2} do \textit{not} extend to $r = s$. Indeed, Equation \eqref{eq:5.1eq3} in Proposition \ref{p:IntfDist} states that
       \[
       g^{\dir-,R}_{(\intc_t^-,t)}(s) = g^{\dir+,R}_{(\intc_t^-,t)}(s)= g^{\dir+,L}_{(\intc_t^+,t)}(s) = g^{\dir-,L}_{(\intc_t^+,t)}(s) = x_0 = \tau_s^{+} = \tau_s^- .
       \]
       See Figure \ref{fig:funnel}. This shows that $(f_{(x_0,s)}^{\dir},s)$-to-$(x,t)$ geodesics can intersect mixed $\dir$-Busemann interfaces only at points where the leftmost and rightmost interfaces split.
   \end{remark}
   \begin{proof}[Proof of Proposition \ref{p:>interface}]  The equalities in \eqref{eq:5.3eq1}-\eqref{eq:5.3eq2} are a consequence of Equations \eqref{eq:5.1eq1}-\eqref{eq:5.1eq2} in Proposition \ref{p:IntfDist}. We prove the inequality in \eqref{eq:5.3eq1}, and the inequality in \eqref{eq:5.3eq2} follows a symmetric argument. By way of contradiction, assume that there exists $r_0 \in (s,t)$ such that 
   \begin{equation}\label{eq75}
     g^{\dir+,R}_{(\intc_t^-,t)}(r_0) \leq \intc_{r_0}^-   
   \end{equation}
   \begin{figure}[t!]
           \includegraphics[width=10cm]{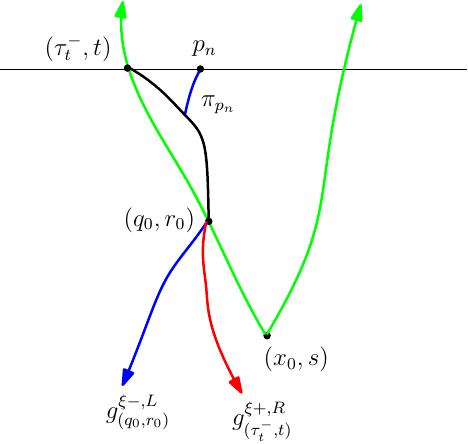}
           \caption{To prove Proposition \ref{p:>interface} we see that if $g^{\dir+,R}_{(\intc_t^-,t)}$ intersects the leftmost interface as in the above figure then there are $p_n$ such that for sufficiently large $n, p_{n}$ the leftmost $(f,s)$-to-$p_{n}$ geodesic $\pi_{p_{n}}$ coalesces with $g^{\dir+,R}_{(\intc_t^-,t)}$ at a time greater than $r_0$ (before $r_0$ when traveling south). So, one can swap $\pi_{p_{n}}$ on $[s,r_0]$ by $g^{\dir-,L}_{(q_0,r_0)}|_{[s,r_0]}$ to obtain another geodesic $\pi'_{p_{n}}$ with $\pi'_{p_{n}}(s)<x_0.$ This is a contradiction to Corollary \ref{c:in between points}.}
           \label{fig:geodesics and splitting points}
       \end{figure}
   We  use an argument similar to the that in the proof of Lemma \ref{lem:Splt}. Let $p_n=(x_n,t)$ be a sequence of points such that $x_n \in (\intc_t^-,\intc_t^+)$ and $x_n \notin \NU_{t,\dir -}$ for all $n \ge 1$, and $x_n \rightarrow \intc_t^-$, where $\NU_{t,\dir-}$ is defined in \eqref{eq:NUdef}. This is possible since $\NU_{t,\dir-}$ is countable on each time level (Lemma \ref{lem:NUcount}). Let $f = f_{(x_0,s)}^\dir$, and, for $n \ge 1$, let $\pi_{p_n}$ be the leftmost  $f$ to $p_n$ geodesic. By Corollary \ref{cor:LRmost_pi}, $\pi_{p_n}=g^{\dir-,L}_{p_n}|_{[s,t]}$. Since we also chose $p_n \notin \NU_{\dir -}$,  we have $g^{\dir-,L}_{p_n} = g^{\dir -,R}_{p_n}$. Since $L$ geodesics are the leftmost geodesics between their points and $R$ geodesics are the rightmost geodesics between their points (Proposition \ref{prop:DL_SIG_cons_intro}\ref{itm:DL_LRmost_geod}), $\pi_{p_n}$ is the unique geodesic between any of its points. By ordering of $\dir-$ geodesics, we get that, for all $r \in [s,t]$,
   \[
   g^{\dir+,R}_{(\intc_t^-,t)}(r)=g^{\dir-,R}_{(\intc_t^-,t)}(r)\leq \pi_{p_n}(r),
   \]
   where the first equality is due to \eqref{eq:5.1eq1}.
   By arguing as before we see that $\pi_{p_{n}}$ converges in the Hausdorff topology to the rightmost  $f$ to $(\intc_t^-,t)$ geodesic, i.e.,\ 
   \begin{equation}
       \pi_{p_{n}}\rightarrow g^{\dir+,R}_{(\intc_t^-,t)}|_{[s,t]}=g^{\dir-,R}_{(\intc_t^-,t)}|_{[s,t]}.
   \end{equation}
    Since $\pi_{p_n}$ is the unique geodesic between its points, Lemma \ref{lem:overlap} implies that the overlap between $\pi_{p_n}$ and $g^{\dir+,R}_{(\intc_t^-,t)}|_{[s,t]}$ is an interval whose endpoints converge to $s$ and $t$. Since $r_0 \in (s,t)$,  \eqref{eq75} implies that for all sufficiently large $n$ (see Figure \ref{fig:geodesics and splitting points}),
   \[
    \pi_{p_{n}}(r_0) =g^{\dir+,R}_{(\intc_t^-,t)}(r_0) \leq \intc^-_{r_0}.
   \]
   This in turn, by ordering of geodesics and Lemma \ref{lem:Splt}\ref{itm:G1} implies
   \[
       \pi_{p_{n}}(s)\leq g^{\dir-,L}_{(\intc_{r_0}^-,r_0)}(s)<x_0.
   \]
    However, since $p_{n}=(x_{n},t)$ where $\intc^-_t<x_{n}$, Corollary \ref{c:in between points} implies $\pi_{p_n}(s) = x_0$, a contradiction.
   \end{proof}

   \section{The interface network}
   \label{s:int_for}
This section is split into two subsections. In Section \ref{sec:interface_dynamics}, we show in Lemma \ref{p:noBub} that left and right $\dir$-mixed Busemann interfaces started from the same initial point intersect infinitely often, and we show in Lemma \ref{lem:coal_int} that interfaces from different initial points on the same time level intersect if and only if they lie in the same interval $\Intr_\alpha^{t,\dir}$ defined in \eqref{eq77}. This is the key to Theorem \ref{thm:int1}\ref{BiInt:it3}-\ref{itm:disj_int_disj}. In Section \ref{sec:along_interface}, we show that the points lying along bi-infinite $\dir$-mixed Busemann interfaces are exactly the points $\Split_{t,\dir}$. This is the key to Theorem \ref{thm:BiInt}. 
   
\subsection{Interface dynamics} \label{sec:interface_dynamics}
 Fix $\dir \in \R$, and let 
   \begin{equation}
       \mathfrak{t}^\dir_{(x_0,s)}(t)=\inf\{r>t:\intc^{\dir,-}_{(x_0,s)}(r)=\intc^{\dir,+}_{(x_0,s)}(r)\} 
   \end{equation}
   be the first time after $t$ that the leftmost and rightmost interfaces meet. 
   \begin{lemma}[Interfaces which meet must meet infinitely often]\label{p:noBub}
       On the event $\Omega_1$, for every $\dir\in\R$, $(x_0,s)\in\R^2$ and $t>s$,
       \begin{equation}\label{eq49}
           \mathfrak{t}^\dir_{(x_0,s)}(t)<\infty.
       \end{equation}
   \end{lemma}
   \begin{proof}
       Fix $\dir\in \R$ and $(x_0,s)\in \R^2$. For ease of notation, we let $(x_0,s)=(0,0)$. For shorthand notation, for $r > 0$, we define $\intc^{\sig}_r = \intc_{(0,0)}^{\dir,\sig}(r)$ for $\sigg \in \{-,+\}$.
       Suppose, by way of contradiction, that $ \mathfrak{t}_{(0,0)}^\dir(t)=\infty$ for some $t \ge 0$. By consistency of interfaces in Proposition \ref{p:rest}, we may assume, without loss of generality, that $t = 0$. For $k \in \Z_{>0}$, let $x_k=(\intc^-_k+\intc^+_k)/2$, and define $g_k=g^{\dir-,L}_{(x_k,k)}$.  Note that, by our assumption, $\mathfrak{t}_{(0,0)}^\dir(t) = \infty$, we can apply Proposition \ref{p:>interface} and ordering of geodesics to conclude that, for $k >  1$,
       \begin{equation} \label{eq:t1pmbd}
           \intc^-_{1}<g^{\dir-,R}_{(\intc^-_k,k)}(1) \leq g_k(1) \leq g^{\dir-,L}_{(\intc^+_k,k)}(1)<\intc^+_{1}.
       \end{equation}
       Using ordering of geodesics again, this further implies that 
       \begin{equation}
           g^{\dir-,L}_{(\intc^-_1,1)}(r)\leq g_k(r)\leq g^{\dir-,R}_{(\intc^+_1,1)}(r) \qquad \forall r\in [0,1]. 
       \end{equation}
By \eqref{eq:t1pmbd}, Equation \eqref{eq:5.1eq3} of Proposition \ref{p:IntfDist} implies that $g_k(0) = 0$ for all $k \ge 1$. By compactness, there exists a subsequence $k^1_{\ell}$ such that \
       \begin{equation}
           g_{k^1_l}|_{[0,1]}\rightarrow \Tilde{g}_1,
       \end{equation}
       uniformly, where The limiting object $\Tilde{g}_1$ is a geodesic by Lemma \ref{lem:precompact}. $\Tilde{g}_1$ is  geodesic on $[0,1]$. By a similar argument, we can find a subsequence  $\{k^2_{\ell}\}\subseteq \{k^1_\ell\}$ such that \ 
       \begin{equation}
           g_{k^2_{\ell}}|_{[0,2]}\rightarrow \Tilde{g}_2,
       \end{equation}
       uniformly on compact sets,  where $\Tilde{g}_2$ is a geodesic on $[t,t+2]$ and $\Tilde{g}_2|_{[0,1]}=\Tilde{g}_1$.  Using a diagonal argument, we can find a sequence $g_{k^\ell_{\ell}}$ such that \ $g_{k^\ell_{\ell}}\rightarrow \Tilde{g}_\infty$ as $\ell\rightarrow \infty$. Since we showed that $g_k(0) = 0$ for all $k \ge 0$, we have   $\Tilde{g}_\infty(0)=0$. We now define
       \begin{equation}
           g_{\infty}(r)=
           \begin{cases}
               \Tilde{g}_\infty(r) & r\geq 0\\
               g^{\dir-,L}_{(0,0)}(r) & r<0.
           \end{cases}
       \end{equation}
       We claim $g_\infty$ is a bi-infinite geodesic. Since $g^{\dir-,L}_{(0,0)}$ is a $\dir-$ geodesic, by Lemma \ref{lem:Buse_eq} and the construction of general $\dir-$ geodesics in Proposition \ref{prop:DL_SIG_cons_intro}\ref{itm:arb_geod_cons}, it suffices to show that, for all $r' > 0$,
       \[
       \Ll(0,0;r',g_\infty(r')) = W^{\dir-}(0,0;r',g_\infty(r')).
       \]
       Indeed, since $g_k =g_{(x_k,k)}^{\dir-,L}$, by \eqref{geod_LR_eq_L}, we have, for all $k \ge r'$,
       \[
       \Ll(0,0;r',g_k(r')) = W^{\dir-}(0,0;r',g_k(r')).
       \]
       and the result follows by taking limits and using continuity of the Busemann functions. This contradicts the non-existence of bi-infinite geodesics in the directed landscape (Proposition \ref{prop:bigeod_non}).
   \end{proof}
We now prove an intermediate lemma, which is a restatement of Theorem \ref{thm:int1}\ref{int1:it3}.
\begin{lemma} \label{lem:int_order}The following holds on the event $\Omega_1$. Let $\sigg \in \{-,+\}$, $x_0 < y_0$ and $s \in \R$. Then, 
           \begin{equation}\label{eq61}
                \intc^{\dir,\sig}_{(x_0,s)}(t)\leq \intc^{\dir,\sig}_{(y_0,s)}(t) \qquad \forall t\geq s.
           \end{equation}
           Furthermore, if $\intc^{\dir,\sig}_{(x_0,s)}(r) = \intc^{\dir,\sig}_{(y_0,s)}(r)$ for some $r \ge s$, then $\intc^{\dir,\sig}_{(x_0,s)}(t) = \intc^{\dir,\sig}_{(y_0,s)}(t)$ for all $t \ge r$.
           \end{lemma}
           \begin{proof}
           By definition, we have
           \be \label{intc_inits}
    \intc^{\dir,\sig}_{(x_0,s)}(s) = x_0 < y_0 = \intc^{\dir,\sig}_{(y_0,s)}(s),
           \ee
           and if $\intc^{\dir,\sig}_{(x_0,s)}(r) = \intc^{\dir,\sig}_{(y_0,s)}(r)$ for some $r > s$, then $\intc^{\dir,\sig}_{(x_0,s)}(t) = \intc^{\dir,\sig}_{(y_0,s)}(t)$ for all $t \ge r$ by by Proposition \ref{p:rest}. The inequality \eqref{eq61} now follows from continuity of $\intc_{(x_0,s)}^{\dir,\sig}$ and  $\intc_{(y_0,s)}^{\dir,\sig}$.
\end{proof}
   
   The following result shows the dynamics of interfaces if they meet (see Figure \ref{fig:intf_ord}). Later, in Lemma \ref{lem:coal_int}, we characterize the pairs of points for which the interfaces meet.  
   \begin{figure}[t!]
       \includegraphics[width=10cm]{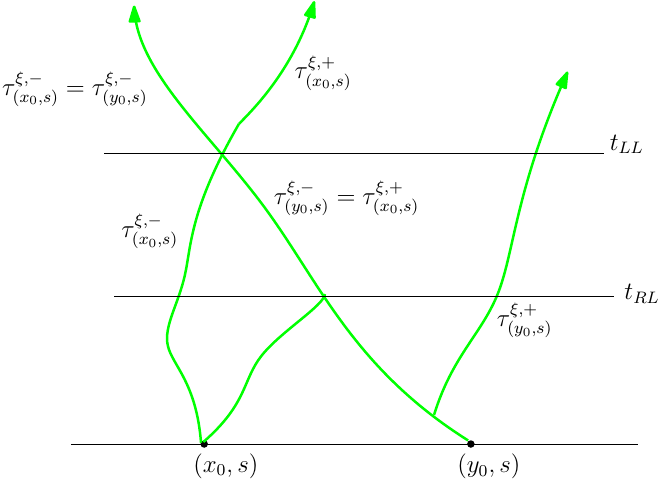}
       \caption{  
       A depiction of Lemma \ref{l:ordrIn}. Four interfaces leave from two ordered points on the line. At time $\mathfrak{t}_{RL}$, $\intc^{\dir,-}_{(y_0,s)}$ meets $\intc^{\dir,+}_{(x_0,s)}$ after which they continue together until $\intc^{\dir,-}_{(x_0,s)}$ and $\intc^{\dir,-}_{(y_0,s)}$ coalesce or $\intc^{\dir,+}_{(x_0,s)}$ and $\intc^{\dir,+}_{(y_0,s)}$ coalesce, whichever happens earlier. After $\mathfrak{t}_{LL} \wedge \mathfrak{t}_{RR}$, it is possible that $\intc^{\dir,+}_{(x_0,s)}$ will branch off.}
       \label{fig:intf_ord} 
   \end{figure}
    \begin{lemma}[Order of interfaces] \label{l:ordrIn}
       The following holds on the event $\Omega_1$ for all $\dir\in\R$, $s\in\R$, and $x_0 <  y_0 \in\R$. For these choices, define
           \begin{equation}\label{eq71}
           \begin{aligned}
                 \mathfrak t_{RL}&=\inf\{r>s:\intc^{\dir,+}_{(x_0,s)}(r)=\intc^{\dir,-}_{(y_0,s)}(r)\}\\
               \mathfrak t_{LL}&=\inf\{r>s:\intc^{\dir,-}_{(x_0,s)}(r)=\intc^{\dir,-}_{(y_0,s)}(r)\}\\
               \mathfrak t_{RR}&=\inf\{r>s:\intc^{\dir,+}_{(x_0,s)}(r)=\intc^{\dir,+}_{(y_0,s)}(r)\}.
           \end{aligned}
           \end{equation}
           Then, 
       \begin{align}
       \label{eq70}
           \intc^{\dir,+}_{(x_0,s)}(r)&\geq \intc^{\dir,-}_{(y_0,s)}(r) \quad\text{for}\quad r\in [\mathfrak t_{RL},\infty), \quad\text{and, moreover,} \\
       \label{eq50}
           \intc^{\dir,+}_{(x_0,s)}(r) &=  \intc^{\dir,-}_{(y_0,s)}(r) \quad\text{for}\quad r\in [\mathfrak t_{RL},\mathfrak t_{LL}\wedge \mathfrak t_{RR}] 
       \end{align}
       \end{lemma}
       \begin{remark}
           In words, Lemma \ref{l:ordrIn} says the following: we look at the interfaces from two ordered vertices $(x_0,s)$ and $(y_0,s)$ with $x_0<y_0$. Then all the terms defined in the lemma could be infinite. i.e., interfaces from $(x_0,s)$ never meet interfaces from $(y_0,s)$. When the two interfaces do meet, $\mathfrak{t}_{RL}$ is finite and  $\tau^{\dir,+}_{(x_0,s)}$ will always stay to the right of $\tau^{\dir,-}_{(y_0,s)}$ once they meet. In addition if either $\mathfrak{t}_{LL}$ or $\mathfrak{t}_{RR}$ is also finite then $\tau^{\dir,+}_{(x_0,s)}$ and $\tau^{\dir,-}_{(y_0,s)}$ cannot separate until $\mathfrak{t}_{LL} \wedge \mathfrak{t}_{RR}$, i.e. they will stay together until either the two leftmost or rightmost interfaces meet (see Figure \ref{fig:intf_ord}).
       \end{remark}
       \begin{proof}[Proof of Lemma \ref{l:ordrIn}]
       We first observe that $\mathfrak t_{LL}\wedge \mathfrak t_{RR} \geq \mathfrak t_{RL}$. Indeed, if $\intc^{\dir,+}_{(x_0,s)}(t) = \intc^{\dir,+}_{(y_0,s)}(t)$ for some $t > s$, then
           \[
           \intc^{\dir,-}_{(y_0,s)}(t) \le \intc^{\dir,+}_{(y_0,s)}(t) = \intc^{\dir,+}_{(x_0,s)}(t).
           \]
            Hence, since $\intc_{(x,s)}^{\dir,\sig}(s)=x$ for all $x \in \R$, $\sigg \in \{-,+\}$,
            continuity of interfaces (Lemma \ref{lem:interface_continuous}) and the intermediate value theorem imply that $\intc^{\dir,-}_{(y_0,s)}(r) = \intc^{\dir,+}_{(x_0,s)}(r)$ for some $r \in (s,t]$. Thus, $\mathfrak t_{RR} \ge \mathfrak t_{RL}$. A similar argument shows that $\mathfrak t_{LL} \ge \mathfrak t_{RL}$.

            We now prove \eqref{eq70}. By definition of $\mathfrak t_{RL}$ and continuity, we have 
            \[
            \hat x := \intc_{(x_0,s)}^{\dir,+}(\mathfrak t_{RL}) = \intc_{(y_0,s)}^{\dir,-}(\mathfrak t_{RL}),
            \]
            so by the consistency of the interfaces in Proposition \ref{p:rest}, for $r \ge \mathfrak t_{RL}$,
            \[
            \intc^{\dir,+}_{(x_0,s)}(r) = \intc_{(\hat x,\mathfrak t_{RL})}^{\dir,+}(r)  \geq \intc_{(\hat x,\mathfrak t_{RL})}^{\dir,-}(r) = \intc^{\dir,-}_{(y_0,s)}(r),
            \]
            giving \eqref{eq70}. 
            
            We now prove \eqref{eq50}. By definition, the equality holds for $r = \mathfrak t_{RL}$. Then, by continuity, it suffices to show the equality holds for all $r$ in the (possibly empty) set $(\mathfrak t_{RL},\mathfrak t_{LL} \wedge \mathfrak t_{RR})$. Assume, by way of  contradiction, that there exists $r_0\in (\mathfrak t_{RL},\mathfrak t_{LL}\wedge \mathfrak t_{RR})$ such that $\intc^{\dir,-}_{(y_0,s)}(r_0) \neq \intc^{\dir,+}_{(x_0,s)}(r_0)$. Then, by \eqref{eq70}, we have $\intc^{\dir,+}_{(x_0,s)}(r_0) > \intc^{\dir,-}_{(y_0,s)}(r_0)$, so we can choose
    \begin{equation}\label{eq63}
           z\in\bigl(\intc^{\dir,-}_{(y_0,s)}(r_0),\intc^{\dir,+}_{(x_0,s)}(r_0)\bigr).    
           \end{equation}
           Since $r_0 < \mathfrak t_{LL} \wedge \mathfrak t_{RR}$, Lemma \ref{lem:int_order} and continuity of interfaces implies that 
         \begin{equation} \label{eq:taupmz_order}
         \begin{aligned}
             \intc^{\dir,-}_{(x_0,s)}(r_0) &< \intc^{\dir,-}_{(y_0,s)}(r_0)  < z < \intc^{\dir,+}_{(x_0,s)}(r_0) < \intc^{\dir,+}_{(y_0,s)}(r_0).
         \end{aligned}
            \end{equation}
We now argue that this implies there exists $r_1 \in [\mathfrak t_{RL},r_0)$ so that 
\be \label{ghits_one}
\text{either}\quad g_{(z,r_0)}^{\dir +R}(r_1)  = \intc_{(x_0,s)}^{\dir,+}(r_1),\quad\text{or}\quad g_{(z,r_0)}^{\dir +R}(r_1)  = \intc_{(y_0,s)}^{\dir,-}(r_1).
\ee
See Figure \ref{fig:LR figure}.
\begin{figure}[t!]
    \includegraphics[width=10cm]{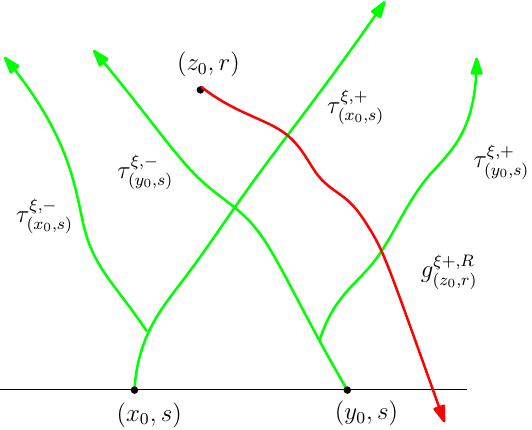}
    \caption{We show a contradiction if the equality in \eqref{eq50} does not hold. If there is a $z$ as in the figure then $g^{\dir+,R}_{(z,r_0)}$ must intersect $\tau^{\dir,+}_{(x_0,s)}$ or $\tau^{\dir,-}_{(y_0,s)}$ in $[\mathfrak{t}_{RL},\infty)$ and the interfaces will look like as in the figure. But then this is a contradiction to Proposition \ref{p:>interface}.}
    \label{fig:LR figure}
\end{figure} 

Indeed, if $g_{(z,r_0)}^{\dir +R}(r) \neq \intc_{(x_0,s)}^{\dir,+}(r)$ for all $r \in [\mathfrak t_{RL},r_0)$, then continuity of geodesics and interfaces along with \eqref{eq:taupmz_order} implies that $g_{(z,r_0)}^{\dir +,R}(r) < \intc^{\dir,+}_{(x_0,s)}(r)$ for all $r \in [\mathfrak t_{RL},r_0)$. Then, we observe that 
\[
\intc_{(y_0,s)}^{\dir,-}(\mathfrak t_{RL}) = \intc_{(x_0,s)}^{\dir,+}(\mathfrak t_{RL}) >  g_{(z,r_0)}^{\dir +R}(\mathfrak t_{RL}).
\]
However, from \eqref{eq:taupmz_order}, we have
\[
\tau_{(y_0,s)}^{\dir,-}(r_0) < z = g_{(z,r_0)}^{\dir +R}(r_0)< \tau^{\dir,+}_{(x_0,s)}(r_0),
\]
so continuity of geodesics and interfaces, along with the intermediate value theorem, imply that the second equality of \eqref{ghits_one} holds for some $r_1 \in (\mathfrak t_{RL},r_0)$.

Now, we obtain a contradiction by showing that \eqref{ghits_one} cannot be true. We rule out each equality in separate cases. 

\medskip \noindent \textbf{Case 1: Showing $g_{(z,r_0)}^{\dir +,R}(r_1) < \intc_{(x_0,s)}^{\dir,+}(r_1)$ for all $r_1 \in [\mathfrak t_{RL},r_0)$.}
We first claim that (see Figure \ref{fig:LR figure}) 
\be \label{maxletRL}
\sup\{r \in [s,r_0): \intc_{(x_0,s)}^{\dir,-}(r) = \intc_{(x_0,s)}^{\dir,+}(r)\} < \mathfrak t_{RL}.
\ee
Indeed, if \eqref{maxletRL} fails, then there exists $r \in[\mathfrak t_{RL},r_0)$ so that 
\[
\intc_{(x_0,s)}^{\dir,-}(r) = \intc_{(x_0,s)}^{\dir,+}(r) \ge \intc_{(y_0,s)}^{\dir,-}(r),
\]
where the inequality follows by \eqref{eq70}. But then, continuity of interfaces and the definition of $\mathfrak t_{LL}$ implies that $ r \ge \mathfrak t_{LL}$, a contradiction since $r < r_0 < \mathfrak t_{LL} \wedge \mathfrak t_{RR}$.

By \eqref{maxletRL} and consistency of interfaces (Proposition \ref{p:rest}), we may apply \eqref{eq:taupmz_order} and ordering of geodesics (Proposition \ref{prop:g_basic_prop}\ref{itm:DL_SIG_mont_x}), followed by Equation \eqref{eq:5.3eq2} of Proposition \ref{p:>interface} to conclude
\[
g_{(z,r_0)}^{\dir +,R}(r_1) \le g_{(\intc_{(x_0,s)}^{\dir,+}(r_0),r_0)}^{\dir +,L}(r_1) < \intc_{(x_0,s)}^{\dir,+}(r_1),
\]
and this proves the desired statment for this case. 

\medskip \noindent \textbf{Case 2: Showing $g_{(z,r_0)}^{\dir +R}(r_1)  > \intc_{(y_0,s)}^{\dir,-}(r_1)$ for all $r_1 \in [\mathfrak t_{RL},r_0)$. } The proof of this case is similar. A symmetric argument shows that 
\[
\sup\{r \in [s,r_0): \intc_{(y_0,s)}^{\dir,-}(r) = \intc_{(y_0,s)}^{\dir,+}(r)\} < \mathfrak t_{RL},
\]
and this allows us to use Equation \eqref{eq:5.3eq1} of Proposition \ref{p:>interface} (after an application of geodesic monotonicity) to conclude
\[
g_{(z,r_0)}^{\dir +,R}(r_1) \ge g_{(\intc_{(y_0,s)}^-(r_0),r_0)}^{\dir +,R}(r_1) > \intc_{(y_0,s)}^{\dir,-}(r_1). \qedhere
\]
       \end{proof}
       For the next result, recall the definition of the sets $\Intr^{s,\dir}_\alpha$ defined below \eqref{eq77}. 
       
   \begin{lemma}[Meeting of interfaces]
   \label{lem:coal_int}
       On the event $\Omega_1$, the following hold for all $\dir\in \DLBusedc$ and $s\in\R$:
       \begin{enumerate} [label=\rm(\roman{*}), ref=\rm(\roman{*})]
           \item \label{Itco1}If $x_0,y_0\in \Intr^{s,\dir}_\alpha$ for some $x_0 \le y_0$ and $\alpha \in \I^{t,\dir}$, then $\mathfrak t_{RL},\mathfrak t_{LL},\mathfrak t_{RR},\mathfrak t_{LR}$ are all finite, where $\mathfrak t_{RL},\mathfrak t_{LL},\mathfrak t_{RR}$ are defined in \eqref{eq71}, and
       \[
          \mathfrak t_{LR} =  \inf\{t > s: \intc_{(x_0,s)}^{\dir,-}(t) = \intc_{(y_0,s)}^{\dir,+}(t)\}.
       \]
       \item \label{Itco2} If $x_0\in\Intr^{s,\dir}_\alpha$ and $y_0\in \Intr^{s,\dir}_\beta$ for $x_0 < y_0$ and $\alpha\neq \beta$, then 
       \begin{equation}
           \intc^{\dir,+}_{(x_0,s)}(r)<\intc^{\dir,-}_{(y_0,s)}(r), \qquad \forall r\geq s.
       \end{equation}
       \end{enumerate}
   \end{lemma}
   \begin{proof}
        \textbf{Item \ref{Itco1}:} If $x_0 = y_0$, then the result follows from Proposition \ref{p:noBub}, so we may assume $x_0 < y_0$. Let $\mathfrak t_{RL},\mathfrak t_{RR},\mathfrak t_{LL}$ be defined as in \eqref{eq71}. We begin by showing the weaker statement that $\mathfrak t_{RL}<\infty$.  As $x_0,y_0\in \Intr^{s,\dir}_\alpha$, monotonicity of $D_s^\dir$ implies $D_s^\dir(z) = D_s^\dir(x_0)$ for all $z \in [x_0,y_0]$. Rearranging, this gives us that 
        \[
        W^{\dir -}(x_0,s;z,s) = W^{\dir +}(x_0,s;z,s), \quad z \in [x_0,y_0].
        \]
        Then, from the definition of $f_{(y_0,s)}^{\dir}$ \eqref{eq: initial condition},
        \begin{align*}
        W^{\dir -}(x_0,s;y_0,s)  + f_{(y_0,s)}^{\dir} &=  
    \begin{cases}
         W^{\dir -}(x_0,s;y_0,s)  + W^{\dir-}(y_0,s;z,s) & z\leq  y_0,\\
         W^{\dir -}(x_0,s;y_0,s)  + W^{\dir+}(y_0,s;z,s) & z  \ge y_0.
     \end{cases} \\
     &= \begin{cases}
     W^{\dir -}(x_0,s;z,s) &z \le y_0 \\
     W^{\dir +}(x_0,s;z,s) &z \ge y_0 
     \end{cases} \\
     &= \begin{cases}
     W^{\dir -}(x_0,s;z,s) &z \le x_0 \\
     W^{\dir -}(x_0,s;z,s) &x_0 \le z \le y_0 \\
     W^{\dir +}(x_0,s;z,s) &z \ge y_0 
     \end{cases} \\
     &= \begin{cases}
     W^{\dir -}(x_0,s;z,s) &z \le x_0 \\
     W^{\dir +}(x_0,s;z,s) &z \ge x_0 
     \end{cases}
        \end{align*}
        and this last line equals $f_{(x_0,s)}^\dir$. Hence, the functions $f_{(y_0,s)}^\dir$ and $f_{(x_0,s)}^\dir$ agree, up to the addition of a constant, so the maximizers of $z \mapsto f_{(x_0,s)}^\dir + \Ll(z,s;y,t)$ and $z \mapsto f_{(y_0,s)}^\dir + \Ll(z,s;y,t)$ agree.

        In particular, for all $(y,t) \in \R^2$,
       \begin{equation}\label{eq66}
            \chi^{L/R}(f^\dir_{(x_0,s)},s;y,t)= \chi^{L/R}(f^\dir_{(y_0,s)},s;y,t).
       \end{equation}
       Suppose, by way of contradiction, that $\mathfrak t_{RL}=\infty$. Then, for any sequence of real numbers $\{r_k\}$ with $r_k \ge s$ and $r_k \to \infty$, we have
       \begin{equation}
           \intc^{\dir,+}_{(x_0,s)}(r_k)<\intc^{\dir,-}_{(y_0,s)}(r_k) \qquad k\in\N.
       \end{equation}
       Choose such a sequence, and set 
       \begin{equation}
       x_k=\big(\intc^{\dir,+}_{(x_0,s)}(r_k)+\intc^{\dir,-}_{(y_0,s)}(r_k)\big)/2
       \end{equation}
       Then, since $\intc^{\dir,+}_{(x_0,s)}(r_k)<x_k<\intc^{\dir,-}_{(y_0,s)}(r_k)$,  Lemma \ref{lem:geod}\ref{it1}-\ref{it2} implies that a geodesic is an $f_{(y_0,s)}^\dir$-to-$(x_k,r_k)$ geodesic if and only if it is the restriction of a $\dir-$ geodesic to $[s,t]$, and a geodesic is a $f_{(x_0,s)}^{\dir}$-to-$(x_k,r_k)$ geodesic if and only if it is the restriction of a $\dir +$ geodesic to $[s,t]$. Then, by \eqref{eq66}, we have 
       \[
g^{\dir-,S}_{(x_k,r_k)}|_{[s,r_k]}=g^{\dir+,S}_{(x_k,r_k)}|_{[s,r_k]}, \qquad S \in \{L,R\}.
       \]
       Then, by Lemma \ref{lem:Splt}\ref{itm:G1}
        and ordering of Busemann geodesics, for $r \in [s,r_k)$, 
       \[
       \intc^{\dir,+}_{(x_0,s)}(r) < g^{\dir +,R}_{(\intc^{\dir,+}_{(x_0,s)}(r_k),r_k)}(r) \le g^{\dir +,S}_{(x_k,r_k)}(r) = g^{\dir -,S}_{(x_k,r_k)}(r)  \le g^{\dir -,L}_{(\intc^{\dir,-}_{(y_0,s)}(r_k),r_k)}(r) < \intc^{\dir,-}_{(y_0,s)}(r).
       \]
      By removing the second and fifth term above, we extend the strict inequality to all $r \in [s,r_k]$: that is,
       \be \label{eq67}
        \intc^{\dir,+}_{(x_0,s)}(r) <  g^{\dir +,S}_{(x_k,r_k)}(r) = g^{\dir -,S}_{(x_k,r_k)}(r)  < \intc^{\dir,-}_{(y_0,s)}(r).
       \ee
       Next, we continue similarly to the proof of Lemma \ref{p:noBub}, to construct a bi-infinite geodesic. We set $g_k=g^{\dir-,L}_{(x_k,r_k)}$. From \eqref{eq67} and a diagonal argument similar to the proof of Lemma \ref{p:noBub}, we see that there exists a subsequence $\{k_\ell\}$ such that $g_{k_\ell}$ converges to a bi-infinite geodesic. The only difference is that we may not always have the same value of $g_k(s)$, but in general,  from \eqref{eq67}, $x_0 < g_k(s) < y_0$, so by compactness, can still obtain the necessary convergent subsequences.  This leads to a contradiction, giving $\mathfrak t_{RL}<\infty$.

       Next we show that $\mathfrak t_{LL}<\infty$ and $\mathfrak t_{RR} < \infty$. We do this by first showing that $\mathfrak t_{LL} \wedge \mathfrak t_{RR} < \infty$. Suppose, by way of contradiction, that this is not the case. Then,  Lemma \ref{lem:int_order} and Equation \eqref{eq50} of Lemma \ref{l:ordrIn} implies that, for all $t \ge \mathfrak t_{RL}$,
       \[
       \intc_{(x_0,s)}^{\dir,-}(t) < \intc_{(y_0,s)}^{\dir,-}(t) = \intc_{(x_0,s)}^{\dir,+} < \intc_{(y_0,s)}^{\dir,+}(t).
       \]
       Then, $\mathfrak t_{(x_0,s)}^{\dir}(\mathfrak t_{RL}) = \infty$, contradicting Lemma \ref{p:noBub}.

       Now that we have shown $\mathfrak t_{LL} \wedge \mathfrak t_{RR} < \infty$, we show that $\mathfrak t_{LL} \vee \mathfrak t_{RR} < \infty$. Assume, without loss of generality that $\mathfrak t_{LL} < \infty$. By consistency of interfaces in Proposition \ref{p:rest}, $\intc_{(x_0,s)}^{\dir,-}(t) = \intc_{(x_0,s)}^{\dir,+}(t)$ for all $t \ge \mathfrak t_{LL}$. Then, using Equation \eqref{eq70} of Lemma \ref{l:ordrIn}, for $t > \mathfrak t_{LL}$,
       \be \label{eq:572}
       \intc_{(x_0,s)}^{\dir,-}(t) = \intc_{(x_0,s)}^{\dir,+}(t) \ge \intc_{(y_0,s)}^{\dir,-}(t).
       \ee
       By Lemma \ref{p:noBub}, there exists $t > \mathfrak t_{LL}$ such that $\intc_{(y_0,s)}^{\dir,+}(t) = \intc_{(y_0,s)}^{\dir,+}(t)$. Combined with \eqref{eq:572}, we have $\intc_{(y_0,s)}^{\dir,+}(t) \le \intc_{(x_0,s)}^{\dir,+}(t)$ for such $t$, so $\mathfrak t_{RR} < \infty$ by the intermediate value theorem. 

       Lastly, we show $\mathfrak t_{LR} < \infty$. For all $t > \mathfrak t_{LL} \vee \mathfrak t_{RR}$, 
       \[
       \intc_{(x_0,s)}^{\dir,-}(t) = \intc_{(y_0,s)}^{\dir,-}(t),\quad\text{and}\quad \intc_{(x_0,s)}^{\dir,+}(t) = \intc_{(y_0,s)}^{\dir,+}(t),
       \]
       and by Lemma \ref{p:noBub}, there exists $t > \mathfrak t_{LL} \vee \mathfrak t_{RR}$ such that these are all equal. 

       \medskip \noindent \textbf{Item \ref{Itco2}:}
        Recall the definition of $f_{(x,s)}^\dir$ \eqref{eq: initial condition}. As $x_0$ and $ y_0$ lie on different intervals in $\Intr^{s,\dir}$, the monotonicity of $D_s^\dir$ implies that $\Delta := D_s^\dir(y_0)-D_s^\dir(x_0)>0$. By additivity of the Busemann functions, for all $w \ge y_0$,
       \begin{equation}
       \label{eq5.6.1}
       \begin{aligned}
    f_{(x_0,s)}^\dir(w) &= W^{\dir +}(x_0,s;w,s)\\
           &= \bigl(W^{\dir +}(x_0,s;y_0,s) - W^{\dir -}(x_0,s;y_0,s)\bigr) + \bigl(W^{\dir -}(x_0,s;y_0,s) + W^{\dir +}(y_0,s;w,s)\bigr) \\
           &= \Delta + W^{\dir -}(x_0,s;y_0,s) + f^\dir_{(y_0,s)}(w).
           \end{aligned}
       \end{equation}
    Let $r\in[s,\infty)$. From the continuity of $z\mapsto d^\dir_{(y_0,s)}(z,r)$, there exists $z_0 < \intc^-_{(y_0,s)}$ such that \ 
       \begin{equation}
           d^\dir_{(y_0,s)}(z_0,r)\in (-\Delta,0).
       \end{equation}
       Recall the definition of $d_{(x_0,s)}^\dir(z_0,r)$ below \eqref{eq:d}:
       \begin{equation}\label{eq73}
           d^\dir_{(x_0,s)}(z_0,r)=\sup_{w\geq x_0}[f_{(x_0,s)}^\dir(w)+\Ll(w,s;z_0,r)]-\sup_{w\leq x_0}[f_{(x_0,s)}^\dir(w)+\Ll(w,s;z_0,r)].
       \end{equation}
       From \eqref{eq5.6.1}, we get
\be \label{eq72p}
\begin{aligned}
&\quad \, \sup_{w\geq x_0}[f_{(x_0,s)}^\dir(w)+\Ll(w,s;z_0,r)] \\&\geq \sup_{w\geq y_0}[f_{(x_0,s)}^\dir(w)+\Ll(w,s;z_0,r)] \\
&= \sup_{w\geq y_0}[f_{(y_0,s)}^\dir(w)+\Ll(w,s;z_0,r)]+\Delta + W^{\dir -}(x_0,s;y_0,s).
\end{aligned}
\ee
Next, observe that for $w \le x_0$, 
\[
f_{(y_0,s)}^\dir(w) + W^{\dir -}(x_0,s;y_0,s) = W^{\dir -}(y_0,s;w,s) + W^{\dir -}(x_0,s;y_0,s) = W^{\dir -}(x_0,s;w,s) = f_{(x_0,s)}^\dir(w).
\]
Hence,
\begin{equation}\label{eq72}
           \begin{aligned}
               &\quad \, \sup_{w\leq x_0}[f_{(x_0,s)}^\dir(w)+\Ll(w,s;z_0,r)]  \\&=\sup_{w\leq x_0}[f_{(y_0,s)}^\dir(w)+\Ll(w,s;z_0,r)] + W^{\dir -}(x_0,s;y_0,s) \\
               &\leq \sup_{w\leq y_0}[f_{(y_0,s)}^\dir(w)+\Ll(w,s;z_0,r)] + W^{\dir -}(x_0,s;y_0,s).
           \end{aligned}
       \end{equation}
       Substituting \eqref{eq72p} and \eqref{eq72} into \eqref{eq73}, we obtain
       \begin{equation}
           d^\dir_{(x_0,s)}(z_0,r)\geq d^\dir_{(y_0,s)}(z_0,r)+\Delta>0,
       \end{equation}
       and this implies
       \begin{equation}
           z_0>\intc^{\dir,+}_{(x_0,s)}(r).
       \end{equation}
       This completes the proof since we chose $z_0 < \intc^{\dir,-}_{(y_0,s)}$.
   \end{proof}
  
  \subsection{The points along interfaces} \label{sec:along_interface}
Proposition \ref{p: existence of bi-infinite competetion interfaces} shows the existence of mixed $\dir$-Busemann bi-infinite interfaces. Later, Lemma \ref{lem:Splt}\ref{itm:branching} shows that the all interfaces (not necessarily bi-infinite ones) traverse through the branching points $\Branch_{t,\dir}$. We now show that $(x,t)$ lies on a leftmost \textit{bi-infinite} mixed $\dir$-Busemann interface if and only if $x \in \Split_{t,\dir}^L \subseteq \Branch_{t,\dir}$ and lies on a rightmost such interface if and only if $x \in \Split_{t,\dir}^R$.
\begin{proposition}
    \label{Cond_bi-inf} 
    On the event $\Omega_1$, for all $\dir \in \DLBusedc$ and $(y,t) \in \R^2$, there exists $\intc \in \Intc^{\dir,-}$ (resp. $\Intc^{\dir +}$) satisfying $\tau(t) = y$, if and only if $y \in \Split^L_{t,\dir}$ (resp. $y \in \Split^R_{t,\dir}$).
\end{proposition}
\begin{proof}
   Proposition \ref{p: existence of bi-infinite competetion interfaces} implies one direction of this statement; namely, if $y \in \Split_{t,\dir}^L$, then there exists $\intc \in \Intc^{\dir,-}$ satisfying $\intc(t) = y$, and analogously for $y \in \Split_{t,\dir}^R$. We now prove the converse. Suppose, by way of contradiction, that there exists $(y,t)$ with $y \notin \Split_{t,\dir}^L$ and $\intc \in \Intc^{\dir,-}$ so that $\intc(t) = y$. 
    By Lemma \ref{l:equ}, we see that there exists $a_0>0$ such that 
    \[
        D_t^\dir(y-x)=D_t^\dir(y), \qquad \forall x\in [0,a_0].
     \]
Written out using the definition of $D_t^\dir$ \eqref{Dfunc}, and taking $x = a_0$, we have 
\[
W^{\dir +}(0,t;y-a_0,t) - W^{\dir -}(0,t;y - a_0,t) = W^{\dir +}(0,t;y,t) - W^{\dir -}(0,t;y,t),
\]
which, when rearranged using the additivity of Busemann functions, implies that 
\be \label{eq69}
W^{\dir +}(y-a_0,t;y,t) = W^{\dir -}(y-a_0,t;y,t).
\ee
Since $\intc \in \Intc^{\dir,-}$, we have $y = \intc_{(\intc(s),s)}^{\dir,-}(t)$, for all $s \leq t$. So Lemma \ref{l:MV}\ref{it:Bu3} implies that 
\[
W^{\dir-}(\intc(s),s;y,t) = W^{\dir +}(\intc(s),s;y,t).
\]
Combined with the additivity of the Busemann functions and \eqref{eq69}, we see that for all $s \leq t$
\be \label{eq:Wxipmaeq}
    \begin{aligned}
        &\quad\; W^{\dir+}(\intc(s),s;y-a_0,t) \\&=W^{\dir+}(\intc(s),s;y,t) + W^{\dir+}(y,t;y-a_0,t)\\
        &=W^{\dir-}(\intc(s),s;y,t) + W^{\dir-}(y,t;y-a_0,t)\\
        &=W^{\dir-}(\intc(s),s;y-a_0,t).
    \end{aligned}
    \ee
Next,  Corollary \ref{cor:LRmost_pi}\ref{it:LR2} implies that $g^{\dir+,R}_{(y,t)}(s) \geq \intc(s)$ for all $s \leq t$. Additionally, by Proposition \ref{prop:DL_all_coal}, $g^{\dir+,R}_{(y-a_0,t)}$ and $g^{\dir+}_{(y,t)}$ coalesce. Hence, there exists $s_0 < t$ such that 
 \be \label{eq:coalyy-a0}
 g^{\dir+,R}_{(y-a_0,t)}(s_0) = g^{\dir+,R}_{(y,t)}(s_0) \ge \intc(s_0)
 \ee
 Since $\intc:\R \to \R$ and $g^{\dir+,R}_{(y-a_0,t)}:(-\infty,t] \to \R$ are continuous function with $\intc(t) = y > y-a_0 = g^{\dir+,R}_{(y-a_0,t)}(t)$, along with \eqref{eq:coalyy-a0}, the intermediate value theorem implies that  there exists $t_0 \in [s_0,t)$ such that 
   \be \label{eq:intc_eq_geod}
   \intc\left( t_0\right)=g^{\dir+,R}_{(y - a_0,t)}(t_0).
   \ee
   See Figure \ref{fig:conv-left-binf}. 
   By \eqref{geod_LR_eq_L} followed by \eqref{eq:Wxipmaeq}, we have
   \[
   \Ll\bigl(g^{\dir+}_{(y,t)}(t_0),t_0;y-a_0,t \bigr)=W^{\dir+}\bigl(g^{\dir+}_{(y,t)}(t_0),t_0;y-a_0,t \bigr)=W^{\dir-}\bigl(g^{\dir+}_{(y,t)}(t_0),t_0;y-a_0,t \bigr).
   \]
   By Lemma \ref{lem:Buse_eq},  $g_{(y-a_0,t)}^{\dir+,R}|_{[t_0,t]}$ is also the restriction of $\dir-$ geodesic. Since $g^{\dir-,R}_{(y-a_0,t)}$ is the rightmost $\dir -$ geodesic rooted at $(y-a_0,t)$, the ordering of geodesics (Proposition \ref{prop:g_basic_prop}\ref{itm:DL_mont_dir}), followed by \eqref{eq:intc_eq_geod} implies
   \be \label{eq:gdir-t0}
   g^{\dir-,R}_{(y-a_0,t)}(t_0) = g_{(y-a_0,t)}^{\dir+,R}(t_0) = \intc(t_0).
   \ee
However, this gives a contradiction. Indeed, Equation \eqref{eq:othpts} of Lemma \ref{lem:Splt}\ref{itm:G1} implies that
   \[
   g^{\dir-,R}_{(y-a_0,t)}(r)<\intc(r), \quad\forall r < t.
   \]
    The analogous statement for  $\Intc^{\dir,+}$ and $\Split_{t,\dir}^R$ follows by a symmetric proof.   
    \begin{figure}[t!]
        \includegraphics[width=10 cm]{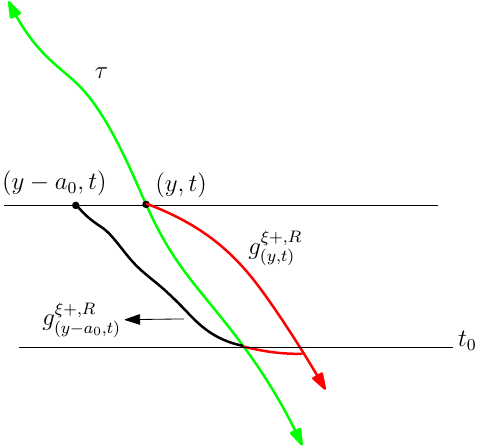}
        \caption{The figure shows that if a point $(y,t)$ lies on an interface $\intc \in \Intc^{\dir,-}$, then it must lie in the set $\Split_{t,\dir}^L$, This is because, since all $\dir+$ geodesics coalesce, $g^{\dir+,R}_{(y-a,t)}$ and $g_{(y,t)}^{\dir+,R}$ coalesce for all $a$. Thus for all $a, g^{\dir+,R}_{(y-a,t)}$ has to intersect $\intc$ at some time level below $t$. Now if $y \notin \Split^L_t$ then one can find $a_0$ and $t_0<t$ such that $g_{(y-a_0,t)}^{\dir+}|_{[t_0,t]}$ is also a restriction of $\dir-$ geodesic. But this is a contradiction to Equation \eqref{eq:othpts} of Lemma \ref{lem:Splt}\ref{itm:G1}.}
        \label{fig:conv-left-binf}
    \end{figure}
\end{proof}

    The following result shows that, for each $t \in \R$, the points on $\Split^L_{t,\dir}$ are in one-to-one correspondence with $\Intc^{\dir,-}$. \begin{corollary}\label{cor:disInt}
       On the event $\Omega_1$, the following holds  for every $\dir\in\DLBusedc$. \vspace{2pt}\\ 
       \begin{enumerate}[label=(\roman*), font=\normalfont]
           \item \label{it:SplitT_bij} For every $t\in\R$, there exist bijections $\sigma_t^-: \Split_{t,\dir}^L \to \Intc^{\dir,-}$ and $\sigma_t^+: \Split_{t,\dir}^R \to \Intc^{\dir,+}$  such that, for any $x\in \Split^L_{t,\dir}$,  $\sigma_{(x,t)}^- :=\sigma_t^-(x)$ is the unique element of $\Intc^{\dir,-}$ satisfying   $\sigma_{(x,t)}^-(t) = x$, and for $x \in \Split_{t,\dir}^R$, $\sigma_{(x,t)}^+ :=\sigma_t^+(x)$ is the unique element of $\Intc^{\dir,+}$ satisfying  $\sigma_{(x,t)}^+(t) = x$.  
           \item \label{it:Split_disj} For any $x<y$ in $\Split^L_{t,\dir}$ (resp $\Split_{t,\dir}^R$), we have
           \begin{equation}
               \sigma_{(x,t)}^-(r)<\sigma_{(y,t)}^-(r), \quad\text{\rm(resp. $\sigma_{(x,t)}^+(r)<\sigma_{(y,t)}^+(r)$)} \qquad \forall r\in\R.
           \end{equation}
           In other words, leftmost (resp. rightmost) bi-infinite interfaces associated with different points in $\Split^L_{t,\dir}$ (resp. $\Split^R_{t,\dir}$) never meet. In particular, the elements in $\Intc^{\dir,\pm}$ are disjoint.  
       \end{enumerate}
   \end{corollary}
   \begin{proof}
       We prove the bijection between $\Split_{t,\dir}^-$ and $\Intc^{\dir,-}$, with the other bijection following similarly. Fix $t\in\R$. We first construct the inverse map $\intc^{\dir,-} \to \Split_{t,\dir}^L$, which we call $\theta_t$. For $\intc \in \Intc^{\dir,-}$, define $\theta_t(\intc) = \intc(t)$. By Proposition \ref{Cond_bi-inf}, $\theta_t$ is a well-defined surjective map $\Intc^{\dir,-} \to \Split_{t,\dir}^L$.

       Now, assume, by way of contradiction, that $\theta_t$ is not a bijection. Then, $\theta_t(\intc^1) = \theta_t(\intc^2)$ for some distinct $\intc^1,\intc^2$ in $\Intc^{\dir,-}$, meaning that $\intc^1(t) = \intc^2(t)$, but $\intc^1(s) \neq \intc^2(s)$ for some $s \in \R$. By the consistency in Proposition \ref{p:rest}, $\intc^1(r) = \intc^2(r)$ for all $r \ge t$, so $\intc^1(s) \neq \intc^2(s)$ for some $s < t$.   
       
       Proposition \ref{Cond_bi-inf} implies that $\intc^1(s),\intc^2(s) \in \Split^L_{s,\dir}$, and since these are distinct points, they belong to different intervals in $\Intr^{s,\dir}$. 
        Lemma \ref{lem:coal_int}\ref{Itco2} implies that $
           \intc^{1}(r)<\intc^{2}(r)$ for all $r \ge s$, 
       in contradiction to the assumption that $\intc^{1}(t)=\intc^{2}(t)$.

We complete the construction of the bijection by defining $\sigma_{(x,t)}^- := \theta_t^{-1}$.

\textbf{Item \ref{it:SplitT_bij}} holds by construction. \textbf{Item \ref{it:Split_disj} } holds by Item \ref{it:SplitT_bij}, continuity of the interfaces and the intermediate value theorem. More specifically, if $\sigma_{(x,t)}^-(r)\ge \sigma_{(y,t)}^-(r)$ for some $x < y$ in $\Split^L_{t,\dir}$ and $r \in \R$, then since
\[
\sigma_{(x,t)}^-(t) = x < y = \sigma^-_{(y,t)}(t),
\]
we have that $\sigma_{(x,t)}^-(s) = \sigma_{(y,t)}^-(s)$  for some $s \in \R$, contradicting Item \ref{it:SplitT_bij}. 
       \end{proof}

We finish this section by answering a question left open in \cite{Busa-Sepp-Sore-22a}. 
   
    \begin{proposition}
    \label{theorem: local variation and splitting points}
    On the event $\Omega_1$, for every $\dir\in\DLBusedc$ and $t\in\R$,
    \[
    \Split_{t, \dir}=\Split^L_{t,\dir} \cup \Split^R_{t,\dir}.
    \]
        
    \end{proposition}
     
    \begin{proof}
     By way of contradiction, assume that  $x \in \Split_{t,\dir}$ but $x \notin \Split^L_{t,\dir} \cup \Split^R_{t,\dir}$. Since, $x \in \Split_{t,\dir}$ it must be the case that $g^{\dir+,R}_{(x,t)}$ and $g^{\dir-,L}_{(x,t)}$ are disjoint since these are the rightmost and leftmost $\xi$-directed geodesics by Propositions \ref{prop:Busani_N3G} and \ref{prop:g_basic_prop}\ref{itm:DL_mont_dir}. Now, as $x \notin \Split^L_{t,\dir} \cup \Split^R_{t,\dir}$, it must be the case that $g^{\dir-,L}_{(x,t)}$ and $g^{\dir+,L}_{(x,t)}$ stay together for some time and $g^{\dir-,R}_{(x,t)}$ and $g^{\dir+,R}_{(x,t)}$ stay together for some time near the initial point $(x,t)$. Precisely, there exist $\delta_1,\delta_2>0$ so that
     \begin{equation}
    \begin{aligned}
        & \label{eq46}g^{\dir-,L}_{(x,t)}(s) < g^{\dir+,R}_{(x,t)}(s) \qquad \text{ for all } s<t;\\
        & g^{\dir-,L}_{(x,t)}(s)=g^{\dir+,L}_{(x,t)}(s)\qquad  \text{  for all $s \in [t-\delta_1,t]$, and }\qquad g^{\dir-,L}_{(x,t)}(s)<g^{\dir+,L}_{(x,t)} \text{ for all } s<t-\delta_1. \\
        & g^{\dir-,R}_{(x,t)}(s)=g^{\dir+,R}_{(x,t)}(s) \qquad\text{ for all $s \in [t-\delta_2,t]$, and }\qquad g^{\dir-,R}_{(x,t)}(s)<g^{\dir+,R}_{(x,t)} \text{ for all } s<t-\delta_2 .
    \end{aligned}
    \end{equation}
    Thus, for $s_1 = (t - \delta_1) \vee (t - \delta_2)$,
    \be \label{eq:gineq1}
    g_{(x,t)}^{\dir +,L}(s_1) = g_{(x,t)}^{\dir -,L}(s_1) < g^{\dir +,R}_{(x,t)}(s_1) = g_{(x,t)}^{\dir -,R}(s_1)
    \ee
    Now, $g^{\dir-,L}_{(x,t)}$ and $g^{\dir-,R}_{(x,t)}$ eventually coalesce and $g^{\dir+,L}_{(x,t)}$ and $g^{\dir+,R}_{(x,t)}$ eventually coalesce by Proposition \ref{prop:DL_all_coal}. Hence, we may choose $s_2$ sufficiently less than $t$ so that
    \be \label{eq:gineq2}
    g_{(x,t)}^{\dir -,R}(s_2) = g_{(x,t)}^{\dir -,L}(s_2) < g^{\dir+,R}_{(x,t)}(s_2) = g_{(x,t)}^{\dir +,L}(s_2).
    \ee
    Comparing \eqref{eq:gineq1} and \eqref{eq:gineq2}, by continuity of geodesics and the intermediate value theorem, there exists $s' < (t - \delta_1) \vee (t - \delta_2)$ such that
    \be \label{eq:ypt}
    y:= g^{\dir+,L}_{(x,t)}(s')=g^{\dir-,R}_{(x,t)}(s').
    \ee
   (see Figure \ref{fig:local variation and splitting points_1}).
    \begin{figure}[t!]
        \centering
        \includegraphics[width=7 cm]{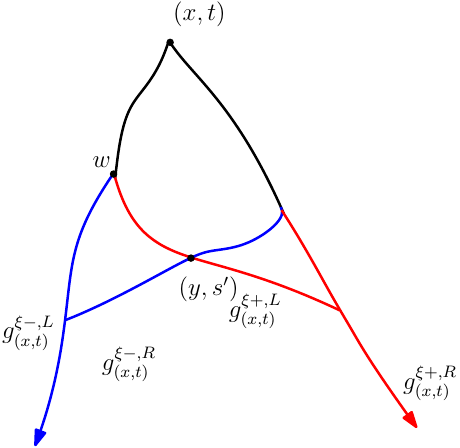}
        \caption{We prove that the situation in the figure can not arise. We argue that, if it does then the restriction of $g^{\dir+,L}_{(x,t)}$ on $[s',t]$  is actually a restriction of an $\dir-$ geodesic. Hence, following $g^{\dir-,R}_{(x,t)}$ after $g^{\dir+,L}_{(x,t)}|_{[s',t]}$ we will get a $\dir-$ geodesic. But this implies that there will be an interior geodesic bubble, which is not possible by Lemma \ref{lem:no_bubbles}.}
        \label{fig:local variation and splitting points_1}
    \end{figure}
    As $s' < (t - \delta_1) \vee (t - \delta_2)$, we assume, without loss of generality, that $t-\delta_1 \ge t - \delta_2$ so that $s' < t - \delta_1$.
    By \eqref{eq:ypt} and the equality of Busemann functions to $\Ll$ along semi-infinite geodesics (Lemma \ref{lem:Buse_eq}),
    \be \label{Leq1}
    \Ll(y,s';x,t)=W^{\dir+}(y,s';x,t)=W^{\dir-}(y,s';x,t).
    \ee
    Now, define $w :=g^{\dir+,L}_{(x,t)}(t-\delta_1)=g^{\dir-,L}_{(x,t)}(t-\delta_1)$. Then, again using Lemma \ref{lem:Buse_eq},
    \be \label{Leq2}
    \Ll(w,t-\delta_1;x,t)=W^{\dir+}(w,t-\delta_1;x,t)=W^{\dir-}(w,t-\delta_1;x,t).
    \ee
    Subtracting the second equality of \eqref{Leq2} from the second equality of \eqref{Leq2} and using additivity of the Busemann functions, we get
    \[
    W^{\dir+}(y,s';w,t-\delta_1)=W^{\dir-}(y,s';w,t-\delta_1).
    \]
    But also as $g^{\dir+,L}_{(x,t)}(t - \delta_1) = w$ and $g^{\dir +,L}_{(x,t)}(s') = y$, we see that (again by Lemma \ref{lem:Buse_eq}),
    \[
    \Ll(y,s';w,t-\delta_1)=W^{\dir+}(y,s';w,t-\delta_1)=W^{\dir-}(y,s';w,t-\delta_1).
    \]
    Then by Lemma \ref{lem:Buse_eq},  we know that $y$ is a maximizer of $z \mapsto W^{\dir-}(0,s;z,s)+\Ll(z,s;w,t-\delta_1)$ over $z \in \R$. Now by Proposition \ref{prop:DL_SIG_cons_intro}\ref{itm:arb_geod_cons}, this implies that if we consider the geodesic $\wt g$ defined as 
    \begin{align*}
    &\wt g(r)=g^{\dir+,L}_{(x,t)}(r) \qquad \text{ for } s'\leq r \leq t.\\
    &\wt g(r)=g^{\dir-,R}_{(x,t)}(r) \qquad \text{ for } r \leq s',
    \end{align*}
    then $g$ is a $\dir-$ geodesic.

    Consider the geodesics $\wt g$ and $g^{\dir-,L}_{(x,t)}$. We know they must coalesce because they are both $\dir -$ geodesics (Proposition \ref{prop:DL_all_coal}). Along with \eqref{eq46}, this implies that there exists $s_0 < s'< t - \delta_1$ so that 
    \begin{align*}
    g^{\dir-,L}_{(x,t)}(r)  &= \wt g(r) \qquad \text{for }t - \delta_1 \le r \le t \\
    g^{\dir -,L}_{(x,t)}(s')  &< \wt g(s') = g_{(x,t)}^{\dir -,R}(s') \\
    g^{\dir-,L}_{(x,t)}(r) &= \wt g(r) \qquad \text{for } r \le s_0
    \end{align*}
    Then, there are two distinct geodesics between the points $(g'(s_0),s_0)$ and $(x,t)$ that agree in neighborhoods of both endpoints.  This is a contradiction to  the fact that geodesics do not form interior bubbles (recorded here as Lemma \ref{lem:no_bubbles}). 
        \end{proof}

   \section{General \texorpdfstring{$\dir$}{}-eternal solutions} \label{sec:gen_soln}
In this section, we prove statements about general $\dir$-eternal solutions to the KPZ fixed point: a precursor to the bijection in Theorem \ref{thm:mr}. Let $\dir \in \R$. We recall that $\dir$-eternal solutions are continuous functions $b:\R^2 \to \R$ which satisfy the following properties. 

\paragraph{\textbf{KPZ Fixed Point Evolution:}}For each $x,s,t \in \R$ with $s < t$, 
\be \label{eq:var}
b(x,t) = \sup_{z \in \R}\{b (z,s)+\Ll(z,s;x,t)\},
\ee 
\paragraph{\textbf{Asymptotic Slope Condition:}}
The function $x \mapsto b(x,0)$ satisfies the asymptotic
\be \label{eq:b_slope}
\lim_{|x| \to \infty} \f{b(x,0)}{x} = 2\dir.
\ee
 We define the following subset of $\dir$-eternal solutions to the KPZ fixed point.
    \[
    \fxi:=\{b: \R^2 \rightarrow \R: b \text { is continuous and satisfies } \eqref{eq:var} \text{ and } \eqref{eq:b_slope} \}.
    \]
Note that the functions $s \mapsto W^{\dir -}(0,0;x,s)$ and $s \mapsto W^{\dir +}(0,0;x,s)$ both belong to $\fxi$ by Proposition \ref{prop:Buse_basic_properties}. By Lemma \ref{lem:KPZ_preserve_lim}, 
 \[
 \text{If }\lim_{|x| \to \infty }\f{b(x,s)}{x} = 2\dir,\quad\text{then}\quad \lim_{|x| \to \infty }\f{b(x,t)}{x} = 2\dir \quad \text{for all }t > s.
 \]
 While the assumption \eqref{eq:b_slope} doesn't immediately imply the slope condition for negative $t$, it is sufficient for the following proofs. A major key in this section is Lemma \ref{lem:dir_from_global} in the Appendix, which shows that, for any $b \in \fxi$, the associated $b$-geodesics are $\dir$-directed.

 For $b \in \fxi$, $(x,t) \in \R^2$, and $s \in \R$, let $g_{(x,t)}^{b,L/R}(s)$ be the leftmost/rightmost maximizers of the function
\[
z \mapsto b(z,s)+\Ll(z,s;x,t)
\]
over $z \in \R$. Also define $g_{(x,t)}^{b,L/R}(t) = x$. 
By Lemma \ref{lem:geodesics_from_b}\ref{itm:geod}, on the event $\Omega_1$, each of the functions $s \mapsto g_{(x,t)}^{b,L}(s)$ and $s \mapsto g_{(x,t)}^{b,R}(s)$ is a semi-infinite geodesic and hence a  continuous function $(-\infty,t] \to \R$.

\begin{lemma}
\label{lem:geod_IC}
The following holds on the event $\Omega_1$. Let $\dir \in \R$, and $b\in \fxi$. Then for each $(x,t) \in \R^2$, either
\[
g_{(x,t)}^{b,L}(t) = g_{(x,t)}^{\dir -,L}(s) \quad\, \forall s \le t,\qquad\text{or}\qquad g_{(x,t)}^{b,L}(s) = g_{(x,t)}^{\dir +,L}(s)\quad\, \forall s\le t.
\]
The same holds for $R$ in place of $L$. 
\end{lemma}
\begin{remark} \label{rmk:eqgoeodesics}
Note that Lemma \ref{lem:geod_IC} holds for all $\dir \in \R$, not just $\dir \in \DLBusedc$. In the case $\dir \notin \DLBusedc$, the conclusion simply states that 
\[
g_{(x,t)}^{b,L}(s) = g_{(x,t)}^{\dir-,L}(s) = g_{(x,t)}^{\dir +,L}(s) \quad \forall s \le t. 
\]
\end{remark}
\begin{proof}[Proof of Lemma \ref{lem:geod_IC}] Note that $g^{b,L/R}_{(x,t)}$ are each infinite geodesics that are the leftmost and rightmost geodesics, respectively, between any of their points. By the asymptotic slope condition on $b$ \eqref{eq:b_slope} and Lemma \ref{lem:dir_from_global}, these geodesics are $\dir$-directed. Hence, each must coalesce with either $g_{(x,t)}^{\dir -,L/R}$ or $g_{(x,t)}^{\dir +,L/R}$ by Propositions \ref{prop:Busani_N3G} and \ref{prop:DL_all_coal}. But because of the property that the $L$ geodesics are the leftmost geodesics between their points and the $R$ geodesics are the rightmost geodesics between their points (Proposition \ref{prop:DL_SIG_cons_intro}\ref{itm:DL_LRmost_geod} and Lemma \ref{lem:geodesics_from_b}\ref{itm:geod}), we have the following: For each $(x,t) \in \R^2$, either
\[
g_{(x,t)}^{b,L}(t) = g_{(x,t)}^{\dir -,L}(s) \quad\, \forall s \le t\qquad\text{or}\qquad g_{(x,t)}^{b,L}(s) = g_{(x,t)}^{\dir +,L}(s)\quad\, \forall s\le t.
\]
The same is also true if we replace $L$ with $R$.
\end{proof}

The next lemma gives a monotonicity result that defines an interface between the $\dir -$ and $\dir +$ geodesics. 
\begin{lemma}
\label{lem:+-ord} On the event $\Omega_1$, for all $\dir \in \DLBusedc$ and $b \in \fxi$, 
 \begin{enumerate}[label=(\roman*), font=\normalfont]
   \item If $x_1 < x_2$ and $g_{(x_1,t)}^{b,L}(s) = g_{(x_1,t)}^{\dir +,L}(s)$ for all $s \le t$, then $g_{(x_2,t)}^{b,L}(s) = g_{(x_2,t)}^{\dir +,L}(s)$ for all $s \le t$.
   \item  if $x_1 < x_2$ and $g_{(x_2,t)}^{b,L}(t) = g_{(x_2,s)}^{\dir -,L}(t)$ for all $s \le t$, then also $g_{(x_1,t)}^{b,L}(t) = g_{(x_1,t)}^{\dir -,L}(s)$ for all $s \le t$. 
    \end{enumerate}
    The same holds if we replace $L$ with $R$.
\end{lemma}
\begin{proof}We only prove (i). The proof of (ii) follows from a symmetric argument. Assume that $g_{(x_1,t)}^{b,L}(s) = g_{(x_1,t)}^{\dir +,L}(s)$ for all $s \le t$, and assume, by way of contradiction, that the conclusion  of (i) fails. Then, by Lemma \ref{lem:geod_IC}, $g_{(x_2,t)}^{b,L}(s) = g_{(x_2,t)}^{\dir -,L}(s)$ for all $s \le t$. Furthermore, by Proposition \ref{prop:g_basic_prop}\ref{itm:eventually_less}, the $\dir -$ geodesics are eventually to the left of the $\dir +$ geodesics. That is, for all $s$ sufficiently less than $t$,
\be \label{gb1}
g_{(x_2,t)}^{b,L}(s) = g_{(x_2,t)}^{\dir -,L}(s) < g_{(x_2,t)}^{\dir +,L}(s).
\ee
Further, by Proposition \ref{prop:DL_all_coal}, $g_{(x_1,t)}^{\dir +,L}$ coalesces with $g_{(x_2,t)}^{\dir +,L}$, so for all $s$ sufficiently less than $t$,
\be \label{gb2}
g_{(x_1,t)}^{b,L}(s) = g_{(x_1,t)}^{\dir +,L}(s) = g_{(x_2,t)}^{\dir +,L}(s).
\ee
Comparing \eqref{gb1} and \eqref{gb2}, we have $g_{(x_2,t)}^{b,L}(s) < g_{(x_1,t)}^{b,L}(s)$ for $s$ sufficiently less than $t$. Since $x_1 < x_2$,  Corollary \ref{cor:b_geod_mont} gives a contradiction.
\end{proof}
In light of Lemma \ref{lem:+-ord}, when $\dir \in \DLBusedc$ and $b \in \fxi$, we define the functions $\intc_t^{b,L/R}:\R\rightarrow\R$ by
\begin{equation} \label{tauLR_Def}
\begin{aligned}
    & \intc_t^{b,L}:=\sup\{x \in \R: g^{b,L}_{(x,t)}=g^{\dir-,L}_{(x,t)} \} = \inf\{x \in \R:g^{b,L}_{(x,t)}=g^{\dir+,L}_{(x,t)} \},\qquad \text{and}\\
    &\intc_t^{b,R}:= \sup\{x \in \R: g^{b,R}_{(x,t)}=g^{\dir-,R}_{(x,t)} \} = \inf\{x \in \R:g^{b,R}_{(x,t)}=g^{\dir+,R}_{(x,t)} \},
\end{aligned}
\end{equation}
where the equality of the $\sup$ and $\inf$ is a consequence of Lemma \ref{lem:+-ord}. Note that $\intc_t^{b,R}$ and $\intc_t^{b,L}$ can be $+\infty$ or $-\infty$. We also emphasize here that the $\intc_t^{b,L/R}$ notation is built from a general eternal solution $b$ and is not the same as the $\intc_{(x_0,s)}^{\dir,\pm}$ notation in \eqref{eq:Buse_interface}. However, we will see in Proposition \ref{p:correspondence}\ref{it:Requiveq} that $\intc_t^{b,R}$ is a leftmost bi-infinite interface, and $\intc_t^{b,L}$ is a rightmost interface.

We have the following result which is a corollary of Lemma \ref{lem:+-ord}.
\begin{corollary}
\label{c:Im_cutoff}
    On the event $\Omega_1$, for all $\dir \in \Xi$, $b \in \fxi$, and $t \in \R$, we have $
    \intc_t^{b,R} \leq \intc_t^{b,L}$. 
\end{corollary}
\begin{remark}
It may seem strange that $\intc_t^{b,R}$ is to the left of $\intc_t^{b,L}$, but the reason is that $\intc_t^{b,R}$ separates rightmost geodesics, while $\intc_t^{b,L}$ separates leftmost geodesics.
\end{remark}
\begin{proof}
    Assume, by way of contradiction, that $-\infty \le \intc_t^{b,L} < \intc_t^{b,R} \le \infty$, and choose $x \in (\intc_t^{b,L}, \intc_t^{b,R})$.  Then by Lemma \ref{lem:+-ord}, 
    \begin{equation} \label{eq:gLRxi-+}
g^{b,L}_{(x,t)}=g_{(x,t)}^{\dir+,L} \qquad \text{and} \qquad g^{b,R}_{(x,t)}=g_{(x,t)}^{\dir-,R}.
    \end{equation}
    By definition of $g^{b,L/R}_{(x,t)}$ as leftmost and rightmost maximizers, $
     g^{b,L}_{(x,t)}(s) \leq g^{b,R}_{(x,t)}(s)$ for all $s \leq t$, so this implies 
    \[
    g_{(x,t)}^{\dir+,L}(s) \leq g_{(x,t)}^{\dir-,R}(s) \qquad \text{for all } s \leq t.
    \]
    But  by Proposition \ref{prop:g_basic_prop}\ref{itm:eventually_less}, this inequality fails for $s$ sufficiently less than $t$, giving a contradiction. 
\end{proof}
We state another corollary.
\begin{corollary} 
    \label{c:cutoff_geod}
    The following holds on the event $\Omega_1$. Let $\dir \in \Xi$, and $b \in \fxi$. If for some $t \in \R, \intc_t^{b,L} \in \R$, then
    \be \label{gtaueq1}
    g_{(\intc_t^{b,L},t)}^{b,L}(s) =g_{(\intc_t^{b,L},t)}^{\dir -,L}(s) \qquad\text{for all }s \le t, 
    \ee
    Additionally, if for some $t \in \R, \intc_t^{b,R} \in \R$, then
    \be \label{gtaueq2}
    g_{(\intc_t^{b,R},t)}^{b,R}(s) =g_{(\intc_t^{b,R},t)}^{\dir +,R}(s) \qquad\text{for all }s \le t. 
    \ee 
\end{corollary}
\begin{remark}
    Since $\tau_t^{b,R} \le \tau_t^{b,L}$ (Corollary \ref{c:Im_cutoff}), the monotonicity in Lemma \ref{lem:+-ord} also implies that 
    \begin{align*}
    g_{(\intc_t^{b,R},t)}^{b,L}(s) &=g_{(\intc_t^{b,R
    },t)}^{\dir -,L}(s) \qquad\text{for all }s \le t,\quad\text{and} \\
    g_{(\intc_t^{b,L},t)}^{b,R}(s) &=g_{(\intc_t^{b,L},t)}^{\dir +,R}(s)\qquad\text{for all }s \le t.
    \end{align*}
\end{remark}
\begin{proof}
    We prove \eqref{gtaueq1}. By Lemma \ref{lem:+-ord} and definition of $\intc_t^{b,L}$, for all $x < \intc_t^{b,L}$,
    \be \label{eq:gbxieq}
    g_{(x,t)}^{b,L}(s) = g^{\dir -,L}_{(x,t)}(s) \quad\text{ for all }s \le t.
    \ee
    Then, by Corollary \ref{cor:b_geod_lim} and  Proposition \ref{prop:g_basic_prop}\ref{itm:DL_SIG_conv_x}, for each $s \le t$,
    \be \label{eq:xi-lim}
     g_{(\intc_t^{b,L},t)}^{b,L}(s) = \lim_{x \nearrow \intc_t^{b,L}}g^{b,L}_{(x,t)}(s) =\lim_{x \nearrow \intc_t^{b,L}}g^{\dir -,L}_{(x,t)}(s) = g_{(\intc_t^{b,L},t)}^{\dir -,L}(s).
    \ee 
    The proof of\eqref{gtaueq2} is symmetric, this time taking limits from the right. 
\end{proof}
As mentioned earlier potentially, $\intc_t^{b,L/R}$ could be $\pm \infty$. When $b$ is not equal to $W^{\dir -}$ or $W^{\dir +}$, we prove that this cannot be the case in a series of lemmas. The first one shows that the increments of $b$ to the left and right of $\intc_t^{b,L/R}$  are equal to the $\pm$ Busemann functions. 
\begin{lemma}
\label{lem:eq_split}
The following holds on the event $\Omega_1$, for all $\dir \in \DLBusedc$ and all $b \in \fxi$.
\begin{enumerate}[label=(\roman*), font=\normalfont]
\item \label{it:bW1} For $t \in \R$ and all $x,y \ge \intc_t^{b,R}$, $b(y,t) - b(x,t) = W^{\dir +}(x,t;y,t)$. For all $x,y \le \intc_t^{b,L}$, $b(y,t) - b(x,t) = W^{\dir -}(x,t;y,t)$.
\item \label{it:bW2} If for all $t \in \R, \tau_t^{b,R}$ are finite, then for all $t,s \in \R$, 
\[
b(\tau_t^{b,R},t) - b(\tau_s^{b,R},s) =W^{\dir+}(\tau_s^{b,R},s;\tau_t^{b,R},t).
\]
\item \label{it:bW3}  If for all $t \in \R, \tau_t^{b,L}$ are finite, then for all $t,s \in \R$, 
\[
b(\tau_t^{b,L},t) - b(\tau_s^{b,L},s) =W^{\dir-}(\tau_s^{b,L},s;\tau_t^{b,L},t).
\]
\end{enumerate}
\end{lemma}
\begin{remark} 
Recall from Corollary \ref{c:Im_cutoff} that $\intc_t^{b,R} \le \intc_t^{b,L}$, so if $x,y \in [\intc_t^{b,R},\intc_t^{b,L}]$, Lemma \ref{lem:eq_split}\ref{it:bW1} implies that 
\[
b(y,t) - b(x,t) = W^{\dir +}(x,t;y,t) = W^{\dir -}(x,t;y,t). 
\]
\end{remark}
\begin{proof}
\textbf{Item \ref{it:bW1}:} We prove the first statement, and the proof of the second is analogous.  Let $x,y \ge \intc_t^{b,R}$. By Corollary \ref{c:cutoff_geod}, $g_{(x,t)}^{b,R}(s) = g_{(x,t)}^{\dir +,R}(s)$  for all $s \le t$, and the same holds for $y$ in place of $x$. By coalescence of the $\dir +$ geodesics (Proposition \ref{prop:DL_all_coal}), we may let $s$ be sufficiently large and negative so that $g_{(y,t)}^{\dir +,R}(s) = g_{(x,t)}^{\dir +,R}(s)$. For such $s$, use $g(s)$ to denote this point for shorthand. Then, 
\begin{align*}
 b(x,t) - W^{\dir +}(0,0;x,t) &= \sup_{z \in \R}\{b(z,s)+\Ll(z,s;x,t)\} - \sup_{z \in \R}\{W^{\dir+}(0,0;z,s)+\Ll(z,s;x,t)\} \\
&=b(g(s),s)+\Ll(g(s),s;x,t) - W^{\dir +}(0,0;g(s),s)-\Ll(g(s),s;x,t) \\
&= b(g(s),s) - W^{\dir +}(0,0;g(s),s).
\end{align*}
The same holds for $y$ in place of $x$. Therefore, 
\[
b(x,t) - W^{\dir +}(0,0;x,t) = b(y,t) - W^{\dir +}(0,0;y,t),
\]
and the result  follows by rearranging.

\medskip \noindent \textbf{Item \ref{it:bW2}:} 
By Corollary \ref{c:cutoff_geod}, $g_{(\tau_t^{b,R},t)}^{b,R}(r) = g_{(\tau_t^{b,R},t)}^{\dir +,R}(r)$ for all $r \le t$, and the same holds for $s$ in place of $t$. By coalescence of the $\dir +$ geodesics, we may let $r$ be sufficiently large and negative so that $g_{(\tau_t^{b,R},t)}^{\dir +,R}(r) = g_{(\tau_s^{b,R},s)}^{\dir +,R}(r)$. For such $r$, use $g(r)$ to denote this point for shorthand. Then,
\begin{align*}
 &\quad \, b(\tau_t^{b,R},t) - W^{\dir +}(0,0;\tau_t^{b,R},t) \\
 &= \sup_{z \in \R}\{b(z,r)+\Ll(z,r;\tau_t^{b,R},t)\} - \sup_{z \in \R}\{W^{\dir+}(0,0;z,r)+\Ll(z,r;\tau_t^{b,R},t)\} \\
&=b(g(r),r)+\Ll(g(r),r;\tau_t^{b,R},t) - W^{\dir +}(0,0;g(r),r)-\Ll(g(r),r;\tau_t^{b,R},t) \\
&= b(g(r),r) - W^{\dir +}(0,0;g(r),r).
\end{align*}
The same holds when $s$ is replaced by $t$.
\[
b(\tau_t^{b,R},t) - W^{\dir +}(0,0;\tau_t^{b,R},t) = b(g(r),r) - W^{\dir +}(0,0;g(r),r) = b(\tau_s^{b,R},s) - W^{\dir +}(0,0;\tau_s^{b,R},s),
\]
and the result follows by rearranging and the additivity of $W^{\dir +}$ and the increments of $b$.

\medskip \noindent \textbf{Item \ref{it:bW3}:} This has a symmetric proof to Item \ref{it:bW2}. 
\end{proof}

\begin{lemma} \label{lem:mont} 
The following holds on the event $\Omega_1$, for all $\dir \in \R$. Whenever $\dir \notin \DLBusedc$ and $b \in \fxi$, we have that 
\be \label{eq:bW}
b(y,t) - b(x,s) = W^{\dir}(x,s;y,t),  \quad \forall (x,s;y,t) \in \R^4.
\ee
In general, for $\dir \in \R$ and $b \in \fxi$, $t \in \R$, and $x < y$, 
\be \label{b_between-+}
W^{\dir -}(x,t;y,t) \le b(y,t) - b(x,t) \le W^{\dir +}(x,t;y,t).
\ee
\end{lemma}
\begin{proof}
We start by proving \eqref{eq:bW}. When $\dir \notin \DLBusedc$, Lemma \ref{lem:geod_IC} and Remark \ref{rmk:eqgoeodesics} states that $g_{p}^{b,R/L} = g_{p}^{\dir,R/L}$ for all $p \in \R^2$, where we have removed the $\pm$ distinction in the superscript. By Proposition \ref{prop:DL_all_coal}, the $\dir$-directed geodesics from all points $p$ coalesce. Let $r$ be sufficiently less than $s \wedge t$ so that 
\[
g(r) := g_{(x,s)}^{\dir,L}(r) = g_{(y,t)}^{\dir,L}(r).
\]
Then,
\begin{align*}
b(y,t) - b(x,s) &= \sup_{z \in \R} \{b(z,r) + \Ll(z,r;y,t)\}- \sup_{z \in \R} \{b(z,r) + \Ll(z,r;x,s)\} \\
&= b\bigl(g(r),r)\bigr) + \Ll\bigl(g(r),r;y,t\bigr) - \Bigl[b\bigl(g(r),r\bigr) + \Ll\bigl(g(r),r;x,s\bigr)\Bigr] \\
&\overset{\eqref{geod_LR_eq_L}}{=} W^{\dir}\bigl(g(r),r;y,t\bigr) - W^{\dir}\bigl(g(r),r;x,s\bigr) = W^\dir(x,s;y,t). 
\end{align*}
We now turn to \eqref{b_between-+}. Since \eqref{eq:bW} holds in the case $\dir \notin \DLBusedc$, we may assume $\dir \in \DLBusedc$. We will show the second inequality, as the first follows similarly. We consider 3 cases.\\
\paragraph{\textbf{Case 1:}} $\intc_t^{b,R} \leq x<y.$ By Lemma \ref{lem:eq_split}\ref{it:bW1}, $b(y,t) - b(x,t)=W^{\dir+}(x,t;y,t)$, so the inequality is an equality in this case.

\paragraph{\textbf{Case 2:}} $\intc_t^{b,R} \in (-\infty, \infty)$ and $x \leq \intc_t^{b,R} \le y$. for shorthand, let $\intc_t = \intc_t^{b,R}$ and note that we also have 
 $x \le \intc_t \le \intc_t^{b,L}$ by Corollary \ref{c:Im_cutoff}. Then,
\begin{align*}
    &\quad \; b(y,t) - b(x,t) \\
    &=b(\intc_t,t) - b(x,t)+b(y,t) - b(\intc_t,t) \\
    &= W^{\dir-}(x,t;\intc_t,t)+W^{\dir+}(\intc_t,t;y,t) \qquad \text{ by Lemma \ref{lem:eq_split}\ref{it:bW1}}\\
    &\leq W^{\dir+}(x,t;\intc_t,t)+W^{\dir+}(\intc_t,t;y,t) \qquad \text{ by monotonicity of the Busemann functions}.\\
    &=W^{\dir+}(x,t;y,t) \qquad\qquad \qquad\qquad \quad \;\; \, \text{ by additivity of } W^{\dir +}.
\end{align*}
\paragraph{\textbf{Case 3:}} $x<y \leq \intc_t^{b,R}$. In this case, by monotonicity of the Busemann functions, 
\[
b(y,t) - b(x,t)=W^{\dir-}(x,t;y,t) \leq W^{\dir+}(x,t;y,t). \qedhere
\] 
\end{proof}

\begin{lemma}\label{lem:neq} 
The following holds on the event $\Omega_1$.
For all $\dir \in \Xi, \sig \in \{+,- \}$, assume that $b \in \fxi$ is such that
\[
b(x_0,t_0) - b(0,0) \neq W^{\dir \sig}(0,0;x_0,t_0),\quad\text{for some}\quad (x_0,t_0) \in \R^2.
\]
Then, for every $t \in \R, \sig \in \{+,-\}$, there exist $x_1 \in \R$ such that
\[
W^{\dir \sig}(0,t;x_1,t) \neq b(x_1,t) - b(0,t).
\]
\end{lemma}
\begin{proof}
We prove the statement when $\sig=+.$ The $\sig=-$ case follows similarly. In this proof, we use the shorthand notation 
\[
b(x,s;y,t) := b(y,t) - b(x,s).
\]
By additivity and \eqref{eq:var}, we have that, for $s < (t_0 \wedge 0)$,
\begin{align*}
b(0,0;x_0,t_0) &= \sup_{z \in \R}\{b(z,s)+\Ll(z,s;x_0,t_0)\} - \sup_{z \in \R}\{b(z,s)+\Ll(z,s;0,0)\} \\
&=  \sup_{z \in \R}\{b(0,s;z,s)+\Ll(z,s;x_0,t_0)\} - \sup_{z \in \R}\{b(0,s;z,s)+\Ll(z,s;0,0)\}
\end{align*}
and the same holds for $W^{\dir +}$ in place of $b$. Hence, if $b(0,0;x_0,t_0) \neq W^{\dir +}(0,0;x_0,t_0)$,  then for all $s < (t_0 \wedge 0)$, $b^\dir(0,s;z,s) \neq W^{\dir +}(0,s;z,s)$ for some $z \in \R$. Now, by way of contradiction, assume that there exists $s_1 \geq (t_0 \wedge 0)$ such that, for all $x \in \R$,
\be \label{b=W}
b(0,s_1;x,s_1)=W^{\dir+}(0,s_1;x,s_1).
\ee
We fix $s_0 \leq (t_0 \wedge 0)$ and choose $w \in \R$ so that
\begin{equation}\label{eq78}
    b(0,s_0;w,s_0) \neq W^{\dir+}(0,s_0;w,s_0).
\end{equation}
We assume, without loss of generality, that $w >0$, and consider the function 
\begin{align} \label{f_Def}
&x \mapsto f(x):=b(0,s_0;x,s_1)-W^{\dir+}(0,s_0;x,s_1).
\end{align}
We may write \eqref{b=W} as 
\[
b(0,s_0;x,s_1) - b(0,s_0;0,s_1)  = W^{\dir +}(0,s_0,x,s_1) - W^{\dir+}(0,s_0;0,s_1),
\]
and since we assume this holds for all $x$, rearranging gives
\be \label{fdef}
f(x) = b(0,s_0;0,s_1) - W^{\dir +}(0,s_0;0,s_1).
\ee
Hence, $f$ is constant in $x$. Now, by Lemma \ref{lem:mont} and additivity we have for all $0<z<Z \in \R$,
\[
0 \leq W^{\dir+}(0,s_0;z,s_0)-b(0,s_0;z,s_0) \leq W^{\dir+}(0,s_0;Z,s_0)-b(0,s_0;Z,s_0).
\]
Similarly, when $Z<z<0$, we have 
\begin{equation}
\label{eq:neg}
0 \geq W^{\dir+}(0,s_0;z,s_0)-b(0,s_0;z,s_0) \geq W^{\dir+}(0,s_0;Z,s_0)-b(0,s_0;Z,s_0).
\end{equation} 
Then, by \eqref{eq78}, for all $Z \geq w$, we have 
\begin{equation}
\label{eq:pos}
0< W^{\dir+}(0,s_0;w,s_0)-b(0,s_0;w,s_0) \leq W^{\dir+}(0,s_0;Z,s_0)-b(0,s_0;Z,s_0).
\end{equation}
Since $b$ and $W^{\dir +}$ are both eternal solutions to the KPZ fixed point, we get
\be \label{bW_global}
\begin{aligned}
b(0,s_0;x,s_1) &= \sup_{z \in \R}\{b(0,s_0;z,s_0)+\Ll(z,s_0;x,s_1)\},\quad\text{and} \\
W^{\dir +}(0,s_0;x,s_1) &= \sup_{z \in \R}\{W^{\dir +}(0,s_0;z,s_0)+\Ll(z,s_0;x,s_1)\},
\end{aligned}
\ee
 by definition of the function $f$ \eqref{f_Def}, we get
\[
f(x)=\sup_{z \in \R}\{b(0,s_0;z,s_0)+\Ll(z,s_0;x,s_1)\}-\sup_{z \in \R}\{W^{\dir+}(0,s_0;z,s_0)+\Ll(z,s_0;x,s_1)\}.
\]
By the asymptotic slope condition on $b$, we know that $|b(0,s_0;z,s_0)| \le A + B |z|$ for some constants $A,B$. Combined with the modulus of continuity bounds on the directed landscape (Lemma \ref{lem:Landscape_global_bound}), wee see that for sufficiently large positive $x > w$,
\[
\sup_{z \le w}\{b(0,s_0;z,s_0)+\Ll(z,s_0;x,s_1)\} \le b(0,x_0,x,s_0) +\Ll(x,s_0,x, s_1).
\]
Hence,  for sufficiently large positive $x$,
\begin{align*}
\sup_{z \in \R}\{b(0,s_0;z,s_0)+\Ll(z,s_0;x,s_1)\}=\sup_{z \geq w}\{b(0,s_0;z,s_0)+\Ll(z,s_0;x,s_1)\},
\end{align*}
and similarly, for large enough $x$, 
\begin{align*}\sup_{z \in \R}\{W^{\dir+}(0,s_0;z,s_0)+\Ll(z,s_0;x,s_1)\}=\sup_{z \geq w}\{W^{\dir+}(0,s_0;z,s_0)+\Ll(z,s_0;x,s_1)\}.
\end{align*}
By \eqref{eq:pos} we get $f(x)<0$ for sufficiently large $x.$ By a similar argument, for sufficiently large negative $x<0$,
\begin{align*}
 &   \sup_{z \in \R}\{b(0,s_0;z,s_0)+\Ll(z,s_0;x,s_1)\}=\sup_{z \leq 0}\{b(0,s_0;z,s_0)+\Ll(z,s_0;x,s_1)\}.\\
 & \sup_{z \in \R}\{W^{\dir+}(0,s_0;z,s_0)+\Ll(z,s_0;x,s_1)\}=\sup_{z \leq 0}\{W^{\dir+}(0,s_0;z,s_0)+\Ll(z,s_0;x,s_1)\}.
\end{align*}
Then, for such $x$, by definition of $f$ \eqref{fdef} and \eqref{bW_global}, followed by \eqref{eq:neg},
\[
f(x) = \sup_{z \leq 0}\{b(0,s_0;z,s_0)+\Ll(z,s_0;x,s_1)\} - \sup_{z \leq 0}\{W^{\dir+}(0,s_0;z,s_0)+\Ll(z,s_0;x,s_1)\} \ge 0,
\]
 But since $f(x) < 0$ for large $x$, the function $f$ cannot be constant, giving a contradiction.
\end{proof}

\begin{lemma}\label{lem:bfin}
The following holds on the event $\Omega_1$: for all $\dir \in \DLBusedc$, whenever $ b \in \fxi$ is such that $b(x_0,t_0) - b(0,0) \neq W^{\dir -}(0,0;x_0,t_0)$ for some $(x_0,t_0) \in \R^2$, and $b(x_1,t_1) \neq W^{\dir -}(0,0;x_1,t_1)$  for some $(x_1,t_1) \in \R^2$, we have that $\intc_t^{b,L}$ and $\intc_t^{b,R}$ are finite for all $t \in \R$. 
\end{lemma} 
\begin{proof}
If not, then for some $t \in \R$, Lemma \ref{lem:eq_split}\ref{it:bW1} implies that that, either $b(x,t) - b(0,t) = W^{\dir +}(0,t;x,t)$ for all $x \in \R$, or $b(x,t) - b(0,t) = W^{\dir -}(0,t;x,t)$ for all $x \in \R$. This contradicts Lemma \ref{lem:neq}. 
\end{proof}

    \section{Eternal solutions and interfaces} \label{sec:global_proofs}
     This Section completes the proofs of the main theorems in the introduction, with the exception of Theorem \ref{thm:int1}, which is proved in Section \ref{sec:geom_fine_prop}.

      We define an equivalence relation on the set$\fxi$ as follows: we say $b_1 \sim b_2$ if and only if $b_1(x,t)=b_2(x,t)+C$ for all $(x,t) \in \R^2$ and for a constant $C$. For $b \in \fxi$ let $[b]$ denote the equivalence class of $b$.  We denote the set of these equivalence classes as
    \[
    \fxi/\sim:=\{[b]: b \in \fxi \}.
    \]
  Since maximizers are preserved under constant shifts, the following is immediate.
\begin{remark}
If $b_1 \sim b_2$, then for all $(x,t) \in \R^2$, $g_{(x,t)}^{b_1,L} = g_{(x,t)}^{b_2,L}$ and $g_{(x,t)}^{b_1,R} = g_{(x,t)}^{b_2,R}$. In particular, for $\dir \in \DLBusedc$ and $b_1 \sim b_2 \in \fxi$, $\intc_t^{b_1,L/R} = \intc_t^{b_2,L/R}$ for all $t \in \R$. 
\end{remark}
The first two items of the following proposition is a restatement of Theorem \ref{thm:mr}.
    \begin{proposition}
    \label{p:correspondence}
        The following holds on the event $\Omega_1$: For every $\dir\in \DLBusedc$,  there are bijections $\iota^-:\overline \Intc^{\dir,-} \;\to \fxi/\sim$ and $\iota^+:\overline \Intc^{\dir,+} \to \fxi/\sim$ having the following properties:
        \begin{enumerate} [label=(\roman*), font=\normalfont]
        \item \label{itm:triv_action} $\intc^-(\ltriv)$ and $\iota^+(\ltriv)$ are both equal to the equivalence class of the function $(x,t) \mapsto W^{\dir +}(0,0;x,t)$, and and $\iota^-(\rtriv)$ and $\iota^+(\rtriv)$ are both equal to equivalence class of the function $(x,t) \mapsto W^{\dir -}(0,0;x,t)$.
        \item \label{it:boutput} If $\intc \in \Intc^{\dir,-}$ (resp. $\intc \in \Intc^{\dir +}$), then the action $\iota^-(\intc)$ (resp. $\iota^+(\intc)$) is the equivalence class of the function
        \be \label{eq:bxs_fun}
        b(x,t)=\begin{cases}
            W^{\dir-}(\intc_0,0;x,t) \text{ if } x \leq \intc_t\\
             W^{\dir+}(\intc_0,0;x,t) \text{ if } x \geq \intc_t,
        \end{cases}
        \ee
        where $\intc_s = \intc(s)$ for $s \in \R$.
        \item \label{it:Requiveq} If $\intc^- \in  \Intc^{\dir,-}$ and $b$ is a representative of the equivalence class $\iota^-(\intc)$, then for all $t \in \R$,
        \[
        \intc^-(t) = \intc_t^{b,R},
        \]
        where $\intc_t^{b,R}$ is defined as in \eqref{tauLR_Def}.  If $\intc^+ \in  \Intc^{\dir,+}$ and $b \in \fxi$ is a representative of the equivalence class  $\iota^+(\intc)$, then for all $t \in \R$,
        \[
        \intc^+(t) = \intc_t^{b,L}. 
        \]
        \end{enumerate} 
       
    \end{proposition}
    \begin{remark}
    Recall by Corollary \ref{c:Im_cutoff} that $\intc_t^{b,R} \le \intc^{b,L}$. This explains why $\intc_t^{b,R}$ is associated with $\iota^-(\intc)$ and $\intc_t^{b,L}$ is associated with $\iota^+(\intc)$. 
    \end{remark}
    \begin{proof}
        We construct the bijection $\iota^-:\overline \Intc^{\dir,-} \to \fxi/\sim$, and the other statement follows symmetrically.  Fix $\intc\in \overline \Intc^{\dir,-}$.  For ease of notation in this proof we will use $\tau_t$ to denote $\tau(t)$.
        Recalling the definition of the set $\overline \Intc^{\dir,-}$ (Definition \ref{def:biInf}), if  $\intc = \ltriv$ or $\intc = \rtriv$, then define $\iota^-(\intc)$ by the statement in Item \ref{itm:triv_action}.
        
        If $\intc \in \Intc^{\dir,-}$, then let
        \begin{equation}
            f^\intc_t=f^{\xi}_{(\intc_t,t)} \text{ for all } t \in \R,
        \end{equation}
        where $f^\xi_{(\tau_t,t)}$ is as defined in \eqref{eq: initial condition}. For $t > s$, define $\mathfrak h^\dir_{(\tau_s,s),t}(x)$ as in \eqref{kpzs_abbr}. 
       For $t \in \R$, define 
        \begin{equation}
            \ell_t^\intc:=
                W^{\dir+}(\intc_0,0;\intc_t,t)= W^{\dir-}(\intc_0,0;\intc_t,t),
        \end{equation}
        where the equality holds by Lemma \ref{l:MV}.
        For $(x,t) \in \R^2$, define $b(x,t)=f_t^\tau(x)+\ell_t^\intc$, and let $\iota^-(\intc) = [b]$. We first claim that $ b \in \mathcal{F}^\dir$.  
        First, observe that by definition and additivity of the Busemann functions, $b(x,t)$ is equal to the function in \eqref{eq:bxs_fun}.
        
        The function $b:\R^2 \to \R$ is continuous because by Lemma \ref{l:MV}\ref{it:Bu3}, 
        \[
         W^{\dir-}(\intc_0,0,\intc_t,t)=W^{\dir+}(\intc_0,0;\intc_t,t).
        \]Choose $s < t$, and observe that
        \begin{align*}
       &\sup_{z \in \R} \{b(z,s)+\Ll(z,s;x,t) \}=\sup_{z \in \R}\{f^\dir_{(\tau_s,s)}(z)+\ell_s^\tau+\Ll(z,s;x,t)\}=\sup_{z \in \R}\{f^{\dir}_{(\tau_s,s)}(z)+\Ll(z,s;x,t) \}+\ell_s^\tau\\
       &=\kpzs^\dir_{(\tau_s,s),t}(x)+W^{\dir+}(\intc_0,0;
       \tau_s,s)=f^\dir_{(\tau_t,t)}+W^{\dir+}(\tau_s,s;\tau_t,t)+W^{\dir+}(\intc_0,0;
       \tau_s,s) \\&=f_t^\tau+W^{\dir+}(\intc_0,0;\tau_t,t)
    =f_t^\tau+\ell_t^\intc=b(x,t),
        \end{align*}
        where the first equality on the second line follows from Lemma \ref{lem:4} and Lemma \ref{l:MV}\ref{it:Bu1}. 
        Finally, \eqref{eq:b_slope} is satisfied because of the asymptotic slope condition for Busemann functions.

                \medskip \noindent \textbf{Proving the map $\iota^-$ is injective:}
                 We do this in two steps:
                \begin{enumerate}
                \item Show that $\iota^-(\ltriv) \neq \iota^-(\intc) \neq \iota^-(\rtriv) \neq \iota^-(\ltriv)$ whenever $\intc \in \Intc^{\dir,-}$.
                \item Show that $\iota^-$ is injective on the smaller set $\Intc^{\dir,-}$. 
                \end{enumerate}
                To prove (1), we show that $\iota(\intc) \neq \iota^-(\rtriv)$, and the others follow similarly. Let $b$ be defined as in \eqref{eq:bxs_fun} so that $\iota^-(\intc) = [b]$. Pick $t \in \R$. Then, for all $x \ge \intc_t$, by additivity of the Busemann functions,
                \begin{align*}
                b(x,t) - W^{\dir -}(0,0;x,t) &= W^{\dir +}(\intc_0,0;x,t)  - W^{\dir -}(0,0;x,t) \\
                &=W^{\dir +}(0,t;x,t) - W^{\dir -}(0,t;x,t) + W^{\dir +}(\intc_0,0;0,t) - W^{\dir -}(0,0;0,t) \\
                &= D_t^\dir(x) + W^{\dir +}(\intc_0,0;0,t) - W^{\dir -}(0,0;0,t).
                \end{align*}
                Since $\dir \in \DLBusedc$, Lemma \ref{lem:dif_to_infinity} implies that, for fixed $t$, the difference $ b(x,t) - W^{\dir -}(0,0;x,t)$ goes to $\infty$ as $x \to \infty$. Hence, $(x,t) \mapsto b(x,t)$ and $(x,t) \mapsto W^{\dir -}(0,0;x,t)$ are not in the same equivalence class, so $\iota^-(\intc) \neq \iota^-(\rtriv)$.

                We move on to proving (2). From Corollary \ref{cor:disInt}, we know that distinct elements in $\Intc^{\dir,-}$ are disjoint. Thus,  without loss of generality, we assume  that $\intc_1,\intc_2 \in \Intc^{\dir,-}$, and
                \[
                \tau_1(s)<\tau_2(s) \text{ for all } s \in \R.
                \]
                By Lemma \ref{lem:coal_int}\ref{Itco1}, for each $s \in \R$, we must have that $\intc_1(s) \in \Intr^{s,\dir}_\alpha$ and $\intc_2(s) \in \Intr^{s,\dir}_\beta$ for indices $\alpha \neq \beta$, where the intervals $\Intr^{s,\dir}_\alpha$ are defined below \eqref{eq77}. 
                By definition of these intervals, we have 
                \begin{equation}
                \label{eq:difference function}
                D^\dir_s(\tau_1(s))<D^\dir_s(\tau_2(s)) \text{ for all } s \in \R.
                \end{equation}
             Let $b_1,b_2$ be the functions as defined in \eqref{eq:bxs_fun} corresponding to the interfaces $\tau_1$ and $\tau_2$ respectively. By definition, we have
                \be
                \begin{aligned}
                \label{eq:b_1b_2}
                b_1(\tau_1(0),0) &=0,\quad \text{and}\quad b_2(\tau_1(0),0)=W^{\dir-}(\tau_2(0),0;\tau_1(0),0)\\
                b_2(\tau_2(0),0)&=0, \quad \text{and}\quad b_1(\tau_2(0),0)=W^{\dir+}(\tau_1(0),0;\tau_2(0),0).
                \end{aligned}
                \ee
                Assume by way of contradiction, that $[b_1]=[b_2]$. Then there exists a constant $C \in \R$ such that
                \[
                b_2(x,0)=b_1(x,0)+C \text{ for all } x \in \R.
                \]
                But then, by \eqref{eq:b_1b_2}, we have
                \begin{align*}
                 -W^{\dir -}(\intc_1(0),0;\intc_2(0),0) &= b_2(\intc_1(0),0) = b_1(\intc_1(0),0) + C = C,\quad\text{and}\\
                W^{\dir +}(\intc_1(0),0;\intc_2(0),0) &= b_1(\intc_2(0),0) = b_2(\intc_2(0),0) - C = -C.
                \end{align*}
                Hence,
                \[
                W^{\dir+}(\tau_1(0),0;\tau_2(0),0)=W^{\dir-}(\tau_1(0),0;\tau_2(0),0).
                \]
                But this implies that
                \[
    D^\dir_0(\tau_1(0))=D^\dir_0(\tau_2(0)),
                \]
                 a contradiction to \eqref{eq:difference function}. Thus $\iota^-$ is injective.

                \medskip \noindent \textbf{Proving that $\iota^-$ is surjective:} 
                 Let $b \in \fxi$. We defined $\iota^-(\rtriv)$ to be the equivalence class of the function $(x,t) \mapsto W^{\dir -}(0,0;x,t)$ and $\iota^-(\rtriv)$ to be the equivalence class of the function $(x,t) \mapsto W^{\dir +}(0,0;x,t)$, so  we may assume that $b$ is not equivalent to either of those functions. Recall the definition of $\tau_t^{b,R}$ from \eqref{tauLR_Def}, and define
                \be \label{taut_taubR_def}
                \tau_t :=\tau_t^{b,R}.
                \ee
                By Lemma \ref{lem:bfin} and our assumption on $b$, $\tau_t \in \R$ for all $t$.
                
                We now claim that $\tau:=\{\tau_t\}_{t \in \R} \in \Intc^{\dir,-}$ and $\iota^-(\tau)=[b]$. To show that $\intc \in \Intc^{\dir,-}$, we need to show that for $s < t$,
                \be \label{is_interface}
                \tau^{\dir,-}_{(\tau_s,s)}(t)=\tau_t.
                \ee 
                Define $p_t := \tau^{\dir,-}_{(\tau_s,s)}(t)$, and assume, by way of contradiction, that $p_t \neq \tau_t$. We consider two cases.

                \medskip \noindent 
                \textbf{Case 1: $p_t<\tau_t$.}  By definition of $\tau_t = \tau_t^{b,R}$ \eqref{tauLR_Def}, we have 
                \[
                g^{b,R}_{(p_t,t)}=g^{\dir-,R}_{(p_t,t)}.
                \]
                Now notice the following: by Corollary \ref{cor:LRmost_pi}\ref{it:LR2},
                \be \label{ysgets}
                y_s:=g^{\dir+,R}_{(p_t,t)}(s) \geq \tau_s.
                \ee
                Further note that, by Lemma \ref{lem:eq_split}\ref{it:bW1},
                \be \label{bxs_fxi_def}
                b(x,s)=\begin{cases}b(\tau_s,s)+W^{\dir+}(\tau_s,s;x,s) \text{ if } x \geq \tau_s.\\
                b(\tau_s,s)+W^{\dir-}(\tau_s,s;x,s) \text{ if } x \leq \tau_s.
                \end{cases}
                \ee
                
                Recall the definition of the initial condition $f^\dir_{(\tau_s,s)}$. From \eqref{bxs_fxi_def}, we see that $x \mapsto b(x,s)$ is a constant shift of the function  $f^\dir_{(\tau_s,s)}$ defined in \eqref{eq: initial condition}. Hence, the maximizers of the functions 
                \[
                z \mapsto b(z,s)+\Ll(z,s;p_t,t)\quad\text{and}\quad z \mapsto f^\dir_{(\tau_s,s)}(z) + \Ll(z,s;p_t,t)
                \]
                are the same. By definition, $g^{b,R}_{(p_t,t)}(s)$ is the rightmost maximizer of the first function, and recalling $p_t = \tau_{(\tau_s,s)}^-(t)$, we have that $g^{\dir+,R}_{(p_t,t)}(s)$ is the rightmost maximizer of the second function by Corollary \ref{cor:LRmost_pi}\ref{it:LR2}. Since both $g^{\dir+,R}_{(p_t,t)}$ and $g^{b,R}_{(p_t,t)}$ are the rightmost geodesics between any two of their points (Lemma \ref{lem:geodesics_from_b} and Proposition \ref{prop:DL_SIG_cons_intro}\ref{itm:DL_LRmost_geod}), we have
                \be \label{gbReqgxi}
                g^{b,R}_{(p_t,t)}|_{[s,t]}=g^{\dir+,R}_{(p_t,t)}|_{[s,t]}.
                \ee

                 Further, by \eqref{ysgets}, definition of $\tau_t = \tau_t^{b,R}$, \eqref{tauLR_Def}, and Corollary \ref{c:cutoff_geod}, we have $g^{b,R}_{(y_s,s)}=g^{\dir+,R}_{(y_s,s)}$. 

                Since $g^{b,R}_{(p_t,t)}(s)=y_s$ by definition, consistency of the choice of rightmost geodesic (Lemma \ref{lem:geodesics_from_b}\ref{itm: consis}) combined with \eqref{gbReqgxi} implies that the equality in \eqref{gbReqgxi} extends to $(-\infty,t]$; that is,
                \[
                g^{b,R}_{(p_t,t)}=g^{\dir+,R}_{(p_t,t)}.
                \]
                But since $p_t<\tau_t$, this is a contradiction to the definition of $\tau_t$ \eqref{tauLR_Def}. 
                
                \medskip \noindent \textbf{Case 2: $p_t>\tau_t$.}  Fix $\tau_t<x<p_t.$ By Lemma \ref{lem:geod}\ref{it1}, we know that $g^{\dir-,R}_{(x,t)}|_{[s,t]}$ is the rightmost $(f^\dir_{(\tau_s,s)},s)$-to-$(x,t)$ geodesic, and furthermore, $z_s:=g^{\dir-,R}_{(x,t)}(s)<\tau_s$. As in the previous case,  \eqref{bxs_fxi_def}, implies that $z_s$ maximizes $z \mapsto b(z,s)+\Ll(z,s;x,t)$ over $z \in \R$. Further,  by \eqref{tauLR_Def} and Corollary \ref{c:cutoff_geod} we have $g^{b,R}_{(z_s,s)}=g^{\dir-,R}_{(z_s,s)}$. This gives us
                \begin{align*}
        g^{b,R}_{(x,t)}|_{[s,t]}=g^{\dir-,R}_{(x,t)} \quad\text{and} \quad g^{b,R}_{(z_s,s)}|_{(-\infty,s]}=g^{\dir-,R}_{(z_s,s)}|_{(-\infty,s]}.
                \end{align*}
                Since $z_s:=g^{\dir-,R}_{(x,t)}(s)$, combining the above and using consistency of the geodesics in Lemma \ref{lem:geodesics_from_b}\ref{itm: consis}, we have
                \[
                g^{b,R}_{(x,t)}=g^{\dir-,R}_{(x,t)}.
                \]
                But since $x>\tau_t$, this is a contradiction to \eqref{tauLR_Def}. 

\medskip \noindent 
                Combing the two cases, we have shown that $p_t=\tau_t$. Thus $\intc \in \Intc^{\dir,-}$, as desired. The only thing remaining to show is $\iota^-(\tau)=[b]$.
        For $t \in \R$, \eqref{bxs_fxi_def} and additivity gives us 
        \begin{align*}
        &b(x,t)=\begin{cases}b(\tau_t,t)+W^{\dir+}(\tau_t,t;\intc_0,0)+W^{\dir+}(\intc_0,0;x,t) \text{ if } x \geq \tau_t\\
        b(\tau_t,t)+W^{\dir-}(\tau_t,t;\intc_0,0)+W^{\dir-}(\intc_0,0;x,t) \text{ if } x \leq \tau_t.
        \end{cases}
        \end{align*}
        Comparing to the definition of an element of the equivalence class $\iota^-(\intc)$ \eqref{eq:bxs_fun}, to show $\iota^-(\tau)=[b]$, it suffices to show that, for all $t,s \in \R$,
        \be \label{eq:Bconst}
        \begin{aligned}
        &b(\tau_t,t)+W^{\dir+}(\tau_t,t;\intc_0,0)=b(\tau_t,t)+W^{\dir-}(\tau_t,t;\intc_0,0)  \\&=b(\tau_s,s)+W^{\dir+}(\tau_s,s;\intc_0,0)=b(\tau_s,s)+W^{\dir-}(\tau_s,s;\intc_0,0)
        \end{aligned}
        \ee
         The first and final equalities of \eqref{eq:Bconst} follow by Lemma \ref{l:MV}\ref{it:Bu3}. It then suffices to show equality of the first and third terms; i.e,
        \[
        b(\tau_t,t)+W^{\dir+}(\tau_t,t;\intc_0,0)=b(\tau_s,s)+W^{\dir+}(\tau_s,s;\intc_0,0).
        \]
        Rearranged, this equation gives
        \[
        b(\tau_t,t) - b(\tau_s,s) =W^{\dir+}(\tau_s,s;\tau_t,t),
        \]
       and since $\tau_t = \tau_t^{b,R}$ by definition \eqref{taut_taubR_def}, this follows by Lemma \ref{lem:eq_split}\ref{it:bW2}, completing the proof. 
         \end{proof}

\subsection{Proof of Corollary \ref{cor:mr}, Theorem \ref{thm:intGeo}, and Theorem \ref{thm:BiInt}:} \label{sec:most_main_proofs}
As noted in the introduction, Theorem \ref{thm:mr} is exactly the first two items of Proposition \ref{p:correspondence}. We move to proving the other main results, with the exception of Theorem \ref{thm:int1}, which is proved in Section \ref{sec:geom_proof}.  

        \begin{proof}[Proof of Corollary \ref{cor:mr}]
        \textbf{Item \ref{it:1F1S}:} If $\dir \notin \DLBusedc$, then Equation \eqref{eq:bW} of Lemma \ref{lem:mont} gives us that, for all $\dir$-eternal solutions $b$ and all $(x,t) \in \R^2$, 
        \[
        b(x,t) = b(0,0) + W^{\dir}(0,0;x,t),
        \]
        so each $\dir$-eternal solution is equivalent to $(x,t) \mapsto W^\dir(0,0;x,t)$.

        \medskip \noindent \textbf{Item \ref{it:1FinfS}:} When $\dir \in \DLBusedc$, then Proposition \ref{p:correspondence} gives us that the set of equivalence classes of $\dir$-eternal solutions, $\fxi/\sim$ are in bijection with $\overline \Intc^{\dir,-}$. This set has two more elements than the set $\Intc^{\dir,-}$, which by Corollary \ref{cor:disInt}\ref{it:SplitT_bij}, is in bijection with the set $\Split_{0,\dir}^L$.  By Proposition \ref{prop:Haus12}, $\Split_{0,\dir}$ is uncountable, and by Remark \ref{rmk:countable}, $\Split_{0,\dir}^L \setminus \Split_{0,\dir}^R$ is countable. By Proposition \ref{theorem: local variation and splitting points}, $\Split_{0,\dir} = \Split_{0,\dir}^L \cup \Split_{0,\dir}^R$, so $\Split_{0,\dir}^L$ is uncountable, and the set $\fxi/\sim$ is therefore also uncountable. 

        \medskip \noindent \textbf{Item \ref{it:not_constant}:} Let $b_1$ and $b_2$ be two non-equivalent $\dir$-eternal solutions. By Proposition \ref{p:correspondence}, there exist distinct $\intc_1,\intc_2 \in \overline \Intc^{\dir,-}$ such that $\iota^-(\intc_1) = [b_1]$ and $\iota^+(\intc_2) = [b_2]$.  Recall that $\overline \Intc^{\dir,-} = \Intc^{\dir,-} \cup\{\ltriv,\rtriv\}$ We first handle the case where at least one of $\intc^1$ or $\intc^2$ is either $\ltriv$ or $\rtriv$. We assume $\intc^1 = \ltriv$, and the other cases follow similarly. In this case, Proposition \ref{p:correspondence}\ref{itm:triv_action} states that  $\iota^-(\ltriv)$  is the equivalence class of $(x,t) \mapsto W^{\dir +}(0,0;x,t)$.
        
        Since $b_1$ and $b_2$ are not equivalent, there exists $(x_0,t_0) \in \R^2$ so that 
             \[
             b^2(x_0,t_0) - b^2(0,0) \neq  b^1(x_0,t_0) - b^1(0,0) = W^{\dir +}(0,0;x_0,t_0).
             \]
             Lemma \ref{lem:neq} implies that, for all $t \in \R$, there exists $x_1 \in \R$ so that 
             \[
             b^2(x_1,t) - b^2(0,t) \neq W^{\dir +}(0,t;x_1,t) = b^1(x_1,t) - b^1(0,t).
             \]
             But then, the function $x \mapsto b^1(x,t) -b^2(x,t) $ is not constant because
             \[
             b^1(x_1,t) - b^2(x_1,t) = \Bigl[b^1(x_1,t) - b^1(0,t) - \bigl(b^2(x_1,t) - b^2(0,t)\bigr)\Bigr] + b^1(0,t) -b^2(0,t) \neq b^1(0,t) -b^2(0,t).
             \]

             With this case handled, we now assume that $\intc^1,\intc^2 \in \Intc^{\dir,-}$. Let $\intc^i_t := \intc^i(t)$ for $t \in \R$ and $i = 1,2$. By Corollary \ref{cor:disInt}\ref{it:Split_disj}, we may, without loss of generality  assume that 
             \begin{equation}
                 \intc^1_t<\intc^2_t, \qquad \forall t\in \R.
             \end{equation}
              Fix $t\in \R$ and set $y_j=\intc^j_t$ for $j\in\{1,2\}$ so that $y_1 < y_2$. The conclusion of the corollary is equivalent to to the statement that
             \begin{equation}\label{eq81}
                x\mapsto b^{1}(x,t)-b^{1}(y_1,t)-[b^{2}(x,t)-b^{2}(y_1,t)] 
             \end{equation}
             is not constant. We show this as follows.  Since $\intc^i \in \Intc^{\dir,-}$ for $i = 1,2$, and $y_i = \intc_t^i$, Proposition \ref{Cond_bi-inf} implies that $y_i \in \Split_{t,\dir}^L$ for $i = 1,2$. In particular, 
             \be \label{eq:Dy2big}
             D_t^\dir(y_2) > D_t^\dir(x) \quad \text{for all }x < y_2.
             \ee 
             
             From the description of the image $\iota^-(\intc)$ in Proposition \ref{p:correspondence}\ref{it:boutput}, for $x\in(y_1,y_2)$,
             \begin{equation}\label{eq80}
             \begin{aligned}
                 &\quad \; b^{\intc_1}(x,t)-b^{\intc_1}(y_1,t)-[b^{\intc_2}(x,t)-b^{\intc_2}(y_1,t)] \\&=W^{\dir+}(y_1,t;,x,t)\
                 -W^{\dir-}(y_1,t;x,t)\
                 =D^\dir_t(x)-D^\dir_t(y_1).
             \end{aligned}
             \end{equation}
             Then, if the function in \eqref{eq81} were constant, this function must be constant on $[y_1,y_2]$, and then \eqref{eq80} would imply that $D_t^\dir(x) = D_t^\dir(y_2)$ for all $x \in [y_1,y_2]$. This contradicts \eqref{eq:Dy2big}. 
        \end{proof}

        \begin{proof}[Proof of Theorem \ref{thm:intGeo}]
        Let $b \in \fxi$, and let  $\intc^-$ and $\intc^+$ be the associated elements of $\overline \Intc^{\dir,-}$ and $\overline \Intc^{\dir,+}$. Let $(x,t) \in \R^2$, and let $g:(-\infty,t] \to \R$ be a $b$-geodesic. By Lemma \ref{lem:geodesics_from_b} and Remark \ref{rmk:b_global_solutions}, for each $s < t$, $g(s)$ is a maximizer of 
        \be \label{b_geod}
        z \mapsto b(z,s) + \Ll(z,s;x,t).
        \ee

If $b$ is equivalent to the function $(x,t) \mapsto W^{\dir-}(0,0;x,t)$, then Theorem \ref{thm:mr}\ref{it:mainthmit1} states that $\intc^-(t) = \intc^+(t) = \rtriv(t) = +\infty$ for all $t \in \R$, so the statement of the theorem reduces to saying that all $b$-geodesics are $\dir-$ geodesics, which is immediate. We similarly get the same when $b$ is equivalent to $(x,t) \mapsto W^{\dir+}(0,0;x,t)$.

We now assume that $b$ is not equivalent to either of the two Busemann functions, so that $\intc^-\in \Intc^{\dir,-}$ and $\intc^+ \in \Intc^{\dir,+}$. Recalling the definition of the functions $f_{(x_0,s)}^\dir$ \eqref{eq: initial condition}, Theorem \ref{thm:mr}\ref{it:mainthmit2} implies that the function $z\mapsto b(z,s)$ is equal to, up to the addition of an $s$-dependent constant, to the function $f_{(\intc^-(s),s)}^\dir$ (and also $f_{(\intc^+(s),s)}^\dir$). Hence, the restrictions of all $b$-geodesics rooted at $(x,t)$ are exactly the $\bigl(f_{(\intc^-(s),s)}^\dir,s\bigr)$-to-$(x,t)$ geodesics (equivalently $\bigl(f_{(\intc^+(s),s)}^\dir,s\bigr)$-to-$(x,t)$ geodesics). The proof now follows from Lemma \ref{lem:geod}.         
       \end{proof}

       \begin{proof}[Proof of Theorem \ref{thm:BiInt}]
       \textbf{Item \ref{itm:Interface_splitting}:} This follows immediately from Corollary \ref{cor:disInt}.

        \medskip \noindent \textbf{Item \ref{BiInt:it4}:} Let $b \in \fxi$, and let $\intc^- = (\iota^-)^{-1}([b])$ and $\intc^+ = (\iota^+)^{-1}([b])$, where $\iota^-$ and $\iota^+$ are the bijections in Proposition \ref{p:correspondence}. Let $t \in \R$. Assuming that neither of $\iota^-$ and $\iota^+$ are $\ltriv$ or $\rtriv$, we seek to show that $\intc^-(t) = \a_\alpha^t$ and $\intc^+(t) = \b_\alpha^t$ for some index $\alpha$.  By Proposition \ref{Cond_bi-inf}, $\intc^-(t)  \in \Split_{t,\dir}^L$ and $\intc^+(t) \in \Split_{t,\dir}^R$. Hence, $\intc^-(t) = \a_\alpha^t$ and $\intc^+(t) = \b_\beta^t$ for some indices $\alpha$ and $\beta$. We show that $\alpha = \beta$. Suppose, by way of contradiction, that $\alpha \neq \beta$. Then, by Proposition \ref{Cond_bi-inf}, there exists an interface $\wt \intc \in \Intc^{\dir,+}$ with $\wt \intc \neq \intc^+$ such that $\wt \intc(t) = \b_\alpha^t$. Let $\wt b$ be a representative of the equivalence class $\iota^+(\wt \intc)$. Since $\iota^+$ is a bijection, $[\wt b] \neq [b]$. By Proposition \ref{p:correspondence}\ref{it:Requiveq},
           \be \label{3intceq}
           \intc^-(t) = \intc_t^{b,R},\quad \intc^+(t) = \intc_t^{b,L}, \quad \text{and}\quad \wt \intc(t) = \intc_t^{\wt b,L}.
           \ee
By Corollary \ref{c:Im_cutoff}, $\intc_t^{b,R} \le \intc_t^{b,L}$. By  \eqref{3intceq}, $\intc_t^{b,R} = \a_\alpha^t$ and $\intc_t^{b,L} = \b_\beta^t$ for $\beta \neq \alpha$, and $\intc_t^{\wt b,L} = \b_\alpha^t$,  so 
\be \label{tau_order}
\intc_t^{\wt b,R} < \intc_t^{b,R} \le \intc_t^{\wt b,L} <  \intc_t^{b,L}.
\ee
To get the strict inequalities above, note that Corollary \ref{cor:disInt}\ref{it:SplitT_bij} implies that $\intc_t^{\wt b,R} \neq \intc_t^{b,R}$ because $b$ and $\wt b$ are non-equivalent $\dir$-eternal solutions. We also assumed that $\intc_t^{b,R} = \a_\alpha^t$ and $\intc_t^{\wt b,L} = \b_\alpha^t$, so by definition of the intervals in $\Intr^{t,\dir}$, $\intc_t^{b,R}$ is the maximal element of $\Split_{t,\dir}^L$ less than or equal to $\intc_t^{\wt b,L}$. Hence, since $\intc_t^{\wt b,R} \le \intc_t^{\wt b,L}$ in general. the first strict inequality in \eqref{tau_order} must hold, and the other follows symmetrically. 

Now, choose $z \in [\intc_t^{b,R},\intc_t^{\wt b,L}]$.  By \eqref{tau_order}, when $x \le z$, we have $x \le\intc_t^{\wt b,L}$ and $x \le \intc_t^{b,L}$, so Lemma \ref{lem:eq_split}\ref{it:bW1} implies that
           \[
           b(x,t) - b(z,t)  = W^{\dir -}(z,t;x,t)= \wt b(x,t) - \wt b(z,t).
           \]
           Similarly, when $x \ge z$, $x \ge \intc_t^{\wt b,R}$ and $x \ge \intc_t^{b,R}$, so Lemma \ref{lem:eq_split}\ref{it:bW1} implies
           \[
           b(x,t) - b(z,t) = W^{\dir +}(z,t;x,t) = \wt b(x,t) - \wt b(z,t) 
           \]
           In particular, for all $x \in \R$,
           \[
           b(x,t) - \wt b(x,t) = b(z,t) - \wt b(z,t),
           \]
           so $x \mapsto b(x,t) - \wt b(x,t)$ is constant. But since $[b] \neq [\wt b]$, this contradicts Corollary \ref{cor:mr}\ref{it:not_constant}. 
       \end{proof}

\section{Finer properties of interfaces and exceptional behavior} \label{sec:geom_fine_prop}
The goal of this Section is to complete the proof of all Items of Theorem \ref{thm:int1}. This proof is completed in Section \ref{sec:geom_proof}. The key remaining ingredients handled in this section are Proposition \ref{p:BP}, which shows that the set $\Branch_{t,\dir}$ is exactly the set of the points through which mixed $\dir$-Busemann interfaces pass at time $t$, and Propositions \ref{p: interfaces do not meet} and \ref{p:fin_bubble}, which combine to allow us to prove that the bi-infinite interfaces in $\Intc^{\dir,-}$ and $\Intc^{\dir,+}$ form interior bubbles.

The first intermediate result of this section gives a sufficient condition for  a point on the plane to have distinct left most and right most interfaces on a fixed interval in time. 
        \begin{lemma}
        \label{lem:CondInt}
            On the event $\Omega_1$, the following holds. Let $\dir \in \R$, and let $(x,s)$ and $(y,t)$ be such that $t>s$ and  $x = g^{\dir+,R}_{(y,t)}(s)=g^{\dir-,L}_{(y,t)}(s)$. Then, $\intc^{\dir,-}_{(x,s)}(t) <y< \intc^{\dir,+}_{(x,s)}(t)$. 
        \end{lemma}
        \begin{proof}
            First, we will show that
            \begin{equation}\label{eq31}
            y \in [\intc^{\dir,-}_{(x,s)}(t),\intc^{\dir,+}_{(x,s)}(t)].    \end{equation}
             Let $f = f_{(x,s)}^\dir$. Assume by way of contradiction that $y< \intc^{\dir,-}_{(x,s)}(t)$. Then by Lemma \ref{lem:geod}\ref{it1}, we have that $g^{\dir-,L}_{(y,t)}|_{[s,t]}$ is an $f$ to $(y,t)$ geodesic and $\chi^{R}(f,s;y,t)<x$. This is a contradiction since $g^{\dir-,L}_{(y,t)}(s) =x$. Hence, $y \ge \intc_{(x,s)}^{\dir,-}(t)$. By a symmetric argument,  $y \leq \intc^{\dir,+}_{(x,s)}$, giving \eqref{eq31}. To finish the proof, assume by way of contradiction, that $y = \intc_{(x,s)}^{\dir,-}(t)$. Then, Lemma \ref{lem:Splt}\ref{itm:G1} implies that 
             \[
             g_{(y,t)}^{\dir-,L}(s) < \tau_{(x,s)}^{\dir,-}(s) = x,
             \]
             contradicting the assumption $g_{(y,t)}^{\dir-,L}(s) = x$. A symmetric argument shows that $y \neq \intc_{(x,s)}^{\dir,+}(t)$.
        \end{proof}

   \begin{lemma}
       \label{lem:NU}
       The following holds on the event $\Omega_1$. For all $\dir \in \R, (x_0,s) \in \R^2$, if  $\intc^{\dir,-}_{(x_0,s)}(t)<\intc^{\dir,+}_{(x_0,s)}(t)$ for some $t > s$, then $\intc_{(x_0,s)}^{\dir,-}(t) \in \NU_{t,\dir-}$ and $\intc_{(x_0,s)}^{\dir,+}(t)\in \NU_{t,\dir+}$.
   \end{lemma}
   \begin{proof}
   Let $\tau_t^\sig = \tau_{(x_0,s)}^{\dir,\sig}(t)$ for $\sigg \in \{-,+\}$. Assume that $\intc_t^- < \intc_t^+$. We prove that $\intc^-_t \in \NU_{t,\dir-}$, and the proof that $\intc^+_t \in \NU_{t,\dir+}$ is symmetric.
       Let 
       \[
       s_1:=\sup \{r \in [s,t]:  \intc^-_r=\intc^+_r\}.
       \]
       Since $\intc^-_s = \intc^+_s = x_0$, the set over which the supremum is taken is nonempty.
       Since mixed $\dir$-Busemann interfaces are continuous (Lemma \ref{lem:interface_continuous}), we have
       \be \label{eq:tr-+}
       \intc^-_r<\intc^+_r \text{ for all } r \in (s_1,t],\quad \text{and}\quad \intc^-_{s_1}=\intc^+_{s_1}.
       \ee
       By Proposition \ref{p:rest}, for $r \in [s,t)$,
       \be \label{tau_agree}
       \tau_{(\tau_{r}^-,r)}(t) = \tau_t^-,\quad\text{and}\quad \tau_{(\tau_{r}^+,r)}(t) = \tau_t^+,
       \ee
       so, combined with \eqref{eq:tr-+}, Equation \eqref{eq:5.1eq1} of Proposition \ref{p:IntfDist} implies that
       \be \label{gdir-+eq}
      g^{\dir-,R}_{(\intc^-_t,t)}|_{[s_1,t]}= g^{\dir+,R}_{(\intc^-_t,t)}|_{[s_1,t]}.
       \ee
     By Lemma \ref{lem:Splt}\ref{itm:G1} followed by Proposition \ref{p:rest}, then Proposition \ref{p:>interface}, then \eqref{gdir-+eq}, we obtain 
       \[
        g^{\dir-,L}_{(\intc^-_t,t)}(r)<  \intc^-_r = \tau_{(\tau_{s_1}^-,s_1)}^{\dir,-}(r) <g^{\dir+,R}_{(\intc^-_t,t)}(r) = g^{\dir -,R}_{(\tau_t^-,t)}(r), \qquad \forall r \in (s_1,t).
       \]
       Thus, $\intc_t^- \in \NU_{t,\dir-}$, completing the proof.
   \end{proof}
   In Lemma \ref{l:LRSP}, we  show that if $x \in \Split_{t,\dir}$, then there exists a mixed $\dir$-Busemann bi-infinite interface which passes through $(x,t)$. The next proposition shows that, the set of locations for general $\dir$-Busemann interfaces (not necessarily bi-infinite) is the larger set $\Branch_{t,\dir}$. 
    \begin{proposition} 
    \label{p:BP}
         The following holds on the event $\Omega_1$ for all $\dir \in \R$. Assume $y \in \mathcal{B}_{t,\dir}\setminus \Split_{t,\dir}$ for some $t \in \R$. Then, there exists $s_0<t$ and $x_0 \in \R$ such that 
         \be \label{eq:ytpm}
         y=\intc^{\dir,-}_{(x_0,s_0)}(t)=\intc^{\dir,+}_{(x_0,s_0)}(t).
         \ee
    \end{proposition}
    \begin{remark} \label{rem:Branchiff}
    Since $\Split_{t,\dir} \subseteq \Branch_{t,\dir}$ Lemmas \ref{lem:Splt}\ref{itm:branching} and \ref{l:LRSP} combine to show that for each $t \in \R$, $y \in \Branch_{t,\dir}$ if and only if there is mixed $\dir$-Busemann interface passing through $(y,t)$. However, the condition $y \in \Branch_{t,\dir} \setminus \Split_{t,\dir}$ is \textit{not} necessary for the existence of a point $(x_0,s)$ to satisfy the equality of interfaces at $(y,t)$, as in \eqref{eq:ytpm}. Indeed, as we will see below in Proposition \ref{p: interfaces do not meet}, there is a certain subset of $\Split_{\dir}$ having this property. 
    \end{remark}
    \begin{proof}[Proof of Proposition \ref{p:BP}]
        The assumption $y \in \mathcal{B}_{t,\dir}$ implies that there exists $s<t$ such that
        \be \label{eq:gsplit1}
    g^{\dir-,L}_{(y,t)}(r) < g^{\dir+,R}_{(y,t)}(r) \qquad \text{for all } r \in (s,t).
        \ee
        But since we also assumed that $y \notin \Split_{t,\dir}$, the geodesics $g^{\dir-,L}_{(y,t)}$ and $g^{\dir+,R}_{(y,t)}$ are not disjoint, so  we may make the choice of $s$ such that  
        \be \label{gpmeq}
    x: =g^{\dir-,L}_{(y,t)}(s)=g^{\dir+,R}_{(y,t)}(s).
        \ee
        See Figure \ref{fig:Branchingpoints}. By \eqref{eq:gsplit1}, we may choose $s_0 \in (s,t)$ and $x_0 \in \R$ such that
         \be \label{x0s0choice}
         g^{\dir-,L}_{(y,t)}(s_0)<x_0<g^{\dir+,R}_{(y,t)}(s_0).
         \ee
         We claim that \eqref{eq:ytpm} holds for this choice of $(x_0,s_0)$ (see Figure \ref{fig:Branchingpoints}). We first make a few observations.
         
         By \eqref{gpmeq} and Lemma \ref{lem:CondInt}, we have
        \begin{equation}\label{eq76}
            \intc^{\dir,-}_{(x,s)}(t)<y<\intc^{\dir,+}_{(x,s)}(t).
        \end{equation}
        We now show that
        \begin{equation}
        \label{eq:5.5eq1}
        \intc^{\dir,-}_{(x,s)}(r) \leq g^{\dir-,L}_{(y,t)}(r) < g^{\dir+,R}_{(y,t)}(r)\leq \intc^{\dir,+}_{(x,s)}(r) \qquad \text{ for all } r \in (s,t).
        \end{equation}
        The middle inequality is \eqref{eq:gsplit1}, so we prove the outer inequalities.  
        First assume, by way of contradiction, that $g^{\dir-,L}_{(y,t)}(r)<\intc^{\dir,-}_{(x,s)}(r)$ for some $r \in (s,t)$. Then, Lemma \ref{lem:geod}\ref{it1} implies that $g^{\dir-,L}_{(y,t)}|_{[s,r]}$ is an $(f_{(x,s)}^\dir,s)$-to-$(g^{\dir-,L}_{(y,t)}(r), r)$ geodesic, and Equation \eqref{eq74} of Lemma \ref{lem:geod}\ref{it1} implies that
        $
    g^{\dir-,L}_{(y,t)}(s) < x$.
     This contradicts the definition of $x$, which proves the first inequality of \eqref{eq:5.5eq1}. The last inequality of \eqref{eq:5.5eq1} follows by a symmetric proof. Note that, if we replace the middle inequality with a weak inequality, then \eqref{eq:5.5eq1} holds for $r = t$ as well.

          \begin{figure}[t!]
             \includegraphics[width=6cm]{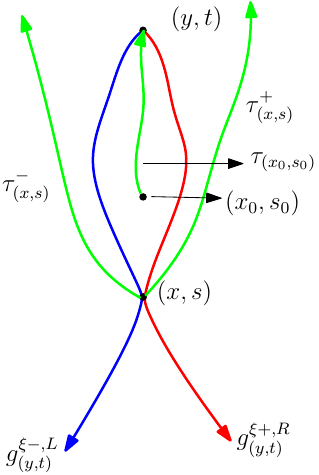}
             \caption{To prove Proposition \ref{p:BP} we consider the branching point $(y,t)$ which is not a splitting point. Then we consider the point $(x_0,s_0)$ as in the figure. We show that $y=\intc^-_{(x_0,s_0)}(t)=\intc^+_{(x_0,s_0)}(t).$ }
             \label{fig:Branchingpoints}
         \end{figure}
         
         Thus, by Lemma \ref{l:MV}\ref{it:Bu3}, for all $r \in (s,t]$,
        \begin{equation}\label{eq:eq5.5eq2}\begin{aligned}
         &W^{\dir+}(x,s;g^{\dir-,L}_{(y,t)}(r),r)=W^{\dir-}(x,s;g^{\dir-,L}_{(y,t)}(r),r),\quad\text{and}\\
         & W^{\dir+}(x,s;g^{\dir+,R}_{(y,t)}(r),r)=W^{\dir-}(x,s;g^{\dir+,R}_{(y,t)}(r),r).
         \end{aligned}
         \end{equation}
         In particular, when $r = t$, we get 
         \be \label{eq:eq5.5eq2.1}
            W^{\dir+}(x,s;y,t)=W^{\dir-}(x,s;y,t).
         \ee
         Applying Equation \eqref{geod_LR_eq_L} of Proposition \ref{prop:DL_SIG_cons_intro} \ref{itm:DL_all_SIG}, together with additivity of the Busemann functions from \eqref{eq:eq5.5eq2} and \eqref{eq:eq5.5eq2.1}, we get that, for all $r \in (s,t)$,
        
         \begin{equation}
         \label{eq5.5eq4}
             \begin{aligned}
                 &\Ll(g^{\dir-,L}_{(y,t)}(r),r;y,t)=W^{\dir+}(g^{\dir-,L}_{(y,t)}(r),r;y,t)=W^{\dir-}(g^{\dir-,L}_{(y,t)}(r),r;y,t).\\
                 &\Ll(g^{\dir+,R}_{(y,t)}(r),r;y,t)= W^{\dir+}(g^{\dir+,R}_{(y,t)}(r),r;y,t)=W^{\dir-}(g^{\dir+,R}_{(y,t)}(r),r;y,t).
             \end{aligned}
         \end{equation}
         By Lemma \ref{lem:Buse_eq}, it follows that for all $r \in (s,t], g^{\dir-,L}_{(y,t)}|_{[r,t]}$ (resp.\ $g^{\dir+,R}_{(y,t)}|_{[r,t]}$) is the restriction of a $\dir +$ (resp.\ $\dir -$) geodesic to $[r,t]$.

         We turn to proving \eqref{eq:ytpm}. We first prove that 
         \be \label{eq:ybetweentau}
         \intc^{\dir,-}_{(x_0,s_0)}(t) \le y \le \intc^{\dir,+}_{(x_0,s_0)}(t)
         \ee
 Assume, by way of contradiction that $y<\intc^{\dir,-}_{(x_0,s_0)}(t)$. We have shown that $g^{\dir+,R}_{(y,t)}|_{[s_0,t]}$ is a restriction of a $\dir-$ geodesic, and since $y<\intc^{\dir,-}_{(x_0,s_0)}(t)$,  Lemma \ref{lem:geod}\ref{it1} implies that $g^{\dir+,R}_{(y,t)}|_{[s_0,t]}$ is an $(f_{(x_0,s)}^\dir,s)$-to-$(y,t)$ geodesic satisfying $g^{\dir+,R}_{(y,t)}(s_0) < x_0$. But this is a contradiction to the choice of $x_0$ and $s_0$ in \eqref{x0s0choice}. This proves the first inequality in \eqref{eq:ybetweentau}. The second inequality follows by a symmetric argument.  

From \eqref{eq:ybetweentau}, it now suffices to show that $\intc_{(x_0,s)}^{\dir,-}(t) = \intc_{(x_0,s)}^{\dir,+}(t)$. By way of contradiction, assume that $\intc_{(x_0,s)}^{\dir,-}(t) < \intc_{(x_0,s)}^{\dir,+}(t)$.  Then, let
\[
s_1 = \sup\{r \in [s_0,t):\intc^{\dir,-}_{(x_0,s_0)}(r)=\intc^{\dir,+}_{(x_0,s_0)}(r) \}.
\]
By continuity, we have $s_1 < t$, and 
         \begin{align*}
         &\intc^{\dir,-}_{(x_0,s_0)}(r)<\intc^{\dir,+}_{(x_0,s_0)}(r)\quad \forall r \in (s_1,t],\quad \text{while} \\ &x_1:=\intc^{\dir,-}_{(x_0,s_0)}(s_1)=\intc^{\dir,+}_{(x_0,s_0)}(s_1).
         \end{align*}
        By the consistency in Proposition \ref{p:rest}, combined with \eqref{eq:ybetweentau},
         \[
\intc^{\dir,-}_{(x_1,s_1)}(t)=\intc^{\dir,-}_{(x_0,s_0)}(t) \le y \le \intc^{\dir,+}_{(x_0,s_0)}(t) \le  \intc^{\dir,+}_{(x_1,s_1)}(t).
         \]
         By Lemma \ref{lem:geod}\ref{it4}, the restriction of every $\dir -$ geodesic and every $\dir +$ geodesic to $[s_1,t]$ is a $(f_{(x_1,s_1)}^{\dir},s)$-to $(y,t)$ geodesic. In particular, $g^{\dir-,L}_{(y,t)}|_{[s_1,t]}$ and $g^{\dir+,R}_{(y,t)}|_{[s_1,t]}$ both are $(f_{(x_1,s_1)}^{\dir},s)$-to-$(y,t)$ geodesics. By Corollary \ref{c:in between points}, 
         \[
         g^{\dir-,L}_{(y,t)}(s_1)=g^{\dir+,R}_{(y,t)}(s_1)=x_1.
         \]
         contradicting \eqref{eq:gsplit1} since $s_1 \ge s_0 > s$. 
         \end{proof}
   Next, we show that it is possible for left and right interfaces to separate after traveling together for some time.
   \begin{proposition}
       \label{p: interfaces do not meet}
       The following holds on the event $\Omega_1$: let $\dir \in \DLBusedc$ and assume that $(x,t) \notin \Split_{\dir}^L$. Let 
       \be \label{t1y1def}
       t_1 := \inf\{r \le t: g_{(x,t)}^{\dir-,L}(r) = g_{(x,t)}^{\dir+,L}(r)\},\quad\text{and set}\quad y_1 := g_{(x,t)}^{\dir-,L}(t_1) = g_{(x,t)}^{\dir+,L}(t_1).
       \ee
       In words, $(y_1,t_1)$ is the splitting time of the two geodesics, as in Figure \ref{fig:interfaces can split}. Alternatively, we can assume that $(x,t) \notin \Split_\dir^R$ and define $(y_1,t_1)$ similarly as above with $L$ replaced by $R$. Then, $t_1 < t$, $y_1 \in \Split_{t_1,\dir}^L \cap \Split_{t_1,\dir}^R$, and there exists $(x_1,s_1) \in \Split_\dir^L$ and $(x_2,s_2) \in \Split_{\dir}^R$ with $s_1\vee s_2 < t$ such that, for  $i = 1,2$,
       \be \label{eq:tausplit}
       \intc_{(x_i,s_i)}^{\dir,-}(t_1) = \intc_{(x_i,s_i)}^{\dir,+}(t_1), \quad\text{and}\quad \intc_{(x_i,s_i)}^{\dir,-}(r) < \intc_{(x_i,s_i)}^{\dir,+}(r),\quad\forall r \in (t_1,t).
       \ee
       In words, the two interfaces agree at time $t_1$, but afterwards split. 
   \end{proposition}
\begin{remark}
Proposition \ref{p: interfaces do not meet} implies that bi-infinite interfaces split and then come back together. Indeed, as $(x_1,s_1) \in \Split_\dir^L$ and $(x_2,s_2) \in \Split_\dir^R$, we know from Proposition \ref{p: existence of bi-infinite competetion interfaces} that there exists an element of $\Intc^{\dir,-}$ that agrees with $\intc_{(x_1,s_1)}^{\dir,-}$ on $[s_1,\infty)$, and an element of $\Intc^{\dir,+}$ that agrees with $\intc_{(x_2,s_2)}^{\dir,+}$ on $[s_2,\infty)$
\end{remark}

   \begin{proof}   

   We will assume $(x,t) \notin \Split_{\dir}^L$; the case where $(x,t) \notin \Split_{\dir}^R$ follows symmetrically. This assumption 
 implies $t_1 < t$, and since $\dir \in \DLBusedc$, Proposition \ref{prop:g_basic_prop}\ref{itm:eventually_less} implies that $t_1 > -\infty$.
   \begin{figure}[t!]
       \includegraphics[width=6cm]{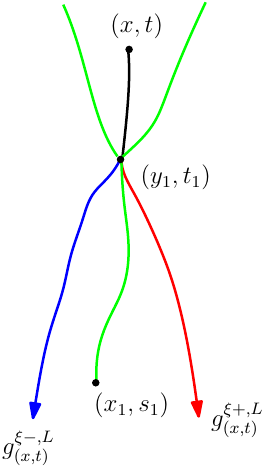}
       \caption{ The above figure shows that it is possible for interfaces to split even after they meet. $(x,t) \notin \Split^L_{\dir}$. Then is such that $g^{\dir-,L}_{(x,t)}$ and $g^{\dir+,L}_{(x,t)}$ coincide in $[t_1,t]$. Then as $(y_1,t_1) \in g^{\dir-,L}_{(x,t)} \cap g^{\dir+,L}_{(x,t)}$, by Lemma \ref{lem:CondInt} $\intc^-_{(y_1,t_1)}(t) \neq \intc^+_{(y_1,t_1)}(t)$. Also $y_1 \in \Split_{t_1,\dir}$. By Lemma \ref{l:LRSP}, we can construct an interface going through $(y_1,t_1)$ starting from $(x_1,s_1)$ and one can argue that interface will be unique at time level $t_1.$ The interfaces are coloured green in the figure.}
       \label{fig:interfaces can split}
   \end{figure}
Given the definition of $t_1$, continuity of geodesics  gives us the second equality stated in \eqref{t1y1def}, namely 
\be \label{y1eq}
y_1 = g_{(x,t)}^{\dir-,L}(t_1) = g_{(x,t)}^{\dir+,L}(t_1).
\ee
Since $g_{(x,t)}^{\dir\pm,L}$ are the leftmost geodesics between their points (Proposition \ref{prop:DL_SIG_cons_intro}\ref{itm:DL_LRmost_geod}), we get 
\[
z_r := g_{(x,t)}^{\dir-,L}(r) = g_{(x,t)}^{\dir+,L}(r),\quad\forall r \in [t_1,t]
\]
By Lemma \ref{lem:geod_unique}, $(z_r,r) \notin \NU_{\dir-} \cup \NU_{\dir +}$ for all $r \in [t_1,t)$. 
This along with \eqref{y1eq} and the consistency of geodesics in Proposition \ref{prop:DL_SIG_cons_intro}\ref{itm:DL_all_SIG} gives us  
\be \label{eq:gLR}
\begin{aligned}
&g_{(z_r,r)}^{\dir+,R}(s) =g_{(z_r,r)}^{\dir+,L}(s) = g_{(x,t)}^{\dir+,L}(s) ,\quad\text{and} \\ &g_{(x,t)}^{\dir-,L}(s) = g_{(z_r,r)}^{\dir-,L}(s) = g_{(z_r,r)}^{\dir-,R}(s),\quad \forall r \in [t_1,t)\text{ and } s \le r.
\end{aligned}
\ee
In particular, setting $s = t_1$ in \eqref{eq:gLR} and using \eqref{y1eq}, we obtain
\[
y_1 = g_{(z_r,r)}^{\dir+,R}(t_1) = g_{(z_r,r)}^{\dir-,L}(t_1), \quad\forall r \in [t_1,t).
\]
Then, by Lemma \ref{lem:CondInt}, we have 
   \be \label{eq:tausep}
   \intc^{\dir,-}_{(y_1,t_1)}(r) < z_r <  \intc^{\dir,+}_{(y_1,t_1)}(r),\quad\forall r \in (t_1,t).
   \ee
   Setting $r =t_1$ in \eqref{eq:gLR} and using the definition of $t_1$ in the inequality below, we obtain, 
   \[
   g_{(y_1,t_1)}^{\dir-,R} =g_{(y_1,t_1)}^{\dir-,L}(s) = g_{(x,t)}^{\dir-,L}(s) < g_{(x,t)}^{\dir+,L}(s) = g_{(y_1,t_1)}^{\dir+,L}(s) =g_{(y_1,t_1)}^{\dir+,R}(s) , \quad\forall s < t_1.
   \]
   so we see by definition of $\Split_{t_1,\dir}^{L/R}$ \eqref{Split_sdir} that $y_1\in \Split_{t_1,\dir}^L \cap \Split_{t_1,\dir}^R$. See Figure \ref{fig:interfaces can split}.

Then, by Lemma \ref{l:LRSP}, there exists $(x_1,s_1)\in \Split_\dir^L$ and $(x_2,s_2) \in \Split_\dir^R$ so that
   \[
   y_1=\intc^{\dir,-}_{(x_1,s_1)}(t_1) = \intc^{\dir,+}_{(x_2,s_2)}(t_1).
   \]
We now show that 
\be \label{eq:taueq}
\intc^{\dir,-}_{(x_1,s_1)}(t_1) = \intc^{\dir,+}_{(x_1,s_1)}(t_1),\quad\text{and}\quad \intc^{\dir,-}_{(x_2,s_2)}(t_1) = \intc^{\dir,+}_{(x_2,s_2)}(t_2).
\ee
We prove the first equality, and the second follows symmetrically. If, by way of contradiction, $\intc^{\dir,-}_{(x_1,s_1)}(t_1) < \intc^{\dir,+}_{(x_1,s_1)}(t_1)$, then  Lemma \ref{lem:NU} implies that $(y_1,t_1) \in \NU_{\dir-}$. But we have already shown that $(z_r,r) \notin \NU_{\dir-}\cup \NU_{\dir+}$ for $r \in [t_1,t)$. In particular, when $r = t_1$, $(y_1,t_1) \notin \NU_{\dir-}\cup \NU_{\dir+}$, giving a contradiction. 

To complete the proof, observe that \eqref{eq:taueq} is the first equality of \eqref{eq:tausplit}. For the second equality, the consistency of interfaces (Proposition \ref{p:rest}) and \eqref{eq:tausep} give us
\[
\intc_{(x_i,s_i)}^{\dir,-}(r) = \intc_{(y_1,t_1)}^{\dir,-}(r) < \intc_{(y_1,t_1)}^{\dir,+}(r) = \intc_{(x_i,s_i)}^{\dir,+}(r),\quad\forall r \in (t_1,t),
\]
as desired. 
\end{proof}
   Next we show that if two interfaces starting from a point meet at some time level $t$, then they can not immediately split again. In particular, this shows that two consecutive ``bubble" regions cannot occur (see Figure \ref{fig:fin_bbles}).
   \begin{proposition}
       \label{p:fin_bubble}
       The following holds on the event $\Omega_1$. Let $\dir \in \DLBusedc$ and $(x_0,s) \in \R^2$ be such that 
       \begin{equation}
      \label{eq: bbl_cond}
      \tau^{\dir,-}_{(x_0,s)}(r) < \tau^{\dir,+}_{(x_0,s)}(r)\quad \text{ for all } r \in (s,t), \text{ and }\quad \tau^{\dir,-}_{(x_0,s)}(t)=\tau^{\dir,+}_{(x_0,s)}(t).
       \end{equation}
      
      Then there exists $\delta>0$ such that 
       \begin{equation}
       \label{eq: fin_step_bubble}
       \tau^{\dir,-}_{(x_0,s)}(r)=\tau^{\dir,-}_{(x_0,s)}(r) \text{ for all } r \in [t,t+\delta].
       \end{equation}
   \end{proposition}
   \begin{figure}
       \centering
       \includegraphics[width=4cm]{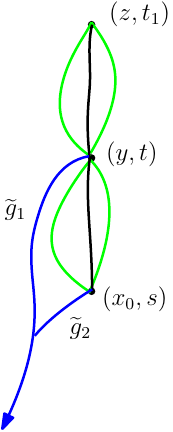}
       \caption{We show that if two interfaces starting from a point meet after splitting then they cannot split immediately again. If they do then one can construct two distinct $\dir-$ geodesics $\widetilde{g}_1$ and $\widetilde{g}_2$ starting from $(z,t_1)$ which stays together in $[t,t_1].$ $\widetilde{g}_1$ is the geodesic that follows the black path in $[t,t_1]$ and then the blue path starting from $(y,t).$ $\widetilde{g}_2$ is the geodesic that follows the black path in $[s,t]$ and then the blue path starting from $(x_0,s).$}
       \label{fig:fin_bbles}
   \end{figure}
   \begin{proof}
   
        Assume, by way of contradiction, that \eqref{eq: bbl_cond} holds, but there exists no such $\delta>0$ as in the statement. Let $y= \tau^{\dir,-}_{(x_0,s)}\left(t\right)=\tau^{\dir,+}_{(x_0,s)}\left (t \right)$. Proposition \ref{p:noBub} and our assumption implies that there exists $t_1>t$ such that 
       \[
       \tau^{\dir,-}_{(y,t)}(t_1)=\tau^{\dir,+}_{(y,t)}(t_1),\quad \text{and}\quad \tau^{\dir,-}_{(y,t)}(r)<\tau^{\dir,+}_{(y,t)}(r) \quad \forall r \in \left( t,t_1 \right ).
       \]
       Let $z=\tau^{\dir,-}_{(y,t)}(t_1)=\tau^{\dir,+}_{(y,t)}(t_1)$. Take an increasing sequence $r_n \nearrow t_1$ with $t < r_n < t_1$ for all $n$. By continuity of interfaces, we have $(z_n ,r_n) \rightarrow (z,t_1)$, where $z_n = \tau^{\dir,-}_{(y,t)}(r_n)$. By Lemma \ref{lem:SIGprecompact}, the sequence of geodesics $g^{\dir+,R}_{(z_n ,r_n)}$ has a subsequence converging to a $\dir$-directed geodesic $g_{(z,t_1)}$ rooted at $(z,t_1)$.  Further, note that by Proposition \ref{p:IntfDist}, for all $n$, 
       \[
       g^{\dir+,R}_{(z_n,r_n)}(t)=y.
       \]
       By taking limits,  $g_{(z,t_1)}(t)=y$.
       Now, by Proposition \ref{prop:Busani_N3G}, we know that $g_{(z,t_1)}$ is either a $\dir-$ geodesic or a $\dir+$ geodesic.  By \eqref{eqn:SIG_weight} and Lemma \ref{l:MV}\ref{it:Bu3} we have 
       \[
       \Ll(y,t;z,t_1)=W^{\dir-}(y,t;z,t_1)=W^{\dir+}(y,t;z,t_1).
       \]
       Therefore,  by Lemma \ref{lem:Buse_eq}, $g_{(z,t_1)}|_{[t,t_1]}$ is the restriction of both a $\dir-$ and $\dir+$ geodesic.
      
       By a similar argument (this time taking a sequence $r_n \nearrow t)$, there  exists  $\dir$-directed geodesic $g_{(y,t)}$ from $(y,t)$ such that 
      \[
       g_{(y,t)}(s)=x_0.
       \]
      Now by Proposition \ref{prop:Busani_N3G} we know that this geodesic will either be a $\dir-$ or a $\dir+$ geodesic.
      By \eqref{eqn:SIG_weight} and Lemma \ref{l:MV}\ref{it:Bu3} we have 
       \[
       \Ll(x_0,s;y,t)=W^{\dir-}(x_0,s;y,t)=W^{\dir+}(x_0,s;y,t).
       \]
       Therefore, $g_{(y,t)}|_{[s,t]}$ is a restriction of $\dir-$ geodesic. 
       We now define the following two backwards infinite paths $\wt g_1,\wt g_2$ rooted at $(z,t_1)$ as follows (see Figure \ref{fig:fin_bbles}) 
       \begin{align*}
           &\widetilde{g}_1|_{[t,t_1]}:=g_{(z,t_1)}|_{[t,t_1]},\quad \text{and}\quad \widetilde{g}_1|_{(-\infty,t]}=g^{\dir-,L}_{(y,t)}|_{(-\infty,t]},\\
           & \widetilde{g}_2|_{[t,t_1]}=g_{(z,t_1)}|_{[t,t_1]},\quad  \widetilde{g}_1|_{[s,t]}=g_{(y,t)}|_{[s,t]},\quad \text{ and }\quad  \widetilde{g}_2|_{(-\infty,s]}=g^{\dir-,L}_{(x_0,s)}|_{(-\infty,s]}.
       \end{align*}
      As a concatenation of $\dir-$ geodesics, Lemma \ref{lem:Buse_eq} and the additivity of Busemann functions implies that $\wt g_1$ and $\wt g_2$ are also $\dir-$ geodesics. By definition, $\wt g_1 = \wt g_2$ on $[t,t_1]$, any by coalescence of $\dir-$ geodesics (Proposition \ref{prop:DL_all_coal}), there exists $T < s$ such that $\wt g_1(r) = \wt g_2(r)$ for all $r \le T$. However, Lemma \ref{lem:Splt}\ref{itm:G1} implies that 
       \[
    \wt g_1(s) = g^{\dir-,L}_{(y,t)}(s)<x_0 = \wt g_1(s).
       \]
       This creates an interior geodesic bubble, a contradiction to Lemma \ref{lem:no_bubbles} (see Figure \ref{fig:fin_bbles}).
   \end{proof}

   \subsection{Proof of Theorem \ref{thm:int1}} \label{sec:geom_proof}
    \begin{proof}[Proof of Theorem \ref{thm:int1}]
    
    \textbf{Item \ref{int1:it1}:} As mentioned below the statement of the theorem, Equation \ref{eq:forw_dir} is \cite[Corollary 4.14]{Rahman-Virag-21}. There, it is stated with probability one for a fixed initial condition, but the only necessary inputs are the modulus of continuity bounds on the directed landscape from Lemma \ref{lem:Landscape_global_bound} and the slope condition 
    \[
    \lim_{x \to \infty} \f{f_{(x_0,s)}^{\dir}(x)}{x} = 2\dir.
    \]
    Since this holds on the event $\Omega_1$ for all such initial conditions (Proposition \ref{prop:Buse_basic_properties}\ref{it:Wslope}), the same proof holds.

    If $\intc \in \Intc^{\dir,-} \cup \Intc^{\dir,+}$, then Corollary \ref{cor:LRmost_pi} implies that, for all $t < 0$,
    \[
    g_{(\intc(0),0)}^{\dir-,L}(t) \le \intc(t) \le g_{(\intc(0),0)}^{\dir+,R}(t).
    \]
    Equation \eqref{eq:back_dir} now follows from the directedness of $g_{(\intc(0),0)}^{\dir-,L}$ and $g_{(\intc(0),0)}^{\dir+,R}$ (Proposition \ref{prop:DL_SIG_cons_intro}\ref{itm:arb_geod_cons}\ref{itm:geo_dir}).

    \medskip \noindent \textbf{Item \ref{int1:it2}:} As discussed in Remark \ref{rem:Branchiff}, this follows by Lemmas \ref{lem:Splt}\ref{itm:branching} and \ref{l:LRSP}. 

    \medskip \noindent \textbf{Item \ref{int1:it3}:} This was stated in Lemma \ref{lem:int_order}.

    \medskip \noindent \textbf{Item \ref{BiInt:it3}:} Let $\dir \in \DLBusedc$, $s \in \R$, $\alpha \in \I^{s,\dir}$, and $x_0,y_0 \in \Intr_\alpha^{s,\dir}$. Lemma \ref{lem:coal_int}\ref{Itco1} implies that, for $\sigg \in \{-,+\}$, 
    \[
    \intc_{(x_0,s)}^{\dir,\sig}(t)= \intc_{(x_0,s)}^{\dir,\sig}(t),\quad \text{for some }t > s.
    \]
    Then, Lemma \ref{lem:int_order} implies coalescence of these interfaces. Additionally, Lemma \ref{lem:coal_int}\ref{Itco1} implies that $\intc_{(x_0,s)}^{\dir,-}(t) = \intc_{(y_0,s)}^{\dir,+}(t)$ for some $t > s$. Then, the consistency of interfaces in Proposition \ref{p:rest} along with Lemma \ref{p:noBub} implies that, for each $n \in \N$, there exists $t_n > n$ such that 
    $
    \intc_{(x_0,s)}^{\dir,-}(t_n) = \intc_{(y_0,s)}^{\dir,+}(t_n).
    $

    \medskip \noindent \textbf{Item \ref{itm:disj_int_disj}:} This follows immediately from Lemmas \ref{lem:int_order} and \ref{lem:coal_int}\ref{Itco2}.

    \medskip \noindent \textbf{Item \ref{int1:it4}:} Take $(x,t) \notin \Split_\dir^L$, and construct a point $(x_0,s) := (x_1,s_1)$ (alternatively $(x_0,s) := (x_2,s_2))$ as in Proposition \ref{p: interfaces do not meet}. Proposition \ref{eq:tausplit} states that there is $t_1 < t$ so that
    \[
    y_1 := \intc_{(x_0,s)}^{\dir,-}(t_1) = \intc_{(x_0,s)}^{\dir,+}(t_1),\quad\text{and}\quad \intc_{(x_0,s)}^{\dir,-}(r) < \intc_{(x_0,s)}^{\dir,+}(r),\quad \forall r \in (t_1,t).
    \]
    Proposition \ref{p:fin_bubble} implies there must exist $\delta > 0$ so that $\intc_{(x_0,s)}^{\dir,-}(r) = \intc_{(x_0,s)}^{\dir,+}(r)$ for all $r \in [t_1-\delta,t_1]$. Next, define
    \[
    t_2 = \inf\{r > t_1:\intc_{(x_0,s)}^{\dir,-}(r) = \intc_{(x_0,s)}^{\dir,+}(r) \}.
    \]
    By Lemma \ref{p:noBub}, $t_2 < \infty$. By definition of $t_2$ and consistency of interfaces,
    \[
    \intc_{(y_1,t_1)}^{\dir,-}(r) = \intc_{(x_0,s)}^{\dir,-}(r) <\intc_{(x_0,s)}^{\dir,-}(r) = \intc_{(y_1,t_1)}^{\dir,-}(r),\quad\forall r \in (t_1,t_2),\quad\text{and}\quad \intc_{(x_0,s)}^{\dir,-}(t_2) <\intc_{(x_0,s)}^{\dir,-}(t_2).  
    \]
    By Proposition \ref{p:fin_bubble}, there exists $\delta > 0$ (redefining the previous $\delta$ to be smaller if necessary) such that 
    \[
    \intc_{(x_0,s)}^{\dir,-}(r) = \intc_{(x_0,s)}^{\dir,+}(r),\quad\forall r \in [t_2,t_2 + \delta]. \qedhere
    \]

       \end{proof}

\section{Eternal solutions as Busemann limits} \label{sec:lim_proofs}
In this section, we prove the results in Section \ref{sec:limits}. For two points $\mbf p,\mbf q$, we let $g_{\mbf p,\mbf q}^{L/R}$ denote the leftmost/rightmost geodesic between them.  We first prove the following intermediate lemma. 
\begin{lemma} \label{lem:gvk_lim}
The following holds on the event $\Omega_1$.  Let $\mbf v_n  = (x_n,t_n)$ be a sequence with $t_n \to - \infty$ and $\f{x_n}{|t_n|} \to \dir$. Then, for each $t \in \R$, each compact set $K \subseteq \R$, and $s < \inf\{t: (x,t) \in K\}$, there exists $N \in \N$ such that, for all $n \ge N$ and $\mbf p = (x,t) \in K$,
\[
\text{Either }\quad g_{\mbf v_n,\mbf p}^L|_{[s,t]} = g_{\mbf p}^{\dir-,L}|_{[s,t]},\quad\text{or}\quad   g_{\mbf v_n,\mbf p}^L|_{[s,t]} = g_{\mbf p}^{\dir+,L}|_{[s,t]}.
\]
The same holds if we replace $L$ everywhere with $R$. 
\end{lemma}
\begin{proof}
Suppose not. Then, there exists a subsequence $n_{\ell}$ such that, for all $\ell \ge 1$, there exists $\mbf p_\ell = (x,t) \in K$ ($x,t$ depending on $\ell$) such that  $\quad g_{\mbf v_{n_\ell},\mbf p_\ell}^L|_{[s,t]} \neq g_{\mbf p_\ell}^{\dir-,L}|_{[s,t]},\quad\text{and}\quad   g_{\mbf v_{n_\ell},\mbf p_\ell}^L|_{[s,t]}\neq g_{\mbf p_\ell}^{\dir+,L}|_{[s,t]}$. 
Since all four of these geodesics are the leftmost geodesic between their points (Proposition \ref{prop:DL_SIG_cons_intro}\ref{itm:DL_LRmost_geod}), and they start at the common point $\mbf p_\ell$, we have  
\be \label{eq:gpneq}
\quad g_{\mbf v_{n_\ell},\mbf p_\ell}^L(r) \neq g_{\mbf p_\ell}^{\dir-,L}(r),\quad\text{and}\quad   g_{\mbf v_{n_\ell},\mbf p_\ell}^L(r)\neq g_{\mbf p_\ell}^{\dir+,L}(r),\quad\text{for all }r \le s,\quad\text{and}\quad \ell \ge 1.
\ee
Take a sequence $s =  s_0 > s_1 > s_2 > s_3 > \cdots$ with $s_i \to -\infty$. Let $\ve > 0$, and define
\[
A_\ve := \inf\{g_{\mbf p}^{(\dir-\ve)-,L}(s_0): \mbf p \in K\},\quad\text{and}\quad B_\ve := \sup \{g_{\mbf p}^{(\dir  +\ve)+,R}(s_0): \mbf p \in K\}.
\]
By Lemma \ref{lem:bounded_maxes}, $-\infty < A_\ve \le B_\ve < \infty$. By directedness of $\mbf v_n$, for all $\ve > 0$ and $i \ge 1$, when $\ell$ is sufficiently large, we have 
\[
g_{(A_\ve,s_0)}^{(\dir- \ve)-,L}(t_{n_\ell}) \le x_{n_\ell} \le g_{(B_\ve,s_0)}^{(\dir- \ve)+,L}(t_{n_\ell}),
\]
and since the geodesics are the leftmost geodesics between their points, this implies that for such large $\ell$,
\be \label{eq:dir}
g_{(A_\ve,s_0)}^{(\dir- \ve)-,L}(r) \le g_{\mbf v_{n_\ell},\mbf p_\ell }^L(r) \le g_{(B_\ve,s_0)}^{(\dir+ \ve)+,L}(r),\quad \text{for all }r \in [t_{n_\ell},s_0],
\ee
where the first and last equalities follow by ordering of geodesics. 
By compactness and Lemma \ref{lem:precompact}, there exists a further subsequence of $\{n_\ell\}$ along which $g_{\mbf v_{n_\ell},\mbf p_\ell}^L|_{[s_1,t]}$ converges to a geodesic on $[s_1,s_0]$ in the uniform topology. Take a further subsequence along which we have convergence to a geodesic on $[s_2,s_0]$, and so on. By a diagonal argument, there exists a subsequence along which $g_{\mbf v_{n_\ell},\mbf p_\ell}^L(r)$ converges, in the topology of uniform convergence on compact sets, to an infinite geodesic, which we call $g$. We claim that
\be \label{eq:gnotin}
g(r) \notin \bigl\{g_{\mbf p_\ell}^{\dir-,L}(r),g_{\mbf p_\ell}^{\dir+,L}(r)\bigr\},\quad\text{for all }r \le s_1.
\ee
 To see this, first observe that for any $r_1 < r_2 < s_0$, $g$ is the unique geodesic between $(g(r_1),r_1)$ and $(g(r_2),r_2)$. Otherwise, an interior geodesic bubble would form, a contradiction to Lemma \ref{lem:no_bubbles}. Recall that $g$ is the uniform-on-compact limit of $g_{\mbf v_{n_\ell},\mbf p_\ell}$ along some subsequence $n_{\ell_m}$. Then, by Lemma \ref{lem:overlap}, since $s_1 < s_0$, for each $S < s_1$, we have that, for all sufficiently large $m$, 
\[
g_{\mbf v_{k_{\ell_m}},\mbf p_{\ell_m}}(r) = g(r),\quad\forall r \in [S,s_1]. 
\]
In particular, for each $r \le s_1$, \eqref{eq:gpneq} implies \eqref{eq:gnotin}. 

Lastly, fix $\mbf q \in K$. By Proposition \ref{prop:DL_all_coal}, there exists $T < s$ such that, 
\be \label{gq_coal}
g_{\mbf q}^{\dir -,L}(r) = g_{\mbf p_\ell}^{\dir -,L}(r),\quad\text{and}\quad g_{\mbf q}^{\dir +,L}(r) = g_{\mbf p_\ell}^{\dir +,L}(r), \quad \text{for all $r \le T$ and $\ell \ge 1$.}
\ee
By \eqref{eq:dir} and directedness of geodesics, the infinite geodesic $g$ has direction $\dir$. By Propositions \ref{prop:Busani_N3G} and \ref{prop:DL_all_coal}, the geodesic $g$ either coalesces with  $g_{\mbf q}^{\dir-}$ or $g_{\mbf q}^{\dir+}$. But this is not possible by \eqref{eq:gnotin} and \eqref{gq_coal}, giving a contradiction.
\end{proof}

\begin{proof}[Proof of Theorem \ref{thm:bn_gen_lim}]
\textbf{Item \ref{itm:bn_precompact}:} By the Arzel\'a-Ascoli theorem, we must show that $b_n$ is pointwise bounded and that, for every compact set $K \subseteq \R^2$, $b_n$ is uniformly equicontinuous on $K$. Without loss of generality, we may assume that the compact set $K$ contains the point $(0,0)$. 

By \cite[Theorem 7.1]{Busa-Sepp-Sore-22a}, there exists $T < \inf \{t: (x,t) \in K\}$ such that, for any two $\dir-$ geodesics $g_1,g_2$ rooted in the set $K$, $g_1|_{(-\infty,T]} = g_2|_{(-\infty,T]}$, and the same is true of any $\dir +$ geodesics. For $u \le T$, let $g^{\dir-}(u)$ and $g^{\dir+}(u)$ denote the common values of the $\dir-$ and $\dir+$ geodesics rooted in the set $K$.  By Lemma \ref{lem:gvk_lim},  there exists $N = N(K) \in \N$ such that, for all $n \ge N$, and $\mbf p= (x,t) \in K$,
\be \label{gpq_choices}
\text{Either }\quad g_{\mbf v_n,\mbf p}^L|_{[T,t]} = g_{\mbf p}^{\dir-,L}|_{[T,t]},\quad\text{or}\quad   g_{\mbf v_n,\mbf p}^L|_{[T,t]} = g_{\mbf p}^{\dir+,L}|_{[T,t]}.
\ee
We now claim that for $n \ge N$ and $\mbf p,\mbf q \in K$,
\be \label{eq:bnpq_bd}
\begin{aligned}
&\quad \,\min\{\Ll(g^{\dir-}(T),T;\mbf p) - \Ll(g^{\dir-}(T),T;\mbf q),\Ll(g^{\dir+}(T),T;\mbf p) - \Ll(g^{\dir+}(T),T;\mbf q) \} \\ &\le b_n(\mbf p) - b_n(\mbf q)   \\ &\le \max\{\Ll(g^{\dir-}(T),T;\mbf p) - \Ll(g^{\dir-}(T),T;\mbf q),\Ll(g^{\dir+}(T),T;\mbf p) - \Ll(g^{\dir+}(T),T;\mbf q) \}
\end{aligned}
\ee
Since $T < \min\{t:(x,t) \in K\}$ and $K$ is compact, the continuity of $\Ll$ and \eqref{eq:bnpq_bd} gives uniform equicontinuity of $b_n$ on $K$. Setting $\mbf q = (0,0)$, \eqref{eq:bnpq_bd} and continuity of $\Ll$ also gives pointwise boundedness of $b_n$. Given $\mbf p = (x,t),\mbf q = (y,s) \in K$, we now prove \eqref{eq:bnpq_bd} using three cases (recalling  \eqref{gpq_choices}).

\medskip \noindent \textbf{Case 1:} $g_{\mbf v_n,\mbf p}^L(T) = g^{\dir-}(T)= g_{\mbf v_n,\mbf q}^L(T)$. Then,
\begin{align*}
b_n(\mbf p) - b_n(\mbf q) &= \Ll(\mbf v_n;\mbf p) - \Ll(\mbf v_n;\mbf q) \\
&= \Ll(\mbf v_n;g^{\dir-}(T),T) + \Ll(g^{\dir-}(T),T;\mbf p) - [\Ll(\mbf v_n;g^{\dir-}(T),T) + \Ll(g^{\dir-}(T),T;\mbf q)] \\
&= \Ll(g^{\dir-}(T),T;\mbf p) - \Ll(g^{\dir-}(T),T;\mbf q).
\end{align*}

\medskip \noindent \textbf{Case 2:} $g_{\mbf v_n,\mbf p}^L(T) = g^{\dir+}(T)= g_{\mbf v_n,\mbf q}^L(T)$. Similarly as in Case 1, we have 
\[
b_n(\mbf p) - b_n(\mbf q) =\Ll(g^{\dir+}(T),T;\mbf p) - \Ll(g^{\dir+}(T),T;\mbf q).
\]

\medskip \noindent \textbf{Case 3:} $g_{\mbf v_n,\mbf p}^L(T) \neq g_{\mbf v_n,\mbf q}^L(T)$. By \eqref{gpq_choices}, we may, without loss of generality, assume that $g_{\mbf v_n,\mbf p}^L(T) = g^{\dir-}(T)$ and $g_{\mbf v_n,\mbf q}^L(T) = g^{\dir+}(T)$. Then,
\begin{align*}
\Ll(\mbf v_n;\mbf p) &= \Ll(\mbf v_n;g^{\dir-}(T),T) + \Ll(g^{\dir-}(T),T; \mbf p), \quad\text{and}\\ \Ll(\mbf v_n;\mbf q) &= \Ll(\mbf v_n;g^{\dir+}(T),T) + \Ll(g^{\dir+}(T),T;\mbf q).
\end{align*}
Then, using the reverse triangle inequality for $\Ll$, we have 
\begin{align*}
&\quad \,\Ll(g^{\dir+}(T),T; \mbf p) - \Ll(g^{\dir+}(T),T; \mbf q) \\
&\le [\Ll(\mbf v_n;g^{\dir+}(T),T) + \Ll(g^{\dir+}(T),T; \mbf p)] - [\Ll(\mbf v_n;g^{\dir+}(T),T) + \Ll(g^{\dir+}(T),T; \mbf q)] \\
&\le\Ll(\mbf v_n;\mbf p) - \Ll(v_n;\mbf q)  \\
&\le [\Ll(\mbf v_n;g^{\dir-}(T),T) + \Ll(g^{\dir-}(T),T; \mbf p)] - [\Ll(\mbf v_n;g^{\dir-}(T),T) + \Ll(g^{\dir-}(T),T; \mbf q)] \\
&\le \Ll(g^{\dir-}(T),T; \mbf p) - \Ll(g^{\dir-}(T),T; \mbf q),
\end{align*}
and this completes the proof of \eqref{eq:bnpq_bd}.

\medskip \noindent \textbf{Item \ref{itm:limits_eternal}:} Let $b$ be a subsequential limit of the sequence $b_n$ with respect to the uniform-on-compact topology. To ease the notation, we assume $b_n \to b$ as $n \to \infty$. We need to prove two things:
\begin{enumerate}[label=(\alph*), font=\normalfont]
\item \label{itm:blim_eternal}
For $s < t$ and $x \in \R$,
\be \label{eq:bxt_lim_global}
b(x,t) = \sup_{z \in \R}[b(z,s) + \Ll(z,s;x,t)].
\ee
\item  \label{itm:blim}$\lim_{|x| \to \infty} \f{b(x,0)}{x} = 2\dir$.
\end{enumerate}

To prove \ref{itm:blim_eternal}, we first observe that, for $t > s > t_n + 1$,
\be \label{eq:bn_prelimit}
\begin{aligned}
b_n(x,t) &= \Ll(\mbf v_n;x,t) - \Ll(\mbf v_n;0,0) \\
&= \sup_{z \in \R}[\Ll(\mbf v_n;z,s) + \Ll(z,s;x,t)] - \Ll(\mbf v_n;0,0) \\
&= \sup_{z \in \R}[b_n(z,s) + \Ll(z,s;x,t)].
\end{aligned}
\ee
By Lemma \ref{lem:gvk_lim}, there exists $N \in \N$ (depending on $x,t,s$) so that, for all $n \ge N$,
\[
\text{Either }\quad g_{\mbf v_n,(x,t)}^L(s) = g_{(x,t)}^{\dir-,L}(s) =: z_-,\quad\text{or}\quad   g_{\mbf v_n,(x,t)}^L(s) = g_{(x,t)}^{\dir+,L}(s) =:z_+.
\]
Then, from \eqref{eq:bn_prelimit},
\be \label{eq:bneqmax}
\begin{aligned}
b_n(x,t) &= \max\Bigl\{\Ll(\mbf v_n; z_-,s) + \Ll(z_-,s;x,t),\Ll(\mbf v_n; z_+,s) + \Ll(z_+,s;x,t)\Bigr\} - \Ll(\mbf v_n;0,0) \\
&= \max\Bigl\{b_n(z_-,s) + \Ll(z_-,s;x,t),b_n(z_+,s) + \Ll(z_+,s;x,t)\Bigr\},
\end{aligned}
\ee
and therefore, for all $z \in \R$,
\be \label{bngeq}
\max\Bigl\{b_n(z_-,s) + \Ll(z_-,s;x,t),b_n(z_+,s) + \Ll(z_+,s;x,t)\Bigr\} \ge b_n(z,s) + \Ll(z,s;x,t).
\ee
Taking limits as $n \to \infty$ in \eqref{eq:bneqmax} and \eqref{bngeq}, we have, for all $z \in \R$,
\[
b(x,t) = \max\Bigl\{b(z_-,s) + \Ll(z_-,s;x,t),b(z_+,s) + \Ll(z_+,s;x,t)\Bigr\} \ge b(z,s) + \Ll(z,s;x,t).
\]
Therefore, \eqref{eq:bxt_lim_global} holds, as desired. 

Next, we prove \ref{itm:blim}. By Proposition \ref{prop:Buse_basic_properties}\ref{it:Busliminfsup}, we have 
\[
\begin{aligned}
&W^{\dir-}(0,0;x,0) \le b(x,0) \le W^{\dir+}(0,0;x,0),\quad\text{for }x \ge 0,\quad\text{and} \\
&W^{\dir+}(0,0;x,0) \le b(x,0) \le W^{\dir-}(0,0;x,0),\quad\text{for }x \le 0.
\end{aligned}
\]
The result now follows by the asymptotic slopes of $W^{\dir-}$ and $W^{\dir +}$ (Proposition \ref{prop:Buse_basic_properties}\ref{it:Wslope}).
\end{proof}

   We complete this section with the proof of Theorem \ref{thm:existence_of_v_seq}.

\begin{proof}[Proof of Theorem \ref{thm:existence_of_v_seq}]
    Let $b$ be a $\dir$-eternal solution. If $b$ is equivalent to $(x,t) \mapsto W^{\dir-}(0,0;x,t)$, then for any compact interval $[x,y] \subseteq \R$ and $t \in \R$, Proposition \ref{prop:DL_all_coal} states that there is a time $S < t$ such that all $\dir-$ geodesics rooted in the set $[x,y] \times \{t\}$ have coalesced by time $S$. For $s \le S$, let $g^{\dir-}(s)$ be the location of this common geodesic, and take a sequence $s \ge t_n \to -\infty$ and $x_n  = g^{\dir-}(t_n)$. Then, for all $n \ge 1$, and $w,z \in [x,y]$, \eqref{geod_LR_eq_L} and additivity of the Busemann functions imply
    \[
    \Ll(\mbf v_n;z,t) - \Ll(\mbf v_n;w,t) = W^{\dir-}(w,t;z,t).
    \]
    We follow a symmetric procedure if $b$ is equivalent to $(x,t) \mapsto W^{\dir+}(0,0;x,t)$. 

    We will now assume that $b$ is not equivalent to either $(x,t) \mapsto W^{\dir-}(0,0;x,t)$ or $(x,t) \mapsto W^{\dir+}(0,0;x,t)$. Let $\intc^- \in \Intc^-$ be the associated leftmost mixed $\dir$-Busemann bi-infinite interface from the bijection in Theorem \ref{thm:mr}. Let $t \in \R$, and let $[x,y] \subseteq \R$. Set $z = \intc^-(t)$. 
    
    By Theorem \ref{it:mainthmit2} and additivity of Busemann functions,
    \be \label{eq:b_written}
    b(w,t) - b(x,t) = \begin{cases}
W^{\dir-}(x,t;w,t),& x \le w \le z, \\
W^{\dir+}(x,t;w,t), &z \le x \le w, \\
W^{\dir-}(x,t;z,t) + W^{\dir+}(z,t;w,t), &x \le z \le w. \\
    \end{cases}
    \ee
    The last case above comes by writing
    \begin{align*}
    &\quad\, W^{\dir +}(\intc_0,0;w,t) - W^{\dir-}(\intc_0,0;x,t) \\
    &= W^{\dir +}(\intc_0,0;w,t) - W^{\dir+}(\intc_0,0;\intc_t,t) - [W^{\dir-}(\intc_0,0;x,t) - W^{\dir-}(\intc_0,0;\intc_t,t)],
    \end{align*}
    where we have used the equality in Lemma \ref{l:MV}\ref{it:Bu3}. 
     By Proposition \ref{prop:DL_all_coal}, we may choose $S < t$ so that, for all $w \in [x,y] \cup \{z\}$,
    \be \label{g_coal}
    \begin{aligned}
    g^{\dir-}(s) &:= g_{(w,t)}^{\dir-,L}(s) = g_{(w,t)}^{\dir-,R}(s) = g_{(z,t)}^{\dir-,L}(s),\; \text{and}\\
    g^{\dir+}(s)&:= g_{(w,t)}^{\dir+,L}(s) = g_{(w,t)}^{\dir+,R}(s) = g_{(z,t)}^{\dir+,L}(s) ,\, \text{for  }s \le S.
    \end{aligned}
    \ee
    By Theorem \ref{thm:BiInt}\ref{itm:Interface_splitting}, $z \in \Split_{t,\dir}^L$, so 
    \be \label{eq:gztsplit}
    g_{(z,t)}^{\dir-,L}(s) < g_{(z,t)}^{\dir+,L}(s),\quad\text{for all }s < t.
    \ee
    We now claim that there exists $T < S$ such that, for all $s \le T$ and all $w \in [g^{\dir-}(s),g^{\dir+}(s)]$, we have 
    \be \label{eq:g_in_set}
        g_{(w,s),(z,t)}^L(S) \in \{g^{\dir-}(S),g^{\dir+}(S)\}.
    \ee
    To see this, first note that, by ordering of geodesics,
    \[
    g_{(w,s),(z,t)}^L(r) \in [g_{(z,t)}^{\dir-,L}(r),g_{(z,t)}^{\dir+,L}(r)], \quad\text{for all }r \in [s,t],
    \]
     and since each is a leftmost geodesic to the point $(z,t)$, if \eqref{eq:g_in_set} fails, then using \eqref{g_coal},
     \[
     g^{\dir-}(r) < g_{(w,s),(z,t)}^L(r) < g^{\dir+}(r),\quad\text{for all }r \in [s,S].
     \]
     Then, by a similar procedure as in the proof of Lemma \ref{lem:gvk_lim}, if such a $T$ did not exist, we could construct a $\dir$-directed geodesic $g$ along a subsequence of point-to-point geodesics such that $g$ does not coalesce with $g_{(z,t)}^{\dir-,L}$ or $g_{(z,t)}^{\dir+,L}$, a contradiction to Propositions \ref{prop:Busani_N3G} and \ref{prop:DL_all_coal}. 
     \begin{figure}
         \includegraphics[width=8 cm]{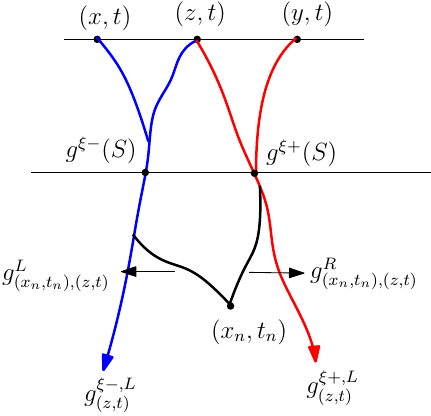}
         \caption{To prove Theorem \ref{thm:existence_of_v_seq} we choose $S$ such that all $\dir-$ (resp.\ all $\dir+$) geodesics starting from the interval $[x,y]$ have coalesced before time $S$. Now we choose a sequence of time $\{t_n\}$ such that for all $n$ and for all $w \in [g^{\dir-}(t_n),g^{\dir+}(t_n)], g^{L}_{(w,t_n),(z,t)}(S) \in \{g^{\dir-}(S),g^{\dir+}(S) \}$. $x_n \in [g^{\dir-}(t_n),g^{\dir+}(t_n)]$ is the maximum value in $[g^{\dir-}(t_n),g^{\dir+}(t_n)]$ such that $g^L_{(x_n,t_n),(z,t)}(S)=g^{\dir-}(S).$ Finally if we set $\mbf v_n:=(x_n,t_n),$ then $\mbf v_n$ is $\dir$ directed and satisfies $b(w,t)-b(x,t)=\Ll \left(\mbf v_n;w,t)-\Ll(\mbf v_n;x,t) \right)$ for all $w \in [x,y]$ and for all sufficiently large $n$.}
         \label{fig:converse_subsequence_problem}
     \end{figure}

     Now, for this choice of $T < S$, let $T > t_1 > t_2 > \cdots$ be an arbitrary sequence. For $n \ge 1$, define
     \be \label{xn_def}
     x_n = \sup\{w \in [g^{\dir-}(t_n),g^{\dir+}(t_n)]: g_{(w,t_n),(z,t)}^L(S) = g^{\dir-}(S)\}.
     \ee
     The set over which the supremum is taken is nonempty because it contains the left endpoint $g^{\dir -}(t_n) = g_{(z,t)}^{\dir-,L}(t_n)$ \eqref{g_coal}. By ordering of geodesics and Equations \eqref{g_coal} and \eqref{eq:g_in_set},
     \be \label{eq:g_options}
     \begin{aligned}
     g_{(w,t_n),(z,t)}^L(S) &= g^{\dir-}(S),\quad\text{for }w < x_n,\quad\text{and} \\
     g_{(w,t_n),(z,t)}^L(S) &= g^{\dir+}(S),\quad\text{for }w > x_n.
     \end{aligned}
     \ee
     Set $\mbf v_n = (x_n,t_n)$. By \eqref{xn_def} and directedness of $\dir \sig$ geodesics, $\f{x_n}{|t_n|} \to \dir$ as $n \to \infty$. We now claim that, for all $n$, 
     \be \label{eq:g_split_geod}
     g_{\mbf v_n,(z,t)}^L(S) = g^{\dir-}(S),\quad\text{and}\quad g_{\mbf v_n,(z,t)}^R(S) \ge g^{\dir+}(S). 
     \ee
     (see Figure \ref{fig:converse_subsequence_problem}).

     To see this, first observe that, by Lemma \ref{lem:precompact}, as $w \nearrow x_n$, the geodesic $g^L_{(w,t_n),(z,t)}$ converges uniformly, along some subsequence, to a geodesic $g^L$, between the points $\mbf v_n = (x_n,t_n)$ and $(z,t)$. By \eqref{eq:g_options}, the limiting geodesic satisfies $g^L(S) = g^{\dir-}(S)$, so by \eqref{eq:g_in_set}, the same is true of the leftmost such geodesic, that is, $g_{\mbf v_n,(z,t)}^L(S) = g^{\dir-}(S)$. Similarly, as $w \searrow x_n$, the geodesic $g^L_{(w,t_n),(z,t)}$ converges uniformly, along some subsequence, to a geodesic $g^R$, between the points $\mbf v_n$ and $(z,t)$, and the limiting geodesic satisfies $g^R(S) = g^{\dir+}(S)$. In the case $x_n = g^{\dir-}(t_n)$ or $g^{\dir+}(t_n)$, the proof is the same, except we only need to take limits from one side. Hence, $g_{\mbf v_n,(z,t)}^R(S) \ge g^{\dir+}(S)$ since $g_{\mbf v_n,(z,t)}^R$  is the rightmost geodesic between $\mbf v_n$ and $(z,t)$.

     Next, Lemma \ref{lem:gvk_lim}  and \eqref{g_coal} imply that there exists $N \in \N$ so that, for $n \ge N$, and each $w \in [x,y]$,
     \[
     \text{Either}\quad g_{\mbf v_n,(w,t)}^L(S) = g^{\dir-}(S),\quad\text{or}\quad g_{\mbf v_n,(w,t)}^L(S) = g^{\dir+}(S),
     \]
     and the same holds for $R$ in place of $L$. Ordering of geodesics, \eqref{eq:g_split_geod}, and $g^{\dir-}(S) < g^{\dir +}(S)$ \eqref{eq:gztsplit} imply that 
     \be \label{eq:6745}
     \begin{aligned}
        g_{\mbf v_n,(w,t)}^L(S) &= g^{\dir-}(S),\quad\text{for }w \le z,\quad\text{and} \\
        g_{\mbf v_n,(w,t)}^R(S) &= g^{\dir+}(S),\quad\text{for }w \ge z.
       \end{aligned} 
     \ee
     Then, when $n \ge N$, we claim that $\Ll(\mbf v_n;w,t) - \Ll(\mbf v_n;x,) = b(w,t) - b(x,t)$. We prove this in three cases.
     
     \medskip \noindent \textbf{Case 1: $x \le w \le z$:}
   Then, rearranging and using \eqref{g_coal}, and \eqref{geod_LR_eq_L} gives
     \begin{align*}
     &\quad \, \Ll(\mbf v_n;w,t) - \Ll(\mbf v_n;x,t) \\
     &= \Ll(\mbf v_n;g^{\dir-}(S),S) + \Ll(g^{\dir-}(S),S; w,t) - \bigl[\Ll(\mbf v_n;g^{\dir-}(S),S) + \Ll(g^{\dir-}(S),S; x,t)\bigr] \\
     &= \Ll(g^{\dir-}(S),S;w,t) - \Ll(g^{\dir-}(S),S;x,t) \\
     &= W^{\dir-}(g^{\dir-}(S),S;w,t) - W^{\dir-}(g^{\dir-}(S),S;x,t) = W^{\dir-}(x,t;w,t),
     \end{align*}

     \medskip \noindent \textbf{Case 2: $z \le x \le w$:} A symmetric argument shows that $\Ll(\mbf v_n;w,t) - \Ll(\mbf v_n;x,t) = W^{\dir+}(x,t;w,t)$.

     \medskip \noindent \textbf{Case 3: $x \le z \le w$:} In this case,
     \begin{align*}
     &\quad \, \Ll(\mbf v_n;w,t) - \Ll(\mbf v_n;x,t) \\
     &= \Ll(\mbf v_n;w,t) - \Ll(\mbf v_n;z,t) + \Ll(\mbf v_n;z,t) - \Ll(\mbf v_n;x,t) \\
     &= W^{\dir+}(z,t;w,t) + W^{\dir-}(x,t;z,t),
     \end{align*}
     where the last equality follows by applied both of the previous cases.
     In all three cases, this equals $b(w,t) - b(x,t)$ by \eqref{eq:b_written}.
\end{proof}

\appendix

         \section{Technical inputs and results} \label{appx:technical}
        The following was also stated in \cite[Lemma A.3]{Busa-Sepp-Sore-22a}. 
        \begin{lemma} \label{lem:cmcf}
Let $S \subseteq \R^n$, and let  $f_n:S \rightarrow \R$ be a sequence of continuous functions, converging uniformly to the function $f:S \rightarrow \R$. Assume that there exists a sequence $\{c_n\}$, of maximizers of $f_n$, converging to some $c \in S$. Then, $c$ is a maximizer of $f$. 
\end{lemma}
\begin{proof}
$f_n(c_n) \ge f_n(x)$ for all $x \in S$, so  it suffices to show that $f_n(c_n) \rightarrow f(c)$. This follows from the uniform convergence of $f_n$ to $f$, the continuity of $f$, and
\[
|f_n(c_n) - f(c)| \le |f_n(c_n) - f(c_n)| +|f(c_n) - f(c)|. \qedhere
\]
\end{proof}

\begin{lemma} \label{lem:sup_continuous}
Let $f:\R \to \R$ be a continuous function such that $\sup_{-\infty < z \le 0} f(z) < \infty$. Then, the function $g:\R \to \R$ defined by 
\[
g(x) = \sup_{-\infty < z \le x} f(z)
\]
is continuous.
\end{lemma}
\begin{proof}
Note that the condition $g(0) < \infty$ implies that $g(x) < \infty$ for all $x \in \R$ since $g(x) < g(0)$ for $x < 0$ and $g(x) = g(0) \vee \sup_{0 \le z \le x} f(z)$ for $x > 0$. Let $x \in \R$ and $\delta > 0$. Note that $g(x+\delta) \ge g(x)$, and
\[
g(x+ \delta) - g(x) = \Bigl(\sup_{x \le z \le x + \delta}f(z) - g(x)\Bigr)\vee 0 \le \Bigl(\sup_{x \le z \le x + \delta}f(z) - f(x)\Bigr)\vee 0,
\]
and the continuity  now follows from continuity of $f$. 
\end{proof}
        
        We now state some results on the modulus of continuity for the directed landscape. 
        \begin{lemma}[\cite{Directed_Landscape}, Corollary 10.7] \label{lem:Landscape_global_bound}
There exists a random constant $C_{\text{DL}}$ such that for all $v = (x,s;y,t) \in \Rup$, we have 
\[
\Bigl|\Ll(x,s;y,t) + \f{(x - y)^2}{t - s}\Bigr| \le C_{\text{DL}} (t - s)^{1/3} \log^{4/3} \Bigl(\f{2(\|v\| + 2)}{t - s}\Bigr)\log^{2/3}(\|v\| + 2),
\]
where $\|v\|$ is the Euclidean norm.
\end{lemma}

\begin{lemma}[\cite{Dauvergne-Sarkar-Virag-2022}, Lemma 3.1] \label{lem:precompact}
The following holds on a single event of probability one. Let $(p_n;q_n) \to (p,q) = (x,s;y,t) \in \Rup$, and let $g_n$ be any sequence of geodesics from $p_n$ to $q_n$. Then, the sequence of graphs $\graph g_n$ is precompact in the Hausdorff metric, and any subsequential limit of $\graph g_n$ is the graph of a geodesic from $p$ to $q$. In particular, since geodesics continuous functions, the convergence of geodesics holds uniformly.  
\end{lemma}

\begin{lemma}[\cite{Dauvergne-Sarkar-Virag-2022}, Lemma 3.3] \label{lem:overlap}
The following holds on a single event of full probability. Let $(p_n;q_n) = (x_n,s_n;y_n,u_n) \in \Rup \to (p;q) = (x,s;y,u) \in \Rup$, and let $g_n$ be any sequence of geodesics from $p_n$ to $q_n$. Suppose that either
\begin{enumerate} [label=\rm(\roman{*}), ref=\rm(\roman{*})]  \itemsep=3pt
    \item \label{uniqn} For all $n$, $g_n$ is the unique geodesic from $(x_n,s_n)$ to $(y_n,u_n)$ and $\graph g_n \to \graph g$ for some geodesic $g$ from $p$ to $q$, or
    \item \label{uniqueg} There is a unique geodesic $g$ from $p$ to $q$.
    \end{enumerate}
    Then, the \textbf{overlap}
    \[
    O(g_n,g) := \{t \in [s_n,u_n]\cap [s,u]: g_n(t) = g(t)\}    
    \]
    is an interval for all $n$ whose endpoints converge to $s$ and $u$. 
\end{lemma}
\begin{remark}
We note that condition~\ref{uniqn} is slightly different from that stated in~\cite{Dauvergne-Sarkar-Virag-2022}. There, it is assumed instead that $(x_n,s_n;y_n,u_n) \in \Q^4 \cap \Rup$ for all $n$. The only use of this requirement in the proof is to ensure that there is a unique geodesic from $(x_n,s_n)$ to $(y_n,u_n)$ for all $n$, so there is no additional justification needed for the statement we use here. 
\end{remark}

The next lemma gives a result on precompactness of semi-infinite geodesics. 
\begin{lemma}\cite[Lemma 14]{Bha24} \label{lem:SIGprecompact}
On an event of probability one, for all $\theta \in \R$, $\mbf p \in \R^2$, and any sequence of points $\mbf p_n \to \mbf p$, if $g_n$ is a sequence of $\dir$-directed semi-infinite geodesics rooted at $p_n$, then the sequence $\{g_n\}$ is precompact in the uniform-on-compact topology, and every subsequential limit is a $\dir$-directed geodesic rooted at $\mbf p$. 
\end{lemma}

\begin{lemma} \label{lem:bounded_maxes}
The following holds on an event of full probability. Let $K \subseteq \R$ be a compact set,   $\dir \in \R$ and $\sigg \in \{-,+\}$.  Then, there exists a random $Z = Z(\dir \sig,K)  \in (0,\infty)$ such that for all $x,s,t \in K$ with $s < t$ and $S \in \{L,R\}$, $|g_{(x,s)}^{\dir \sig,S}(t)| \le Z$.
\end{lemma}

         \begin{lemma}\cite[arXiv version, Lemma B.9]{Busa-Sepp-Sore-22a} \label{lem:KPZ_preserve_lim}
The following holds simultaneously for all initial data and all $t > s$ on the event of probability one from Lemma~\ref{lem:Landscape_global_bound}. Let $\h \in \UC$ (the space of upper semi-continuous functions $\R \to \R$) be initial data for the KPZ fixed point, sampled at time $s$. 
For $t > s$, let 
\[
h_t(y;\h) = \sup_{x \in \R}\{\h(x) + \Ll(x,s;y,t)\}.
\]
Then, simultaneously for all $t > s$,
\be \label{hliminfbd1}
\liminf_{x \to +\infty} \f{h_t(x;\h)}{x} \ge \liminf_{x \to + \infty} \f{\h(x)}{x},\qquad\text{and}\qquad \limsup_{x \to - \infty} \f{h_t(x;\h)}{x} \le \limsup_{x \to -\infty} \f{\h(x)}{x}.
\ee

Furthermore, assuming that $\h:\R \to \R$ is \textbf{continuous} and satisfies
\be \label{liminfsupfinite}
\liminf_{x \to \pm \infty} \f{\h(x)}{x} > -\infty \qquad\text{and}\qquad \limsup_{x \to \pm \infty} \f{\h(x)}{x} < +\infty,
\ee
then also
\be \label{hliminfbd2}
\limsup_{x \to + \infty} \f{h_t(x;\h)}{x} \le \limsup_{x \to +\infty} \f{\h(x)}{x},\qquad\text{and}\qquad \liminf_{x \to -\infty} \f{h_t(x;\h)}{x} \ge \liminf_{x \to - \infty} \f{\h(x)}{x}.
\ee

In particular, for \textbf{continuous} initial data $\h$ satisfying~\eqref{liminfsupfinite}, if either \rm{(}or both\rm{)} of the limits
$
\lim_{x \to \pm \infty} \f{\h(x)}{x}
$
exist \rm{(}potentially with different limits on each side\rm{)}, then for $t > s$,
\[
\lim_{x \to \pm \infty} \f{h_t(x;\h)}{x} = \lim_{x \to \pm \infty} \f{\h(x)}{x}.
\]
\end{lemma}

         To prove Lemma \ref{lem:geod_IC}, we need to prove the following:
         \begin{lemma}\label{lem:dir_from_global}
         Let $b \in \mathcal{F}_{\text{KPZ}}^\xi$.
         For each point $(x,t) \in \R^2$ and $s < t$, let $g_{(x,t)}^{b,L/R}(s)$ be the leftmost/rightmost maximizer of 
         \[
         z \mapsto b(z,s) + \Ll(z,s;x,t) \quad \text{over } z \in \R.
         \]
         Then, for each $s < t$, $g_{(x,t)}^{b,L/R}(s)$ is finite, and
         \be \label{eq:gxtlim}
         \lim_{s \to -\infty} \f{g_{(x,t)}^{b,L/R}(s)}{|s|} = \dir.
         \ee
         \end{lemma}

         To prove this lemma, we first prove the following. Here, we assume that $\dir > 0$, but we can obtain analogous lemmas for $\dir < 0$ and $\dir = 0$. 
         \begin{lemma}\label{lem:bxs_bd}
         Let $b \in \mathcal{F}_{\text{KPZ}}^\xi$. Then, on the event of probability one from Lemma \ref{lem:Landscape_global_bound}, for every $\ve > 0$, there exists a constant $C_{\ve} > 0$ (random, depending on the particular $b$ and $\dir$) so that, for all $x \in \R$ and $s < 0$,
         \[
         b(x,s)  \le C_{\ve} + 2\dir x -\dir^2  s + \ve |x| + \ve |\dir s| + C_{\text{DL}}|s|^{1/3} \log^2\Bigl(2\sqrt{2 x^2 + (1 + 2\dir^2) s^2} + 4\Bigr),
         \]
         where $C_{\text{DL}}$ is the constant from Lemma \ref{lem:Landscape_global_bound}. 
         \end{lemma}
         \begin{proof}
         First, we make the observation that for $x \in \R$ and $s < 0$,
         \begin{align*}
         b(x - \dir |s|,0) = \sup_{z \in \R} [b(z,s) + \Ll(z,s;x - \dir |s|,0)]\ge b(x,s) + \Ll(x,s;x - \dir |s|,0),
         \end{align*}
         and by rearranging,
         \be \label{eq:bxs_bd}
         b(x,s) \le b(x - \dir |s|,0) - \Ll(x,s;x - \dir |s|,0).
         \ee
         By the asymptotic assumption $\lim_{|z| \to +\infty} \f{b(z,0)}{z} = 2\dir$,  for each $\ve > 0$, exists a constant $C_{\ve} > 0$ so that $b(z,0) \le C_{\ve} + 2\dir z + \ve |z|$ for all $z \in \R$. By substituting this bound, along with the bound on $\Ll$ from Lemma \ref{lem:Landscape_global_bound} into \eqref{eq:bxs_bd}, we get
         \[
         b(x,s) \le C_{\ve} + 2\dir x -\dir^2 |s| + \ve |x| + \ve \dir|s| + C_{\text{DL}}|s|^{1/3} \log^2\Bigl(2\sqrt{x^2 + s^2 + (x - \dir|s|)^2} + 4\Bigr).
         \]
         The final bound comes from $(a+b)^2 \le 2a^2 + 2b^2$
         \end{proof}

         \begin{proof}[Proof of Lemma \ref{lem:dir_from_global}]
         For ease of notation, we prove the lemma for the point $(x,t) = (0,0)$. We note that the method of proof works  to show that the limits \eqref{eq:gxtlim} hold for all $(x,t) \in \R^2$, simultaneously on a single event of probability one, as the key inputs are the bounds from Lemmas \ref{lem:Landscape_global_bound} and \ref{lem:dir_from_global}, which are stated for all parameters simultaneously. We will also assume that $\dir > 0$, with the cases $\dir < 0$ and $\dir = 0$ following analogous proofs. The finiteness of maximizers comes directly from the bounds in Lemmas \ref{lem:bxs_bd} and \ref{lem:Landscape_global_bound} (while the bounds in Lemma \ref{lem:bxs_bd} are stated for $s < 0$, we can get them for arbitrary $s < t$ using the same proof because $b(x,t)$ for $t > 0$ satisfies the asymptotic slope $\lim_{|x| \to \infty}\f{b(x,t)}{x} = 2\dir$ by Lemma \ref{lem:KPZ_preserve_lim}.

         By assumption, for all $s < 0$,
         \be \label{eq:b00}
         b(0,0) = \sup_{z \in \R}[b(z,s) + \Ll(z,s;0,0)].
         \ee
         In particular, the right-hand side does not depend on $s$, and in other words, is $O(1)$ as $s \to -\infty$. It therefore suffices to prove that for all $\ve' > 0$,
         \be \label{eq:bL_lim}
         \limsup_{s \to -\infty} \sup_{z \notin [(\dir - \ve')|s|,(\dir + \ve')|s| ]}[b(z,s) + \Ll(z,s;0,0)] = -\infty.
         \ee
         Let $\ve > 0$ be given (this will be different than $\ve'$ used above). We will assume that $\ve$ is sufficiently small   so that $\ve < \dir \wedge 1$ and 
         \be \label{eps_condition}
         2\ve \dir + \ve^{5/4} - \ve^{1/2} < 0
         \ee
         (this will be different than the choice of $\ve'$ above).
         By combining Lemmas \ref{lem:Landscape_global_bound} and \ref{lem:bxs_bd}, we have that, for all $z \in \R$,
         \begin{align} \label{eq:bzs_bd}
             b(z,s) + \Ll(z,s;0,0) \le M_U(z,s) := C_{\ve} + 2\dir z - \dir^2 |s| + \ve |z| + \ve \dir |s| - \f{z^2}{|s|}  +  F(z,s),
         \end{align}
         where $C_{\ve} > 0$ is the constant from Lemma \ref{lem:bxs_bd}, and
         \begin{align} \label{F}
         F(z,s) := 2C_{\text{DL}}|s|^{1/3} \log^2\Bigl(2\sqrt{2 z^2 + (1 + 2\dir^2)s^2} + 4\Bigr).
         \end{align}
         Observe that 
         \begin{align} \label{eq:F_derivative}
         \partial_z F(z,s) = \f{8 C_{\text{DL}}|s|^{1/3} z\log\Bigl(2\sqrt{2 z^2 + (1 + 2\dir^2)s^2} + 4\Bigr)}{\Bigl(2\sqrt{2 z^2 + (1 + 2\dir^2)s^2} + 4\Bigr)\sqrt{2z^2 + (1+ 2\dir^2)s^2} } 
         \end{align}
        Note that, for all $s < 0$ and $z \in \R$, $|z| \Bigl(2z^2 + (1+ 2\dir^2)s^2\Bigr)^{-1/2} \le \f{1}{\sqrt 2} < 1$  and that $\log(x)/x$ is decreasing for $x > e$. Then, since $2\sqrt{2z^2 + (1+2\dir^2)s^2} + 4 \ge 2|s|$, we have, for all $s < 0$ and $z \in \R$, 
         \be \label{partial_z_bd}
        |\partial_z F(z,s)| \le \f{ 4 C_{\text{DL}} |s|^{1/3}\log(2|s|)   }{|s|}.
         \ee
         In particular, this bound does not depend on $z$ and decays to $0$ as $|s| \to \infty$. We also make the observation from \eqref{eq:F_derivative} that 
         \be \label{eq:Fp_sign}
            \partial_z F(z,s)  \begin{cases} > 0, &z > 0 \\
            < 0, &z < 0.
            \end{cases}
         \ee
Now, from \eqref{eq:bzs_bd}, we have 
\begin{align*}
&\quad \, \sup_{z \le 0}[b(z,s) +\Ll(z,s;0,0)] \\&\le C_{\ve} - \dir^2 |s| + \ve \dir |s| + \sup_{z \le 0} \Bigl[2\dir z -2 \ve z - \f{z^2}{|s|}\Bigr] + \sup_{z \le 0}[\ve z + F(z,s)] \\
&=C_{\ve} - \dir^2 |s| + \ve \dir |s| + (\dir - \ve)^2|s| + F(0,s) \\
&=C_{\ve} + (\ve^2 -\ve \dir)|s| + F(0,s) \to -\infty \text{ as }s \to -\infty,
\end{align*}
where the first equality follows by computing the first supremum and noting that the second supremum occurs at $z = 0$ (for sufficiently large $|s|$) since $\partial_z [\ve z + F(z,s)] = \ve + \partial_z F(z,s)$, which, by \eqref{partial_z_bd}, is positive for all $z$ when $|s|$ is sufficiently large. The convergence to $-\infty$ holds by the assumption $\ve < \dir$ and since $F(0,s) = o(s)$ by definition \eqref{F}. Thus, it now suffices to prove \eqref{eq:bL_lim} where we only consider the supremum for positive $z$. 

Consider the function $M_U(z,t)$ defined in \eqref{eq:bzs_bd}. $M_U(z,s)$ is a smooth function in $z$ and converges to $-\infty$  as $|z| \to \infty$ so that maximizers over $z > 0$ exist and occur when $\partial_z M_U(z,s) = 0$. Note that, for $z > 0$,
\be \label{eq:MU_prime}
\partial_z M_U(z,s) = 2\dir  + \ve - \f{2z}{|s|} + \partial_z F(z,s).
\ee
By \eqref{eq:Fp_sign} and \eqref{partial_z_bd}, we may choose $s < 0$ with $|s|$ sufficiently large so that $0 < \partial_z F(z,s) < \ve$ uniformly for all $z \ge 0$. Hence, any maximizer of $M_U(z,s)$ over $z \ge 0$ must lie in the interval $[(\dir + \f{\ve}{2})|s|,(\dir + \ve)|s|]$. Furthermore from \eqref{eq:MU_prime}, for such sufficiently large $|s|$, $z \mapsto M_U(z,s)$ is increasing for $0 \le z \le (\dir + \f{\ve}{2})|s|$ and decreasing for $z \ge (\dir + \ve)|s|$. Now, define the interval
\[
I_{s,\ve} := [(\dir - \ve^{1/4})|s|,(\dir + \ve^{1/4})|s|],
\]
and note that for sufficiently small $\ve$,
\[
[(\dir + \f{\ve}{2})|s|,(\dir + \ve)|s|] \subseteq I_{s,\ve}.
\]
As we have seen, the function $M_U(z,s)$ is increasing to the left of the interval $I_{s,\ve}$ and decreasing to the right, so to find the maximum of $M_U(z,t)$ outside of $I_{s,\ve}$, it suffices to compute the value at the endpoints of the interval. First, note that 
\begin{align*}
&\quad \, M_U((\dir - \ve^{1/4})|s|)  \\
&= C_\ve + 2\dir\bigl((\dir - \ve^{1/4})|s|\bigr) - \dir^2|s| + \ve (\dir - \ve^{1/4})|s| + \ve \dir |s| - \f{\bigl((\dir - \ve^{1/4})|s|\bigr)^2}{|s|} + F((\dir - \ve^{1/4})|s|,s) \\
&= C_\ve + \bigl(2\ve \dir - \ve^{5/4} - \ve^{1/2}\bigr)|s| + F((\dir - \ve^{1/4})|s|,s).
\end{align*}
On the other hand,
\begin{align*}
&\quad \, M_U\bigl((\dir + \ve^{1/4})|s|\bigr) \\
&= C_\ve + 2\dir\bigl((\dir + \ve^{1/4})|s|\bigr) - \dir^2|s| + \ve (\dir + \ve^{1/4})|s| + \ve \dir |s| - \f{\bigl((\dir + \ve^{1/4})|s|\bigr)^2}{|s|} + F((\dir + \ve^{1/4})|s|,s) \\
&= C_\ve + \bigl(2\ve \dir + \ve^{5/4} - \ve^{1/2}   \bigr)|s| + F((\dir + \ve^{1/4})|s|,s). 
\end{align*}
Hence, by the assumption \eqref{eps_condition} on $\ve$ and since $F((\dir \pm \ve^{1/4})|s|,s) = o(s)$ as $s \to -\infty$ \eqref{F}, we have that 
\[
\limsup_{s \to -\infty} \sup_{z \notin I_{s,\ve}} [b(z,s) + \Ll(z,s;0,0)] \le \limsup_{s \to -\infty} \sup_{z \notin I_{s,\ve}}M_U(z,s) = -\infty, 
\]
and this proves \eqref{eq:bL_lim}, upon setting $\ve' = \ve^{1/4}$. 
         \end{proof}

We now prove the following lemma, which is similar to \cite[Theorem 5.9]{Busa-Sepp-Sore-22a}. Here, we are constructing infinite geodesics from an arbitrary eternal solution $b$.

\begin{lemma}\label{lem:geodesics_from_b}
Let $b:\R^2 \to \R$ be a function satisfying, for all $s < t$, and $x \in \R$
\be \label{eq:global_appx}
b(x,t) = \sup_{z \in \R} [b(z,s) + \Ll(z,s;x,t)], 
\ee
and assume that for each $t > s$ and $x \in \R$, leftmost and rightmost maximizers in the equation above exist in $\R$. Define $g_{(x,t)}^{b,L/R}(s)$ to be the leftmost/rightmost maximizers, and set 
\[
g_{(x,t)}^{b,L/R}(t) = x.
\]
Then, the following hold
\begin{enumerate} [label=(\roman*), font=\normalfont]
\item \label{itm:geod} $g_{(x,t)}^{b,L}:(-\infty,t] \to \R$ \rm{(}resp. $g_{(x,t)}^{b,R}:(-\infty,t] \to \R$\rm{)} is a semi-infinite geodesic that is the leftmost (resp. rightmost) directed landscape geodesic between any two of its points. In particular, $g_{(x,t)}^{b,L}$ and $g_{(x,t)}^{b,R}$ are continuous functions $(-\infty,t] \to \R$.
\item \label{itm: consis} Let $s < r < t$ and $x \in \R$, and let $
w = g_{(x,t)}^{b,L}(r)$.
Then,
\[
g_{(x,t)}^{b,L}(s) = g_{(w,r)}^{b,L}(s),
\]
and the same holds when replacing $L$ with $R$. 
\item \label{itm: general}
More generally, let $t = s_0 > s_1 > \cdots$ be any decreasing sequence with $\lim_{n \to \infty} s_n = -\infty$. Set $g(t) = x$, and for each $i \ge 1$, let $g(s_i)$ be \textit{any} maximizer of $z \mapsto b(z,s_{i}) + \Ll(z,s_i; g(s_{i - 1}),s_{i - 1})$ over $z \in \R$. Then, pick \textit{any} point-to-point geodesic of $\Ll$ from $(g(s_{i}),s_{i})$ to $(g(s_{i-1}),s_{i-1})$, and for $s_{i} < s < s_{i-1}$, let $g(s)$ be the location of this geodesic at time $s$. Then, $g:(-\infty,t] \to \R$ is a semi-infinite geodesic. 
\end{enumerate}
\end{lemma}
\begin{remark} \label{rmk:b_global_solutions}
Given the eternal solution $b$, we call the geodesics constructed in Lemma \ref{lem:geodesics_from_b} \textbf{$b$-geodesics}.
\end{remark}

Before proving Lemma \ref{lem:geodesics_from_b} we prove the following intermediate lemma.
\begin{lemma}
\label{lem:L_b}
Assume that $b$ satisfies \eqref{eq:global_appx}. Then, for $s < t$ and $x,y \in \R$,
\be \label{eq:Lleb}
\Ll(x,s;y,t) \le b(y,t) - b(x,s),
\ee
and
\be \label{eq:iff_eq}
\text{Equality occurs if and only if } x \text{ maximizes } z \mapsto b(z,s) +\Ll(z,s;y,t) \text{ over } z \in \R.
\ee
\end{lemma}
\begin{proof}
Write
\[
b(y,t) = \sup_{z \in \R}[b(z,s) + \Ll(z,s;y,t)]  \ge b(x,s) + \Ll(x,s;y,t),
\]
with equality holding if and only if $x$ is a maximizer. The result now follows by subtracting $b(x,s)$ from both sides.  
\end{proof}

\begin{proof}[Proof of Lemma \ref{lem:geodesics_from_b}] 
\textbf{Item \ref{itm:geod}}: We prove the statement for the leftmost maximizers; the proof for rightmost maximizers is symmetric. For simplicity of notation, we fix $(x,t)$ and let $g = g_{(x,t)}^{b,L}$. To prove the first part, it suffices to show the following: for $s < r < t$, $g(r)$ is the leftmost maximizer of 
\be \label{zLfun}
z \mapsto \Ll(g(s),s;z,r) + \Ll(z,r;x,t).
\ee
This is sufficient because it implies $(g(r),r)$ lies on the leftmost geodesic from $(g(s),s))$ to $(x,t)$ by \cite[Lemma 13.2]{Directed_Landscape}, so the path $g$ is a semi-infinite geodesic and in particular, the leftmost geodesic between any two of its points. Continuity of the path then follows from continuity of the leftmost and rightmost directed landscape geodesics \cite[Lemma 13.2]{Directed_Landscape}. In particular, the continuity also holds at $s = t$. 

By \eqref{eq:iff_eq}, 
\[
\Ll(g(s),s;x,t) = b(x,t) - b(g(s),s)
\]
By the reverse triangle inequality for the directed landscape and \eqref{eq:Lleb} ,
\be \label{b_bd}
\begin{aligned}
b(x,t) - b(g(s),s) &= \Ll(g(s),s;x,t) \\
&\le \Ll(g(s),s;g(r),r) + \Ll(g(r),r;x,t) \\
&\le [b(g(r),r) - b(g(s),s)] + [b(x,t) - b(g(r),r)] \\ &= b(x,t) - b(g(s),s)
\end{aligned}
\ee
so the inequalities above are all equalities. Since the function in \eqref{zLfun} is bounded above by $\Ll(g(s),s;x,t)$ for all $z$, $g(r)$ is a maximizer. It remains to show that $g(r)$ is the leftmost maximizer. Let $z_r$ be the leftmost maximizer of the function in \eqref{zLfun}. Since $g(r)$ is also a maximizer $z_r \le g(r)$. Then,
\begin{align*}
b(x,t) - b(g(s),s) &= \Ll(g(s),s;x,t) \\
&= \Ll(g(s),s;z_r,r) + \Ll(z_r,r;x,t) \\
& \le [b(z_r,r) - b(g(s),s)] + [b(x,t) - b(z_r,r)],
\end{align*}
where in the last step, we applied \eqref{eq:Lleb} to both terms. But then the right-hand side equals the left side, so the two inequalities are equalities, and $\Ll(z_r,r;x,t) = b(x,t) - b(z_r,r)$. By \eqref{eq:iff_eq}, $z = z_r$ maximizes $z \mapsto b(z,r) + \Ll(z,r;x,t)$ over $z \in \R$. Since $g(r)$ is the leftmost such maximizer by definition, we have $g(r) \le z_r$, which completes the other inequality.  

\medskip \noindent \textbf{Item \ref{itm: consis}:}  From \eqref{b_bd}, we see that 
\[
\Ll(g(s),s;g(r),r) = b(g(r),r) - b(g(s),s).
\]
By \eqref{eq:iff_eq}, $g(s)$ maximizes 
\be \label{zr}
z \mapsto b(z,s) + \Ll(z,s;g(r),r) \text{ over } z \in \R.
\ee
We wish to show $g(s)$ is the leftmost maximizer of \eqref{zr}; i.e., $g(s) = g_{(w,r)}^L(s)$, where $w = g(r)$. For simplicity of notation, let $g'(s) = g_{(w,r)}^L(s)$. We have established that $g'(s) \le g(s)$. Observe that 
\begin{align*}
b(x,t) - b(g'(s),s) &= [b(g(r),r) - b(g'(s),s)] + [b(x,t) - b(g(r),r)]  \\
&= \Ll(g'(s),s;g(r),r) + \Ll(g(r),r;x,t)  \\
&\le \Ll(g'(s),s;x,t),
\end{align*}
where the second equality comes from\eqref{eq:iff_eq}, noting $g(r)$ maximizes $z \mapsto b(z,r) + \Ll(z,r;x,t)$  and $g'(s)$ maximizes $z \mapsto b(z,s) + \Ll(z,s;g(r),r)$ by definition. But we also know by \eqref{eq:Lleb} that $b(x,t) - b(g'(s),s) \ge \Ll(g'(s),s;x,t)$, so we in fact have equality of both sides. By another application of \eqref{eq:iff_eq}, $g'(s)$ is a maximizer of $z \mapsto b(z,s) + \Ll(z,s;x,t)$. Since $g(s)$ is the leftmost such maximizer, $g(s) \le g'(s)$, completing the proof. 

\medskip \noindent \textbf{Item \ref{itm: general}:} Lemma \ref{lem:L_b} implies that, for $i \ge 1$,
\[
\Ll\bigl(g(s_i),s_i; g(s_{i-1}),s_{i-1}\bigr) = b\bigl(g(s_{i-1}),s_{i-1}\bigr) -b\bigl(g(s_{i}),s_{i}\bigr). 
\]
Similar reasoning as in the similar items implies that 
\[
\Ll\bigl(g(s),s; g(u),u\bigr) = b\bigl(g(s),s\bigr) -b\bigl(g(u),u\bigr)
\]
for \textit{any} $s < u \le t$. As inequality holds in general by Lemma \ref{lem:L_b}, the path must be a semi-infinite geodesic by similar reasoning as in the previous items. 
\end{proof}

\begin{corollary}\label{cor:b_geod_mont}
Let $b:\R^2 \to \R$ satisfy the assumptions of Lemma \ref{lem:geodesics_from_b}, and let $g_{(x,t)}^{b,L/R}(s)$ be the associated leftmost/rightmost maximizers for $s < t$. Then, for all $s <t$ and $x < y$
\[
g_{(x,t)}^{b,L}(s) \le g_{(y,t)}^{b,L}(s)\qquad\text{and}\qquad g_{(x,t)}^{b,R}(s) \le g_{(y,t)}^{b,R}(s).
\]
\end{corollary}
\begin{proof}
We show the statement for $L$, with the statement for $R$ following an analogous proof. From Lemma \ref{lem:geodesics_from_b}, \ref{itm:geod}, $s \mapsto g_{(x,t)}^{b,R}(s)$ and $s \mapsto g_{(x,t)}^{b,L}(s)$ are continuous functions $(-\infty,t] \to \R$ satisfying 
\[
g_{(x,t)}^{b,L}(t) = x < y = g_{(y,t)}^{b,L}(t) = y.
\]
Hence, it suffices to show that, if $g_{(x,t)}^{b,L}(r) =  g_{(y,t)}^{b,L}(r)$ for some $r < t$, then $
g_{(x,t)}^{b,L}(s) =  g_{(y,t)}^{b,L}(s)
$
for all $s \le r$ (i.e., if the two paths ever intersect, they stay together after their first meeting point). This is a direct consequence
Lemma  \ref{lem:geodesics_from_b}\ref{itm: consis}: if 
\[
w := g_{(x,t)}^{b,L}(r) = g_{(y,t)}^{b,L}(r)
\]
for some $r < t$, then for $s < r$,
\[
g_{(x,t)}^{b,L}(s) = g_{(w,r)}^{b,L}(s) = g_{(y,t)}^{b,L}(s).  
\]
\end{proof}

\begin{corollary} \label{cor:b_geod_lim}
Assume the same as Corollary \ref{cor:b_geod_mont}. Then, for all $s < t$ and $x \in \R$,
\[
\lim_{y \nearrow x} g_{(y,t)}^{b,L}(s) = g_{(x,t)}^{b,L}(s)\qquad\text{and}\qquad\lim_{y \searrow x} g_{(y,t)}^{b,R}(s) = g_{(x,t)}^{b,R}(s). 
\]
\end{corollary}

\begin{proof}
    We prove the first limit, and the second follows symmetrically.  By the monotonicity in Lemma \ref{cor:b_geod_mont}, the limits exist, and
    \be \label{eq:gbL_bd}
    \lim_{y \nearrow x} g_{(y,t)}^{b,L}(s) \le g_{(x,t)}^{b,L}(s).
    \ee 
    Furthermore, this monotonicity guarantees that, for $y \in [x-1,x]$, all maximizers of the function
    \[
    z \mapsto b(z,s) + \Ll(z,s;y,t)
    \]
    lie in the common compact interval
    \[
    I := [g_{(x-1,t)}^{b,L}(s),g_{(x,t)}^{b,R}(s)].
    \]
    By continuity of the directed landscape (Lemma \ref{lem:Landscape_global_bound}), as $y \nearrow x$, the function
    \[
    z \mapsto b(z,s) +\Ll(z,s;y,t)
    \]
    converges uniformly on $I$ to $z \mapsto b(z,s) + \Ll(z,s;x,t)$. By Lemma \ref{lem:cmcf}, $\lim_{y \nearrow x} g_{(y,t)}^{b,L}(s) $ maximizes $z \mapsto b(z,s) + \Ll(z,s;x,t)$ over $z \in \R$. By \eqref{eq:gbL_bd}, the limit lies to the left of the leftmost such maximizer, so \eqref{eq:gbL_bd} is in fact an equality. 
\end{proof}

\begin{lemma} \label{lem:geod_unique}
For $\sigg \in \{-,+\}$, $S \in \{L,R\}$, $(x,t) \in \R^2$ and $s < t$, if $x_0 = g_{(x,t)}^{\dir \sig,S}(s)$, then 
\[
g_{(x_0,s)}^{\dir \sig,L} = g_{(x_0,s)}^{\dir \sig,R}.
\]
In other words, $(x_0,s) \notin \NU_{\dir\sig}$.
\end{lemma}
\begin{proof} On any point on the geodesic $(x_0,s)$ with $s < t$, there is a unique semi-infinite geodesic from that point since all $\dir \sig$ geodesics coalesce (Proposition \ref{prop:DL_all_coal}) and interior bubbles cannot form (Lemma \ref{lem:no_bubbles}).
\end{proof}

\bibliographystyle{plain}
\bibliography{arxivsubmission}
\end{document}